\documentclass[12pt,lot, lof]{puthesis}
\newcommand{\proquestmode}{}

\title{Complexity Aspects of Fundamental Questions in Polynomial Optimization}

\submitted{September 2020}  
\copyrightyear{2020}  
\author{Jeffrey Zhang}
\adviser{Amir Ali Ahmadi}  
\department{Operations Research and Financial Engineering}

\usepackage{graphicx}

\usepackage{verbatim}

\usepackage{multirow}
\usepackage{longtable}

\usepackage{booktabs}

\usepackage[margin=1in]{geometry}
\usepackage{graphicx}
\usepackage{amsmath}
\usepackage{amssymb}
\usepackage{amsthm}

\usepackage{algorithm}
\usepackage{algpseudocode}
\usepackage{enumerate}
\usepackage{diagbox}
\usepackage{color}

\ifdefined\printmode

\usepackage{url}

\else

\ifdefined\proquestmode

\usepackage{hyperref}
\hypersetup{bookmarksnumbered}

\makeatletter
\hypersetup{pdftitle=\@title,pdfauthor=\@author}
\makeatother

\else

\usepackage{hyperref}
\hypersetup{colorlinks,bookmarksnumbered}

\makeatletter
\hypersetup{pdftitle=\@title,pdfauthor=\@author}
\makeatother

\fi 
\fi 

\theoremstyle{plain}
\newtheorem{theorem}{Theorem}[section]
\newtheorem{lem}[theorem]{Lemma}
\newtheorem{prop}[theorem]{Proposition}
\newtheorem{cor}[theorem]{Corollary}
\newtheorem{defn}[theorem]{Definition}
\theoremstyle{definition}
\theoremstyle{remark}
\newtheorem{example}{Example}[section]
\newtheorem{remark}{Remark}[section]

\def\0{{\bf 0}}
\def\1{{\bf 1}}
\def \M{ \mathcal{M}}
\def \nulls{ \mathcal{N}}
\def \cols{ \mathcal{C}}
\def \bmat{\left[\begin{matrix}}
	\def \emat{\end{matrix}\right]}
	\def \bvec{\left(\begin{matrix}}
\def \evec{\end{matrix}\right)}
\def \xyvec{\left[\begin{matrix}x\\y\end{matrix}\right]}
\def \xy1vec{\left[\begin{matrix}x\\y\\1\end{matrix}\right]}
\def \QED{\begin{flushright}\Halmos\end{flushright}\end{proof}}
\def \defeq{\mathrel{\mathop{:}}=}
\def \Tr{\mathrm{Tr}}
\def \xbar{\bar{x}}
\def \ybar{\bar{y}}
\def \gx{\grad p(x)}
\def \Hx{\Hess p(x)}
\def \gxb{\grad p(\xbar)}
\def \Hxb{\Hess p(\xbar)}
\def \gp3d{\grad p_3(d)}
\def \R{\mathbb{R}}
\def \Rn{\R^n}
\def \beq{\begin{equation}}
\def \eeq{\end{equation}}
\def \otm{\{1, \ldots, m\}}

\def \Snn{ \mathbb{S}^{n \times n}}
\def \grad{\nabla}

\def \Hess{\nabla^2}
\def \beq{\begin{equation}}
\def \eeq{\end{equation}}
\def \baeq{\begin{equation*}\begin{aligned}}
\def \eaeq{\end{aligned}\end{equation*}}
\newcommand{\baeql}[1]{\begin{equation}\label{#1}\begin{aligned}}
\def \eaeql{\end{aligned}\end{equation}}
\def \otn{\{1, \ldots, n\}}

\newcommand{\ceil}[1]{\lceil #1 \rceil}

\ifodd 0
\renewcommand{\maketitlepage}{}
\renewcommand*{\makecopyrightpage}{}
\renewcommand*{\makeabstract}{}

\else

\abstract{\begin{small}In this thesis, we settle the computational complexity of some fundamental questions in polynomial optimization. These include the questions of (i) finding a local minimum, (ii) testing local minimality of a candidate point, and (iii) deciding attainment of the optimal value. Our results characterize the complexity of these three questions for all degrees of the defining polynomials left open by prior literature.

Regarding questions (i) and (ii), we show that unless P=NP, there cannot be a polynomial-time algorithm that finds a point within Euclidean distance $c^n$ (for any constant $c$) of a local minimum of an $n$-variate quadratic function over a polytope. This result answers a question of Pardalos and Vavasis that appeared in 1992 on a list of seven open problems in complexity theory for numerical optimization. By contrast, through leveraging techniques from algebraic geometry, we show that a local minimum of a cubic polynomial can be found efficiently by semidefinite programming. Interestingly, we prove that second-order points of cubic polynomials admit an efficient semidefinite representation, even though their critical points are NP-hard to find. We also give an efficiently-checkable necessary and sufficient condition for local minimality of a point for a cubic polynomial.

Regarding question (iii), we prove that testing whether a quadratically constrained quadratic program with a finite optimal value has an optimal solution is NP-hard. We also show that testing coercivity of the objective function, compactness of the feasible set, and the Archimedean property associated with the description of the feasible set are all NP-hard. The latter property is the assumption on which convergence of the Lasserre hierarchy relies. We also give a new characterization of coercive polynomials that lends itself to a hierarchy of semidefinite programs.

In our final chapter, we present a semidefinite programming relaxation for the problem of finding approximate Nash equilibria in bimatrix games. We show that for a symmetric game, a $1/3$-Nash equilibrium can be efficiently recovered from any rank-2 solution to this relaxation. We also propose semidefinite programming relaxations for NP-hard problems related to Nash equilibria, such as that of finding the highest achievable welfare under any Nash equilibrium.\end{small}}

\acknowledgements{I would like to first and foremost thank my adviser, Amir Ali Ahmadi, for making all of this possible. Thank you for being an amazing adviser, and for teaching me everything you have. Thanks to you, I feel ready to be a professor and a researcher. Thank you for all the long hours, for not letting me take the easy road, and for all the feedback on every little thing. Thank you for giving me every opportunity to grow as a student, as a future teacher, and as a person. Thank you for broadening my horizons, both academically and literally. I know that so much of where I am now is due to you, and I know that all the lessons will stick with me for a long future to come.

I also want to thank the other professors who helped me along my way; Bob Vanderbei, for being there for me from the very beginning in so many ways, and I hope my life can be as interesting as yours one day. Nicolas Boumal and Anirudha Majumdar for also serving on my thesis committee. Jianqing Fan, for introducing me to ORFE which has become a home to me. Matt Weinberg, for a most interesting course and new research ideas. I would also like to thank all the members of ORFE who have made my experience what it was. Thank you Kim, Tara, Michael, Tiffany, and Melissa for putting up with me and answering questions when I had them.

I would also like to thank all the friends I made along the way; Bachir and Cemil, who made working with Amirali all the more interesting, Galen and Dan, for all the hours spent in the STWG, Thomas, Elahe, Yiqiao, Sinem, Yair, Kaizheng, Suqi, for making the time spent more interesting. A special thanks to Georgina, who was a guiding light in the early years when I needed one. Thank you also to David, Chris, Alan, Dan, Andre, Josh, Connor, Demi, and Jenny for keeping me grounded and always different perspectives on life.

Last but not least I want to thank my family. My brother Leon for all the motivation, ideas, and discussion, and being with me my whole life. My father Heping, for setting me in the path to being a professor, keeping me motivated, and all the advice over the years. And finally my mother Julan, without whose love and passion I may not be here today.
}
\dedication{To my mother, Julan, for all her loving care.}

\fi  

\begin{document}
	
\makefrontmatter
\chapter{Introduction}\label{Chap: Introduction}

In this thesis, we concern ourselves with \emph{polynomial optimization problems (POPs)}, i.e., problems of the type
\beq\label{Eq: Polynomial Optimization Problem}
\begin{aligned}
& \underset{x \in \Rn}{\inf}
& & p(x) \\
& \text{subject to}
&& q_i(x) \ge 0, \forall i \in \otm,\\
\end{aligned}
\eeq
where $p,q_1, \ldots q_m$ are polynomial functions. In Chapters~\ref{Chap: Polynomials}-\ref{Chap: Attainment}, we address the complexity of the following questions for a problem in the form of (\ref{Eq: Polynomial Optimization Problem}):
\begin{enumerate}[{\bf Q1:}]
	\item Is a given a point $x$ a local minimum of (\ref{Eq: Polynomial Optimization Problem})?
	\item Does (\ref{Eq: Polynomial Optimization Problem}) have a local minimum (and if so, can one be found efficiently)?
	\item If (\ref{Eq: Polynomial Optimization Problem}) has a finite optimal value, does it have an optimal solution?
\end{enumerate}
Precise definitions and in-depth study of related problems can be found in their respective chapters.

POPs have wide modeling capabilities and arise ubiquitously in applications, either as exact models of objective functions or as approximations thereof. Perhaps the most well-known special case of POPs is linear programming, but POPs have much richer expressive power. For example, in full generality any decision problem in NP, a class of yes-no decision problems with the property that any yes answer can be certified efficiently, can be posed as a POP.\footnote{This follows straightforwardly from the fact that NP-complete problems can be formulated as POPs; see for example Section~\ref{POLY Sec: NP-hardness results}, Section~\ref{QPs Sec: Main Result}, or Section~\ref{AOS Sec: AOS}.} An example of a POP we will see in this thesis is the search for Nash equilibria in bimatrix games, but other examples arise from optimal power flow \cite{huneault1991survey}, the quadratic assignment problem \cite{loiola2007survey}, robotics and control \cite{ahmadi2016some}, and statistics and machine learning \cite{tibshirani1996regression,suykens1999least}. Even when the goal is not to minimize a polynomial function, optimization algorithms that involve minimizing Taylor expansions of functions solve POPs as a subroutine.

With such modeling power comes the price of computational intractability, and unfortunately POPs become intractable to solve even when the degrees of the defining polynomials are low. The study of local minima in Q1 and Q2 is in large part motivated by this intractability of finding global minima in polynomial optimization. It is common for optimization algorithms to instead search for local minima, with the hope that local minima are easier to find. This notion is not new; for example~\cite{more1990solution} provides an explicit example of a class of POP where global minima are hard to find but local minima are not. There has also been renewed interest in finding local minima due to the growth of machine learning applications, where local minima of highly nonconvex functions are sought for in practice with simple first-order methods like gradient descent. Our goal in Chapters~\ref{Chap: Polynomials} and \ref{Chap: QPs} of this thesis is to more formally understand the complexity of finding local minima (as well as some related questions).

We point out that a priori there are no complexity implications between questions Q1 and Q2 stated above. For example, there is no reason to expect that an efficient algorithm for verifying that a given point is a local minimum would provide any guidance on how one decide if a problem has a local minimum (this is in essence the dilemma of the famous question ``P = NP?''). Conversely, even if local minimality of a given point cannot always be efficiently certified, that does not rule out the existence of algorithms that can efficiently find particular local minima that are easy to certify; see e.g. Question 3 of~\cite{pardalos1992open}. Thus the complexities of these two questions need to be studied separately.

One of the first hardness results on local minima in the literature is due to Murty and Kabadi~\cite{murty1987some}, who show that the problems of deciding whether a given point is a local minimum of a quadratic program or a local minimum of a quartic polynomial are NP-hard (see Table \ref{Table: Complexity Checking}). These were two of the first results indicating that local minima are not necessarily ``easier'' than global minima. In regards to Q1, along with prior classical results, they only leave open the complexity of deciding whether a given point is a local minimum of a cubic polynomial. We show that this last case is polynomial-time solvable in Section~\ref{POLY Sec: Local Min} of this thesis.

\begin{table}[H]
	\centering
	\begin{tabular}{c|c|c|c}
		{\bf \diagbox{$deg(p)$}{$\max deg(q_i)$}} &  $\emptyset$ & 1 & $\ge$2 \\\hline
		1 & P & P & NP-hard\\
		  & & & \\\hline
		2 & P & NP-hard & NP-hard\\
		& & \cite{murty1987some},\cite{pardalos1988checking} & \\\hline
		3 & P & NP-hard & NP-hard \\
		& (Theorem \ref{POLY Thm: Checking Min Poly Time}) & & \\\hline
		$\ge$4 & NP-hard & NP-hard & NP-hard\\
		& \cite{murty1987some} & & \\
		\end{tabular}
		\caption{Complexity of deciding whether a point is a local minimum of a POP, based on the degree of the objective $p$ and the maximum degree of any constraint function $q_i$. Entries without a reference are either classical or implied by a stronger hardness result in the table.\protect\footnotemark}
		\label{Table: Complexity Checking}
\end{table}
\footnotetext{The NP-hardness of the case of linear objective and quadratic constraints is implied by the NP-hardness of the quadratic programming case. Indeed, minimizing a quadratic function $p(x)$ over a polyhedron $Ax = b$ can be reduced to minimizing a variable $\gamma$ over the set $\{(x,\gamma)\ |\ Ax = b,, p(x) = \gamma\}$.}

We next comment on Q2, the question of deciding whether a POP has a local minimum. The complexity of Q2 based on the degrees of the defining polynomials is presented in Table~\ref{Table: Complexity Existence}. This problem has not been as extensively studied in prior literature, though it is more closely related to the problem of searching for a local minimum as compared to Q1. In fact, for the cases labeled ``P'' in Table~\ref{Table: Complexity Existence}, some natural algorithms that find local minima implicitly check that they exist; see Section~\ref{POLY Sec: Introduction} for details in the unconstrained case. In Chapter~\ref{Chap: Polynomials} of this thesis, we show that this is also the case for cubic polynomials. In particular, the problem of deciding if a cubic polynomial has a local minimum (and then finding one) can be done by solving a polynomial number of polynomially-sized semidefinite programs (SDPs), hence the label in Table~\ref{Table: Complexity Existence}. By contrast, we show that Q2 is intractable in the same cases that Q1 is.

\begin{table}[H]
	\centering
	\begin{tabular}{c|c|c|c}
		{\bf \diagbox{$deg(p)$}{$\max deg(q_i)$}} & $\emptyset$ & 1 & $\ge$2 \\\hline
		1 & P & P & NP-hard\\
		& & & \\\hline
		2 & P & NP-hard & NP-hard\\
		& & (Theorem~\ref{QPs Thm: LM QP NP hard}) & \\\hline
		3 & SDP & NP-hard & NP-hard \\
		& (Algorithm \ref{POLY Alg: Complete Cubic SDP}) & & \\\hline
		$\ge$4 & NP-hard & NP-hard & NP-hard\\
		& (Theorem \ref{QPs Thm: LM Degree 4 NP hard}) & & \\
	\end{tabular}
	\caption{Complexity of deciding whether a POP has a local minimum, based on the degree of the objective $p$ and the maximum degree of any constraint function $q_i$. Entries without a reference are either classical or implied by a stronger hardness result in the table.}
	\label{Table: Complexity Existence}
\end{table}

In many settings the existence of a local minimum in a POP is guaranteed; the focus is then on finding a local minimum without needing to consider whether one exists. One way this can be the case is when the feasible set is bounded, as is commonly the case in both applications and POPs encoding classical combinatorial problems. For the specific case of quadratic programs with bounded feasible sets, the question of the complexity of finding a local minimum has appeared explicitly in the literature in~\cite{pardalos1992open}:

\begin{quote}
    \emph{``What is the complexity of finding even a local minimizer for nonconvex quadratic programming, assuming the feasible set is compact? Murty and Kabadi (1987) and Pardalos and Schnitger (1988) have shown that it is NP-hard to test whether a given point for such a problem is a local minimizer, but that does not rule out the possibility that another point can be found that is easily verified as a local minimzer.''}
\end{quote}
We settle this question in Chapter~\ref{Chap: QPs}, where we show that unless P=NP, no polynomial-time algorithm can even find a point within a Euclidean distance of $c^n$ (for any constant $c \ge 0$) of a local minimum.

The final complexity question we study for POPs is testing the existence of an optimal solution when the optimal value of the problem is finite. This problem is in part motivated by a question of Nie, Dummel, and Sturmfels \cite{nie2006minimizing}, who provide an algorithm for solving an unconstrained POP under the assumption that the optimal value is attained. The authors remark
\begin{quote}
	\emph{``This assumption is non-trivial, and we do not address (the important and difficult) question of how to verify that a given polynomial $f(x)$ has this property.''}
\end{quote}
Most prior work on this question has focused on identifying cases where the existence of optimal solutions is guaranteed. Perhaps the most classical example is the Bolzano-Weirstrauss theorem for continuous functions over compact sets. Such theorems in the setting of POPs are commonly referred to as Frank-Wolfe type theorems, due to the eponymous authors' result that quadratic programs attain their optimal value when that value is finite~\cite{frank1956algorithm}. For POPs, this result was extended to cubic programs in \cite{andronov1982solvability}, where it is shown that finite optimal values are always attained if the objective is at most cubic, and the constraints are affine. However, we will show that these are the only cases where this is true, and that deciding whether a POP attains its optimal value is NP-hard in the remaining cases (see Table \ref{Table: Complexity Attainment}).

\begin{table}[H]
	\centering
	\begin{tabular}{c|c|c|c}
		{\bf \diagbox{$deg(p)$}{$\max deg(q_i)$}} & $\emptyset$ & 1 & $\ge$2 \\\hline
		1 & YES & YES & NP-hard\\
		& & & (Theorem \ref{AOS Thm: AOS NP-hard o1c2}) \\\hline
		2 & YES & YES & NP-hard\\
		& & \cite{frank1956algorithm} & \\\hline
		3 & YES & YES & NP-hard \\
		& & \cite{andronov1982solvability} & \\\hline
		$\ge$4 & NP-hard & NP-hard & NP-hard\\
		& (Theorem \ref{AOS Thm: AOS NP-hard o4c0}) & & \\
	\end{tabular}
	\caption{Complexity of deciding whether a POP with a finite optimal value has an optimal solution, based on the degree of the objective $p$ and the maximum degree of any constraint function $q_i$. Entries without a reference are either classical or implied by a stronger hardness result in the table. Note that whenever the degree of the objective is at most three, and the feasible set is a polyhedron, there is no algorithm required; the answer is simply `yes'.}
	\label{Table: Complexity Attainment}
\end{table}

In this thesis, we also study an application of POPs to the problem of finding Nash equilibria in bimatrix games. Nash equilibria are a fundamental concept in economics, but they also arise frequently in disciplines such as biology and finance. Finding a Nash equilibrium however, is computationally intractable~\cite{daskalakis2009complexity}. In the final chapter of this thesis, we formulate the problem of finding Nash equilibria in bimatrix games as a POP, and explore semidefinite programming relaxations for finding approximate Nash equilibria. We also apply these techinques to certain decision problems related to Nash equilibria.

\section{Preliminaries}\label{INTRO Sec: Prelims}

In this thesis, we will study the complexity of Q1-Q3 in the Turing model of computation. Since polynomial functions of a given degree are finitely parameterized, they allow for a convenient study of complexity questions in this setting. The size of a given instance is determined by the number of bits required to write down the coefficients of the polynomial (and, in the case of Q1, the entries of the point $x$), which are all taken to be rational numbers. For the purposes of analyzing the complexity of these three questions for POPs, we consider the relevant setting in applications where the degrees of any polynomials are fixed and the number of variables in the POP increases. We are interested in the existence or non-existence of efficient algorithms for solving Q1-Q3 in this setting, as established theory (e.g. quantifier elimination theory~\cite{tarski1951decision, seidenberg1954new}) already yields exponential-time algorithms for them. We also point out that our intractability results are in the strong sense, meaning that the problems remain NP-hard even if the bitsize of all numerical data is $O(\log(n))$, where $n$ is the number of variables in the problem. Unless P = NP, not even a pseudo-polynomial time algorithm (an algorithm whose running time is polynomial in the magnitude of the numerical data, but not their bitsize) can exist that solves a strongly NP-hard problem on all instances. This is in contrast to problems such as knapsack~\cite{garey2002computers}, which can be solved tractably, e.g. by dynamic programming, when the size of the numerical data is ``small''. See \cite{garey2002computers} or \cite[Section 2]{ahmadi2019complexity} for more details on the distinction between weakly and strongly NP-hard problems.

A prevalent tool in this thesis will be sum of squares programming. More details will be provided in each chapter as they are used, but we provide an introduction here. We say that a polynomial $p$ is a \emph{sum of squares} (sos) if there exist polynomials $q_1,\ldots,q_r$ such that $p= \sum_{i=1}^r q_i^2$. This is an algebraic sufficient condition for global nonnegativity of a polynomial which is in general not necessary~\cite{hilbert1888darstellung}. While deciding nonnegativity of a polynomial is in general NP-hard (e.g., as a consequence of~\cite{murty1987some}), deciding whether a polynomial is sos can be done via semidefinite programming. This is because a polynomial $p$ of degree $2d$ in $n$ variables is a sum of squares if and only if there exists an ${n+d \choose d} \times {n+d \choose d}$ positive semidefinite matrix $Q$ satisfying the identity
\beq \label{Eq: Sos Gram} p(x) = z(x)^TQz(x),\eeq
where $z(x)$ denotes the vector of all monomials in $x$ of degree less than or equal to $d$. Note that because of this equivalence, one can also impose the constraint that a polynomial $p$ with unknown coefficients is sos by semidefinite programming (see, e.g., \cite{parrilo2003semidefinite}). Given a rank-$r$ psd matrix $Q$ that satisfies (\ref{Eq: Sos Gram}), one can write $Q$ as $\sum_{i=1}^r v_iv_i^T$ (e.g. via a Cholesky factorization), and obtain an sos decomposition of $p$ as $p=\sum_{i=1}^r (v_i^Tz(x))^2$.

Sum of squares polynomials have gained interest in the field of polynomial optimization because the problem of finding the infimum of some polynomial $p$ can be straightforwardly reformulated into the problem
\beq
\begin{aligned}
	& \underset{\gamma \in \R}{\sup}
	&&\gamma\\
	&\mbox{subject to}
	&& p(x) - \gamma \mbox{ is a nonnegative polynomial.}
\end{aligned}
\eeq
This formulation has the interpretation of finding the largest lower bound on a polynomial. Unfortunately, this problem cannot be efficiently solved, since testing whether a polynomial of degree at least 4 is nonnegative is NP-hard and thus the constraint ``$p(x)-\gamma$ is a nonnegative polynomial'' cannot be imposed in a tractable fashion. Therefore, to obtain what is known as a ``sum of squares relaxation'', this constraint is replaced by a sum of squares constraint, which can be imposed tractably:
\beq\label{Eq: Sum of Squares Relaxation}
\begin{aligned}
	& \underset{\gamma \in \R}{\sup}
	&&\gamma\\
	&\mbox{subject to}
	&& p(x) - \gamma \mbox{ is sos.}
\end{aligned}
\eeq
Since the constraint ``$p(x)$ is sos'' is a semidefinite constraint, (\ref{Eq: Sum of Squares Relaxation}) is an SDP. As any sos polynomial is nonnegative, this gives a lower bound on the infimum of $p(x)$.

There are many extensions for constrained problems, with one of the more well-known being based on Putinar's Positivstellensatz~\cite{putinar1993positive}. Putinar's Positivstellensatz states that if the so-called Archimedean property is satisfied, then if a polynomial $p$ is positive for all $x$ in the set $\{x \in \Rn\ |\ q_i(x) \ge 0, \forall i = 1, \ldots, m\}$, there exist sos polynomials $\sigma_0, \ldots, \sigma_m$ such that
\begin{equation}\label{Eq: Putinar's Positivstellensatz}
p(x) = \sigma_0(x) + \sum_{i=1}^m \sigma_i(x)q_i(x).
\end{equation}
The Archimedean property requires existence of a scalar $R$ such that the polynomial ${R - \sum_{i=1}^n x_i^2}$ belongs to the \emph{quadratic module} of $q_1,\ldots,q_m$, i.e., the set of polynomials that can be written as $$\tau_0(x)+\sum_{i=1}^m \tau_i(x) q_i(x),$$ where $\tau_0,\ldots,\tau_m$ are sum of squares polynomials. This is an algebraic notion of the compactness of the set $\{x \in \Rn \ |\ q_i(x) \ge 0, \forall i = 1, \ldots, m\}$, which is stronger than the geometric notion. Similar to the construction of (\ref{Eq: Sum of Squares Relaxation}), for any positive integer $d$ the following problem gives a lower bound on the optimal value of (\ref{Eq: Polynomial Optimization Problem}):
\beq\label{Eq: Putinar Relaxation}
\begin{aligned}
	\gamma^d \defeq & \underset{\gamma \in \R, \sigma_i}{\sup}
	&&\gamma\\
	&\mbox{subject to}
	&& p(x) - \gamma = \sigma_0(x) + \sum_{i=1}^n \sigma_i(x)q_i(x),\\
	&&& \sigma_i \mbox{ is an sos polynomial of degree at most 2}d, \forall i = 0, \ldots, m.
\end{aligned}
\eeq
Taking $d = 1, 2, \ldots$ defines a sequence of problems referred to as the \emph{Lasserre hierarchy}~\cite{lasserre2001global}. There are two primary properties of the Lasserre hierarchy which are of interest. The first is that for any fixed $d$, the $d$-th problem in this sequence is an SDP of size polynomial in $n$. The second is that under the Archimedean property, $\underset{d \to \infty}{\lim} \gamma^d = p^*$, where $p^*$ is the optimal value of the POP in (\ref{Eq: Polynomial Optimization Problem}). While this is a powerful property of the hierarchy, and in practice the hierarchy is often exact at low levels, in general the level of the hierarchy needed can be arbitrarily high, and the semidefinite programs involved become expensive very quickly. Additionally, the Archimedean property is NP-hard to check, as we show in Section~\ref{AOS Sec: Sufficient Conditions} of this thesis.

\section{Outline of this thesis}
{\bf Complexity Results in Polynomial Optimization} Chapters~\ref{Chap: Polynomials}-\ref{Chap: Attainment} this thesis focus on fundamental problems in polynomial optimization from an algorithmic perspective. Chapter~\ref{Chap: Polynomials} concerns itself with the complexity of local minima and related notions in unconstrained polynomial optimization. In particular, it establishes that local minima and second-order points of cubic polynomials can be found by solving polynomially many semidefinite programs of polynomial size. In the negative direction, it establishes that the problems of deciding if quartic polynomials have second-order points and whether cubic polynomials have critical points are NP-hard. Notably, our approach for finding local minima of cubic polynomials relies on circumventing the search for critical points. Chapter~\ref{Chap: QPs} establishes our intractability results related to local minima in quadratic programming. In particular, we show that unless P=NP, no polynomial-time algorithm can find points within distance $c^n$, for any constant $c \ge 0$, of a local minimum in an $n$-variate quadratic program with a bounded feasible set. Chapter \ref{Chap: Attainment} focuses on the problem of testing attainment of optimal values, and settles the complexity of deciding whether a POP with a finite optimal value has an optimal solution.

{\bf Semidefinite Relaxations for Bimatrix Games}
In the second part of this thesis, we explore an application of semidefinite programming to game theory, in particular semidefinite programming relaxations for finding Nash equilibria in bimatrix games. We show that for a symmetric game, a symmetric $1/3$-Nash equilibrium can be efficiently recovered from any rank-2 solution to this relaxation. We also present semidefinite programming relaxations for NP-hard decision problems related to Nash equilibria, such as that of finding the highest achievable welfare under any Nash equilibrium.

\section{Related Publications}
The material in this thesis is based on the following work:

{\bf Chapter \ref{Chap: Polynomials}}: Ahmadi, A. A., and Zhang, J. Complexity Aspects of Local Minima and Related Notions. Available at arXiv:2008.06148.

{\bf Chapter \ref{Chap: QPs}}: Ahmadi, A. A., and Zhang, J. On the Complexity of Finding a Local Minimizer of a Quadratic Function over a Polytope. Available at arXiv:2008.05558.

{\bf Chapter \ref{Chap: Attainment}}: Ahmadi, A. A., and Zhang, J. On the Complexity of Testing Attainment of the Optimal Value in Nonlinear Optimization. \emph{Mathematical Programming}, 2019. Available at arXiv:1803.07683.

{\bf Chapter \ref{Chap: Nash}}: Ahmadi, A. A., and Zhang, J. Semidefinite Programming and Nash Equilibria in Bimatrix Games. To appear in \emph{INFORMS Journal on Computing}. Available at arXiv:1706.08550.

\chapter{On Local Minima and Related Notions in Unconstrained Polynomial Optimization}\label{Chap: Polynomials}

\section{Introduction}\label{POLY Sec: Introduction}

In this chapter of the thesis, we address the complexity of questions Q1 and Q2 from Chapter~\ref{Chap: Introduction}, but more generally for the following four types of points for a given polynomial $p:\Rn \to \R$:
\begin{enumerate}[(i)]
    \item a \emph{critical point}, i.e., a point $x$ where the gradient $\grad p(x)$ is zero,
    \item a \emph{second-order point}, i.e., a point $x$ where $\grad p(x) = 0$ and the Hessian $\Hess p(x)$ is positive semidefinite (psd), i.e. has nonnegative eigenvalues,
    \item a \emph{local minimum}, i.e., a point $x$ for which there exists a scalar $\epsilon > 0$ such that ${p(x) \le p(y)}$ for all $y$ with $\|y - x\| \le \epsilon$,
    \item a \emph{strict local minimum}, i.e., a point $x$ for which there exists a scalar $\epsilon > 0$ such that $p(x)<p(y)$ for all $y \ne x$ with $\|y - x\| \le \epsilon$.
\end{enumerate}
We note the following straightforward implications between (i)-(iv):
\begin{center}
    strict local minimum $\Rightarrow$ local minimum $\Rightarrow$ second-order point $\Rightarrow$ critical point.
\end{center}

Notions (i)-(iv) appear ubiquitously in nonconvex continuous optimization as surrogates for global minima, since it is well understood that finding a global minimum of $f$ is in general an intractable problem. With regard to each of these four notions, we restate Q1 and Q2:

\begin{enumerate}[{\bf Q1*:}]
    \item Given a polynomial $p: \Rn \to \R$ and a point $x \in \Rn$, is $x$ of a given type (i)-(iv)?
    \item Given a polynomial $p: \Rn \to \R$, does $p$ have a point of a given type (i)-(iv) (and if so, can one be found efficiently)?
\end{enumerate}

As discussed in Chapter~\ref{Chap: Introduction}, a priori there are no complexity implications between these two questions.

Let us first comment on the complexity of Q1* and Q2* for some simple and classical cases. For Q1*, checking whether a given point is a critical point of a polynomial function (of any degree) can trivially be done in polynomial time simply by evaluating the gradient at that point. To check that a given point is a second-order point, one can additionally compute the Hessian matrix at that point and check that it is positive semidefinite. This can be done in polynomial time, e.g., by performing Gaussian pivot steps along the main diagonal of the matrix~\cite[Section 1.3.1]{murty1988linear} or by computing its characteristic polynomial and checking that the signs of its coefficients alternate~\cite[p. 403]{horn2012matrix}. Since for affine or quadratic polynomials, any second-order point is a local minimum, the only remaining case of Q1* for them is that of strict local minima. Affine polynomials never have strict local minima, making the question uninteresting. A point is a strict local minimum of a quadratic polynomial if and only if it is a critical point and the associated Hessian matrix is positive definite (pd), i.e., has positive eigenvalues. The latter property can be checked in polynomial time, for example by computing the leading principal minors of the Hessian and checking that they are all positive. As for Q2*, the affine case is again uninteresting since there is a critical point (which will also be a second-order point and a local minimum) if and only if the coefficients of all degree-one monomials are zero. For quadratic polynomials, since the entries of the gradient are affine, searching for critical points can be done in polynomial time by solving a linear system. A candidate critcal point will be a second-order point (and a local minimum) if and only if the Hessian is psd, and a strict local minimum if and only if the Hessian is pd.

Other than the aforementioned cases, the only prior result in the literature that we are aware of is due to Murty and Kabadi~\cite{murty1987some}, which settles the complexity of Q1* for degree-4 polynomials. Our contribution in this chapter is to settle the complexity of the remaining cases for both Q1* and Q2*. A summary of the results is presented in Tables~\ref{POLY Table: Complexity Checking} and \ref{POLY Table: Complexity Existence}. Entries denoted by ``P'' indicate that the problem can be solved in polynomial time. The notation ``SDP'' indicates that the problem of interest can be reduced to solving either one or polynomially-many semidefinite programs (SDP) whose sizes are polynomial in the size of the input. (In fact, the reduction also goes in the other direction for second-order points and local minima; see Theorems~\ref{POLY Thm: SOP SDPF} and \ref{POLY Thm: LM SDPSF}.) Finally, recall that a strong NP-hardness result implies that the problem of interest remains NP-hard even if the size (i.e. bit length) of the coefficients of the polynomial is $O(\log(n))$, where $n$ is the number of variables. Therefore, unless P=NP, even a pseudo-polynomial time algorithm (i.e., an algorithm whose running time is polynomial in the magnitude of the coefficients, but not necessarily their bit length) cannot exist for the indicated problems in these tables.

\begin{table}[H]
    \centering
    \begin{tabular}{c|c|c|c|c}
    {\bf Q1*: property vs. degree} & $1$ & $2$ & $3$ & $\ge 4$ \\\hline
    Critical point & P & P & P & P\\\hline
    Second-order point & P & P & P & P\\\hline
    Local minimum & P & P & P & strongly NP-hard \cite{murty1987some}\footnotemark\addtocounter{footnote}{-1}\\
    & & & (Theorem \ref{POLY Thm: Checking Min Poly Time}) & \\\hline
    Strict local minimum & P & P & P & strongly NP-hard \cite{murty1987some}\footnotemark\\
    & & & (Corollary \ref{POLY Cor: Check Strict Local Min}) &
    \end{tabular}
    \caption{Complexity of deciding whether a given point is of a certain type, based on the degree of the polynomial. Entries without a reference are classical.}
    \label{POLY Table: Complexity Checking}
\end{table}
\footnotetext{The proof in~\cite{murty1987some} is based on a reduction from the ``matrix copositivity'' problem. However, \cite{murty1987some} only shows that this problem (and thus deciding if a quartic polynomial has a local minimum) is weakly NP-hard, since the reduction to matrix copositivity there is from the weakly NP-hard problem of Subset Sum. Nonetheless, their result can be strengthened by observing that testing matrix copositivity is in fact strongly NP-hard. This claim is implicit, e.g., in~\cite[Corollary 2.4]{de2002approximation}. The NP-hardness of testing whether a point is a strict local minimum of a quartic polynomial is not explicitly stated in~\cite{murty1987some}, though it follows in the weak sense from the weak NP-hardness of Problem 8 of~\cite{murty1987some}. Again, with some work, this can be strengthened to a strong NP-hardness result.} 

\begin{table}[H]
    \centering
    \begin{tabular}{c|c|c|c|c}
    {\bf Q2*: property vs. degree} & $1$ & $2$ & $3$ & $\ge 4$ \\\hline
    Critical point & P & P & strongly NP-hard & strongly NP-hard \\
    & & & (Theorem \ref{POLY Thm: Critical point cubic NP-hard}) & (Theorem \ref{POLY Thm: Critical point cubic NP-hard})\\\hline
    Second-order point & P & P & SDP & strongly NP-hard\\
    & & & (Corollary \ref{POLY Cor: Complete Cubic SDP SOP}) & (Theorem \ref{POLY Thm: SOP quartic NP-hard})\\\hline
    Local minimum & P & P & SDP & strongly NP-hard \\
    & & & (Algorithm \ref{POLY Alg: Complete Cubic SDP}) & (Theorem \ref{QPs Thm: LM Degree 4 NP hard})\\\hline
    Strict local minimum & P & P & SDP & strongly NP-hard\\
    & & & (Algorithm \ref{POLY Alg: Complete Cubic SDP}, Remark \ref{POLY Rem: Find SLM}) & (Corollary \ref{QPs Cor: SLM Degree 4 NP hard})
    \end{tabular}
    \caption{Complexity of deciding whether a polynomial has a point of a certain type, based on the degree of the polynomial. Entries without a reference are classical.}
    \label{POLY Table: Complexity Existence}
\end{table}

The majority of the technical work in this chapter is spent on the case of cubic polynomials. It is somewhat surprising that many of the problems of interest to us are tractable for cubics, especially the search for local minima. This is in contrast to the intractability of other interesting problems related to cubic polynomials, for example, minimizing them over the unit sphere~\cite{nesterov2003random}, or checking their convexity over a box~\cite{ahmadi2019complexity}. It is also interesting to note that second-order points of cubic polynomials are easier to find than their critical points, despite being a more restrictive type of point. This shows that the right approach to finding second-order points involves bypassing the search for critical points as an initial step.

\subsection{Organization and Main Contributions of the Chapter}

Section \ref{POLY Sec: NP-hardness results} covers the NP-hardness results from Table~\ref{POLY Table: Complexity Existence}. The remainder of the chapter is devoted to our results on cubic polynomials, which fills in the remaining entries of Tables~\ref{POLY Table: Complexity Checking} and \ref{POLY Table: Complexity Existence}. In Section~\ref{POLY Sec: Local Min}, we give a characterization of local minima of cubic polynomials (Theorem~\ref{POLY Thm: TOC}) and show that it can be checked in polynomial time (Theorem \ref{POLY Thm: Checking Min Poly Time}). In Section~\ref{POLY Sec: Geometry}, we give some geometric facts about local minima of cubic polynomials. For example, we show that the set of local minima of a cubic polynomial $p$ is convex (Theorem~\ref{POLY Thm: WLM Convex}), and we relate this set to the second-order points of $p$ and to the set of minima of $p$ over points where $\Hess p$ is positive semidefinite (Theorem~\ref{POLY Thm: Closure} and Theorem~\ref{POLY Thm: Local Minima SOP}). In Section~\ref{POLY SSec: Cubic Spectrahedra}, we show that the interior of any spectrahedron is the projection of the local minima of some cubic polynomial (Theorem~\ref{POLY Thm: Cubic Spectrahedron}). In Section~\ref{POLY Sec: Complexity}, we use this result to show that deciding if a cubic polynomial has a local minimum or a second-order point is at least as hard as some semidefinite feasibility problems.

In Section~\ref{POLY Sec: Finding Local Min}, we start from a ``sum of squares'' approach to finding second-order points of a cubic polynomial (Theorem~\ref{POLY Thm: cubic sdp} and Theorem~\ref{POLY Thm: solution recovery sdp}), and build upon it (Section~\ref{POLY SSec: Simplification}) to arrive at an efficient semidefinite representation of these points (Corollary~\ref{POLY Cor: Complete Cubic SDP SOP}). This also leads to an algorithm for finding local minima of cubic polynomials by solving polynomially-many SDPs of polynomial size (Algorithm~\ref{POLY Alg: Complete Cubic SDP}). In Section~\ref{POLY Sec: Conclusions}, we take preliminary steps towards some interesting future research directions, such as the design of an unregularized third-order Newton method that would use as a subroutine our algorithm for finding local minima of cubic polynomials (Section \ref{POLY SSec: Cubic Newton}).

\subsection{Preliminaries and Notation}

We review some standard facts about local minina; more preliminaries specific to cubic polynomials appear in Section~\ref{POLY SSec: Cubic Preliminaries}. Three well-known optimality conditions in unconstrained optimization are the \emph{first-order necessary condition} (FONC), the \emph{second-order necessary condition} (SONC), and the \emph{second-order sufficient condition} (SOSC). Respectively, they are that the gradient at any local minimum is zero, the Hessian at any local minimum is psd, and that any critical point at which the Hessian is positive definite is a strict local minimum. A vector $d \in \Rn$ is said to be a \emph{descent direction} for a function $p: \Rn \to \R$ at a point $\xbar \in \Rn$ if there exists a scalar $\epsilon > 0$ such that $p(\xbar + \alpha d) < p(\xbar)$ for all $\alpha \in (0,\epsilon)$. Existence of a descent direction at a point clearly implies that the point is not a local minimum. However, in general, the lack of a descent direction at a point does not imply that the point is a local minimum (see, e.g., Example \ref{POLY Ex: no local min}).

Next, we establish some basic notation which will be used throughout the chapter. We denote the set of $n\times n$ real symmetric matrices by $\Snn$. For a matrix $M\in\Snn$, the notation $M \succeq 0$ denotes that $M$ is positive semidefinite, $M \succ 0$ denotes that it is positive definite, and $\Tr(M)$ denotes its trace, i.e. the sum of its diagonal entries. For a matrix $M$, the notation $\nulls(M)$ denotes its null space, and $\cols(M)$ denotes its column space. All vectors are taken to be column vectors. For two vectors $x$ and $y$, the notation $(x,y)$ denotes the vector $\bvec x \\ y \evec$. The notation $0_n$ denotes the vector of length $n$ containing only zeros. The notation $e_i$ denotes the $i$-th coordinate vector, i.e., the vector with a one in its $i$-th entry and zeros everywhere else.

\section{NP-hardness Results}\label{POLY Sec: NP-hardness results}

In this section, we present reductions that show our NP-hardness results from Tables \ref{POLY Table: Complexity Checking} and \ref{POLY Table: Complexity Existence}. For concreteness, we construct these reductions from the (simple) MAXCUT problem, though our proof can work with any NP-hard problem that can be encoded by quadratic equations with ``small enough'' coefficients. Recall that in the (simple) MAXCUT problem, we are given as input an undirected and unweighted graph $G$ on $n$ vertices and an integer $k \le n$. We are then asked whether there is a cut in $G$ of size $k$, i.e. a partition of the vertices into two sets $S_1$ and $S_2$ such that the number of edges with one endpoint in $S_1$ and one endpoint in $S_2$ is equal to $k$. It is well known that the (simple) MAXCUT problem is strongly NP-hard~\cite{garey2002computers}.
	
If we denote the adjacency matrix of $G$ by $E \in \Snn$, it is straightforward to see that $G$ has a cut of size $k$ if and only if the following system of quadratic equations is feasible:
	
	\begin{equation}\label{POLY Eq: MAXCUT QP}
    \begin{aligned}
	q_0(x) &\defeq \frac{1}{4}\sum_{i=1}^n \sum_{j=1}^n E_{ij}(1-x_ix_j) - k = 0,\\
	q_i(x) &\defeq x_i^2 - 1 = 0, i = 1, \ldots, n.
	\eaeql
	Indeed, the second set of constraints enforces each variable $x_i$ to be $-1$ or $1$, and any $x \in \{-1,1\}^n$ encodes a cut in $G$ by assigning vertices with $x_i = 1$ to one side of the partition, and those with $x_i=-1$ to the other. Observe that with this encoding, $x_ix_j$ equals $1$ whenever the two vertices $i$ and $j$ are on the same side and $-1$ otherwise. The size of the cut is therefore given by $\frac{1}{4}\sum_{i=1}^n \sum_{j=1}^n E_{ij}(1-x_ix_j)$, noting that every edge is counted twice.

\begin{theorem}\label{POLY Thm: Critical point cubic NP-hard}
	It is strongly NP-hard to decide whether a polynomial $p: \Rn \to \R$ of degree greater than or equal to three has a critical point.
\end{theorem}

\begin{proof}
	Let $d \ge 3$ be fixed. Given an instance of the (simple) MAXCUT problem with a graph on $n$ vertices, let the quadratic polynomials $q_0, \ldots, q_n$ be as in (\ref{POLY Eq: MAXCUT QP}), and consider the following degree-$d$ polynomial in $2n+2$ variables $(x_1, \ldots, x_n, y_0, y_1, \ldots, y_n,z)$:
	$$p(x,y,z) = \sum_{i=0}^{n} y_iq_i(x) + z^d.$$
	Note that all coefficients of this polynomial take $O(\log(n))$ bits to write down. We show that $p(x,y,z)$ has a critical point if and only if the quadratic system $q_0(x) = 0, \ldots, q_{n}(x) = 0$ is feasible. Observe that the gradient of $p$ is given by
	
	$$\bvec \\ \frac{\partial p}{\partial x}\\\\\hline\\ \frac{\partial p}{\partial y} \\\\\hline \frac{\partial p}{\partial z}\evec = \bvec \sum_{i=0}^n y_i\frac{\partial q_i}{\partial x_1}(x) \\ \vdots \\ \sum_{i=0}^n y_i\frac{\partial q_i}{\partial x_n}(x)\\ \hline q_0(x) \\ \vdots \\ q_n(x) \\\hline dz^{d-1}\evec.$$
	
	If $\xbar \in \Rn$ is a solution to (\ref{POLY Eq: MAXCUT QP}), then the point $(\xbar, 0_{n+1},0)$ is a critical point of $p$. Conversely, if $(\xbar, \ybar, \bar z)$ is a critical point of $p$, then, since $\frac{\partial p}{\partial y}(\xbar,\ybar,\bar z) = 0$, $\xbar$ must be a solution to (\ref{POLY Eq: MAXCUT QP}).
\end{proof}

\begin{theorem}\label{POLY Thm: SOP quartic NP-hard}
	It is strongly NP-hard to decide whether a polynomial $p: \Rn \to \R$ of degree greater than or equal to four has a second-order point.
\end{theorem}

\begin{proof}
	Let $d \ge 4$ be fixed. Given an instance of the (simple) MAXCUT problem with a graph on $n$ vertices, let the quadratic polynomials $q_0, \ldots, q_n$ be as in (\ref{POLY Eq: MAXCUT QP}), and consider the following degree-$d$ polynomial in $3n+3$ variables $(x_1, \ldots, x_n, y_0, y_1, \ldots, y_n, z_0, z_1, \ldots, z_n,w)$:
	$$p(x,y,z,w) = \sum_{i=0}^{n} \left(y_i^2 q_i(x) - z_i^2 q_i(x)\right) + w^d.$$
	Note that all coefficients of this polynomial take $O(\log(n))$ bits to write down. We show that $p(x,y,z,w)$ has a second-order point if and only if the quadratic system ${q_0(x) = 0, \ldots, q_{n}(x) = 0}$ is feasible. 
	
	Observe that $\frac{\partial^2 p}{\partial y^2}$ is an $(n+1) \times (n+1)$ diagonal matrix with $2q_0(x), \ldots, 2q_n(x)$ on its diagonal. Similarly, $\frac{\partial^2 p}{\partial z^2}$ is an $(n+1) \times (n+1)$ diagonal matrix with $-2q_0(x), \ldots, -2q_n(x)$ on its diagonal. Suppose first that $(\xbar, \ybar, \bar z, \bar w)$ is a second-order point of $p$. Since $\Hess p(\xbar, \ybar, \bar z, \bar w) \succeq 0$, and since $\frac{\partial^2 p}{\partial y^2}$ and $\frac{\partial^2 p}{\partial z^2}$ are both principal submatrices of $\Hess p$, it must be that ${q_0(\xbar) = 0, \ldots, q_n(\xbar) = 0}$.
	
	Now suppose that $\xbar \in \Rn$ is a solution to (\ref{POLY Eq: MAXCUT QP}). We show that $(\xbar, 0_{n+1}, 0_{n+1},0),$ is a second-order point of $p$. Note that $\frac{\partial p}{\partial x}$ is quadratic in $y$ and $z$, $\frac{\partial p}{\partial y}$ is linear in $y$, $\frac{\partial p}{\partial z}$ is linear in $z$, and $\frac{\partial p}{\partial w} = dw^{d-1}$. Thus $(\xbar, 0_{n+1}, 0_{n+1},0)$ is a critical point of $p$. Now observe that the entries of $\frac{\partial^2 p}{\partial x^2}$ are quadratic in $y$ and $z$ or are zero, the entries of $\frac{\partial^2 p}{\partial x \partial y}$ are linear in $y$ or are zero, the entries of $\frac{\partial^2 p}{\partial x \partial z}$ are linear in $z$ or are zero, $\frac{\partial^2 p}{\partial w^2} = d(d-1)w^{d-2}$, $\frac{\partial^2 p}{\partial y^2}(\xbar, 0_{n+1}, 0_{n+1},0)$ and $\frac{\partial^2 p}{\partial z^2}(\xbar, 0_{n+1}, 0_{n+1},0)$ are both zero, and all other entries of $\Hess p$ are zero. Thus $\Hess p(\xbar, 0_{n+1}, 0_{n+1},0) = 0$, and we conclude that $(\xbar, 0_{n+1}, 0_{n+1},0)$ is a second-order point of $p$.
\end{proof}

The remaining two NP-hardness results from Table \ref{POLY Table: Complexity Existence} are stated next, but proven in Chapter~\ref{Chap: QPs}, since a corollary of them is the main result of that chapter.

\begin{theorem}\label{POLY Thm: LM quartic NP-hard}
    It is strongly NP-hard to decide whether a polynomial $p: \Rn \to \R$ of degree greater than or equal to four has a local minimum. The same statement holds for testing existence of a strict local minimum.
\end{theorem}

\section{Checking Local Minimality of a Point for a Cubic Polynomial}\label{POLY Sec: Local Min}

As the reader can observe from Tables~\ref{POLY Table: Complexity Checking} and \ref{POLY Table: Complexity Existence} from Section~\ref{POLY Sec: Introduction}, the remaining entries all have to do with the case of cubic polynomials. To answer these questions about cubics, we start in this section by showing that the problem of deciding if a given point is a local minimum (or a strict local minimum) of a cubic polynomial is polynomial-time solvable. This answers the remaining cases in Table~\ref{POLY Table: Complexity Checking}. We first make certain observations about cubic polynomials that will be used throughout the chapter. 

\subsection{Preliminaries on Cubic Polynomials}\label{POLY SSec: Cubic Preliminaries}

It is easy to observe that a univariate cubic polynomial has either no local minima, exactly one local minimum (which is strict), or infinitely many non-strict local minima (in the case that the polynomial is constant). Further observe that if a point $\xbar \in \Rn$ is a (strict) local minimum of a function $p: \Rn \to \R$, then for any fixed point $\ybar \in \Rn$, the restriction of $p$ to the line going through $\xbar$ and $\ybar$ ---i.e. the univariate function $q(\alpha) \defeq p(\xbar + \alpha (\ybar - \xbar))$---has a (strict) local minimum at $\alpha = 0$. Since the restriction of a multivariate cubic polynomial to any line is a univariate polynomial of degree at most three, the previous two facts imply that (i) if a cubic polynomial has a strict local minimum, then it must be the only local minimum (strict or non-strict), and that (ii) if a cubic polynomial has multiple local minima, then the polynomial must be constant on the line connecting any two of these (necessarily non-strict) local minima.

Observe that for any cubic polynomial $p$, the error term of the second-order Taylor expansion is given by the cubic homogeneous component of $p$. More formally, for any point $\xbar \in \Rn$ and direction $v \in \Rn$,
\beq \label{POLY Eq: cubic taylor}
p(\xbar + \lambda v) = p_3(v)\lambda^3 + \frac{1}{2}v^T \Hxb v \lambda^2 + \gxb^Tv\lambda + p(\xbar),
\eeq 
where $p_3$ is the collection of terms of $p$ of degree exactly 3.

Note that the Hessian of any cubic $n$-variate polynomial is an affine matrix of the form $\sum_{i=1}^n x_iH_i + Q$, where $H_i$ and $Q$ are all $n \times n$ symmetric matrices and the $H_i$ satisfy
\begin{equation}\label{POLY Eq: Valid Hessian}
(H_i)_{jk} = (H_j)_{ik} = (H_k)_{ij}
\end{equation}
for any $i,j,k \in \{1, \ldots, n\}$. This is because an $n \times n$ symmetric matrix $A(x) \defeq A(x_1, \ldots, x_n)$ is a valid Hessian matrix if and only if $\frac{\partial}{\partial x_i} A_{jk}(x) = \frac{\partial}{\partial x_j} A_{ik}(x) = \frac{\partial}{\partial x_k} A_{ij}(x)$ for all $i,j,k \in \{1, \ldots, n\}$. If $\sum_{i=1}^n x_iH_i + Q$ is a valid Hessian matrix, then the cubic polynomial which gives rise to it is of the form
\beq \label{POLY Eq: Cubic Poly Form} \frac{1}{6}\sum_{i=1}^n x^T x_iH_ix + \frac{1}{2}x^TQx + b^Tx + c.\eeq
In this chapter, it is sometimes convenient for us to parametrize a cubic polynomial in the above form. As the scalar term in (\ref{POLY Eq: Cubic Poly Form}) is irrelevant for deciding local minimality or finding local minima, in the remainder of this chapter, we take $c=0$ without loss of generality. Observe that the gradient of the polynomial in (\ref{POLY Eq: Cubic Poly Form}) is $\frac{1}{2}\sum_{i=1}^n x_iH_ix + Qx + b$, or equivalently a vector whose $i$-th entry is $\frac{1}{2}x^TH_ix + e_i^TQx + b_i$.

\subsection{Local Minimality of a Point for a Cubic Polynomial}

In this section, we give a characterization of local minima of cubic polynomials and show that this characterization can be checked in polynomial time. Recall that we use the notation $p_3$ to denote the cubic homogeneous component of a cubic polynomial $p$, and $\nulls(M)$ (resp. $\cols(M)$) to denote the null space (resp. column space) of a matrix $M$.

\begin{theorem}\label{POLY Thm: TOC}
	A point $\xbar \in \Rn$ is a local minimum of a cubic polynomial $p: \Rn \to \R$ if and only if the following three conditions hold:
	\begin{itemize}
		\item $\grad p(\xbar) = 0,$
		\item $\Hess p(\xbar) \succeq 0,$
		\item $\gp3d = 0, \forall d \in \nulls(\Hxb).$
	\end{itemize}
\end{theorem}

Note that the first two conditions are the well-known FONC and SONC. Throughout the chapter, we refer to the third condition as the \emph{third-order condition (TOC)} for optimality. This condition is requiring the gradient of the cubic homogeneous component of $p$ to vanish on the null space of the Hessian of $p$ at $\xbar$. We remark that the FONC, SONC, and TOC together are in general neither sufficient nor necessary for a point to be a local minimum of a polynomial of degree higher than three. The first claim is trivial (consider, e.g., $p(x) = x^5$ at $x = 0$); for the second claim see Example \ref{POLY Ex: TOC not necessary}.

\begin{remark}\label{POLY Rem: TONC}
It is straightforward to see that any local minimum $\xbar$ of a cubic polynomial $p$ satisfies a condition similar to the TOC, that $p_3(d) = 0, \forall d \in \nulls(\Hxb)$. Indeed, if $\xbar$ is a second-order point and $d \in \nulls(\Hxb)$, then Equation (\ref{POLY Eq: cubic taylor}) gives $p(\xbar + \lambda d) = p_3(d)\lambda^3 + p(\xbar)$. Hence, if $p_3(d)$ is nonzero, then either $d$ or $-d$ is a descent direction for $p$ at $\xbar$, and so $\xbar$ cannot be a local minimum. This observation was made in~\cite{anandkumar2016efficient} for three-times differentiable functions, and is referred to as the ``third-order necessary condition'' (TONC) for optimality. Note that because $p_3$ is homogeneous of degree three, from Euler's theorem for homogeneous functions we have $3 p_3(x) = x^T\grad p_3(x)$. We can then see that $\grad p_3(d) = 0 \Rightarrow p_3(d) = 0$, and therefore the TOC is a stronger condition than the TONC. Indeed, the FONC, SONC, and TONC together are not sufficient for local optimality of a point for a cubic polynomial; see Example \ref{POLY Ex: no local min}. Intuitively, this is because the FONC, SONC, and TONC together avoid existence of a descent direction for cubic polynomials, but as the proof of Theorem \ref{POLY Thm: TOC} will show, existence of a ``descent parabola'' must also be avoided.
\end{remark}

We will need the following fact from linear algebra for the proof of Theorem \ref{POLY Thm: TOC}.

\begin{lem}\label{POLY Lem: Smallest Nonzero Eigenvalue}
	Let $M \in \Snn$ be a symmetric positive semidefinite matrix and denote its smallest positive eigenvalue by $\lambda_+$. Then if $z \in \cols(M)$ and $\|z\| = 1, z^TMz \ge \lambda_+$.
\end{lem}
\begin{proof}
	Suppose $M$ has eigenvalues $\lambda_1 \ge \lambda_2 \ge \cdots \ge \lambda_k > \lambda_{k+1} = \cdots = \lambda_n = 0$ (so $\lambda_+ = \lambda_k$). Let $v_1, \ldots, v_n$ be a set of corresponding mutually orthogonal unit-norm eigenvectors of $M$. Observe that any $z \in \cols(M)$ can be written as $z = \sum_{i=1}^n \alpha_iv_i$, for some scalars $\alpha_i$ with $\alpha_i = 0$ for $i =k+1, \ldots, n$. This is because the column space is orthogonal to the null space, and the eigenvectors corresponding to zero eigenvalues span the null space.
	
	Since $v_1, \ldots, v_k$ are mutually orthogonal unit vectors, we have
	$$z^TMz = \left(\sum_{i=1}^k \alpha_iv_i\right)^T \left(\sum_{i=1}^k \lambda_iv_iv_i^T\right) \left(\sum_{i=1}^k \alpha_iv_i\right) = \sum_{i=1}^k \alpha_i^2\lambda_i v_i^Tv_iv_i^Tv_i = \sum_{i=1}^k \alpha_i^2 \lambda_i,$$
	and
	$$1 = \|z\|^2 = \left(\sum_{i=1}^k \alpha_iv_i\right)^T\left(\sum_{i=1}^k \alpha_iv_i\right) = \sum_{i=1}^k \alpha_i^2v_i^Tv_i = \sum_{i=1}^k \alpha_i^2.$$
	These two equations combined imply that $z^TMz \ge \lambda_k = \lambda_+$.
\end{proof}

\begin{proof}[Proof (of Theorem \ref{POLY Thm: TOC})]
	As any local minimum must satisfy the FONC and SONC, it suffices to show that a second-order point is a local minimum for a cubic polynomial if and only if it also satisfies the TOC.
	
	We first observe that for any second-order point $\xbar$, scalars $\alpha$ and $\beta$, and vectors $d \in \nulls(\Hxb)$ and $z \in \Rn$, the following identity holds:
	\begin{equation}\label{POLY Eq: Taylor NC decomposition}
    \begin{aligned}
	p(\xbar + \alpha d+\beta z)	&=  p_3(\alpha d+\beta z) + \frac{1}{2}(\alpha d+\beta z)^T\Hxb(\alpha d+\beta z) + p(\xbar)\\
	&= \beta^3p_3(z) + \frac{\beta^2}{2}z^T\Hess p_3(\alpha d)z + \beta\grad p_3(\alpha d)^Tz + p_3(\alpha d) + \frac{\beta^2}{2} z^T \Hxb z + p(\xbar)\\
	&= \beta^3p_3(z) + \frac{\alpha \beta^2}{2}z^T\Hess p_3(d)z + \alpha^2 \beta\grad p_3(d)^Tz + \alpha^3p_3(d) + \frac{\beta^2}{2} z^T \Hxb z + p(\xbar).\\
	\eaeql
	The first equality follows from (\ref{POLY Eq: cubic taylor}) and the FONC. The second equality follows from the Taylor expansion of $p_3(\alpha d+\beta z)$ around $\alpha d$ and using the fact that $d \in \nulls(\Hxb)$. The last equality follows from homogeneity of $p_3$.
		
	\noindent\textbf{(second-order point) + TOC $\Rightarrow$ local minimum:}
	
	Let $\xbar$ be any second-order point at which the TOC holds. Note that any vector $v \in \Rn$ can be written as $\alpha d + \beta z$ for some (unique) scalars $\alpha$ and $\beta$, and unit vectors $d \in \nulls(\Hxb)$ and $z \in \cols(\Hxb)$. Since from the TOC we have $\grad p_3(d) = 0$ (which also implies that $p_3(d) = 0$, as seen e.g. by Euler's theorem for homogeneous functions mentioned above), the identity in (\ref{POLY Eq: Taylor NC decomposition}) reduces to \beq
	\label{POLY Eq: pxbar - p}
	p(\xbar + v) - p(\xbar) =\beta^2 \left(\beta p_3(z) + \frac{\alpha}{2} z^T \Hess p_3(d)z + \frac{1}{2} z^T \Hxb z\right).
	\eeq	

	Let $\lambda > 0$ be the smallest nonzero eigenvalue of $\Hxb$. From Lemma \ref{POLY Lem: Smallest Nonzero Eigenvalue} we have that $z^T \Hxb z \ge~\lambda$. Thus, if $\alpha$ and $\beta$ satisfy 
	\beq\label{POLY Eq: Min Radius}| \alpha| + | \beta| \le \lambda \cdot \left(\underset{\|z\|=1, \|d\|=1}{\max} \max \{z^T \Hess p_3(d)z, 2p_3(z)\}\right)^{-1},\eeq the expression on the right-hand side of (\ref{POLY Eq: pxbar - p}) is nonnegative. As the set ${\{\|z\|=1\} \cap \{\|d\|=1\}}$ is compact and $p_3$ is continuous and odd, the quantity $$\gamma \defeq \underset{\|z\|=1, \|d\|=1}{\max} \max \{z^T \Hess p_3(d)z, 2p_3(z)\}$$ is finite and nonnegative, and thus $\lambda/\gamma$ is positive (or potentially $+\infty$). Finally, note that for any $v \in \Rn$ such that $\|v\| \le \lambda/\gamma$, the corresponding $\alpha$ and $\beta$ satisfy (\ref{POLY Eq: Min Radius}), and thus $p(\xbar + v) - p(\xbar) \ge 0$ as desired.\\\\
	\textbf{Local minimum $\Rightarrow$ TOC:}
	
	Note that if $\xbar$ is a local minimum, then we must have $p_3(d) = 0$ whenever $d \in \nulls(\Hxb)$ (see Remark \ref{POLY Rem: TONC}). We also assume that $p_3$ is not the zero polynomial, as then the TOC would be automatically satisfied. 
	
	Now suppose for the sake of contradiction that there exists a vector $\hat{d} \in \nulls(\Hxb)$ such that $\grad p_3(\hat{d}) \ne 0$. Consider the sequence of points given by
	\begin{equation}\label{POLY Eq: Descent Parabola}
	\hat{x}_i \defeq \xbar + \alpha_i \hat{d} + \beta_i z,
	\eeq
	where
	$$z = -\frac{\grad p_3(\hat{d})}{\|\grad p_3(\hat{d})\|}, \alpha_i = \frac{1}{i}\sqrt{\frac{z^T \Hxb z}{|\grad p_3(\hat{d})^Tz|}}, \beta_i = \frac{1}{i^2}.$$
	Observe that $\hat{x}_i \to \xbar$ as $i \to \infty$.
	From (\ref{POLY Eq: Taylor NC decomposition}), we have
	$$p(\xbar + \alpha_i \hat{d}+\beta_i z) - p(\xbar)
	= p_3(z)\beta_i^3 + \frac{1}{2}z^T\Hess p_3(\hat{d})z \alpha_i \beta_i^2 + \grad p_3(\hat{d})^Tz \alpha_i^2 \beta_i + \frac{1}{2} z^T \Hxb z \beta_i^2.$$
	Note that because $\alpha_i \propto \sqrt{\beta_i}$, the third and fourth terms of the right-hand side of the above expression will be the dominant terms as $i \to \infty$. For our choices of $\alpha_i$ and $\beta_i$, the sum of these two dominant terms simplifies to $-\frac{1}{2i^4}z^T\Hxb z$. Observe that for any $w \in \nulls(\Hxb)$ and any $\alpha \in \R, p_3(\hat{d}+\alpha w) = 0$. Since the gradient of $p_3$ is orthogonal to its level sets, we must then have $\grad p_3(\hat{d})^Tw = 0$ for any $w \in \nulls(\Hxb)$. Thus, $\grad p_3(\hat{d})$ is in the orthogonal complement of $\nulls(\Hxb)$, i.e. in $\cols (\Hxb)$, and hence $z^T \Hxb z > 0$. Thus, for any sufficiently large $i$, $p(\hat{x}_i) < p(\xbar)$, and so $\xbar$ is not a local minimum.
	
\end{proof}

\begin{remark} Note that the points $\hat{x}_i$ constructed in (\ref{POLY Eq: Descent Parabola}) trace a parabola as $i$ ranges from $-\infty$ to $+\infty$. Thus as a corollary of the proof of Theorem~\ref{POLY Thm: TOC}, we see that if a point $\xbar \in \Rn$ is not a local minimum of a cubic polynomial $p: \Rn \to \R$, then there must exist a ``descent parabola'' that certifies that; i.e. a parabola $q(t): \R \to \Rn$ and a scalar $\bar{\alpha}$ satisfying $q(0) = \xbar$ and $p(q(\alpha)) < p(\xbar)$ for all $\alpha \in (0, \bar{\alpha})$.
\end{remark}

Theorem~\ref{POLY Thm: TOC} gives rise to the following algorithmic result.


\begin{theorem}\label{POLY Thm: Checking Min Poly Time}
	Local minimality of a point $\xbar \in \Rn$ for a cubic polynomial $p: \Rn \to \R$ can be checked in polynomial time.
\end{theorem}

\begin{proof}
	In view of Theorem \ref{POLY Thm: TOC}, we show that the FONC, SONC, and TOC can be checked in polynomial time (in the Turing model of computation). Checking that the gradient of $p$ vanishes at $\xbar$ and that the Hessian at $\xbar$ is positive semidefinite can be done in polynomial time as explained in Section~\ref{POLY Sec: Introduction}. We give the following polynomial-time algorithm for checking the TOC:
	
	\begin{algorithm}[H]
		\caption{Algorithm for checking the TOC.}\label{POLY Alg: Check Local Min}
		\begin{algorithmic}[1]
			\State {\bf Input:} Coefficients of a cubic polynomial $p: \Rn \to \R$, a point $\xbar \in \Rn$
			\State Compute $\Hxb$
			\State Compute a rational basis $\{v_1, \ldots, v_k\}$ for the null space of $\Hxb$
			\State Check if coefficients of $g(\lambda) \defeq \grad p_3 (\sum_{i=1}^k \lambda_i v_i)$ are all zero
    		\State \quad \texttt{{\bf if}} YES
    			\State \quad \quad $\xbar$ is a local minimum of $p$
    		\State \quad \texttt{{\bf if}} NO
    			\State \quad\quad $\xbar$ is a not local minimum of $p$
		\end{algorithmic}
	\end{algorithm}
	
	Note that the entries of the function $g: \R^k \to \Rn$ that appears in this algorithm are homogeneous quadratic polynomials in $\lambda \defeq (\lambda_1, \ldots, \lambda_k)$, where $k$ is the dimension of $\nulls(\Hxb)$. For the TOC to hold, $g$ must be zero for all $\lambda \in \R^k$, which happens if and only if all coefficients of every entry of $g$ are zero.
	
	A rational basis for the null space of a symmetric matrix can be computed in polynomial time, for example through the Bareiss algorithm \cite{bareiss1968sylvester}. For completeness, we give a less efficient but also polynomial-time algorithm which solves a series of linear systems. The first linear system finds a nonzero vector $v_1 \in \Rn$ such that $\Hxb^T v_1 = 0$. The successive linear systems solve for nonzero vectors $v_i \in \Rn$ such that $\Hxb^T v_i = 0, v_j^Tv_i = 0, \forall j = 1, \ldots, i-1$. To ensure nonzero solutions, some entry of the vector is fixed to 1, and if the system is infeasible, the next entry is fixed to 1 and the system is re-solved. Once the only feasible vector is the zero vector, the basis is complete.
	
	The next step is to compute the coefficients of $g$. To do this, one can first compute the coefficients of $\grad p_3$. There are $n \times {n+1 \choose 2}$ coefficients to compute, and each is a coefficient of $p_3$, multiplied by 1, 2, or 3. If the $m$-th entry of $\grad p_3$ is given by $\sum_{i=1}^n \sum_{j\ge i}^n c_{ij}x_ix_j$, then the $m$-th entry of $g$ is equal to $g_m(\lambda) = \sum_{a=1}^n \sum_{b=1}^n (\sum_{i=1}^n \sum_{j\ge i}^n c_{ij}(v_a)_i(v_b)_j)\lambda_a\lambda_b$, where the vectors $\{v_i\}$ are our rational basis for $\nulls(\Hxb)$. Observe that $g_m$ is a polynomial in $\lambda$ whose coefficients can be computed with a polynomial number of additions and multiplications over polynomially-sized scalars, and thus checking if all these coefficients are zero for every $m$ can be done in polynomial time.
	
	\end{proof}
	
	Let us end this subsection by also giving an efficient characterization of strict local minima of cubic polynomials.
	
	\begin{cor}\label{POLY Cor: Strict Local Min}
		A point $\xbar \in \Rn$ is a strict local minimum of a cubic polynomial $p: \Rn \to \R$ if and only if
		\begin{itemize}
			\item $\gxb = 0,$
			\item $\Hxb \succ 0.$
		\end{itemize}
	\end{cor}	
	
	\begin{proof}
	    The fact that these two conditions are sufficient for local minimality is immediate from the SOSC. To show the converse, in view of the FONC, we only need to show that positive definiteness of the Hessian is necessary. Suppose for the sake of contradiction that for some nonzero vector $d \in \Rn$, we have $d^T \Hxb d = 0$ (note that in view of the SONC, we cannot have $d^T \Hxb d< 0$). From (\ref{POLY Eq: cubic taylor}), we have $p(\xbar+\alpha d) = p(\xbar) + p_3(d)\alpha^3$. Hence, $\alpha = 0$ is not a strict local minimum of the univariate polynomial $p(\xbar + \alpha d)$, and so $\xbar$ is not a strict local minimum of $p$.
	\end{proof}
	
	\begin{cor}\label{POLY Cor: Check Strict Local Min}
	Strict local optimality of a point $\xbar \in \Rn$ for a cubic polynomial $p: \Rn \to \R$ can be checked in polynomial time.
	\end{cor}
	\begin{proof}
		This follows from the characterization in Corollary~\ref{POLY Cor: Strict Local Min}. Checking the FONC is straightforward as before. As explained in Section~\ref{POLY Sec: Introduction}, to check that $\Hxb$ is positive definite, one can equivalently check that all $n$ leading principal minors of $\Hxb$ are positive. This procedure takes polynomial time since determinants can be computed in polynomial time.
	\end{proof}

	\subsection{Examples}\label{POLY SSec: Local Min Examples}
	
	We give a few illustrative examples regarding the application and context of Theorem \ref{POLY Thm: TOC}.
	
	\begin{figure}[h]
	\includegraphics[height=.28\textheight,keepaspectratio]{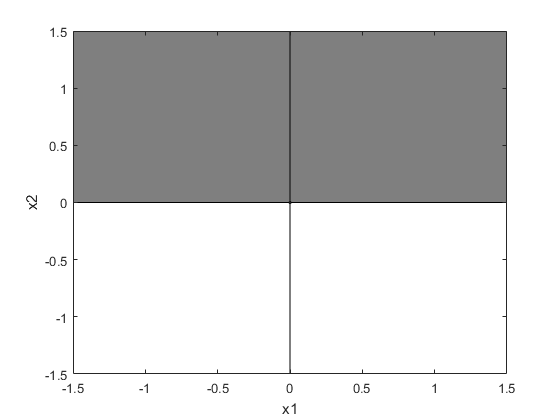}
	\includegraphics[height=.28\textheight,keepaspectratio]{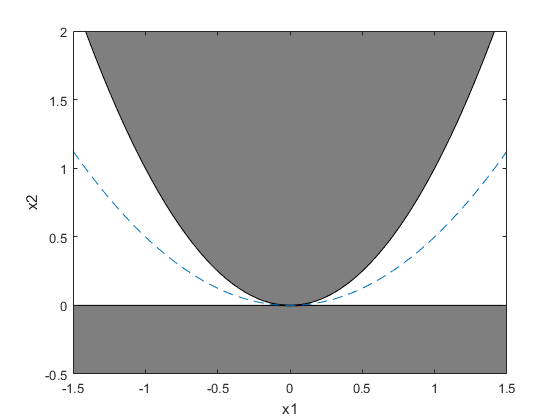}
	\caption{Contour plots of $x_1^2x_2$ (left) and $x_2^2 - x_1^2x_2$ (right) from Examples \ref{POLY Ex: local min} and \ref{POLY Ex: no local min}. The polynomials are zero on the black lines, positive on the gray regions, and negative on the white regions. The dashed line in the right-side figure denotes a descent parabola at the origin.}
	\label{POLY Fig: Local Min contours}
	\end{figure}
	
	\begin{example}\label{POLY Ex: local min}
		{\bf A cubic polynomial with local minima}
	
	Consider the polynomial $p(x_1,x_2) = x_1^2x_2$. By inspection (see Figure \ref{POLY Fig: Local Min contours}), one can see that points of the type $\{(x_1,x_2)\ |\ x_1 = 0, x_2 > 0\}$ are local minima of $p$, as $p$ is nonnegative when $x_2 > 0$, zero whenever $x_1 = 0$, and positive whenever $x_2 > 0$ and $x_1 \ne 0$. As a sanity check, we use Theorem~\ref{POLY Thm: TOC} to verify that the point $(0,1)$ is a local minimum of $p$ (the same reasoning applies to all other local minima).
	
	Through straightforward computation, we find
	
	$$\gx = \bvec 2x_1x_2 \\ x_1^2 \evec, \grad p_3(x) = \bvec 2x_1x_2 \\ x_1^2 \evec, \Hx = \bmat 2x_2 & 2x_1 \\ 2x_1 & 0 \emat.$$
	We can see that the FONC and SONC are satisfied at $(0,1)$.	The null space of $\Hess p(0,1)$ is spanned by $(0,1)$. We have
	$$\grad p_3\left(\alpha \bvec 0 \\ 1 \evec\right) = \bvec 2(0)(\alpha) \\ (0)^2 \evec = 0,$$
	which shows that the TOC is satisfied, verifying that $(0,1)$ is a local minimum of $p$.
	
	One can also verify that $\{(x_1,x_2)\ |\ x_1 = 0, x_2 > 0\}$ are the only local minima. Indeed, the critical points of $p$ are those where $x_1 = 0$, and the second-order points are those where $x_1 = 0$ and $x_2 \ge 0$. To see that $(0,0)$ is not a local minimum, observe that $(1,1) \in \nulls (\Hess p(0,0))$, but $\grad p_3(1,1) = (2,1) \ne 0$, and thus the TOC is violated.
	\end{example}
	
	\begin{example}\label{POLY Ex: no local min}
		{\bf A cubic polynomial with no local minima}
	
	We use Theorem~\ref{POLY Thm: TOC} to show that the polynomial $p(x_1,x_2) = x_2^2 - x_1^2x_2$ has no local minima. We have
	
	$$\gx = \bvec -2x_1x_2 \\ 2x_2 - x_1^2 \evec, \grad p_3(x) = \bvec -2x_1x_2 \\ -x_1^2 \evec, \Hx = \bmat -2x_2 & -2x_1 \\ -2x_1 & 2 \emat.$$
	Observe that $(0,0)$ is the only second-order point of $p$. The null space of $\Hess p(0,0)$ is spanned by $(1,0)$. We have
	
	$$\grad p_3\left(\alpha \bvec 1 \\ 0 \evec\right) = \bvec -2(\alpha)(0) \\ -(\alpha)^2 \evec = \bvec 0 \\ -\alpha^2 \evec \ne 0,$$
	which shows that the TOC is violated, and hence $(0,0)$ is not a local minimum. Note that the TONC is in fact satisfied at $(0,0)$, since $p_3(\alpha, 0) = 0$ for any scalar $\alpha$.
	
It is also interesting to observe that there are no descent directions for $p$ at $(0,0)$ (this is implied, e.g., by satisfaction of the TONC, along with the FONC and SONC). However, we can use the proof of Theorem \ref{POLY Thm: TOC} to compute a descent parabola, thereby more explicitly demonstrating that $(0,0)$ is not a local minimum. The column space of $\Hess p(0,0)$ is spanned by $(0,1)$. Then, following the proof of Theorem \ref{POLY Thm: TOC} with $z = (0,1)$ and $\hat{d} = (1,0)$, we have $z^T \Hess p(0,0) z = 2$ and $|\grad p_3(\hat{d})^T z| = 1$. The parabola prescribed is then the set $\{(x_1, x_2)\ |\ x_2 = \frac{1}{2} x_1^2\}$. Indeed, one can now verify that except at $(0,0)$, $p$ is negative on the entire parabola; see the dashed line in Figure \ref{POLY Fig: Local Min contours}.
	\end{example}
	
	\begin{example}\label{POLY Ex: TOC not necessary}
	{\bf A quartic polynomial with a local minimum that does not satisfy the TOC}
	
	
	We show in this example that for polynomials of degree higher than three, the TOC is not a necessary condition for local minimality. Consider the quartic polynomial given by $p(x_1,x_2) = 2x_1^4 + 2x_1^2x_2 + x_2^2$. The point $(0,0)$ is a local minimum, as $p(0,0) = 0$ and $p(x_1,x_2) = x_1^4 + (x_1^2 + x_2)^2$ is nonnegative. However, the Hessian of $p$ at $(0,0)$ is
	$$\Hess p(0,0) = \bmat 0 & 0 \\ 0 & 2 \emat,$$
	which has a null space spanned by $(1,0)$. We observe that $\grad p_3(x_1,x_2) = \bmat 4x_1x_2 \\ 2x_1^2 \emat$ does not vanish on this null space, as it evaluates, for example, to $(0,2)$ at $(1,0)$.
	\end{example}
	
	\section{On the Geometry of Local Minima of Cubic Polynomials}\label{POLY Sec: Geometry}
	
	We have shown that deciding local minimality of a given point for a cubic polynomial is a polynomial-time solvable problem. We now turn our attention to the remaining unresolved entries in Table~\ref{POLY Table: Complexity Existence} from Section~\ref{POLY Sec: Introduction}, which are on the problems of deciding whether a cubic polynomial has a second-order point, a local minimum, or a strict local minimum. In Sections~\ref{POLY Sec: Complexity} and \ref{POLY Sec: Finding Local Min}, we will show that these problem can all be reduced to semidefinite programs of tractable size. In the current section, we present a number of geometric results about local minima and second-order points of cubic polynomials which are used in those sections, but are possibly of independent interest. For the remainder of this chapter, we use the notation $SO_p$ to denote the set of second-order points of a polynomial $p$, $LM_p$ to denote the set of its local minima, and $\bar{S}$ to denote the closure of a set $S$.
	
	\subsection{Convexity of the Set of Local Minima}\label{POLY SSec: Local Min Convex}
	
	We begin by showing that for any cubic polynomial $p$, the set $LM_p$ is convex. We go through two lemmas; the first is a simple algebraic observation, and the second contains information about some critical points. Recall that the Hessian of a cubic polynomial $p$ written in the form of (\ref{POLY Eq: Cubic Poly Form}) is given by $\sum_{i=1}^n x_iH_i + Q$. Furthermore, its gradient is given by $\frac{1}{2}\sum_{i=1}^n x_iH_ix + Qx + b$, or equivalently a vector whose $i$-th entry is $x^TH_ix + e_i^TQx + b_i$.
	
	\begin{lem}\label{POLY Lem: Hessian switch}
		Let $H_1, \ldots, H_n \subseteq \Snn$ satisfy (\ref{POLY Eq: Valid Hessian}). Then for any two vectors $y, z \in \Rn$, $$\left(\sum_{i=1}^n y_iH_i\right)z = \left(\sum_{i=1}^n z_iH_i\right)y.$$
	\end{lem}
	\begin{proof}
		Observe that for any index $k \in \{1, \ldots, n\}$, we have
		\baeq
		\left(\left(\sum_{i=1}^n y_iH_i\right)z\right)_k &= \sum_{i=1}^n \sum_{j=1}^n (H_i)_{kj}y_iz_j\\
		&= \sum_{i=1}^n \sum_{j=1}^n (H_j)_{ki}y_iz_j\\
		&= \left(\left(\sum_{j=1}^n z_jH_j\right)y\right)_k,
		\eaeq
		where the second equality follows from (\ref{POLY Eq: Valid Hessian}).
	\end{proof}
	
	\begin{lem}\label{POLY Lem: Gradient Invariance}
		Let $\xbar \in \Rn$ be a local minimum of a cubic polynomial $p: \Rn \to \R$, and let $d~\in~\nulls(\Hxb)$. Then for any scalar $\alpha$, $\xbar +\alpha d$ is a critical point of $p$.
	\end{lem}
	\begin{proof}
		Let $p$ be given in our canonical form as $\frac{1}{6}\sum_{i=1}^n x^Tx_iH_ix + \frac{1}{2}x^TQx + b^Tx$. We have
		
		\baeq
		\grad p(\xbar+ \alpha d) &= \left(\frac{1}{2} \sum_{i=1}^n \xbar_iH_i + \alpha d_iH_i\right)(\xbar+\alpha d) + Q(\xbar+\alpha d) + b\\
		&= \left(\frac{1}{2}\sum_{i=1}^n \xbar_iH_i\right)\xbar + Q\xbar + b\\
		&+ \frac{1}{2} \sum_{i=1}^n \alpha d_iH_i\xbar + \frac{1}{2}\sum_{i=1}^n \alpha \xbar_iH_id + \alpha Qd\\
		&+ \frac{\alpha^2}{2}\sum_{i=1}^{n} d_iH_id\\
		&=\grad p(\xbar) + \alpha \Hess p(\xbar)d + \alpha^2\grad p_3(d)\\
		&= 0 + 0 + 0 = 0,\eaeq
		where the third equality follows form Lemma \ref{POLY Lem: Hessian switch}, and the last follows from the FONC and TOC.
	\end{proof}
	
	\begin{theorem}\label{POLY Thm: WLM Convex}
		The set of local minima of any cubic polynomial is convex.
	\end{theorem}
	\begin{proof}
		If for some cubic polynomial $p$, the set $LM_p$ of its local minima is empty or a singleton, the claim is trivially established. Otherwise, let $\xbar, \ybar \in LM_p$ with $\xbar \ne \ybar$. Consider any convex combination $z \defeq \xbar + \alpha(\ybar-\xbar)$, where $\alpha \in (0,1)$. We show that $z$ satisfies the FONC, SONC, and TOC, and therefore by Theorem~\ref{POLY Thm: TOC}, $z \in LM_p$.
		
		Note from (\ref{POLY Eq: cubic taylor}) that the restriction of $p$ to the line passing through $\xbar$ and $\ybar$ is
		$$p(\xbar + \alpha (\ybar - \xbar)) = p_3(\ybar - \xbar)\alpha^3 + \frac{1}{2}(\ybar - \xbar)^T \Hxb (\ybar - \xbar) \alpha^2 + \gxb^T(\ybar - \xbar)\alpha + p(\xbar).$$
		Since this univariate cubic polynomial has two local minima at $\alpha = 0$ and $\alpha = 1$, it must be constant. In particular, the coefficient of $\alpha^2$ must be zero, and because $\Hxb$ is psd, that implies $\ybar-\xbar \in \nulls(\Hxb)$.  Hence, by Lemma \ref{POLY Lem: Gradient Invariance}, the FONC holds at $z$. To show the SONC and TOC at $z$, note that because $\Hess p(x)$ is affine in $x$, $\Hess p(z)$ can be written as a convex combination of $\Hxb$ and $\Hess p(\ybar)$, both of which are psd. The SONC is then immediate. To see why the TOC holds, recall that the null space of the sum of two psd matrices is the intersection of the null spaces of the summand matrices. Thus $\nulls(\Hess p(z)) \subseteq \nulls(\Hxb)$, and the TOC is satisfied.
	\end{proof}
	
	As a demonstration of Theorem \ref{POLY Thm: WLM Convex}, Figure \ref{POLY Fig: Critical nonconvex} shows the critical points and the local minima of the cubic polynomial \beq\label{POLY Eq: Nonconvex CPs}x_1^3 + 3x_1^2x_2 + 3x_1x_2^2 + x_2^3 - 3x_1 - 3x_2.\eeq Note that the critical points form a nonconvex set, while the local minima constitute a convex subset of the critical points.
	
	\begin{figure}[h]
	\centering
	\includegraphics[height=.35\textheight,keepaspectratio]{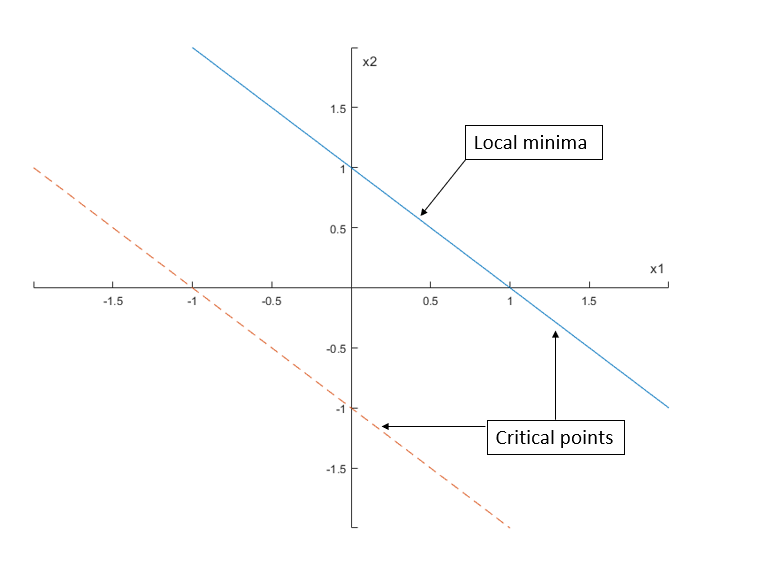}
	\caption{The critical points of the polynomial (\ref{POLY Eq: Nonconvex CPs}). One can verify that the set of critical points is $\{(x_1, x_2)\ |\ (x_1 + x_2)^2 = 1\}$, and that the set of local minima is $\{(x_1, x_2)\ |\ x_1 + x_2 = 1\}$. The points on the dashed line are local maxima.}
	\label{POLY Fig: Critical nonconvex}
	\end{figure}
	
	Unlike the above example, $LM_p$ (or even $\overline{LM_p}$ as $LM_p$ is in general not closed) may not be a polyhedral\footnote{Recall that a \emph{polyhedron} is a set defined by finitely many affine inequalities.} set for cubic polynomials. For instance, the polynomial \beq\label{POLY Eq: Circle LM}p(x_1,x_2,x_3,x_4) = -x_1x_3^2 + x_1x_4^2 + 2x_2x_3x_4 + x_3^2 + x_4^2,\eeq has $LM_p = \{x\in \R^4 \ |\ x_1^2+x_2^2 < 1, x_3 = x_4 = 0\}$ (see Figure \ref{POLY Fig: Local Min Spectrahedron}). This is in contrast to quadratic polynomials, whose local minima always form a polyhedral set. We show in Theorem \ref{POLY Thm: WLM Closure Spectrahedron}, however, that $\overline{LM_p}$ is always a spectrahedron\footnote{Recall that a \emph{spectrahedron} is a set of the type $S = \{x \in \Rn|A_0 + \sum_{i=1}^n x_iA_i \succeq 0\}$, where $A_0, \ldots A_n$ are symmetric matrices of some size $m \times m$~\cite{vinzant2014spectrahedron}.}. We first need the following lemma.
	
	\begin{figure}[H]
	\centering
	\includegraphics[height=.35\textheight,keepaspectratio]{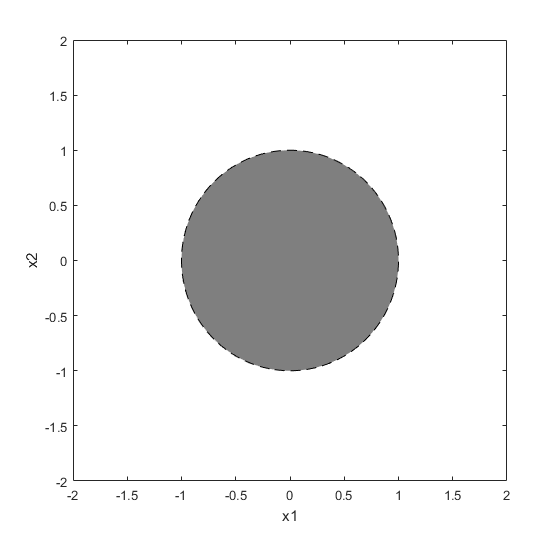}
	\caption{The projection of the set of local minima of the polynomial in (\ref{POLY Eq: Circle LM}) onto the $x_1$ and $x_2$ variables. This example shows that $\overline{LM_p}$ is not always a polyhedral set.}
	\label{POLY Fig: Local Min Spectrahedron}
	\end{figure}
	
	\begin{lem}\label{POLY Lem: constant null}
			For any cubic polynomial $p: \Rn \to \R$, suppose $\xbar \in \Rn$ and $\ybar \in \Rn$ satisfy
			\begin{itemize}
				\item $\xbar \in SO_p$,
				\item $\Hess p(\ybar) \succeq 0$,
				\item $p(\xbar) = p(\ybar)$.
			\end{itemize}
			 Then $p(\xbar + \alpha (\ybar - \xbar)) = p(\xbar)$ for any scalar $\alpha$, and $\ybar - \xbar \in \nulls (\Hxb)$.
		\end{lem}
		
		 Note in particular that this lemma applies if $\ybar$ is simply a second-order point, since $p$ must take the same value at any two second-order points. This is because any non-constant univariate cubic polynomial can have at most one second-order point.
	
		\begin{proof}
			Consider the Taylor expansion of $p$ around $\xbar$ in the direction $\ybar - \xbar$ (see (\ref{POLY Eq: cubic taylor})):
			$$q(\alpha) \defeq p(\xbar + \alpha (\ybar - \xbar)) = p_3(\ybar - \xbar)\alpha^3 + \frac{1}{2}(\ybar - \xbar)^T \Hxb (\ybar - \xbar) \alpha^2 + \gxb^T(\ybar - \xbar)\alpha + p(\xbar).$$
			Note that $q$ is a univariate cubic polynomial which has a second-order point at $\alpha = 0$. It is straightforward to see that if a univariate cubic polynomial is not constant and has a second-order point, then any other point which takes the same function value as the second-order point must have a negative second derivative. As this is not the case for $q$ (in view of $\alpha = 0$ and $\alpha = 1$), $q$ must be constant, i.e., $p(\xbar + \alpha (\ybar - \xbar)) = p(\xbar)$ for any $\alpha$. Now observe that for $p(\xbar + \alpha (\ybar - \xbar))$ to be constant, we must have $(\ybar - \xbar)^T \Hxb (\ybar - \xbar) = 0$. As $\Hxb \succeq 0$, we have $\ybar - \xbar \in \nulls (\Hxb)$.
		\end{proof}
	
	\begin{theorem}\label{POLY Thm: WLM Closure Spectrahedron}
		For a cubic polynomial $p:\Rn \to \R, \overline{LM_p}$ is a spectrahedron.
	\end{theorem}
	\begin{proof}
	    If $LM_p$ is empty, the claim is trivial. Otherwise, let $\xbar \in LM_p$. We show that $\overline{LM_p}$ is given by the spectrahedron
	    \beq\label{POLY Eq: WLM Closure Spectrahedron}
	    M \defeq \{x \in \Rn\ |\ \Hx \succeq 0, \Hxb(x - \xbar) =0\}.
	    \eeq
	    First consider any $\ybar \in LM_p$. From the SONC we know that $\Hess p(\ybar) \succeq 0$ and from Lemma \ref{POLY Lem: constant null}, we know that $\ybar - \xbar \in \nulls(\Hxb)$. Thus $\ybar \in M$. Since $M$ is closed, we get that $\overline{LM_p} \subseteq M$.
	    
	    Now consider any $\ybar \in M$. By the definition of $M$, $\ybar$ satisfies the SONC, and by Lemma \ref{POLY Lem: Gradient Invariance}, it also satisfies the FONC. Since for any scalar $\alpha \in (0,1)$, $\Hess p(\xbar + \alpha (\ybar - \xbar))$ is a convex combination of the two psd matrices $\Hxb$ and $\Hess p(\ybar)$, $\nulls(\Hess p(\xbar + \alpha (\ybar - \xbar))) \subseteq \Hxb$ and thus $\xbar + \alpha (\ybar - \xbar)$ satisfies the TOC (since $\xbar$ does). Thus $\ybar$ can be written as the limit of local minima of $p$ (e.g. $\{\xbar + \alpha (\ybar - \xbar)\}$ as $\alpha \to 1$).
	\end{proof}

	\begin{remark}
	We will soon show that for a cubic polynomial $p$, if $LM_p$ is nonempty, then $\overline{LM_p} = SO_p$ (see Theorem \ref{POLY Thm: Closure}). In Section \ref{POLY Sec: Finding Local Min}, we will give other representations of $SO_p$, which in contrast to the representation in (\ref{POLY Eq: WLM Closure Spectrahedron}), do not rely on access to or even existence of a local minimum.
	\end{remark}

	\subsection{Local Minima and Solutions to a ``Convex'' Problem}\label{POLY SSec: convex pop}

	In Section \ref{POLY Sec: Finding Local Min}, we present an SDP-based approach for finding local minima of cubic polynomials. (We note again that the SDP representation in (\ref{POLY Eq: WLM Closure Spectrahedron}) is useless for this purpose as it already assumes access to a local minimum.) Many common approaches for computing local minima of twice-differentiable functions involve first finding critical points of the function, and then checking whether they satisfy second-order conditions. However, as discussed in the introduction and in Section~\ref{POLY Sec: NP-hardness results}, such approaches are unlikely to be effective for cubic polynomials as critical points of these functions are in fact NP-hard to find (see Theorem~\ref{POLY Thm: Critical point cubic NP-hard}). Interestingly, however, we show in Section~\ref{POLY Sec: Finding Local Min} that by bypassing the search for critical points, one can directly find second-order points and local minima of cubic polynomials by solving semidefinite programs of tractable size. The key to our approach is to relate the problem of finding a local minimum of a cubic polynomial $p$ to the following optimization problem:
	\begin{equation}\label{POLY Eq: convex pop}
    \begin{aligned}
	& \underset{x \in \Rn}{\inf}
	& & p(x) \\
	& \text{subject to}
	&& \Hx \succeq 0.\\
	\eaeql
	The connection between solutions of (\ref{POLY Eq: convex pop}) and local minima of $p$ is established by Theorem \ref{POLY Thm: Closure} below. The feasible set of (\ref{POLY Eq: convex pop}) has interesting geometric properties (see, e.g., Corollary \ref{POLY Thm: Cubic Spectrahedron}) and will be referred to with the following terminology in the remainder of the chapter.
	
	\begin{defn}\label{POLY Def: ConvR}
		The \emph{convexity region} of a polynomial $p: \Rn \to \R$ is the set $$CR_p \defeq \{x \in \Rn\ |\ \Hx \succeq 0\}.$$
	\end{defn}
	Observe that for any cubic polynomial, its convexity region is a spectrahedron,
	and thus a convex set. As $p$ is a convex function when restricted to its convexity region, one can consider (\ref{POLY Eq: convex pop}) to be a convex problem in spirit. 
	
	\begin{theorem}\label{POLY Thm: Closure}
        Let $p$ be a cubic polynomial with a second-order point. Then the following sets are equivalent:
        \renewcommand{\labelenumi}{(\roman{enumi})}
		\begin{enumerate}
			\item $SO_p$
			\item Minima of (\ref{POLY Eq: convex pop}).
		\end{enumerate}
		Furthermore, if $p$ has a local minimum, then these two sets are equivalent to:
		\renewcommand{\labelenumi}{(\roman{enumi})}
		\begin{enumerate}\addtocounter{enumi}{2}
			\item $\overline{LM_p}$.
		\end{enumerate}
	\end{theorem}
	\begin{proof}
	    $(i) \subseteq (ii)$.\\
		Let $\ybar \in SO_p$ and $\xbar$ be any feasible point to (\ref{POLY Eq: convex pop}). If we consider the univariate cubic polynomial $q(\alpha) \defeq p(\xbar + \alpha(\ybar - \bar{x}))$, i.e., the restriction of $p$ to the line passing through $\xbar$ and $\ybar$, we can see that $\alpha = 1$ is a second-order point of $q$. Note that if any univariate cubic polynomial has a second-order point, then that second-order point is a minimum of it over its convexity region. In particular, because $\xbar$ is feasible to (\ref{POLY Eq: convex pop}) and thus $\alpha = 0$ is in the convexity region of $q$, we have $p(\ybar) = q(1) \le q(0) = p(\xbar)$. As $\ybar$ is feasible to (\ref{POLY Eq: convex pop}) and has objective value no higher than any other feasible point, it must be optimal to (\ref{POLY Eq: convex pop}).
	
	    $(ii) \subseteq (i)$\\
		Let $\ybar$ be a minimum of (\ref{POLY Eq: convex pop}) (we know that such a point exists because we have shown $SO_p$ is a subset of the minima of (\ref{POLY Eq: convex pop}), and $SO_p$ is nonempty by assumption). Let $\xbar \in SO_p$ and $d \defeq \ybar - \xbar$. Observe that $p(\ybar) = p(\xbar)$, and so by Lemma \ref{POLY Lem: constant null}, we must have $d \in \nulls(\Hxb)$. It follows that $\grad p(\ybar) = \grad p_3(d)$ (cf. the proof of Lemma \ref{POLY Lem: Gradient Invariance}). Now suppose for the sake of contradiction that $\ybar$ is not a second-order point. Since $\ybar$ is feasible to (\ref{POLY Eq: convex pop}), we must have $\grad p(\ybar) = \grad p_3(d) \ne 0$. As $p(\xbar) = p(\xbar + \alpha d)$ for any scalar $\alpha$ due to Lemma \ref{POLY Lem: constant null}, we must have $p_3(d) = \frac{1}{6} d^T \Hess p_3(d)d = 0$ (see (\ref{POLY Eq: cubic taylor})). Thus we can write
		\begin{align*}
		\big(d - \alpha \gp3d\big)^T\Hess p(\ybar)\big(d-\alpha \gp3d\big) =& \big(d - \alpha \gp3d\big)^T\big(\Hxb + \Hess p_3(d)\big)\big(d-\alpha \gp3d\big)\\
		=& \alpha^2 \gp3d^T\Hxb \gp3d - 2 \alpha \gp3d^T \Hess p_3(d)^Td\\
		&+ \alpha^2 \gp3d^T\Hess p_3(d)\gp3d\\
		=& \alpha^2 \left(\gp3d^T\Hxb \gp3d + \gp3d^T\Hess p_3(d)^T\gp3d\right)\\
		&- 4\alpha \gp3d^T\gp3d,
		\end{align*}
		where the last equality follows from that $\grad p_3(d) = \frac{1}{2} \Hess p_3(d)^Td$ due to Euler's theorem for homogeneous functions. Note that the right-hand side of the above expression is negative for sufficiently small $\alpha > 0$, and so $\Hess p(\ybar)$ is not psd, which contradicts feasibility of $\ybar$ to (\ref{POLY Eq: convex pop}).
		\\
		
		For the second claim of the theorem, suppose that $p$ has a local minimum. The following arguments will show $(i) = (ii) = (iii).$
		
		$(iii) \subseteq (i)$\\
		Clearly any local minimum of $p$ is a second-order point. Since the gradient and the Hessian of $p$ are continuous in $x$ and as the cone of psd matrices is closed, the limit of any convergent sequence of second-order points is a second-order point.
		

		$(ii) \subseteq (iii)$.\\
		Let $\ybar$ be any minimum of (\ref{POLY Eq: convex pop}). Consider any local minimum $\xbar$ of $p$ and let $z_\alpha \defeq \xbar + \alpha(\ybar-\xbar)$. As both $\Hess p(\ybar)$ and $\Hess p(\xbar)$ are psd, any point $z_\alpha$ with $\alpha \in [0,1)$ satisfies the SONC and TOC, by the same arguments as in the proof of Theorem \ref{POLY Thm: WLM Convex}.
		
		Now note that since $\xbar$ is a second-order point, it is also a minimum of (\ref{POLY Eq: convex pop}) (as $(i) \subseteq (ii)$) and thus $p(\ybar) = p(\xbar)$. From Lemma \ref{POLY Lem: constant null}, we then have $\ybar - \xbar \in \nulls(\Hxb)$, and so from Lemma \ref{POLY Lem: Gradient Invariance}, $z_\alpha$ satisfies the FONC for any $\alpha$. Thus, in view of Theorem \ref{POLY Thm: TOC}, for any $\alpha \in [0,1), z_\alpha$ is a local minimum of $p$. Therefore $\ybar$ can be written as the limit of a sequence of local minima (i.e., $\{z_\alpha\}$ as $\alpha \to 1$), and hence $\ybar \in \overline{LM_p}$.
	\end{proof}
	
	\begin{remark}\label{POLY Rem: SOP Spectrahedron}
	Note that as a consequence of Theorems \ref{POLY Thm: WLM Closure Spectrahedron}	and \ref{POLY Thm: Closure}, if a cubic polynomial $p$ has a local minimum, then $SO_p$ is a spectrahedron. In fact, $SO_p$ is a spectrahedron for \emph{any} cubic polynomial $p$; see Theorem~\ref{POLY Thm: solution recovery sdp}. In that theorem, we will give a more useful spectrahedral representation of $SO_p$ which does not rely on knowledge of a local minimum.
	\end{remark}

	\begin{cor}\label{POLY Cor: optval convex pop}
	Let $p$ be a cubic polynomial with a second-order point. Then the optimal value of (\ref{POLY Eq: convex pop}) is the value that $p$ takes at any of its second-order points (and in particular, at any of its local minima if they exist).
	\end{cor}
	\begin{proof}
	This is immediate from the equivalence of $(i)$ and $(ii)$ in Theorem \ref{POLY Thm: Closure}.
	\end{proof}

	\subsection{Distinction Between Local Minima and Second-Order Points}

	We have shown that the optimization problem in (\ref{POLY Eq: convex pop}) gives an approach for finding second-order points of a cubic polynomial $p$ without computing its critical points. However, not all second-order points are local minima, and so in this subsection, we characterize the difference between the two notions more precisely. We first recall the concept of the relative interior of a (convex) set (see, e.g., \cite[Chap. 6]{rockafellar1970convex}).
	
	\begin{defn}\label{POLY Def: Relative interior}
		The relative interior of a nonempty convex set $S \subseteq \Rn$ is the set
		$$ri(S) \defeq \{x \in S\ |\ \forall y \in S, \exists \lambda > 1\ s.t.\ \lambda x + (1-\lambda)y \in S \}.$$
	\end{defn}
	This definition generalizes the notion of interior to sets which do not have full dimension. One can show that for a convex set $S$, $ri(S)$ is convex, $ri(\bar{S}) = ri(S)$, and $\overline{ri(S)} = \bar{S}$~\cite{rockafellar1970convex}. In general, for a nonempty convex set $S$, we have $ri(\bar{S}) = ri(S) \subseteq S$, but we may not have $ri(\bar{S}) = S$. (For example, let $S$ be a line segment with one of its endpoints removed.) It turns out, however, that for a cubic polynomial $p$ with a local minimum, $ri(\overline{LM_p}) = LM_p$. 
	
	\begin{theorem}\label{POLY Thm: Local Minima SOP}
		Let $p: \Rn \to \R$ be a cubic polynomial with a local minimum. Then the following three sets are equivalent:
		\renewcommand{\labelenumi}{(\roman{enumi})}
		\begin{enumerate}
			\item $LM_p$
			\item $ri(SO_p)$
			\item Intersection of critical points of $p$ with $ri(CR_p)$.
		\end{enumerate}
	\end{theorem}

	\begin{proof}
		$(ii) \subseteq (i)$\\
		Recall from Theorem \ref{POLY Thm: WLM Convex} that $LM_p$ is convex, and from Theorem \ref{POLY Thm: Closure} that $SO_p = \overline{LM_p}$. Then we have $ri(SO_p) = ri(\overline{LM_p}) = ri(LM_p) \subseteq LM_p$.
		
		$(i) \subseteq (ii)$\\
		We prove the contrapositive. Let $\xbar$ be a point which is not in $ri(SO_p)$. If $\xbar$ is not a second-order point, then it clearly cannot be a local minimum. Suppose now that $\xbar \in SO_p \backslash ri(SO_p)$. Then there is another second-order point $\ybar$ such that $\ybar + \lambda (\xbar - \ybar)$ is not a second-order point for any $\lambda > 1$. Note from Lemma~\ref{POLY Lem: constant null} and the statement after it that $p(\ybar + \lambda (\xbar - \ybar))$ is a constant univariate function of $\lambda$. Now for any $\epsilon > 0$, define the point $\bar{z}_\epsilon \defeq \xbar + \frac{\epsilon}{2 \|\xbar - \ybar\|} (\xbar - \ybar)$. Since $\bar{z}_\epsilon$ is not a second-order point and thus not a local minimum, there is a point $z_\epsilon$ satisfying $\|\bar{z}_\epsilon - z_\epsilon\| < \frac{\epsilon}{2}$ and $$p(z_\epsilon) < p(\bar{z}_\epsilon) = p(\frac{\epsilon}{2 \|\xbar - \ybar\|} (\xbar - \ybar)) = p(\xbar).$$ Furthermore, by the triangle inequality, $z_\epsilon$ also satisfies $\|z_\epsilon - \xbar\| < \epsilon$. Thus, by considering $\{z_\epsilon\}$ as $\epsilon \to 0$, we can conclude that $\xbar$ is not a local minimum.
		
		$(i) \subseteq (iii)$\\
		Consider any local minimum $\xbar$ of $p$, which clearly must also be a critical point of $p$, and a member of $CR_p$. Suppose for the sake of contradiction that $\xbar \not\in ri(CR_p)$. Then there exists $y \in CR_p$ such that for any scalar $\alpha > 0, \Hess p(\xbar + \alpha(\xbar - y))$ is not psd. In particular, for any $\alpha > 0$ there exists a unit vector $z_\alpha \in \Rn$ such that $z_\alpha^T \Hess p(\xbar + \alpha(\xbar - y)) z_\alpha < 0$.
		
		We now show that for any $\alpha$, $z_\alpha$ can be taken to be in $\cols(\Hxb)$. This is because, as we will show, if $z_\alpha = d + v$, where $d \in \nulls(\Hxb)$ and $v \in \cols(\Hxb)$,
		\beq\label{POLY Eq: Cross Term Dies}
		(d+v)^T\Hess p(\xbar + \alpha(\xbar - y))(d+v) = v^T\Hess p(\xbar + \alpha(\xbar - y))v.\eeq
		Observe that if $p$ is written in the form (\ref{POLY Eq: Cubic Poly Form}), for any $d~\in~\nulls(\Hxb)$, we have
		\baeq
		d^T \Hess p(\xbar + \alpha(\xbar - y)) d
		&= d^T\left(\sum_{i=1}^n (\xbar_i + \alpha(\xbar_i - y_i))H_i  + Q\right)d\\
		&= d^T\left(\sum_{i=1}^n \xbar_iH_i + Q\right)d + \alpha \sum_{i=1}^n (d^TH_id)(\xbar_i-y_i) = 0,
		\eaeq
		where the last equality follows from that $d \in \nulls(\Hxb)$, and the TOC, recalling that the $i$-th entry of $\gp3d$ is $\frac{1}{2}d^TH_id$. Note in particular that the expression above also holds for $\alpha = -1$, and so $d \in \nulls(\Hess p(y))$. Now observe that because we can write $$\Hess p(\xbar + \alpha(\xbar - y)) = (1+\alpha)\Hxb - \alpha \Hess p(y),$$
		we have $\Hess p(\xbar + \alpha(\xbar - y))d = 0$. Thus, we have shown (\ref{POLY Eq: Cross Term Dies}), and we can take ${z_\alpha \in \cols(\Hxb)}$.
		
		Note that if $z_\alpha \in \cols(\Hxb)$, then by Lemma \ref{POLY Lem: Smallest Nonzero Eigenvalue} we have $z_\alpha^T\Hxb z_\alpha \ge \lambda$, where $\lambda$ is the smallest nonzero eigenvalue of $\Hxb$. Thus, for small enough $\alpha$, the quantity $z_\alpha^T \Hess p(\xbar + \alpha(\xbar - y)) z_\alpha$ is positive and so we arrive at a contradiction.
		
		$(iii) \subseteq (i)$\\
		Let $\xbar$ be a critical point which is in $ri(CR_p)$. Clearly $\xbar \in SO_p$. Consider any local minimum $\ybar$ of $p$, and	observe that for any $\alpha \ne 0$, we can write \beq \label{POLY Eq: xbar convex combination} \xbar = \frac{1}{\alpha}(\alpha \xbar + (1-\alpha)\ybar) + \frac{\alpha-1}{\alpha}\ybar.\eeq
		As $\xbar \in ri(CR_p)$ and $\ybar \in CR_p, \alpha \xbar + (1-\alpha) \ybar \in CR_p$ for some $\alpha > 1$. In particular, for that $\alpha, \Hess p(\alpha \xbar + (1-\alpha)\ybar) \succeq 0$ and thus in view of (\ref{POLY Eq: xbar convex combination}), we can see that $\nulls(\Hess p(\xbar)) \subseteq \nulls(\Hess p(\ybar))$. Hence, because the TOC holds at $\ybar$, it must also hold at $\xbar$. Thus $\xbar$ is a local minimum.
	\end{proof}
	
	Figure \ref{POLY Fig: SOP LM} demonstrates the relation between $LM_p$ and $SO_p$ for the polynomial ${p(x_1,x_2) = x_1^2x_2}$. For this example, $SO_p = \{(x_1,x_2)\ |\ x_1 = 0, x_2 \ge 0\}$, and $LM_p = \{(x_1,x_2)\ |\ x_1 = 0, x_2 > 0\}$ (see Example~\ref{POLY Ex: local min}).
	\begin{figure}[H]
	\centering
	\includegraphics[height=.25\textheight,keepaspectratio]{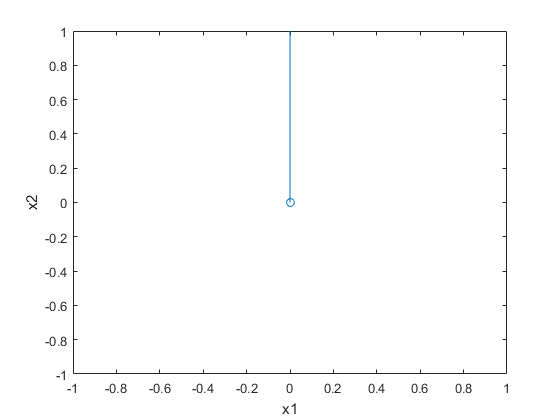}
	\includegraphics[height=.25\textheight,keepaspectratio]{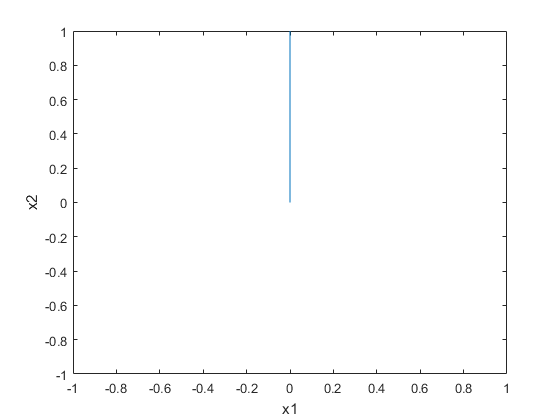}
	\caption{The set of local minima (left) and second-order points (right) of the cubic polynomial $p(x_1,x_2) = x_1^2x_2$. Note that $SO_p$ is the closure of $LM_p$ (Theorem \ref{POLY Thm: Closure}) and $LM_p$ is the relative interior of $SO_p$ (Theorem \ref{POLY Thm: Local Minima SOP}).}
	\label{POLY Fig: SOP LM}
	\end{figure}

	Theorem \ref{POLY Thm: Local Minima SOP} gives rise to the following interesting geometric fact about local minima of cubic polynomials.

	\begin{cor}\label{POLY Cor: Relint NS}
		Let $\xbar$ and $\ybar$ be two local minima of a cubic polynomial. Then $$\nulls (\Hxb) = \nulls (\Hess p(\ybar)).$$
	\end{cor}
	
	\begin{proof}
	It is known (\cite[Corollary 1]{ramana1995some}) that for a spectrahedron ${\{x \in \Rn\ |\ A_0 + \sum_{i=1}^n x_iA_i \succeq 0\}}$ and any two points $x$ and $y$ in its relative interior, $\nulls(A_0 + \sum_{i=1}^n x_iA_i) = \nulls(A_0 + \sum_{i=1}^n y_iA_i)$. In view of the facts that for any cubic polynomial $p$, $CR_p$ is a spectrahedron and $LM_p \subseteq ri(CR_p)$ (from Theorem \ref{POLY Thm: Local Minima SOP}), the result is immediate.
	\end{proof}
	
	\subsection{Spectrahedra and Convexity Regions of Cubic Polynomials}\label{POLY SSec: Cubic Spectrahedra}
	
	We end this section with a result relating general spectrahedra and convexity regions of cubic polynomials. Recall from the end of Section \ref{POLY SSec: Cubic Preliminaries} that if $S \defeq \{x \in \Rn\ |\ A_0 + \sum_{i=1}^n x_iA_i \succeq 0\}$ is a special spectrahedron, where $A_0, \ldots, A_n$ are $n \times n$ symmetric matrices satisfying $$(A_i)_{jk} = (A_j)_{ik} = (A_k)_{ij}$$ for any $i,j,k \in \{1 ,\ldots, n\}$, then $S$ is the convexity region of the cubic polynomial $$p(x) = \frac{1}{6}\sum_{i=1}^n x^Tx_iA_ix + \frac{1}{2}x^TA_0x.$$
	
	The following theorem shows that if the number of variables is allowed to increase, then \emph{any} spectrahedron can be represented by the convexity region of a cubic polynomial.
	
	\begin{theorem}\label{POLY Thm: Cubic Spectrahedron}
		Let a spectrahedron $S \subseteq \Rn$ be given by ${S \defeq \{x \in \Rn\ |\ A_0 + \sum_{i=1}^n x_iA_i \succeq 0\}}$, where $A_0, \ldots, A_n \in \mathbb{S}^{m \times m}$. There exists a cubic polynomial $p$ in at most $m+n$ variables such that $S$ is a projection of its convexity region; i.e., $$S = \{x \in \Rn\ |\ \exists y \in \R^m \mbox{ such that } (x,y) \in CR_p\}.$$ Furthermore, the interior of $S$ is a projection of the set of local minima of $p$.
	\end{theorem}

	\begin{proof}
		Let $A(x) \defeq A_0 + \sum_{i=1}^n x_iA_i$. We first present a characterization of the interior of $S$ following the developments in Section 2.4 of \cite{ramana1995some}. Let $\nulls_A \defeq \nulls(A_0) \cap  \ldots \cap \nulls(A_n)$, and $V$ be a full-rank matrix whose columns span the orthogonal complement of $\nulls_A$. Suppose that $\nulls_A$ is $(m-k)$-dimensional. Then there exist matrices $B_0, \ldots, B_n \in \mathbb{S}^{k \times k}$ with $\nulls(B_0) \cap \ldots \cap \nulls(B_n) = \{0_k\}$ such that $$B(x) \defeq B_0 + \sum_{i=1}^n x_iB_i = V^TA(x)V.$$
		In \cite[Corollary 5]{ramana1995some}, it is shown that $B(x)$
		\beq\label{POLY Eq: Pencils same}\{x \in \Rn\ |\ A(x) \succeq 0\} = \{x \in \Rn\ |\ B(x) \succeq 0\}\eeq
		and that the set $\{x \in \Rn\ |\ B(x) \succ 0\}$ gives the interior of $S$. Now consider the following cubic polynomial in $n+k$ variables:
		\beq\label{POLY Eq: Cubic with LM} p(x,y) \defeq y^TB(x)y.\eeq
		Observe that the partial derivative of $p$ with respect to $y$ is $2B(x)y$, the partial derivative of $p$ with respect to $x_i$ is $y^TB_iy$, and the Hessian of $p$ is
		$$\Hess p(x,y) = 2\bmat 0 & C(y)^T \\ C(y) & B(x) \emat,$$
		where $C(y)$ is an $k \times n$ matrix whose $i$-th column equals $B_iy$. One can then immediately see that if $(\xbar, \ybar) \in CR_p$, then we must have $B(\xbar) \succeq 0$. Conversely, if $B(\xbar) \succeq 0$, then $(\xbar,0_k) \in CR_p$. Hence, in view of (\ref{POLY Eq: Pencils same}), we have shown that the spectrahedron $S$ is the projection of $CR_p$ onto the $x$ variables.
		
		We now show that $LM_p = \{x \in \Rn\ |\ B(x) \succ 0\} \times \{0_k\}$. This would prove the second claim of the theorem. First let $\xbar$ be such that $B(\xbar) \succ 0$. Note that $p(\xbar, 0_k) = 0$ and that for any two vectors $\chi \in \Rn$ and $\psi \in \mathbb{R}^k$, $$p(\xbar + \chi, \psi) = \psi^T\left(B(\xbar) + \sum_{i=1}^n B_i\chi_i\right)\psi.$$ Since $B(\xbar) \succ 0$, then for any $\chi$ of sufficiently small norm, $B(\xbar) + \sum_{i=1}^n B_i\chi_i$ is still positive definite, and hence for any $\psi$, $p(\xbar + \chi,\psi) \ge 0 = p(\xbar,0_k)$. Thus $(\xbar, 0_k)$ is a local minimum of $p$.
		
		Now let $(\xbar, \ybar)$ be a local minimum of $p$. From the SONC, we must have $B(\xbar) \succeq 0$ and $C(\ybar) = 0$, which implies that $B_i\ybar = 0_k, \forall i \in \otn$. Since
		$$\frac{\partial p}{\partial y}(\xbar,\ybar) = 2B(\xbar)\ybar = 2\left(B_0 + \sum_{i=1}^n \xbar_iB_i\right)\ybar = 2B_0\ybar + 2\sum_{i=1}^n \xbar_i(B_i\ybar),$$ it further follows from the FONC that $B_0\ybar = 0$. As $\nulls(B_0) \cap \ldots \cap \nulls(B_n) = \{0_k\}$ by construction, it follows that we must have $\ybar = 0_k$. Next, observe that $\nulls(\Hess p(\xbar, 0_k)) = \R^n \times \nulls(B(\xbar))$. Let $d \in \nulls(B(\xbar))$, and note that for any $i \in \otn$, $(e_i,d) \in \nulls(\Hess p(\xbar, 0_k))$ and $\frac{\partial p_3}{\partial y}(e_i,d) = B_id$. Then from the TOC, we must have $B_id = 0_k, \forall i \in \{1, \ldots, n\}$. Furthermore, since $d \in \nulls(B(\xbar))$, it follows that $B_0d = 0_k$ as well. Again, as $\nulls(B_0) \cap \ldots \cap \nulls(B_n) = \{0_k\}$ by construction, it follows that we must have $d = 0_k$ and thus $B(\xbar) \succ 0$.
	\end{proof}
		
\section{Complexity Justifications for an Exact SDP Oracle}\label{POLY Sec: Complexity}

In the next section, we show that second-order points and local minima of cubic polynomials can be found by solving polynomially-many semidefinite programs with a polynomial number of variables and constraints. One caveat however is that the inputs and outputs of these semidefinite programs can sometimes be algebraic but not necessarily rational numbers. As a result, we cannot claim that second-order points and local minima of cubic polynomials can be found in polynomial time in the Turing model of computation. In this subsection, we give evidence as to why establishing the complexity of these problems in the Turing model is at the moment likely out of reach.

\begin{defn}
The \emph{SDP Feasibility Problem} (SDPF) is the following decision question: Given $m \times m$ symmetric matrices $A_0, \ldots, A_n$ with rational entries, decide whether there exists a vector $x \in \Rn$ such that $A_0 + \sum_{i=1}^n x_iA_i \succeq 0$.
\end{defn}

\begin{defn}
The \emph{SDP Strict Feasibility Problem} (SDPSF) is the following decision question: Given $m \times m$ symmetric matrices $A_0, \ldots, A_n$ with rational entries, decide whether there exists a vector $x \in \Rn$ such that $A_0 + \sum_{i=1}^n x_iA_i \succ 0$.
\end{defn}

Even though semidefinite programs can be solved to arbitrary accuracy in polynomial time~\cite{vandenberghe1996semidefinite}, the complexities of the decision problems above remain as two of the outstanding open problems in semidefinite programming. At the moment, it is not known if these two decision problems even belong to the class NP~\cite{ramana1993algorithmic, porkolab1997complexity, de2006aspects}. We show next that the complexities of these problems are a lower bound on the complexities of testing existence of second-order points and local minima of cubic polynomials. (In Section~\ref{POLY Sec: Finding Local Min}, we accomplish the more involved task of giving the reduction in the opposite direction.)

\begin{theorem}\label{POLY Thm: SOP SDPF}
    If the problem of deciding whether a cubic polynomial has any second-order points is in P (resp. NP), then SDPF is in P (resp. NP).
\end{theorem}

\begin{proof}
    Given matrices $A_0, \ldots, A_n \in \mathbb{S}^{m \times m}$, let $A(x) \defeq A_0 + \sum_{i=1}^n x_iA_i$. By noting that the cubic polynomial $p(x,y) = y^TA(x)y$ has as its Hessian
    $$\Hess p(x,y) = 2\bmat 0 & B(y)^T \\ B(y) & A(x) \emat,$$
	where $B(y)$ is an $m \times n$ matrix whose $i$-th column equals $A_iy$, we can see that if $A(\xbar) \succeq 0$ for some $\xbar \in \Rn$, then $\Hess p(\xbar, 0_k) \succeq 0$. Since $p$ is quadratic in the variables $y$, $\grad p(\xbar, 0_k) = 0_{m+n}$, and hence $(\xbar, 0_k)$ is a second-order point of $p$. Conversely, if $A(x) \not\succeq 0$ for any $x \in \Rn$, then clearly $\Hess p(x, y) \not\succeq 0$ for any $x \in \Rn$ and $y \in \R^m$, and thus $p$ cannot have any second-order points.
	
	The above reduction shows that any polynomial-time algorithm (or polynomial-time verifiable certificate) for existence of second-order points of cubic polynomials translates into one for SDPF.
\end{proof}

\begin{theorem}\label{POLY Thm: LM SDPSF}
	If the problem of deciding whether a cubic polynomial has any local minima is in P (resp. NP), then SDPSF is in P (resp. NP).
\end{theorem}

\begin{proof}
	Given matrices $A_0, \ldots, A_n \in \mathbb{S}^{m \times m}$, let $A(x) \defeq A_0 + \sum_{i=1}^n x_iA_i$ and consider the set $S \defeq \{x \in \Rn\ |\ A(x) \succeq 0\}$. It is not difficult to see that there exists $\xbar \in \Rn$ such that $A(\xbar) \succ 0$ if and only if $S$ has a nonempty interior and $\nulls_A \defeq \nulls(A_0) \cap \nulls(A_1) \cap \nulls(A_2) \ldots \cap \nulls(A_n) = \{0_m\}$.\footnote{The ``only if'' direction is straightforward and the ``if'' direction follows from \cite[Corollary 5]{ramana1995some}.} The latter condition can be checked in polynomial time by solving linear systems. The former can be reduced---due to the second claim of Theorem \ref{POLY Thm: Cubic Spectrahedron}---to deciding if the cubic polynomial constructed in (\ref{POLY Eq: Cubic with LM}) has a local minimum. Note that the polynomial in (\ref{POLY Eq: Cubic with LM}) has coefficients polynomially sized in the entries of the matrices $A_i$, since the matrix $V$ in the proof of Theorem \ref{POLY Thm: Cubic Spectrahedron} can be taken to be the identity matrix when $\nulls_A = \{0_m\}$.
\end{proof}

In addition to the difficulties alluded to in the above two theorems, the following three examples point to concrete representation issues that one encounters in the Turing model when dealing with local minima of cubic polynomials. The same complications are known to arise for SDP feasibility problems~\cite{de2006aspects}.

\begin{example}
	{\bf A cubic polynomial with only irrational local minima.} Consider the univariate cubic polynomial $p(x) = x^3 - 6x$. One can easily verify that its unique local minimum is at $x=\sqrt{2},$ which is irrational even though the coefficients of $p$ are rational.
\end{example}

\begin{example}
{\bf A cubic polynomial with an irrational convexity region.} Consider the quintary cubic polynomial $p(x,y) = y^TA(x)y$, where
$$A(x) = \bmat 2 & x & 0 & 0 \\
x & 1 & 0 & 0 \\
0 & 0 & 2x & 2 \\
0 & 0 & 2 & x\emat.$$
One can easily verify that $x = \sqrt{2}$ is the only scalar satisfying $A(x) \succeq 0$. Since the matrix $2A(x)$ is a principal submatrix of $\Hess p(x,y)$, any point in the convexity region of $p$ must satisfy $x = \sqrt{2}$ (even though the coefficients of $p$ are rational).
\end{example}

\begin{example}
	{\bf A family of cubic polynomials whose local minima have exponential bitsize.} Consider the family of cubic polynomials $p_n(x,y) = y^TA_n(x)y$ in $3n$ variables, where
	$$A_n(x) = \bmat x_1 & 2 & 0 & 0 & \cdots & 0 & 0\\
	2 & 1 & 0 & 0 & \cdots & 0 & 0 \\
	0 & 0 & x_2 & x_1 & \cdots & 0 & 0\\
	0 & 0 & x_1 & 1 & \cdots & 0 & 0 \\
	\cdots & \cdots & \cdots & \cdots & \ddots & \cdots & \cdots \\
	0 & 0 & 0 & 0 & \cdots & x_n & x_{n-1}\\
	0 & 0 & 0 & 0 & \cdots & x_{n-1} & 1\emat.$$
	We show that even though these polynomials have some rational local minima, it takes exponential time to write them down. From the proof of Theorem \ref{POLY Thm: Cubic Spectrahedron}, one can infer that the set of local minima of $p_n$ is the set $\{x \in \Rn\ |\ A_n(x) \succ 0\} \times \{0_{2n}\}$. However, observe that to have $A_n(x) \succ 0$ (or even $A_n(x) \succeq 0)$, we must have $$x_1 \ge 4, x_2 \ge 16, \ldots, x_n \ge 2^{2^n}.$$  Hence, any local minimum of $p_n$ has bit length at least $O(2^n)$ even though the bit length of the coefficients of $p_n$ is $O(n)$.
\end{example}

\section{Finding Local Minima of Cubic Polynomials}\label{POLY Sec: Finding Local Min}

In this section, we derive an SDP-based approach for finding second-order points and local minima of cubic polynomials. This, along with the results established in Section \ref{POLY Sec: NP-hardness results}, will complete the entries of Table \ref{POLY Table: Complexity Existence} from Section~\ref{POLY Sec: Introduction}. We begin with some preliminaries that are needed to present the theorems of this section.

\subsection{Preliminaries from Semidefinite and Sum of Squares Optimization}\label{POLY SSec: SDP Prelims}

\subsubsection{The Oracle E-SDP}\label{POLY SSSec: ESDP}

Recall that a \emph{spectrahedron} is a set of the type
$$\left\{x \in \Rn\ |\ A_0 + \sum_{i=1}^n x_iA_i \succeq 0\right\},$$
where $A_0, \ldots, A_n$ are symmetric matrices of some size $m \times m$. A \emph{semidefinite representable set} (also known as a \emph{spectrahedral shadow}) is a set of the type
\beq\label{POLY Eq: SDR Set}\left\{x \in \Rn\ |\ \exists y \in \R^k \text{ such that } A_0 + \sum_{i=1}^n x_iA_i + \sum_{i=1}^k y_iB_i \succeq 0\right\},\eeq
for some integer $k \ge 0$ and symmetric $m \times m$ matrices $A_0, A_1, \ldots, A_n, B_1, \ldots, B_k$. These are exactly sets which semidefinite programming can optimize over.

We show in Theorem \ref{POLY Thm: solution recovery sdp} and Corollary \ref{POLY Cor: Complete Cubic SDP SOP} that the set of second-order points of any cubic polynomial is a spectrahedron and describe how a description of this spectrahedron can be obtained from the coefficients of $p$ only.\footnote{Recall that the results of Section \ref{POLY Sec: Geometry} by contrast established spectrahedrality of the set of second-order points under the assumption of existence of a local minimum (see Remark \ref{POLY Rem: SOP Spectrahedron}). Furthermore, the spectrahedral representation that we gave there (see Theorem \ref{POLY Thm: WLM Closure Spectrahedron}) required knowledge of a local minimum.} Since relative interiors of semidefinite representable sets (and in particular spectrahedra) are semidefinite representable \cite[Theorem 3.8]{netzer2010semidefinite}, it follows from our Theorem \ref{POLY Thm: Local Minima SOP} that the set of local minima of any cubic polynomial is semidefinite representable.

Due to the complexity results and representation issues presented in Section \ref{POLY Sec: Complexity}, we assume in this section that we can do arithmetic over real numbers and have access to an oracle which solves SDPs exactly. This oracle---which we call \emph{E-SDP}---takes as input an SDP with real data and outputs the optimal value as a real number if it is finite, or reports that the SDP is infeasible, or that it is unbounded.\footnote{Though this will not be needed for our purposes, it is straightforward to show that for an SDP with $n$ scalar variables, the oracle E-SDP can be called twice to test attainment of the optimal value, and a total of $n+1$ times to recover an optimal solution.} The following lemma shows that E-SDP can find a point in the relative interior of a semidefinite representable set. This will be relevant for us later in this section when we search for local minima of cubic polynomials.

\begin{lem}\label{POLY Lem: Relint recovery}
	Let $S$ be a nonempty semidefinite representable set in $\R^n$. Then a point in $ri(S)$ can be recovered in $2n$ calls to E-SDP.
\end{lem}

\begin{proof}
    Consider the following procedure. Let $S_1 = S$, and for $i \in \otn$ let $$S_{i+1} = S_i \cap \{x \in \Rn\ |\ x_i = x_i^*\},$$ where the scalar $x_i^*$ is chosen to be any ``intermediate'' value of $x_i$ on $S_i$. More precisely, let $\xbar_i$ (resp. $\underline{x}_i$) be the supremum (resp. infimum) of $x_i$ over $S_i$ (these two values may or may not be finite). If $\bar{x}_i = \underline{x}_i$, then set $x_i^* = \bar{x}_i$. Otherwise, set $x_i^*$ to be any scalar satisfying $\underline{x}_i < x_i^* < \bar{x}_i$. Note that for each $i$, $x_i^*$ can be computed using $2$ calls to E-SDP. Hence, after $2n$ calls to E-SDP, we arrive at a set $S_{n+1}$ which is a singleton by construction. 
		
	We next show, by induction, that the point in $S_{n+1}$ belongs to $ri(S)$. First note that as $S$ is nonempty, $ri(S)$ is nonempty~\cite[Theorem 6.2]{rockafellar1970convex}, which implies that $S_1 \cap ri(S) = ri(S)$ is nonempty. Now suppose that $S_i \cap ri(S)$ is nonempty for $i \in \{1,\ldots, k\}$. We show that $S_{k+1} \cap ri(S)$ is nonempty.
	
	First suppose that $k$ is such that $\bar{x}_k = x_k^* = \underline{x}_k$. In this case, because $\forall x \in S_k, x_k = x_k^*$,
	$$S_{k+1} \cap ri(S) = S_k \cap \{x \in \Rn\ |\ x_k = x_k^*\} \cap ri(S) = S_k \cap ri(S) \ne \emptyset.$$
	Now suppose that $\underline{x}_k < x_k^* < \bar{x}_k$.	By the definition of $\xbar_k$, there exists a sequence of points $\{y_j\} \subseteq S_k$ such that $(y_j)_k \to \bar{x}_k$. We recall that for any $z \in ri(S), y \in \bar{S}$, and $\lambda \in (0,1]$, $\lambda z + (1-\lambda)y \in ri(S)$~\cite[Theorem 6.1]{rockafellar1970convex}. Now let $z \in S_k \cap ri(S)$. Since $S_k$ is convex, for any $y \in S_k \cap \bar{S}$ and $\lambda \in (0,1]$, $\lambda z + (1-\lambda)y \in S_k \cap ri(S)$. In particular, since $S_k \cap \bar{S} = S_k$, the sequence $\{z_j\} \defeq \{\frac{1}{j}z + \frac{j-1}{j}y_j\}$ satisfies $\{z_j\} \subseteq S_k \cap ri(S)$ and $(z_j)_k \to \xbar_k$. Similarly, there exists a sequence of points $\{w_j\} \subseteq S_k \cap ri(S)$ such that $(w_j)_k \to \underline{x}_k$. As $S_k \cap ri(S)$ is convex, there must then be a point $x \in S_k \cap ri(S)$ satisfying $x_k = x_k^*$, and so $$S_{k+1} \cap ri(S) = S_k \cap \{x \in \Rn\ |\ x_k = x_k^*\} \cap ri(S)$$ is not empty.
\end{proof}

\subsubsection{Overview of Sum of Squares Polynomials}\label{POLY SSSec: Sos}
In order to describe our SDP-based approach for finding local minima of cubic polynomials, we also need to briefly review the connection between sum of squares polynomials and matrices to semidefinite programming. Approaches to finding local minima based on sum of squares have been studied before, such as in~\cite{nie2015hierarchy}. The SDP approach in this chapter, however, is based partially on finding critical points and does not formally study the case of cubic polynomials.

Recall that a (multivariate) polynomial $p:\R^n\to\R$ is \emph{nonnegative} if $p(x) \ge 0, \forall x \in \R^n$, and that a polynomial $p$ is said to be a \emph{sum of squares} (sos) if $p= \sum_{i=1}^r q_i^2$ for some polynomials $q_1,\ldots,q_r$. The notion of sum of squares also extends to polynomial matrices (i.e., matrices whose entries are multivariate polynomials). We say that symmetric polynomial matrix $M(x):\Rn \to \R^m \times \R^m$ is an \emph{sos-matrix} if it has a factorization as $M(x)=R(x)^TR(x)$ for some $r\times m$ polynomial matrix $R(x)$~\cite{helton2010semidefinite}. Observe that if $M$ is an sos-matrix, then $M(x) \succeq 0$ for any $x \in \Rn$. One can check that $M(x)$ is an sos-matrix if and only if the scalar-valued polynomial $y^TM(x)y$ in variables $(x_1,\ldots,x_n,y_1,\ldots,y_m)$ is sos. Indeed, the ``only if'' direction is clear, the ``if'' direction is because when $y^TM(x)y = \sum_{i=1}^r q_i^2(x,y)$ for some polynomials $q_1, \ldots, q_r$, each $q_i$ must be linear in $y$ and thus writable as $q_i(x)=\sum_{j=1}^m y_jq_{ij}(x)$ for some polynomials $q_{ij}$. Then if $R(x)$ is the $r \times m$ matrix where $R_{ij}(x) = q_{ij}(x)$, we will have $M(x) = R^T(x)R(x)$.

\subsection{A Sum of Squares Approach for Finding Second-Order Points}\label{POLY SSec: SDP Approach}

We have shown in Theorem \ref{POLY Thm: Closure} that if a cubic polynomial $p$ has a second-order point, the solutions of the optimization problem in (\ref{POLY Eq: convex pop}) exactly form the set $SO_p$ of its second-order points. The same theorem further showed that if $p$ has a local minimum, then the solutions of (\ref{POLY Eq: convex pop}) also coincide with $\overline{LM_p}$, i.e. the closure of the set of its local minima. Our goal in this section is to develop a semidefinite representation of $SO_p$ which can be obtained directly from the coefficients of $p$ (Corollary \ref{POLY Cor: Complete Cubic SDP SOP}). To arrive to this representation, we first present an sos relaxation of problem (\ref{POLY Eq: convex pop}), which we prove to be tight when $SO_p$ is nonempty (Theorem \ref{POLY Thm: cubic sdp}). We then provide a more efficient representation of the SDP underlying this sos relaxation in Section \ref{POLY SSec: Simplification}. This will lead to an algorithm (Algorithm~\ref{POLY Alg: Complete Cubic SDP}) for finding local minima of cubic polynomials which is presented in Section \ref{POLY SSSec: SOP and Algorithm}.

\begin{theorem}\label{POLY Thm: cubic sdp}
	If a cubic polynomial $p: \Rn \to \R$ has a second-order point, the optimal value of the following semidefinite program\footnote{To clarify, $x$ is not a decision variable in this problem. The decision variables are $\gamma$, the coefficients of $\sigma$, and the coefficients of the entries of $S$. The identity in the first constraint must hold for all $x$, and this can be enforced by matching the coefficient of each monomial on the left with the corresponding coefficient on the right.} is attained and is equal to the value of $p$ at all second-order points:
	
	\begin{equation}\label{POLY Eq: cubic sdp}
    \begin{aligned}
	& \underset{\gamma \in \R, \sigma(x), S(x)}{\sup}
	& & \gamma \\
	& \text{\emph{subject to}}
	&& p(x) - \gamma = \sigma(x) + \Tr(S(x)\Hx),\\
	&&& \sigma(x) \text{\emph{ is a degree-2 sos polynomial}},\\
	&&& S(x) \text{\emph{ is an }} n \times n \text{\emph{ sos-matrix with degree-2 entries.}}\\
	\eaeql
	
\end{theorem}

\begin{proof}
Let $\xbar$ be a second-order point of $p$ and $\gamma^*$ be the optimal value of (\ref{POLY Eq: cubic sdp}). Consider any feasible solution $(\gamma, \sigma, S)$ to (\ref{POLY Eq: cubic sdp}) (nonemptiness of the feasible set is established in the next paragraph). Since $\Hxb \succeq 0$ and $S(\xbar) \succeq 0$, we have $\Tr(\Hxb S(\xbar)) \ge 0$. Since $\sigma(\xbar) \ge 0$ as well, it follows that $p(\xbar) \ge \gamma$. Hence, $p(\xbar) \ge \gamma^*$.

To show that $p(\xbar) \le \gamma^*$ and that the value $\gamma^* = p(\xbar)$ is attained, we establish that
$$(\gamma,\sigma,S) = \left(p(\xbar), \frac{1}{3}(x-\xbar)^T\Hxb(x-\xbar), \frac{1}{6}(x-\xbar)(x-\xbar)^T\right)$$
is feasible to (\ref{POLY Eq: cubic sdp}). Note that $\frac{1}{3}(x-\xbar)^T\Hxb(x-\xbar)$ is an sos polynomial (as $\Hxb$ can be factored into $V^TV$), and that $\frac{1}{6}(x-\xbar)(x-\xbar)^T$ is an sos-matrix by construction. To show that the first constraint in (\ref{POLY Eq: cubic sdp}) is satisfied, consider the Taylor expansion of $p$ around $\xbar$ in the direction $x - \xbar$ (see (\ref{POLY Eq: cubic taylor}), noting that $\gxb = 0$):
\beq \label{POLY Eq: SOP cubic taylor}p(\xbar + (x-\xbar)) = p(\xbar) + \frac{1}{2}(x-\xbar)^T\Hxb(x-\xbar) + p_3(x-\xbar).\eeq
Observe that if $p$ is written in the form (\ref{POLY Eq: Cubic Poly Form}), then we have
\begin{align*}
p_3(x-\xbar) &= \frac{1}{6}(x-\xbar)^T\left(\sum_{i=1}^n (x_i-\xbar_i)H_i\right)(x-\xbar)\\
&=\frac{1}{6}(x-\xbar)^T\left(\sum_{i=1}^n (x_i-\xbar_i)H_i + Q - Q\right)(x-\xbar)\\
&=\frac{1}{6}(x-\xbar)^T\left(\sum_{i=1}^n x_iH_i+ Q - \sum_{i=1}^n \xbar_iH_i - Q\right)(x-\xbar)\\
&=\frac{1}{6}(x-\xbar)^T\Hx(x-\xbar) - \frac{1}{6}(x-\xbar)^T\Hxb(x-\xbar).
\end{align*}
Note further that due to the cyclic property of the trace, we have
$$\frac{1}{6}(x-\xbar)^T\Hx(x-\xbar) = \Tr\left((\frac{1}{6}(x-\xbar)(x-\xbar)^T)\Hx\right).$$
Hence, (\ref{POLY Eq: SOP cubic taylor}) reduces to the following identity
\beq\label{POLY Eq: cubic sos identity}
	p(x) - p(\xbar) = \frac{1}{3}(x-\xbar)^T\Hxb(x-\xbar) + \Tr\left((\frac{1}{6}(x-\xbar)(x-\xbar)^T)\Hx\right),
\eeq
and thus the claim is established.
\end{proof}

Since (\ref{POLY Eq: cubic sdp}) is a tight sos relaxation of (\ref{POLY Eq: convex pop}) when $SO_p$ is nonempty, it is interesting to see how an optimal solution to (\ref{POLY Eq: convex pop}) can be recovered from an optimal solution to (\ref{POLY Eq: cubic sdp}). This is shown in the next theorem, keeping in mind that optimal solutions to (\ref{POLY Eq: convex pop}) are second-order points of $p$ (see Theorem \ref{POLY Thm: Closure}).

\begin{theorem}\label{POLY Thm: solution recovery sdp}
		Let $p: \Rn \to \R$ be a cubic polynomial with a second-order point, and let $(\gamma^*, \sigma^*, S^*)$ be an optimal solution of (\ref{POLY Eq: cubic sdp}) applied to $p$. Then, the set
		\beq\label{POLY Eq: second order points}
		\Gamma \defeq \{x \in \Rn\ |\ \Hx \succeq 0, \sigma^*(x) = 0, \Tr(S^*(x)\Hx) = 0\}
		\eeq
		is a spectrahedron, and $\Gamma = SO_p$.
\end{theorem}

\begin{proof}

We first show that $\Gamma = SO_p$. Let $\xbar$ be a second-order point of $p$. From Theorem \ref{POLY Thm: cubic sdp} and the first constraint of (\ref{POLY Eq: cubic sdp}) we have
$$0 = p(\xbar) - p(\xbar) = p(\xbar) - \gamma^* = \sigma^*(\xbar) + \Tr(S^*(\xbar)\Hxb).$$
As $\sigma^*(\xbar)$ and $\Tr(S^*(\xbar)\Hxb)$ are both nonnegative, the above equation implies they must both be zero, and hence $SO_p \subseteq \Gamma$. To see why $\Gamma \subseteq SO_p$, let $\ybar$ be a point in $\Gamma$ and $\hat{x}$ be an arbitrary second-order point (which by the assumption of the theorem exists). Observe from Theorem \ref{POLY Thm: cubic sdp} and the first constraint of (\ref{POLY Eq: cubic sdp}) that
$$p(\ybar) - p(\hat{x}) = p(\ybar) - \gamma^* = \sigma^*(\ybar) + \Tr(S^*(\ybar)\Hess p(\ybar)) = 0.$$ Additionally, because $\Hess p(\ybar) \succeq 0$, it follows from Corollary \ref{POLY Cor: optval convex pop} that $\ybar$ is optimal to (\ref{POLY Eq: convex pop}), and thus is a second-order point by Theorem \ref{POLY Thm: Closure}.
	
Now we show that $\Gamma$ is a spectrahedron by ``linearizing'' the quadratic and cubic equations that appear in (\ref{POLY Eq: second order points}). Since $\sigma^*$ is a quadratic sos polynomial, it can be written equivalently as $\sigma^*(x) = \sum_{i=1}^m q_i^2(x)$ for some affine polynomials $q_1,\ldots,q_m$. Similarly, since $S^*$ is an sos-matrix with quadratic entries, it can be written as $S^*(x) = R(x)^TR(x)$ for some $k \times n$ matrix $R$ with affine entries. First note that as $\Hx$ is affine in $x$ and $\sigma^*$ is a sum of squares of affine polynomials, the set
$$\{x \in \Rn\ |\ \Hx \succeq 0, \sigma^*(x) = 0\} = \{x \in \Rn\ |\ \Hx \succeq 0, q_1(x) = 0,\ldots, q_m(x) = 0\}$$
is clearly a spectrahedron.

Now let $y$ be any point in $ri(CR_p)$. Such a point exists because $CR_p$ is nonempty by assumption, and relative interiors of nonempty convex sets are nonempty \cite[Theorem 6.2]{rockafellar1970convex}. Now let $r_i$ be the $i$-th column of the matrix $R^T$. We claim that $\Gamma$ is equivalent to the following set:
\beq \label{POLY Eq: SOP SDR}\big\{x \in \Rn\ |\ \Hx \succeq 0, q_1(x) = 0, \ldots, q_m(x) = 0, \Hess p(y) r_1(x) = 0, \ldots, \Hess p(y) r_k(x) = 0\big\}.\eeq
Note that this set is a spectrahedron, and that the final $k$ equality constraints are enforcing that each column of $R^T$ be in the null space of $\Hess p(y)$.

To prove the claim, first let $x$ be in (\ref{POLY Eq: SOP SDR}). Note that $\nulls(\Hess p(y)) \subseteq \nulls(\Hx)$, as $y \in ri(CR_p)$ and so $\Hess p(y) = \lambda \Hx + (1-\lambda) \Hess p(z)$ for some $z \in CR_p$ and $\lambda \in (0,1)$. Then,
$$\Tr(S^*(x)\Hx) = \sum_{i=1}^k r_i^T(x)\Hx r_i(x) = 0.$$
Hence $(\ref{POLY Eq: SOP SDR}) \subseteq (\ref{POLY Eq: second order points})$.

To show the reverse inclusion, let $x$ be a point in (\ref{POLY Eq: second order points}). It is easy to check that $\Tr(AB) = 0$ for two psd matrices $A = C^TC$ and $B$ if and only if the columns of $C^T$ belong to the null space of $B$. Hence, we must have $r_i(x) \in \nulls (\Hx)$. Assume first that $x\in ri(CR_p)$. Then we must have $r_i(x) \in \nulls (\Hx) = \nulls (\Hess p(y))$ as $CR_p$ is a spectrahedron and any two matrices in the relative interior of a spectrahedron have the same null space \cite[Corollary~1]{ramana1995some}. To see why we must also have $r_i(x) \in \nulls (\Hess p(y))$ for any $x \in CR_p\backslash ri(CR_p)$, observe that $\nulls(\Hess p(y))$ is closed, the vector-valued functions $r_i$ are continuous in $x$, and the preimage of a closed set under a continuous function is closed.
\end{proof}

\subsection{A Simplified Semidefinite Representation of Second-Order Points and an Algorithm for Finding Local Minima}\label{POLY SSec: Simplification}

In this subsection, we derive a semidefinite representation of the set $SO_p$, which will be given in (\ref{POLY Eq: Complete Cubic SDP SOP}). In contrast to the semidefinite representation in (\ref{POLY Eq: SOP SDR}), which requires first solving (\ref{POLY Eq: cubic sdp}) and then performing some matrix factorizations, the representation in (\ref{POLY Eq: Complete Cubic SDP SOP}) can be immediately obtained from the coefficients of $p$. To find a second-order point of an $n$-variate cubic polynomial via the representation in (\ref{POLY Eq: Complete Cubic SDP SOP}), one needs to solve an SDP with $\frac{(n+2)(n+1)}{2}$ scalar variables and two semidefinite constraints of size $(n+1) \times (n+1)$. This is in contrast to finding a second-order point via the representation in (\ref{POLY Eq: SOP SDR}), which requires solving two SDPs: (\ref{POLY Eq: cubic sdp}) which has $\left(\frac{n(n+1)}{2}+1\right)\left(\frac{(n+2)(n+1)}{2}\right)+1$ scalar variables and two semidefinite constraints of sizes $(n+1) \times (n+1)$ and $n(n+1) \times n(n+1)$ (coming from the two sos constraints), and then the SDP associated with (\ref{POLY Eq: SOP SDR}), which has $n$ scalar variables and a semidefinite constraint of size $n \times n$. Another purpose of this subsection is to present our final result, which is an algorithm for testing for existence of a local minimum (Algorithm \ref{POLY Alg: Complete Cubic SDP} in Section \ref{POLY SSSec: SOP and Algorithm}).

\subsubsection{A Simplified Sos Relaxation}

Recall from the proof of Theorem \ref{POLY Thm: cubic sdp} that if $p$ has a second-order point $\xbar$, then there is an optimal solution to (\ref{POLY Eq: cubic sdp}) of the form
\beq\label{POLY Eq: Ideal solution}
(\gamma, \sigma, S) = \left(p(\xbar), \frac{1}{3}(x-\xbar)^T\Hxb(x-\xbar), \frac{1}{6}(x-\xbar)(x-\xbar)^T\right).\eeq In particular, for this solution, the coefficients of $\sigma$ and $S$ can both be written entirely in terms of the entries of $\xbar$ and the coefficients of $p$. In what follows, we attempt to optimize over solutions to (\ref{POLY Eq: cubic sdp}) which are of the form in (\ref{POLY Eq: Ideal solution}). However, imposing this particular structure on the solution requires nonlinear equality constraints (in fact, it turns out quadratic constraints suffice). Instead, we will impose an SDP relaxation of these nonlinear constraints and show that the relaxation is exact. We follow a standard technique in deriving SDP relaxations for quadratic programs, where the outer product $xx^T$ of some variable $x$ is replaced by a new matrix variable $X$ satisfying $X - xx^T \succeq 0$. The latter matrix inequality that can be imposed as a semidefinite constraint via the Schur complement~\cite{boyd2004convex}. The variable $\xbar$ will be represented by a variable $y \in \Rn$, and the symmetric matrix variable $Y \in \Snn$ will represent $yy^T$. In addition, we will need another scalar variable $z$.

Assume $p$ is given in the form (\ref{POLY Eq: Cubic Poly Form}), and let us expand $\sigma$ in (\ref{POLY Eq: Ideal solution}) (disregarding the factor $\frac{1}{3}$) as follows:
\begin{small}
\begin{align*}
(x-\xbar)^T \Hxb (x - \xbar) &= x^T \left(\sum_{i=1}^n \xbar_iH_i + Q\right)x - 2\xbar^T\left(\sum_{i=1}^n \xbar_iH_i + Q\right)x + \xbar^T\left(\sum_{i=1}^n \xbar_iH_i + Q\right)\xbar\\&= x^T\left(\sum_{i=1}^n \xbar_iH_i + Q\right)x - 2\sum_{i=1}^n \Tr(H_i \xbar \xbar^T)x_i - 2\xbar^TQx + \xbar^T\left(\sum_{i=1}^n \xbar_iH_i + Q\right)\xbar,
\end{align*}\end{small}
where in the last equality we used Lemma \ref{POLY Lem: Hessian switch}. If we replace any occurrence of $\xbar$ with $y$, any occurrence of $\xbar \xbar^T$ with $Y$ and any occurrence of $\xbar^T(\sum_{i=1}^n \xbar_iH_i + Q)\xbar$ with $z$, we can rewrite the above expression as
\beq \label{POLY Eq: new cubic sdp first term}
\sigma_{Y,y,z}(x) \defeq \sum_{j=1}^n \sum_{k=1}^n \left(\sum_{i=1}^n (H_i)_{jk}y_i + Q_{jk}\right)x_jx_k - 2\sum_{i=1}^n (\Tr(H_iY) + e_i^TQy)x_i + z. \eeq
Similarly, the matrix $S$ in (\ref{POLY Eq: Ideal solution}) can be written as $xx^T - xy^T - yx^T + Y$ (disregarding the factor $\frac{1}{6}$). Note that if $Y - yy^T\succeq 0$, then the matrix $xx^T - xy^T - yx^T + Y$ is an sos-matrix (as a polynomial matrix in $x$). By making these replacements, we arrive at an SDP which attempts to look for a solution to the sos program in (\ref{POLY Eq: cubic sdp}) which is of the structure in (\ref{POLY Eq: Ideal solution}). This is the following SDP\footnote{Note that $x$ is not a decision variable in this SDP as the first constraint needs to hold for all $x$.}:
\begin{equation}\label{POLY Eq: new cubic sdp}
    \begin{aligned}
	& \underset{\gamma \in \R, Y \in \Snn, y \in \Rn, z \in \R}{\sup}
	& & \gamma \\
	& \text{subject to}
	&& p(x) - \gamma = \frac{1}{3}\sigma_{Y,y,z}(x) + \frac{1}{6}\Tr\left(\Hx(xx^T - xy^T - yx^T + Y)\right),\\
	&&&\sigma_{Y,y,z} \text{ is sos},\\
	&&&\bmat Y & y\\ y^T & 1 \emat \succeq 0.
\eaeql

Through straightforward algebra and matching coefficients, the first constraint (keeping in mind that $p$ is as in (\ref{POLY Eq: Cubic Poly Form})) can be more explicitly written as:
\begin{align*}
b_i &= -e_i^TQy - \frac{1}{2} \Tr(H_iY), i = 1, \ldots, n,\\
- \gamma &= \frac{1}{6}\Tr(QY) + \frac{z}{3}.
\end{align*}

These constraints reflect that the coefficients of the linear terms and the scalar coefficient match on both sides; the cubic and quadratic coefficients are automatically the same. We can rewrite (\ref{POLY Eq: new cubic sdp first term}) as
$$\sigma_{Y,y,z}(x) = \bvec x\\ 1\evec^T T(Y,y,z) \bvec x \\ 1 \evec,$$
where
$$T(Y,y,z) \defeq \bmat \sum_{i=1}^n y_iH_i + Q & \sum_{i=1}^n \Tr(H_iY)e_i+Qy \\ (\sum_{i=1}^n \Tr(H_iY)e_i+Qy)^T & z\emat.$$
The constraint in (\ref{POLY Eq: new cubic sdp}) that $\sigma$ be sos is the same as the matrix $T$ being psd. Putting everything together, the problem in (\ref{POLY Eq: new cubic sdp}) can be rewritten as the following SDP:\footnote{Recall that the data to this SDP is obtained from the representation of $p$ in the form of (\ref{POLY Eq: Cubic Poly Form}).}

\begin{equation}\label{POLY Eq: Small cubic SDP}
    \begin{aligned}
	& \underset{Y \in \Snn, y \in \Rn, z \in \R}{\inf}
	& & \frac{1}{6}\Tr(QY) + \frac{z}{3} \\
	& \text{subject to}
	&& \frac{1}{2}\Tr(H_iY) + e_i^TQy + b_i = 0, \forall i = 1, \ldots, n,\\
	&&& T(Y,y,z) \succeq 0,\\
	&&& \bmat Y & y\\ y^T & 1 \emat \succeq 0.
\eaeql

It is interesting to observe that the first constraint is a relaxation of the quadratic constraint which would impose $\grad p(y) = 0$, and that the constraint $T(Y,y,z) \succeq 0$ in particular implies $\Hess p(y) \succeq 0$. One can think of (\ref{POLY Eq: Small cubic SDP}) as another SDP relaxation of (\ref{POLY Eq: convex pop}) which is tight when $p$ has a second-order point.

\subsubsection{Combining the SDP in (\ref{POLY Eq: Small cubic SDP}) with its Dual}
In this subsection, we write down an SDP (given in (\ref{POLY Eq: complete cubic SDP})) whose optimal value can be related to the existence of second-order points of a cubic polynomial. To arrive at this SDP, we first take the dual of (\ref{POLY Eq: Small cubic SDP}). It will turn out that the constraints in the dual follow a very similar structure to those in the primal, and that any feasible solution of the primal yields a feasible solution of the dual. We then combine the primal-dual pair of SDPs to arrive at a single SDP, which is the one in (\ref{POLY Eq: complete cubic SDP}). To this end, let us write down the dual of (\ref{POLY Eq: Small cubic SDP}):
\begin{small}
\baeq
&\underset{R,S,r,s,\lambda,\sigma,\rho,\gamma}{\sup}
& & \gamma &&\\
& \text{subject to}
&& \frac{1}{6}\Tr(QY) + \frac{z}{3} - \gamma &=& \sum_{i=1}^n \lambda_i\left(\frac{1}{2}\Tr(H_iY)+e_i^TQy+b_i\right)\\
&&&&& + \Tr\left(\bmat Y & y \\ y^T & 1\emat \bmat R & r \\ r^T & \rho\emat \right)+ \Tr\left(T(Y,y,z) \bmat S & s \\ s^T & \sigma \emat\right), \forall (Y,y,z)\\
&&& \hspace{1.6cm} \bmat R & r \\ r^T & \rho\emat &\succeq& 0,\\
&&& \hspace{1.6cm} \bmat S & s \\ s^T & \sigma \emat &\succeq& 0,
\eaeq
\end{small}
where $R,S \in \Snn, r,s,\lambda \in \Rn,$ and $\sigma,\rho,\gamma \in \R$. The right-hand side of the first constraint simplifies to
$$b^T\lambda + \rho + \Tr(QS) +\Tr\left(\left(\sum_{i=1}^n (\frac{1}{2}\lambda_i+2s_i)H_i + R \right)Y\right) + \left(Q(\lambda+2s) + \sum_{i=1}^n \Tr(H_iS)e_i + 2r\right)^Ty + \sigma z.$$
After matching coefficients, the dual problem can be rewritten as
\begin{small}
\baeq
    & \underset{R,S,r,s,\lambda,\rho}{\sup}
	&& -b^T\lambda - \rho - \Tr(QS)\\
	& \text{subject to}
	&& \sum_{i=1}^n (\frac{1}{2}\lambda_i+2s_i)H_i + R = \frac{1}{6}Q,\\
	&&& Q(\lambda+2s) + \sum_{i=1}^n \Tr(H_iS)e_i + 2r= 0,\\
	&&& \bmat R & r \\ r^T & \rho\emat \succeq 0,\\
	&&& \bmat S & s \\ s^T & \frac{1}{3} \emat \succeq 0,
\eaeq
\end{small}
Substituting $R$ and $r$ using the first two constraints into the first psd constraint and then multiplying by 6, we arrive at the problem
\baeq
    & \underset{S, s, \lambda, \rho}{\sup}
	&& -b^T\lambda - \rho - \Tr(QS)&\\
	& \text{subject to}
	&& \bmat \sum_{i=1}^n (-3\lambda_i - 12s_i)H_i + Q & Q(-3\lambda-6s) - 3\sum_{i=1}^n \Tr(H_iS)e_i \\ \left(Q(-3\lambda-6s) - 3\sum_{i=1}^n \Tr(H_iS)e_i \right)^T & 6\rho \emat &\succeq 0,\\
	&&& \hspace{11.75cm} \bmat S & s \\ s^T & \frac{1}{3}\emat &\succeq 0.
\eaeq

    Replacing $S$ with $\frac{1}{3}S$, $s$ with $-\frac{1}{3}s$, and $\rho$ with $\frac{1}{6}\rho$, we can reparameterize this problem and arrive at our final form for the dual of (\ref{POLY Eq: Small cubic SDP}):
    \begin{equation}\label{POLY Eq: small cubic SDP dual}
    \begin{aligned}
        & \underset{S, s, \lambda, \rho}{\sup}
    	&& -b^T\lambda - \frac{1}{6}\rho - \frac{1}{3}\Tr(QS)&\\
    	& \text{subject to}
    	&& \bmat \sum_{i=1}^n (4s_i - 3\lambda_i)H_i + Q & Q(2s - 3\lambda) - \sum_{i=1}^n \Tr(H_iS)e_i \\ \left(Q(2s-3\lambda) - \sum_{i=1}^n \Tr(H_iS)e_i\right)^T & \rho \emat &\succeq 0,\\
    	&&& \hspace{10.65cm} \bmat S & s \\ s^T & 1\emat &\succeq 0.
    \eaeql
    
    One can easily verify that if $(Y,y,z)$ is feasible to (\ref{POLY Eq: Small cubic SDP}), then $(Y,y,y,z)$ is feasible to (\ref{POLY Eq: small cubic SDP dual}). Replacing $(S,s,\lambda,\gamma)$ with $(Y,y,y,z)$ in (\ref{POLY Eq: small cubic SDP dual}) gives an SDP whose constraints are the two psd constraints in (\ref{POLY Eq: Small cubic SDP}) and whose objective function is $-b^Ty - \frac{1}{6}z - \frac{1}{3}\Tr(QY)$. We now create a new SDP, which has the same decision variables and constraints as (\ref{POLY Eq: Small cubic SDP}), but whose objective function is the difference between the objective function of (\ref{POLY Eq: Small cubic SDP}) and $-b^Ty - \frac{1}{6}z - \frac{1}{3}\Tr(QY)$. The optimal value of this new SDP is an upper bound on the duality gap of the primal-dual SDP pair (\ref{POLY Eq: Small cubic SDP}) and (\ref{POLY Eq: small cubic SDP dual}). If our cubic polynomial $p$ is written in the form (\ref{POLY Eq: Cubic Poly Form}) and $$T(Y,y,z) = \bmat \sum_{i=1}^n y_iH_i + Q & \sum_{i=1}^n \Tr(H_iY)e_i+Qy \\ (\sum_{i=1}^n \Tr(H_iY)e_i+Qy)^T & z\emat$$
    as before, the new SDP we just described can be written as

    \begin{equation}\label{POLY Eq: complete cubic SDP}
        \begin{aligned}
    	& \underset{Y \in \Snn, y \in \Rn, z \in \R}{\inf}
    	& & \frac{1}{2}\Tr(QY) + b^Ty + \frac{z}{2}\\
    	& \text{subject to}
    	&& \frac{1}{2}\Tr(H_iY)+e_i^TQy+b_i=0, \forall i = 1, \ldots, n,\\
    	&&& T(Y,y,z) \succeq 0,\\
    	&&& \bmat Y & y \\ y^T & 1 \emat \succeq 0.
    \eaeql
The following theorem relates the optimal value of this SDP to the existence of second-order points of $p$.

\begin{theorem}\label{POLY Thm: Complete Cubic SDP}
    For a cubic polynomial $p$ given in the form (\ref{POLY Eq: Cubic Poly Form}), consider the SDP in (\ref{POLY Eq: complete cubic SDP}). For any feasible solution $(Y,y,z)$ to (\ref{POLY Eq: complete cubic SDP}), the objective value of (\ref{POLY Eq: complete cubic SDP}) is nonnegative. Furthermore, the optimal value of (\ref{POLY Eq: complete cubic SDP}) is zero and is attained if and only if $p$ has a second-order point.
\end{theorem}
\begin{proof}
Suppose $(Y,y,z)$ is a feasible solution to (\ref{POLY Eq: complete cubic SDP}). Note that $(Y,y,z)$ is feasible to (\ref{POLY Eq: Small cubic SDP}) and $(Y,y,y,z)$ is feasible to (\ref{POLY Eq: small cubic SDP dual}), and so
$$\frac{1}{2}\Tr(QY) + b^Ty + \frac{z}{2} = \frac{1}{6}\Tr(QY) + \frac{z}{3} - \left(-b^Ty - \frac{1}{6}z - \frac{1}{3}\Tr(QY)\right) \ge 0$$
by weak duality applied to (\ref{POLY Eq: Small cubic SDP}) and (\ref{POLY Eq: small cubic SDP dual}). Hence, the objective of (28) is nonnegative at any feasible solution.

Now suppose that $p$ has a second-order point $\xbar$. We claim that the triplet $$\left(\xbar\xbar^T, \xbar, \xbar^T\left(\sum_{i=1}^n \xbar_iH_i + Q\right)\xbar\right)$$ is feasible to (\ref{POLY Eq: complete cubic SDP}) and achieves an objective value of zero. Indeed, the first constraint of (\ref{POLY Eq: complete cubic SDP}) is satisfied because its left-hand side reduces to $\gxb$, which is zero. The third constraint is satisfied since the matrix $(\xbar, 1)(\xbar, 1)^T$ is clearly psd. The second constraint is satisfied since $T(\xbar\xbar^T, \xbar, \xbar^T(\sum_{i=1}^n \xbar_iH_i + Q)\xbar)$ can be written as
\begin{small}
$$\bmat \sum_{i=1}^n \xbar_iH_i + Q & (\sum_{i=1}^n \xbar_iH_i+Q)\xbar \\ \xbar^T(\sum_{i=1}^n \xbar_iH_i+Q) & \xbar^T(\sum_{i=1}^n \xbar_iH_i + Q)\xbar\emat = \bmat (\sum_{i=1}^n \xbar_iH_i + Q)^{\frac{1}{2}} \\ \xbar^T(\sum_{i=1}^n \xbar_iH_i + Q)^{\frac{1}{2}}\emat \bmat (\sum_{i=1}^n \xbar_iH_i + Q)^{\frac{1}{2}} \\ \xbar^T(\sum_{i=1}^n \xbar_iH_i + Q)^{\frac{1}{2}}\emat^T.$$\end{small} The objective value at $(\xbar\xbar^T, \xbar, \xbar^T(\sum_{i=1}^n \xbar_iH_i + Q)\xbar)$ is
\begin{align*}
    &\frac{1}{2}\Tr(Q\xbar\xbar^T) + b^T\xbar + \frac{1}{2}\xbar^T\left(\sum_{i=1}^n \xbar_iH_i + Q\right)\xbar\\
    =& \frac{1}{2}\xbar^TQ\xbar - \left(\frac{1}{2} \sum_{i=1}^n \xbar_iH_i\xbar +Q\xbar\right)^T \xbar + \frac{1}{2}\xbar^T\left(\sum_{i=1}^n \xbar_iH_i + Q\right)\xbar\\
    =& 0.
\end{align*}
Since we have already shown that the objective function of (\ref{POLY Eq: complete cubic SDP}) is nonnegative over its feasible set, it follows that when $p$ has a second-order point, the optimal value of (\ref{POLY Eq: complete cubic SDP}) is zero and is attained.

To prove the converse, suppose the optimal value of (\ref{POLY Eq: complete cubic SDP}) is zero and is attained. Let $(Y^*, y^*, z^*)$ be an optimal solution to (\ref{POLY Eq: complete cubic SDP}). We will show that $y^*$ is a second-order point for $p$. Clearly $\Hess p(y^*)$ is psd, since $T(Y^*, y^*, z^*) \succeq 0$. To show that $\grad p(y^*) = 0$, let us start by letting $D \defeq Y^* - y^*y^{*T}$, and $d \defeq \sum_{i=1}^n \Tr(H_iD)e_i$. Note that
$$\frac{1}{2}\Tr(H_iy^*y^{*T}) + \frac{1}{2}\Tr(H_iD) + e_i^TQy^* + b_i = 0 \overset{(\ref{POLY Eq: Cubic Poly Form})}{\Rightarrow} -2(\grad p(y^*))_i = \Tr(H_iD),$$
or equivalently $d = -2\grad p(y^*)$. In the remainder of the proof, we show that $d = 0$.

Since $\sum_{i=1}^n y_i^*H_iy^*$ is the vector whose $i$-th entry is $y^{*T}H_iy^*$, we have that 
\begin{equation}\label{POLY Eq: Top right vector decomposition}
\sum_{i=1}^n \Tr(H_iY^*)e_i+Qy^*= \left(\sum_{i=1}^n y_i^*H_i + Q\right) y^* + d.
\end{equation} Then from the generalized\footnote{Here, $A^{+}$ refers to any pseudo-inverse of $A$, i.e. a matrix satisfying $AA^{+}A = A$.} Schur complement condition applied to $T(Y^*,y^*,z^*)$, we have
\begin{align*}
    z^* &\ge \left(\left(\sum_{i=1}^n y_i^*H_i + Q\right)y^* + d\right)^T \left(\sum_{i=1}^n y_i^*H_i + Q\right)^{+} \left(\left(\sum_{i=1}^n y_i^*H_i + Q\right)y^* + d\right)\\
    &= y^{*T}\left(\sum_{i=1}^n y_i^*H_i + Q\right)y^* + 2d^T\left(\sum_{i=1}^n y_i^*H_i + Q\right)^{+}\left(\sum_{i=1}^n y_i^*H_i + Q\right)y^* + d^T\left(\sum_{i=1}^n y_i^*H_i + Q\right)^{+}d.
\end{align*}
It is not difficult to verify that since $T(Y^*,y^*,z^*) \succeq 0$, we have $$\sum_{i=1}^n \Tr(H_iY^*)e_i+Qy^* \in \cols(\sum_{i=1}^n y_i^*H_i + Q),$$ and thus (\ref{POLY Eq: Top right vector decomposition}) implies $d \in \cols(\sum_{i=1}^n y_i^*H_i + Q)$. Therefore, there exists a vector $v \in \Rn$ such that $d = (\sum_{i=1}^n y_i^*H_i + Q) v$. We then have
\begin{align*}
    &d^T\left(\sum_{i=1}^n y_i^*H_i + Q\right)^{+}\left(\sum_{i=1}^n y_i^*H_i + Q\right) y^*\\
    =& v^T \left(\sum_{i=1}^n y^*_iH_i + Q\right) \left(\sum_{i=1}^n y^*_iH_i + Q\right)^{+} \left(\sum_{i=1}^n y^*_iH_i + Q\right) y^*\\
    =& v^T \left(\sum_{i=1}^n y_i^*H_i + Q\right) y^*\\
    =& d^T y^*.
\end{align*}
Now let $$\delta \defeq z^* - y^{*T}\left(\sum_{i=1}^n y_i^*H_i + Q\right)y^* - 2d^Ty^* - d^T\left(\sum_{i=1}^n y_i^*H_i + Q\right)^{+}d$$ and observe that $\delta \ge 0$. We can then write the objective value of (\ref{POLY Eq: complete cubic SDP}) at $(Y^*,y^*,z^*)$ in terms of $D, d$, and $\delta$:

\begin{equation}\label{POLY Eq: Complete SDP nonnegative}
\begin{aligned}
    &\ \frac{1}{2}\Tr(QY^*) + b^Ty^* + \frac{1}{2}z^*\\
	=&\ \frac{1}{2}\left(y^{*T}Qy^* + \Tr(QD)\right) + \sum_{i=1}^n \left(-e_i^TQ y^*- \frac{1}{2}\Tr(H_iy^*y^{*T}) - \frac{1}{2}\Tr(H_iD)\right)y_i^*\\
	&+ \frac{1}{2}\left(y^{*T}\left(\sum_{i=1}^n y_i^*H_i + Q\right)y^* + 2d^Ty^* + d^T\left(\sum_{i=1}^n y_i^*H_i + Q\right)^{+}d+ \delta\right)\\
	=&\ \left(\frac{1}{2}-1+\frac{1}{2}\right)y^{*T}Qy^* + \left(-\frac{1}{2} + \frac{1}{2}\right)\sum_{i=1}^n y^{*T}y_i^*H_iy^*\\
	&+ \frac{1}{2}\Tr(QD) + \left(-\frac{1}{2}+1\right)\sum_{i=1}^n\Tr(H_iD)y_i^* + \frac{1}{2}d^T\left(\sum_{i=1}^n y_i^*H_i + Q\right)^{+}d + \frac{\delta}{2}\\
	=&\ \frac{1}{2}\Tr\left(\left(\sum_{i=1}^n y_i^*H_i + Q\right)D\right) + \frac{1}{2}d^T\left(\sum_{i=1}^n y_i^*H_i + Q\right)^{+}d + \frac{\delta}{2}\\
	 \ge &\ 0,
\end{aligned}
\end{equation}
where in the last inequality we used the facts that $D \succeq 0$ and that the pseudo-inverse of a psd matrix is psd.
    
Since the left-hand side of the above equation is zero by assumption, and since all three terms on the right-hand side are nonnegative, it follows that $(\sum_{i=1}^n y_i^*H_i + Q)^+ d = 0$. As the null space of $(\sum_{i=1}^n y_i^*H_i + Q)^+$ is the same as the null space of $(\sum_{i=1}^n y_i^*H_i + Q)$, we have $(\sum_{i=1}^n y_i^*H_i + Q)d = 0$. However, because $d \in \cols(\sum_{i=1}^n y_i^*H_i + Q)$, it must be that $d=0$.

\end{proof}

\subsubsection{An Algorithm for Finding Local Minima}\label{POLY SSSec: SOP and Algorithm}

Theorem \ref{POLY Thm: Complete Cubic SDP} leads to the following characterization of second-order points of a cubic polynomial.

\begin{cor}\label{POLY Cor: Complete Cubic SDP SOP}
    Let $p: \Rn \to \R$ be a cubic polynomial written in the form (\ref{POLY Eq: Cubic Poly Form}). Then the set of its second-order points is equal to
    \beq\label{POLY Eq: Complete Cubic SDP SOP}
    \begin{aligned}\{y \in \Rn\ |\ & \exists Y \in \Snn,z \in \R \text{ such that }\\
    &\frac{1}{2}\Tr(QY) + b^Ty + \frac{z}{2} = 0, \frac{1}{2}\Tr(H_iY)+e_i^TQy+b_i = 0, \forall i = 1, \ldots, n,\\
    &T(Y,y,z) \succeq 0, \bmat Y&y\\y^T&1\emat \succeq 0\}.
    \end{aligned}\eeq
\end{cor}
\begin{proof}
Recall from the proof of Theorem \ref{POLY Thm: Complete Cubic SDP} that if $\xbar$ is a second-order point of $p$, then the triplet $(\xbar\xbar^T, \xbar, \xbar^T(\sum_{i=1}^n \xbar_iH_i + Q)\xbar)$ is feasible solution to (\ref{POLY Eq: complete cubic SDP}) with objective value zero. Hence any second-order point belongs to (\ref{POLY Eq: Complete Cubic SDP SOP}). Conversely, recall that if $(Y,y,z)$ is a feasible solution to (\ref{POLY Eq: complete cubic SDP}) with objective value zero, then $y$ is a second-order point of $p$. Therefore any point in (\ref{POLY Eq: Complete Cubic SDP SOP}) is a second-order point of $p$.
\end{proof}

In view of Theorem \ref{POLY Thm: Closure}, we observe that if $p$ has a local minimum, the set in (\ref{POLY Eq: Complete Cubic SDP SOP}) is a semidefinite representation of $\overline{LM_p}$. This observation gives rise to the following algorithm which tests if a cubic polynomial has a local minimum.

\begin{algorithm}[H]
	\caption{Algorithm for finding a local minimum of a cubic polynomial using a polynomial number of calls to E-SDP.}\label{POLY Alg: Complete Cubic SDP}
	\begin{algorithmic}[1]
		\State {\bf Input:} A cubic polynomial $p: \Rn \to \R$ in the form (\ref{POLY Eq: Cubic Poly Form})
		\State \texttt{{\bf TEST1}} test using E-SDP if (\ref{POLY Eq: Complete Cubic SDP SOP}) is empty
		\State \quad \texttt{{\bf if}} YES
			\State \quad \quad \Return NO LOCAL MINIMUM
		\State \quad \texttt{{\bf if}} NO
			\State \quad\quad Find (via Lemma \ref{POLY Lem: Relint recovery}) a point $x^*$ in the relative interior of (\ref{POLY Eq: Complete Cubic SDP SOP})

        \mbox{}
		\State \texttt{{\bf TEST2}} test (via Theorem \ref{POLY Thm: Checking Min Poly Time}) if $x^*$ is a local minimum
		\State \quad \texttt{{\bf if}} YES
			\State \quad \quad \Return $x^*$
		\State \quad \texttt{{\bf if}} NO
			\State \quad\quad \Return NO LOCAL MINIMUM
	\end{algorithmic}
\end{algorithm}

{\bf Complexity and correctness of Algorithm \ref{POLY Alg: Complete Cubic SDP}.} By design, if $p$ has no local minimum, Algorithm \ref{POLY Alg: Complete Cubic SDP} will return \texttt{NO LOCAL MINIMUM} since \texttt{TEST2} answers \texttt{NO} for every point. If $p$ has a local minimum, then $SO_p$ is nonempty. Since $SO_p$ is given by (\ref{POLY Eq: Complete Cubic SDP SOP}) due to Corollary~\ref{POLY Cor: Complete Cubic SDP SOP}, \texttt{TEST1} answers \texttt{YES}. Then, by Theorem \ref{POLY Thm: Local Minima SOP}, any point in the relative interior of (\ref{POLY Eq: Complete Cubic SDP SOP}) is a local minimum. Hence $x^*$ will pass \texttt{TEST2}. Note that this algorithm makes $2n+1$ calls to E-SDP, and then runs Algorithm~\ref{POLY Alg: Check Local Min}.\footnote{In fact, the number of calls to E-SDP can be reduced to $2n$ if the very first call to E-SDP uses $x_1$ as the objective function.}

\begin{remark}\label{POLY Rem: Find SLM}
{\bf Finding strict local minima.}
If we are specifically interested in searching for a strict local minimum of a cubic polynomial, we can simply check if the point $x^*$ returned by Algorithm~\ref{POLY Alg: Complete Cubic SDP} satisfies $\Hess p(x^*) \succ 0$. If the answer is yes, we return $x^*$; if the answer is no, we declare that $p$ has no strict local minimum. Clearly, if a local minimum $x^*$ satisfies $\Hess p(x^*) \succ 0$, it must be a strict local minimum due to the SOSC. Furthermore, recall from Section \ref{POLY SSec: Cubic Preliminaries} that if $p$ has a strict local minimum, then it has a unique local minimum, and thus that must be the output of Algorithm \ref{POLY Alg: Complete Cubic SDP}.
\end{remark}

\section{Conclusions and Future Directions}\label{POLY Sec: Conclusions}

In this chapter, we considered the notions of (i) critical points, (ii) second-order points, (iii) local minima, and (iv) strict local minima for multivariate polynomials. For each type of point, and as a function of the degree of the polynomial, we studied the complexity of deciding (1) if a given point is of that type, and (2) if a polynomial has a point of that type. See Tables~\ref{POLY Table: Complexity Checking} and \ref{POLY Table: Complexity Existence} in Section~\ref{POLY Sec: Introduction} for a summary of how our results complement prior literature. The majority of our work was dedicated to the case of cubic polynomials, where some new tractable cases were revealed based in part on connections with semidefinite programming. In this final section, we outline two future research directions which also have to do with cubic polynomials.

\subsection{Approximate Local Minima}\label{POLY SSec: Approximate Local Minima}

In Sections~\ref{POLY Sec: Complexity} and \ref{POLY Sec: Finding Local Min}, we established polynomial-time equivalence of finding local minima and second-order points of cubic polynomials and some SDP feasibility problems (see Corollary~\ref{POLY Cor: Complete Cubic SDP SOP}, Algorithm~\ref{POLY Alg: Complete Cubic SDP}, Theorem~\ref{POLY Thm: SOP SDPF}, Theorem~\ref{POLY Thm: LM SDPSF}). Unless some well-known open problems around the complexity of SDP feasibility are resolved (see Section~\ref{POLY Sec: Complexity}), one cannot expect to make claims about finding local minima of cubic polynomials in polynomial time in the Turing model of computation. Nonetheless, it is known that under some assumptions, one can solve semidefinite programs to arbitrary accuracy in polynomial time (see, e.g. \cite{ramana1997exact,alizadeh1995interior,vandenberghe1996semidefinite,porkolab1997complexity,nesterov1994interior,grotschel2012geometric}). It is therefore reasonable to ask if one can find local minima of cubic polynomials to arbitrary accuracy in polynomial time. This is a question we would like to study more rigorously in future work. We present a partial result in this direction in Theorem \ref{POLY Thm: Eps local min} below.

Recall from Section \ref{POLY SSec: SDP Approach} that our ability to find local minima of a cubic polynomial $p$ depended on our ability to minimize $p$ over its convexity region $CR_p$. We show next that we can find an $\epsilon$-minimizer of $p$ over $CR_p$ by approximately solving a semideifnite program.

\begin{theorem}\label{POLY Thm: Eps local min}
    For a cubic polynomial $p$ given in the form (\ref{POLY Eq: Cubic Poly Form}), consider the SDP in (\ref{POLY Eq: complete cubic SDP}). If the objective value at a feasible point $(Y,y,z)$ is $\epsilon \ge 0$, then $p(y) \le p(x) + \frac{2}{3}\epsilon$, ${\forall x \in CR_p}$.
\end{theorem}
\begin{proof}
Consider a feasible solution $(Y,y,z)$ to (\ref{POLY Eq: complete cubic SDP}). Let $\gamma^*$ be the infimum of $p$ over $CR_p$. Observe that
$$-\frac{1}{6}\Tr(QY) - \frac{z}{3} \le \gamma^*.$$
This is because the SDPs in (\ref{POLY Eq: complete cubic SDP}) and (\ref{POLY Eq: Small cubic SDP}) have the same constraints, and the optimal value of (\ref{POLY Eq: Small cubic SDP}) is the negative of the optimal value of (\ref{POLY Eq: new cubic sdp}), which by construction is a lower bound on $\gamma^*$. Similarly as in the proof of Theorem \ref{POLY Thm: Complete Cubic SDP}, let $D \defeq Y - yy^T$, ${d \defeq \sum_{i=1}^n \Tr(H_iD)e_i}$, and
$$\delta \defeq z - y^T\left(\sum_{i=1}^n y_iH_i + Q\right)y - 2d^Ty - d^T\left(\sum_{i=1}^n y_iH_i + Q\right)^{+}d.$$ We can then write:
	\begin{align*}
	\frac{1}{6}\Tr(QY) + \frac{z}{3} &= \frac{1}{6}\Tr(QY) + \frac{z}{3} - \sum_{i=1}^n \left(\frac{1}{2} \Tr(H_iY) + e_i^TQy + b_i\right)y_i\\
	&= \frac{1}{6}\left(\Tr(Qyy^T) + \Tr(QD)\right)\\
	&+\frac{1}{3}\left(y^T\left(\sum_{i=1}^n y_iH_i + Q\right)y + 2d^Ty + d^T\left(\sum_{i=1}^n y_iH_i + Q\right)^{+}d + \delta \right)\\
	&- \frac{1}{2}\left(\Tr\left(\sum_{i=1}^n y_iH_iyy^T\right) + \Tr\left(\sum_{i=1}^n y_iH_iD\right)\right) - y^TQy - b^Ty\\
	&= -\frac{1}{6}\sum_{i=1}^n y^Ty_iH_iy - \frac{1}{2}y^TQy - b^Ty\\
	&+ \frac{1}{6}\Tr\left(\left(\sum_{i=1}^n y_iH_i + Q\right)D\right) + \frac{1}{3}\left(d^T\left(\sum_{i=1}^n y_iH_i + Q\right)^{+}d\right) + \frac{\delta}{3}\\
	&= -p(y) + \frac{1}{6}\Tr\left(\left(\sum_{i=1}^n y_iH_i + Q\right)D\right) + \frac{1}{3}\left(d^T\left(\sum_{i=1}^n y_iH_i + Q\right)^{+}d\right) + \frac{\delta}{3}\\
	&\le -p(y) + \frac{2}{3}\epsilon,
	\end{align*}
	where the first equality is due to the first constraint in (\ref{POLY Eq: complete cubic SDP}), and the last inequality follows from the last equation of (\ref{POLY Eq: Complete SDP nonnegative}) with $(Y^*,y^*,z^*)$ replaced by $(Y,y,z)$ and the fact that $\sum_{i=1}^n y_iH_i+Q$ and $D$ are both psd matrices. We therefore conclude that
	$$p(y) - \frac{2}{3}\epsilon \le -\frac{1}{6}\Tr(QY) - \frac{z}{3} \le \gamma^*.$$
	 We then have that $p(y) \le p(x) + \frac{2}{3}\epsilon, \forall x \in CR_p$ as desired.
	 
\end{proof}

\subsection{Unregularized Third-Order Newton Methods}\label{POLY SSec: Cubic Newton}

We end our chapter with an interesting application of the problem of finding a local minimum of a cubic polynomial. Recall that Newton's method for minimizing a twice-differentiable function proceeds by approximating the function with its second-order Taylor expansion at the current iterate, and then moving to a critical point\footnote{If the function to be minimized is convex, this critical point will be a global minimum of the quadratic approximation.} of this quadratic approximation. It is natural to ask whether one can lower the iteration complexity of Newton's method for three-times-differentiable functions by using third-order information. An immediate difficulty, however, is that the third-order Taylor expansion of a function around any point will not be bounded below (unless the coefficients of all its cubic terms are zero). In previous work (see, e.g. \cite{nesterov2019implementable}), authors have gotten around this issue by adding a regularization term to the third-order Taylor expansion. In future work, we aim to study an unregularized third-order Newton method which in each iteration moves to a local minimum of the third-order Taylor approximation by applying Algorithm \ref{POLY Alg: Complete Cubic SDP}. We would like to explore the convergence properties of this algorithm and conditions under which the algorithm is well defined at every iteration.

As a first step, let us consider the univariate case. For a function $f: \R \to \R$, the iterations of (classical) Newton's method read
\begin{equation}\label{POLY Eq: Second Order Iterates}
x_{k+1} = x_k - \frac{f'(x_k)}{f''(x_k)}.
\end{equation}
The update rule of a third-order Newton method, which in each iteration moves to the local minimum of the third-order Taylor approximation, is given by
\begin{equation}\label{POLY Eq: Third Order Iterates}
x_{k+1} = x_k - \frac{f''(x_k) - \sqrt{f''(x_k)^2 - 2f'(x_k)f'''(x_k)}}{f'''(x_k)}.
\end{equation}
We have already observed that in some settings, these iterations can outperform the classical Newton iterations. For example, consider the univariate function
\begin{equation}\label{POLY Eq: 363 Function}
f(x) = 20x\arctan(x) - 10 \log(1+x^2) + x^2,
\end{equation}
which is strongly convex and has a (unique) global minimum at $x = 0$, where $f(x) = 0$; see Figure~\ref{POLY Fig: Newton Iterates}. The first three derivatives of this function are
$$f'(x) = 20\arctan(x) + 2x,$$
$$f''(x) = 2+\frac{20}{1+x^2},$$
$$f'''(x) = \frac{-40x}{(1+x^2)^2}.$$
One can show that the basin of attraction of the global minimum of $f$ under the classical Newton iterations in (\ref{POLY Eq: Second Order Iterates}) is approximately $[-1.7121, 1.7121]$. Starting Newton's method with $|x_0| \ge 1.7122$ results in the iterates eventually oscillating between $\pm 13.4942$. In contrast, the iterates of our proposed third-order Newton method in (\ref{POLY Eq: Third Order Iterates}) are globally convergent to the global minimum of $f$. The iterations of both methods starting at $x_0 = 1.5$ are compared in Table \ref{POLY Tab: Newton Iterates} and Figure \ref{POLY Fig: Newton Iterates}, showing faster convergence to the global minimum for the third-order approach.

\begin{table}[H]
	{\begin{tabular}{| c | c | c |}\hline
         $k$ & $x_k$ & $f(x_k)$\\\hline
         0 &  1.5 & 19.9473\\\hline
         1 & -.2327 & .5910\\\hline
         2 & -.0030 & 1.0014e-4\\\hline
         3 & -8.3227e-9 & 1.4546e-15\\\hline
         4 & 2.3490e-9 & 1.1587e-16\\\hline
    \end{tabular}
	\hspace{1cm}
	\begin{tabular}{| c | c | c |}\hline
         $k$ & $x_k$ & $f(x_k)$\\\hline
         0 &  1.5 & 19.9473\\\hline
         1 & -1.2786 & 15.1411\\\hline
         2 & .8795 & 7.7329\\\hline
         3 & -.3396 & 1.2477\\\hline
         4 & .0230 & .0058\\\hline
    \end{tabular}}
	\centering
	\caption{Iterations of the third-order Newton method (left) and the classical Newton method (right) on the function $f$ in (\ref{POLY Eq: 363 Function}) starting at $x_0 = 1.5$.}
	\label{POLY Tab: Newton Iterates}
\end{table}

\begin{figure}[H]
	\centering
	\includegraphics[height=.4\textheight,keepaspectratio]{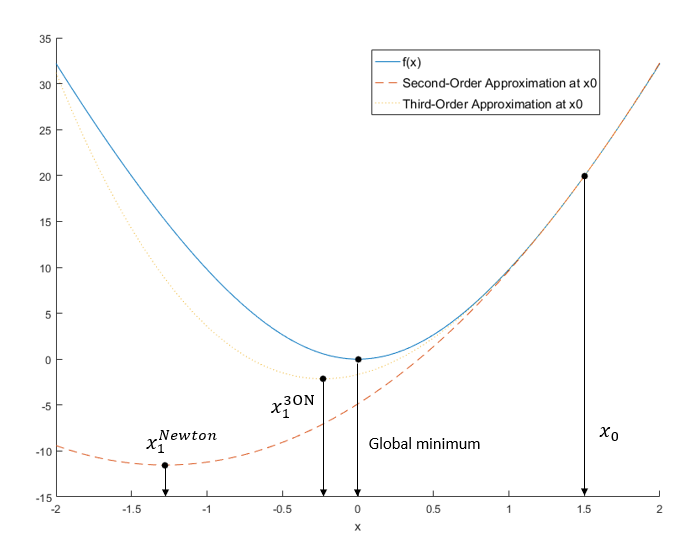}
	\caption{The plots of the function $f$ in (\ref{POLY Eq: 363 Function}) and its second and third-order Taylor expansions around $x_0 = 1.5$. One can see that one iteration of the third-order Newton method in (\ref{POLY Eq: Third Order Iterates}) brings us closer the global minimum of $f$ compared to one iteration of the Newton method in (\ref{POLY Eq: Second Order Iterates}).} 
	\label{POLY Fig: Newton Iterates}
\end{figure}

In addition to potential benefits regarding convergence, we have also observed that the behavior of the algorithm can be less sensitive to the initial condition when compared to Newton's method. As an example, we used Newton's method to find the critical points $\{1,-1,i,-i\}$ of $f(x) = x^5 - 5x$ on the complex plane, using the iterates (\ref{POLY Eq: Second Order Iterates}), (\ref{POLY Eq: Third Order Iterates}), and iterates given by
\begin{equation}\label{POLY Eq: Third Order Iterates Max}
x_{k+1} = x_k - \frac{f''(x_k) + \sqrt{f''(x_k)^2 - 2f'(x_k)f'''(x_k)}}{f'''(x_k)},
\end{equation}
which can be interpreted as the iterates for moving to the local maximum of a third-order approximation of $f$. For each of the three iterations, the plots below demonstrate which initial conditions converge to the same critical point. As can be seen, sensitivity of Newton's method to the initial condition demonstrates fractal behavior, while the third-order iterates do not.

\begin{figure}[H]
	\centering
	\includegraphics[height=.25\textheight,keepaspectratio]{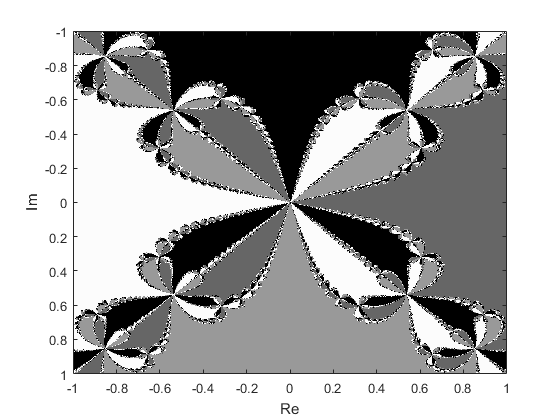}\\
	\includegraphics[height=.25\textheight,keepaspectratio]{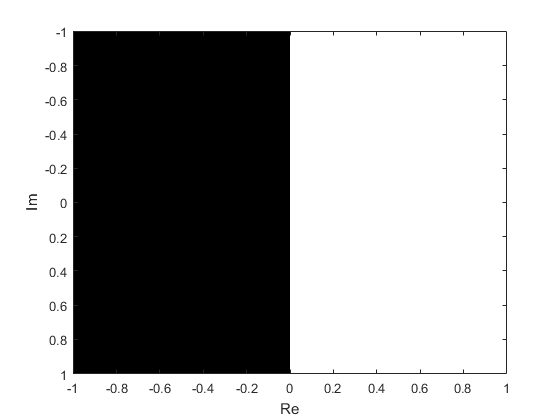}
	\includegraphics[height=.25\textheight,keepaspectratio]{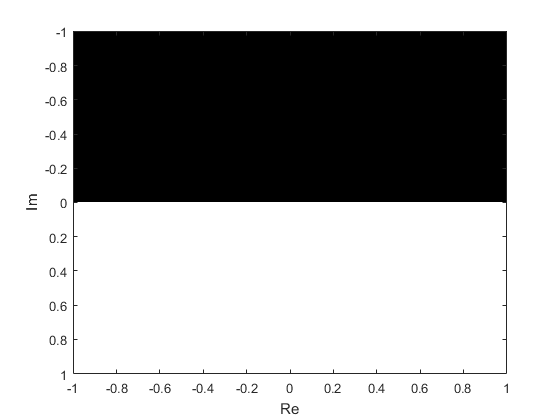}
	\caption{Sensitivity of the limits of the iterates (\ref{POLY Eq: Second Order Iterates}), (\ref{POLY Eq: Third Order Iterates}), and (\ref{POLY Eq: Third Order Iterates Max}) respectively to initial conditions. Regions with the same color denote initial conditions which converge to the same critical point.} 
	\label{POLY Fig: Newton Fractals}
\end{figure}
\chapter{On the Complexity of Finding a Local Minimizer of a Quadratic Function over a Polytope}\label{Chap: QPs}

\section{Introduction}\label{QPs Sec: Introduction}

In this chapter of the thesis, we consider \emph{quadratic programs}, which are polynomial optimization problems of the form (\ref{Eq: Polynomial Optimization Problem}) where the objective function $p$ is quadratic and all constraint functions $q_i$ are affine. Recall that a \emph{local minimum} of a function $f: \Rn \to \R$ over a set $\Omega \subseteq \Rn$ is a point $\xbar \in \Omega$ for which there exists a scalar $\epsilon > 0$ such that $p(\xbar) \le p(x)$ for all $x \in \Omega$ with $\|x - \xbar\| \le \epsilon$. In the case where $p$ and all the constraint functions $q_i$ are affine (i.e., linear programming), it is well known that a local minimum (which also has to be a global minimum) can be found in polynomial time in the Turing model of computation~\cite{karmarkar1984new,khachiyan1979polynomial}. Perhaps the next simplest constrained optimization problems to consider are quadratic programs, which can be written as
\beq\label{QPs Eq: Polynomial Optimization Problem}
\begin{aligned}
& \underset{x \in \Rn}{\min}
& & x^TQx + c^Tx \\
& \text{subject to}
&& a_i^Tx \le b_i, \forall i \in \otm,\\
\end{aligned}
\eeq
where $Q \in \R^{n \times n}$, $c, a_1, \ldots, a_m \in \Rn$, and $b_1, \ldots, b_m~\in~\R$. The matrix $Q$ is taken without loss of generality to be symmetric. When complexity questions about quadratic programs are studied in the Turing model of computation, all these data are rational and the input size is the total number of bits required to write them down. It is well known that finding a global minimum of a quadratic program is NP-hard, even when the matrix $Q$ has a single negative eigenvalue~\cite{pardalos1991quadratic}. It is therefore natural to ask whether one can instead find a local minimum of a quadratic program efficiently. In fact, this precise question appeared in 1992 on a list of seven open problems in complexity theory for numerical optimization \cite{pardalos1992open}:
\begin{quote}
    \emph{``What is the complexity of finding even a local minimizer for nonconvex quadratic programming, assuming the feasible set is compact? Murty and Kabadi (1987, \cite{murty1987some}) and Pardalos and Schnitger (1988, \cite{pardalos1988checking}) have shown that it is NP-hard to test whether a given point for such a problem is a local minimizer, but that does not rule out the possibility that another point can be found that is easily verified as a local minimizer.''}
\end{quote}
A few remarks on the phrasing of this problem are in order. First, note that in this question, the feasible set of the quadratic program is assumed to be compact (i.e., a polytope). Therefore, there is no need to focus on the related and often prerequisite problem of deciding the \emph{existence} of a local minimum (since any global minimum e.g. is a local minimum). The latter question makes sense in the case where the feasible set of the quadratic program is unbounded; the complexity of this question is also addressed in this chapter (Theorem~\ref{QPs Thm: LM QP NP hard}). Second, as the quote points out, the question of finding a local minimum is also separate from a complexity viewpoint from that of testing if a given point is a local minimum. This related question has been studied more extensively and its complexity has already been settled for optimization problems whose objective and constraints are given by polynomial functions of any degree; see~\cite{murty1987some,pardalos1988checking,cubicpaper}.

To point out some of the subtle differences between these variations of the problem more specifically, we briefly review the reduction of Murty and Kabadi \cite{murty1987some}, which shows the NP-hardness of deciding if a given point is a local minimum of a quadratic program. In~\cite{murty1987some}, the authors show that the problem of deciding if a symmetric matrix $Q$ is \emph{copositive}---i.e. whether $x^TQx \ge 0$ for all vectors $x$ in the nonnegative orthant---is NP-hard. From this, it is straightforward to observe that the problem of testing whether a given point is a local minimum of a quadratic function over a polyhedron is NP-hard: Indeed, the origin is a local minimum of $x^TQx$ over the nonnegative orthant if and only if the matrix $Q$ is copositive. However, it is not true that $x^TQx$ has a local minimum over the nonnegative orthant if and only if $Q$ is copositive. Although the ``if'' direction holds, the ``only if'' direction does not. For example, consider the matrix
$$Q = \bmat 0 & 1 \\ 1 & -2\emat,$$ which is clearly not copositive, even though the point $(1,0)^T$ is a local minimum of $x^TQx$ over the nonnegative orthant.


Our main results in this chapter are as follows. We show that unless P=NP, no polynomial-time algorithm can find a point within Euclidean distance $c^n$ (for any constant $c \ge 0$) of a local minimum of an $n$-variate quadratic program with a bounded feasible set (Theorem~\ref{QPs Thm: Finding LM NP-hard}). See also Corollaries~\ref{QPs Cor: QP Pseudo} and \ref{QPs Cor: 2n NP-hard}. To prove this, we show as an intermediate step that deciding whether a quartic polynomial or a quadratic program has a local minimum is strongly NP-hard\footnote{This implies that these
problems remain NP-hard even if the bitsize of all numerical data are $O(\log n)$, where $n$ is the number of variables. For a
strongly NP-hard problem, even a pseudo-polynomial time algorithm---i.e., an algorithm whose running time is polynomial in the magnitude of the numerical data of the problem but not necessarily in their bitsize---cannot exist unless P=NP.  See \cite{garey2002computers} or \cite[Section 2]{ahmadi2019complexity} for more details.} (Theorems~\ref{QPs Thm: LM Degree 4 NP hard} and \ref{QPs Thm: LM QP NP hard}). Finally, we show that unless P=NP, there cannot be a polynomial-time algorithm that decides if a quadratic program with a bounded feasible set has a unique local minimum and if so returns this minimum (Theorem~\ref{QPs Thm: UniqueQP NP Hard}).

Overall, our results suggest that without additional problem structure, questions related to finding local minima of quadratic programs are not easier (at least from a complexity viewpoint) than those related to finding global minima. It also suggests that any efficient heuristic that aims to find a local minimum of a quadratic program must necessarily fail on a ``significant portion'' of instances; see e.g. Corollary 2.2 of \cite{hemaspaandra2012sigact} for a more formal complexity theoretic statement.

\subsection{Notation and Basic Definitions}\label{QPs SSec: Notation}

For a vector $x \in \Rn$, the notation $x^2$ denotes the vector in $\Rn$ whose $i$-th entry is $x_i^2$, and $diag(x)$ denotes the diagonal $n \times n$ matrix whose $i$-th diagonal entry is $x_i$. The notation $x \ge 0$ denotes that the vector $x$ belongs to the nonnegative orthant, and for such a vector, $\sqrt{x}$ denotes the vector in $\Rn$ whose $i$-th entry is $\sqrt{x_i}$. For two matrices $X, Y \in \R^{m \times n}$, we denote by $X \cdot Y$ the matrix in $\R^{m \times n}$ whose $(i,j)$-th entry is $X_{ij}Y_{ij}$. For vectors $x, y \in \Rn$, the notation $y_x$ (sometimes $(y)_x$ if there is room for confusion with other indices) denotes the vector containing the entries of $y$ where $x_i$ is nonzero in the same order (the length of $y_x$ is hence equal to the number of nonzero entries in $x$). Similarly, for a vector $x \in \Rn$ and a matrix $Y \in \R^{n \times n}$, the notation $Y_x$ (sometimes $(Y)_x$ if there is room for confusion with other indices) denotes the principal submatrix of $Y$ consisting of rows and columns of $Y$ whose indices correspond to indices of nonzero entries of $x$. The notation $I$ (resp. $J$) refers to the identity matrix (resp. the matrix of all ones); the dimension will be clear from context. For a symmetric matrix $M \in \R^{n \times n}$, the notation $M \succeq 0$ (resp. $M \succ 0$) denotes that $M$ is \emph{positive semidefinite} (resp. \emph{positive definite}), i.e. that it has nonnegative (resp. positive) eigenvalues. As mentioned already, we say that $M$ is \emph{copositve} if $x^TMx \ge 0, \forall x \ge 0$. The \emph{simplex} in $\Rn$ is denoted by $\Delta_n \defeq \{x \in \Rn\ |\ x \ge 0, \sum_{i=1}^n x_i = 1\}$. Finally, for a scalar $c$, the notation $\ceil{c}$ denotes the ceiling of $c$, i.e. the smallest integer greater than or equal to $c$.

We recall that a \emph{form} is a homogeneous polynomial; i.e. a polynomial whose monomials all have the same degree. A form $p: \Rn \to \R$ is said to be \emph{nonnegative} if ${p(x) \ge 0, \forall x \in \Rn}$, and \emph{positive definite} if $p(x) > 0, \forall x \ne 0$. A \emph{critical point} of a differentiable function $f:\Rn \to \R$ is a point $x\in\Rn$ at which the gradient $\grad  f(x)$ is zero. A \emph{second-order point} of a twice-differentiable function $f:\Rn \to \R$ is a critical point $x$ at which the Hessian matrix $\Hess f(x)$ is positive semidefinite.

All graphs in this chapter are undirected, unweighted, and have no self-loops. The \emph{adjacency matrix} of a graph $G$ on $n$ vertices is the $n \times n$ symmetric matrix whose $(i,j)$-th entry equals one if vertices $i$ and $j$ share an edge in $G$ and zero otherwise. The \emph{complement} of a graph $G$, denoted by $\bar{G}$, is the graph with the same vertex set as $G$ and such that two distinct vertices are adjacent if and only if they are not adjacent in $G$. An \emph{induced subgraph} of $G$ is a graph containing a subset of the vertices of $G$ and all edges connecting pairs of vertices in that subset.



\section{The Main Result}\label{QPs Sec: Main Result}

\subsection{Complexity of Deciding Existence of Local minima}

To show that a polynomial-time algorithm for finding a local minimum of a quadratic function over a polytope (i.e., a bounded polyhedron) implies P = NP, we show as an intermediate step that it is NP-hard to decide whether a quadratic program with an unbounded feasible set has a local minimum (Theorem~\ref{QPs Thm: LM QP NP hard}). To achieve this intermediate step, we first establish the following hardness result.

\begin{theorem}\label{QPs Thm: LM Degree 4 NP hard}
	It is strongly NP-hard to decide if a degree-4 polynomial has a local minimum.
\end{theorem}

We will prove this theorem by presenting a polynomial-time reduction from the STABLESET problem, which is known to be (strongly) NP-hard~\cite{garey2002computers}. Recall that in the STABLESET problem, we are given as input a graph $G$ on $n$ vertices and a positive integer $r \le n$. We are then asked to decide whether $G$ has a \emph{stable set} of size $r$, i.e. a set of $r$ pairwise non-adjacent vertices. We denote the size of the largest stable set in a graph $G$ by the standard notation $\alpha(G)$. We also recall that a \emph{clique} in a graph $G$ is a set of pairwise adjacent vertices. The size of the largest clique in $G$ is denoted by $\omega(G)$. The following theorem of Motzkin and Straus~\cite{motzkin1965maxima} relates $\omega(G)$ to the optimal value of a quadratic program.

\begin{theorem}[\cite{motzkin1965maxima}]\label{QPs Thm: Motzkin Straus}
	Let $G$ be a graph on $n$ vertices with adjacency matrix $A$ and clique number $\omega$. The optimal value of the quadratic program
	
	\begin{equation}\label{QPs Eq: MS QP}
	\begin{aligned}
	& \underset{x \in \Rn}{\max}
	& & x^TAx \\
	& \text{\emph{subject to}}
	&& x \ge 0,\\
	&&& \sum_{i=1}^n x_i = 1
	\end{aligned}
	\end{equation}
	is $1 - \frac{1}{\omega}$.
\end{theorem}

For a scalar $k$ and a symmetric matrix $A$ (which will always be an adjacency matrix), the following notation will be used repeatedly in our proofs:
\begin{equation}\label{QPs Def: MAk} M_{A,k} \defeq kA + kI - J,\end{equation}
\begin{equation}\label{QPs Def: qAk}q_{A,k}(x) \defeq x^TM_{A,k}x,\end{equation}
and
\begin{equation}\label{QPs Def: pAk}p_{A,k}(x) \defeq (x^2)^T M_{A,k} x^2.\end{equation}
Note that nonnegativity of the quadratic form $q_{A,k}$ over the nonnegative orthant is equivalent to (global) nonnegativity of the quartic form $p_{A,k}$ and to copositivity of the matrix $M_{A,k}$.

The following corollary of Theorem~\ref{QPs Thm: Motzkin Straus} will be of more direct relevance to our proofs. The first statement in the corollary has been observed e.g. by de Klerk and Pashechnik~\cite{de2002approximation}, but its proof is included here for completeness. The second statement, which will also be needed in the proof of Theorem~\ref{QPs Thm: LM Degree 4 NP hard}, follows straightforwardly.

\begin{cor}\label{QPs Cor: Motzkin Straus 2}
	For a scalar $k > 0$ and a graph $G$ with adjacency matrix $A$, the matrix $M_{A,k}$ in (\ref{QPs Def: MAk}) is copositive if and only if $\alpha(G) \le k$. Furthermore, if $\alpha(G) < k$, the quartic form $p_{A,k}$ in (\ref{QPs Def: pAk}) is positive definite.\footnote{The converse of this statement also holds, but we do not need it for the proof of Theorem~\ref{QPs Thm: LM Degree 4 NP hard}.}
\end{cor}
\begin{proof}
First observe that $\alpha(G) = \omega(\bar{G})$, and that the adjacency matrix of $\bar{G}$ is $J - A - I$. Thus from Theorem~\ref{QPs Thm: Motzkin Straus}, the maximum value of $x^T(J-A-I)x$ over $\Delta_n$ is $1 - \frac{1}{\alpha(G)}$, and hence the minimum value of $x^T(A+I)x$ over $\Delta_n$ is $\frac{1}{\alpha(G)}$. Therefore, for any $k > 0$, $\alpha (G) \le k$ if and only if $x^T(A+I)x \ge \frac{1}{k}$ for all $x \in \Delta_n$, which holds if and only if $x^T(k(A+I)-J)x \ge 0$ for all $x \in \Delta_n.$ The first statement of the corollary then follows from the homogeneity of $x^T(k(A+I)-J)x$.


To show that $p_{A,k}$ is positive definite when $\alpha(G) < k$, observe that
$$kA+kI-J = (\alpha(G)(A+I)-J) + (k - \alpha(G))(A+I).$$
Considering the two terms on the right separately, we observe that $(x^2)^T (\alpha(G)(A+I)-J) x^2$ (i.e., $p_{A,\alpha(G)})$ is nonnegative since $M_{A,\alpha(G)}$ is copositive, and that $(x^2)^T (k-\alpha(G))(A+I)x^2$ is positive definite. Therefore, their sum $(x^2)^T(kA+kI-J )x^2$ is positive definite.
\end{proof}

We now present the proof of Theorem~\ref{QPs Thm: LM Degree 4 NP hard}. While the statement of the theorem is given for degree-4 polynomials, it is straightforward to extend the result to higher-degree polynomials. We note that degree four is the smallest degree for which deciding existence of local minima is intractable. For degree-3 polynomials, it turns out that this question can be answered by solving semidefinite programs of polynomial size~\cite{cubicpaper}.


\begin{proof}[Proof (of Theorem~\ref{QPs Thm: LM Degree 4 NP hard})]
	We present a polynomial-time reduction from the STABLESET problem. Let a graph $G$ on $n$ vertices with adjacency matrix $A$ and a positive integer $r \le n$ be given. We show that $G$ has a stable set of size $r$ if and only if the quartic form $p_{A,r-0.5}$ defined in (\ref{QPs Def: pAk}) has no local minimum. This is a consequence of the following more general fact that we prove below: For a noninteger scalar $k$, the quartic form $p_{A,k}$ has no local minimum if and only if $\alpha(G) \ge k$.
	
	We first observe that if $\alpha(G) < k$, then $p_{A,k}$ has a local minimum. Indeed, recall from the second claim of Corollary~\ref{QPs Cor: Motzkin Straus 2} that under this assumption, $p_{A,k}$ is positive definite. Since $p_{A,k}$ vanishes at the origin, it follows that the origin is a local minimum. Suppose now that $\alpha(G) \ge k$. Since $k$ is noninteger, this implies that $\alpha(G) > k$. We show that in this case, $p_{A,k}$ has no local minimum by showing that the origin must be the only second-order point of $p_{A,k}$. Since any local minimum of a polynomial is a second-order point, only the origin can be a candidate local minimum for $p_{A,k}$. However, by the first claim of Corollary~\ref{QPs Cor: Motzkin Straus 2}, the matrix $M_{A,k}$ is not copositive and hence $p_{A,k}$ is not nonnegative. As $p_{A,k}$ is homoegenous, this implies that $p_{A,k}$ takes negative values arbitrarily close to the origin, ruling out the possibility of the origin being a local minimum.
	
	To show that when $\alpha(G) > k$, the origin is the only second-order point of $p_{A,k}$, we compute the gradient and Hessian of $p_{A,k}$. We have
	$$\grad p_{A,k}(x) = 4x \cdot M_{A,k} x^2,$$
	and
	$$\Hess p_{A,k}(x) = 8M_{A,k} \cdot xx^T + 4diag(M_{A,k} x^2).$$
	Suppose for the sake of contradiction that $p_{A,k}$ has a nonzero second-order point $\xbar$. Since $\xbar$ is a critical point, $\grad p_{A,k}(\xbar)=0$ and thus $(M_{A,k} \xbar^2)_{\xbar} = 0$. It then follows that
	$$(\Hess p_{A,k}(\xbar))_{\xbar} = 8(M_{A,k})_{\xbar} \cdot \xbar_{\xbar}\xbar_{\xbar}^T.$$
    Because $\Hess p_{A,k}(\xbar) \succeq 0$ and thus all its principal submatrices are positive semidefinite, we have $8(M_{A,k})_{\xbar} \cdot \xbar_{\xbar}\xbar_{\xbar}^T \succeq 0$. Since
    $$(M_{A,k})_{\xbar} \cdot \xbar_{\xbar}\xbar_{\xbar}^T = diag(\xbar_{\xbar}) (M_{A,k})_{\xbar} diag(\xbar_{\xbar}),$$
    and since $diag(\xbar_{\xbar})$ is an invertible matrix, it follows that $(M_{A,k})_{\xbar} \succeq 0$.
    
	
	We now consider the induced subgraph $G_{\xbar}$ of $G$ with vertices corresponding to the indices of the nonzero entries of $\xbar$. Note that the adjacency matrix of $G_{\xbar}$ is $A_{\xbar}$. Furthermore, observe that $M_{A_{\xbar},k} = (M_{A,k})_{\xbar}$, and therefore $M_{A_{\xbar},k}$ is positive semidefinite and thus copositive. We conclude from the first claim of Corollary~\ref{QPs Cor: Motzkin Straus 2} that $\alpha(G_{\xbar}) \le k$. We now claim that
	$$M_{A_{\xbar},k}\xbar_{\xbar}^2 = (M_{A,k})_{\xbar}\xbar_{\xbar}^2 = (M_{A,k}\xbar^2)_{\xbar} = 0.$$
    The first equality follows from that $M_{A_{\xbar},k} = (M_{A,k})_{\xbar}$, the second from that the indices of the nonzero entries of $\xbar$ are the same as those of $\xbar^2$, and the third from that $\grad p_{A,k} (\xbar) = 0$. Hence, $p_{A_{\xbar},k}(\xbar_{\xbar}) = 0$. Since $\xbar_{\xbar}$ is nonzero, $p_{A_{\xbar},k}$ is not positive definite. By the second claim of Corollary~\ref{QPs Cor: Motzkin Straus 2}, we must have $\alpha(G_{\xbar}) \ge k$. Therefore, $\alpha(G_{\xbar}) = k$. However, because $k$ was assumed to be noninteger, we have a contradiction.
\end{proof}

It turns out that the proof of Theorem~\ref{QPs Thm: LM Degree 4 NP hard} also shows that it is NP-hard to decide if a quartic polynomial has a strict local minimum. Recall that a \emph{strict local minimzer} of a function $f: \Rn \to \R$ over a set $\Omega \subseteq \Rn$ is a point $\xbar \in \Omega$ for which there exists a scalar $\epsilon > 0$ such that $p(\xbar) < p(x)$ for all $x \in \Omega\backslash \xbar$ with $\|x - \xbar\| \le \epsilon$.

\begin{cor}\label{QPs Cor: SLM Degree 4 NP hard}
	It is strongly NP-hard to decide if a degree-4 polynomial has a strict local minimum.
\end{cor}
\begin{proof}
Observe from the proof of Theorem~\ref{QPs Thm: LM Degree 4 NP hard} that for a graph $G$ and a noninteger scalar $k$, the quartic form $p_{A,k}$ has a local minimum if and only if $\alpha(G)<k$. In the case where $p_{A,k}$ does have a local minimum, we showed that $p_{A,k}$ is positive definite, and thus the local minimum (the origin) must be a strict local minimum.
\end{proof}

We now turn our attention to local minima of quadratic programs.

\begin{theorem}\label{QPs Thm: LM QP NP hard}
	It is strongly NP-hard to decide if a quadratic function has a local minimum over a polyhedron. The same is true for deciding if a quadratic function has a strict local minimum over a polyhedron.
\end{theorem}

\begin{proof}
	We present a polynomial-time reduction from the STABLESET problem to the problem of deciding if a quadratic function has a local minimum over a polyhedron. The reader can check that same reduction is valid for the case of strict local minima.
	
	Let a graph $G$ on $n$ vertices with adjacency matrix $A$ and a positive integer $r \le n$ be given. Let $k = r - 0.5$, $q_{A,k}$ be the quadratic form defined in (\ref{QPs Def: qAk}), and consider the optimization problem
	\begin{equation}\label{QPs Eq: SS QP}
	\begin{aligned}
	& \underset{x \in \Rn}{\min}
	& & q_{A,k}(x) \\
	& \text{subject to}
	&& x \ge 0.
	\end{aligned}
	\end{equation}
    We show that a point $x \in \Rn$ is a local minimum of (\ref{QPs Eq: SS QP}) if and only if $\sqrt{x}$ is a local minimum of the quartic form $p_{A,k}$ defined in (\ref{QPs Def: pAk}). By the arguments in the proof of Theorem~\ref{QPs Thm: LM Degree 4 NP hard}, we would have that (\ref{QPs Eq: SS QP}) has no local minimum if and only if $G$ has a stable set of size $r$.
	
	Indeed, if $x$ is not a local minimum of (\ref{QPs Eq: SS QP}), there exists a sequence $\{y_j\} \subseteq \Rn$ with $y_j \to x$ and such that for all $j$, $y_j \ge 0$ and $q_{A,k}(y_j) < q_{A,k}(x)$. The sequence $\{\sqrt{y_j}\}$ would then satisfy $p_{A,k}(\sqrt{y_j}) < p_{A,k}(\sqrt{x})$ and $\sqrt{y_j} \to \sqrt{x}$, proving that $\sqrt{x}$ is not a local minimum of $p_{A,k}$. Similarly, if $x$ is not a local minimum of $p_{A,k}$, there exists a sequence $\{z_j\} \subseteq \Rn$ such that $z_j \to x$ and $p_{A,k}(z_j) < p_{A,k}(x)$ for all $j$. The sequence $\{z_j^2\}$ would then prove that $x^2$ is not a local minimum of (\ref{QPs Eq: SS QP}).
	
\end{proof}

\subsection{Complexity of Finding a Local minimum of a Quadratic Function Over a Polytope}

We now address the original question posed by Pardalos and Vavasis concerning the complexity of finding a local minimum of a quadratic program with a compact feasible set. Note again that if the feasible set is compact, the existence of a local minimum is guaranteed. In fact, there will always be a local minimum that has rational entries with polynomial bitsize~\cite{vavasis1990quadratic}.

\begin{theorem}\label{QPs Thm: Finding LM NP-hard}
    If there is a polynomial-time algorithm that finds a point within Euclidean distance $c^n$ (for any constant $c \ge 0$) of a local minimum of an $n$-variate quadratic function over a polytope, then $P = NP$.
\end{theorem}

\begin{proof}
Fix any constant $c \ge 0$. We show that if an algorithm could take as input a quadratic program with a bounded feasible set and in polynomial time return a point within distance $c^n$ of any local minimum, then this algorithm would solve the the STABLESET problem in polynomial time.

Let a graph $G$ on $n$ vertices with adjacency matrix $A$ and a positive integer $r \le n$ be given. Let $k = r - 0.5$, $q_{A,k}$ be the quadratic form defined in (\ref{QPs Def: qAk}), and consider the quadratic program\footnote{Without loss of generality, we suppose that $c$ is rational. If $c$ is irrational, one can e.g. replace it in (\ref{QPs Eq: SS Approx QP}) with $\ceil{c}$.}
\begin{equation}\label{QPs Eq: SS Approx QP}
	\begin{aligned}
	& \underset{x \in \Rn}{\min}
	& & q_{A,k}(x) \\
	& \text{subject to}
	&& x \ge 0,\\
	&&& \sum_{i=1}^n x_i \le 3 c^n \sqrt{n}.
	\end{aligned}
\end{equation}
Note that the feasible set of this problem is bounded. Moreover, the number of bits required to write down this quadratic program is polynomial in $n$. This is because the scalar $3c^n\sqrt{n}$ takes $2+n\lceil \log_2(c+1) \rceil+\frac{1}{2}\lceil \log_2(n+1) \rceil$ bits to write down, and the remaining $O(n^2)$ numbers in the problem data are bounded in magnitude by $n$, so they each take $O(\log_2(n))$ bits to write down.

We will show that if $\alpha(G) < k$, the origin is the unique local minimum of (\ref{QPs Eq: SS Approx QP}), and that if $\alpha(G) \ge k$ (equivalently $\alpha(G) > k$), any local minimum $\xbar$ of (\ref{QPs Eq: SS Approx QP}) satisfies $\sum_{i=1}^n \xbar_i = 3 c^n \sqrt{n}$. Since the (Euclidean) distance from the origin to the hyperplane $\{ x \in \Rn | \sum_{i=1}^n x_i = 3 c^n \sqrt{n}\}$ is $3c^n$, there is no point that is within distance $c^n$ of both the origin and this hyperplane. Thus, the graph $G$ has no stable set of size $r$ (or equivalently $\alpha(G) < k$) if and only if the Euclidean norm of all points within distance $c^n$ of any local minimum of (\ref{QPs Eq: SS Approx QP}) is less than or equal to $c^n$.

To see why $\alpha(G) < k$ implies that the origin is the unique local minimum of (\ref{QPs Eq: SS Approx QP}), recall from the second claim of Corollary~\ref{QPs Cor: Motzkin Straus 2} that the quartic form $p_{A,k}$ defined in (\ref{QPs Def: pAk}) must be positive definite. Thus, for any nonzero vector $x \ge 0$, we have $q_{A,k}(x) > 0$. This implies that the origin is a local minimum of (\ref{QPs Eq: SS Approx QP}). Moreover, since $q_{A,k}$ is homogeneous, we have that no other feasible point can be a local minimum. Indeed, for any nonzero vector $x \ge 0$ and any nonnegative scalar $\epsilon < 1$, $q_{A,k}(\epsilon x) < q_{A,k}(x)$.

To see why when $\alpha(G) > k$, the last constraint of (\ref{QPs Eq: SS Approx QP}) must be tight at all local minima, recall from the proof of Theorem~\ref{QPs Thm: LM QP NP hard} that when $\alpha(G) > k$, the optimization problem in (\ref{QPs Eq: SS QP}) has no local minimum. Therefore, for any vector $x$ that is feasible to (\ref{QPs Eq: SS Approx QP}) and satisfies $\sum_{i=1}^n x_i < 3 c^n \sqrt{n}$, there exists a sequence $\{y_i\} \subseteq \Rn$ with $y_i \to x$, and satisfying $$y_i \ge 0, \sum_{i=1}^n y_i < 3 c^n \sqrt{n}, q_{A,k}(y_i) < q_{A,k}(x), \forall i.$$ As the points $y_i$ are feasible to (\ref{QPs Eq: SS Approx QP}), any vector $x$ satisfying $\sum_{i=1}^n x_i < 3 c^n \sqrt{n}$ cannot be a local minimum of (\ref{QPs Eq: SS Approx QP}). Thus, if $\alpha(G) > k$, any local minimum $\xbar$ of (\ref{QPs Eq: SS Approx QP}) satisfies $\sum_{i=1}^n \xbar_i = 3 c^n \sqrt{n}$.
\end{proof}

By replacing the quantity $3c^n \sqrt{n}$ in the proof of Theorem~\ref{QPs Thm: Finding LM NP-hard} with $3n^{c+0.5}$ and $2n$ respectively, we get the following two corollaries.

\begin{cor}\label{QPs Cor: QP Pseudo}
If there is a pseudo-polynomial time algorithm that finds a point within Euclidean distance $n^c$ (for any constant $c \ge 0$) of a local minimum of an $n$-variate quadratic function over a polytope, then $P = NP$.
\end{cor}
\begin{cor}\label{QPs Cor: 2n NP-hard}
If there is a polynomial-time algorithm that finds a point within Euclidean distance $\epsilon \sqrt{n}$ (for any constant $\epsilon \in [0,1)$) of a local minimum of a restricted set of quadratic programs over $n$ variables whose numerical data are integers bounded in magnitude by $2n$, then $P = NP$.
\end{cor}

In~\cite{pardalos1992open}, Pardalos and Vavasis also propose two follow-up questions about quadratic programs with compact feasible sets. The first is about the complexity of finding a ``KKT point''. As is, our proof does not have any implications for this question since the origin is always a KKT point of the quadratic programs that arise from our reductions. The second question asks whether finding a local minimum is easier in the special case where the problem only has one local minimum (which is thus also the global minimum). Related to this question, we can prove the following claim.

\begin{theorem}\label{QPs Thm: UniqueQP NP Hard}
If there is a polynomial-time algorithm which decides whether a quadratic program with a bounded feasible set has a unique local minimum, and if so returns this minimum\footnote{This unique local (and therefore global) minimum is guaranteed to have rational entries with polynomial bitsize; see~\cite{vavasis1990quadratic}.}, then P=NP.
\end{theorem}

\begin{proof}
Suppose there was such an algorithm (call it Algorithm U). We show that Algorithm U would solve the STABLESET problem in polynomial time. Let a graph $G$ on $n$ vertices with adjacency matrix $A$ and a positive integer $r \le n$ be given, and input the quadratic program (\ref{QPs Eq: SS Approx QP}), with $k = r-0.5$, into Algorithm U. Observe from the proof of Theorem~\ref{QPs Thm: Finding LM NP-hard} that there are three possibilities for this quadratic program: (i) the origin is the unique local minimum, (ii) there is a unique local minimum and it is on the hyperplane ${\{ x \in \Rn | \sum_{i=1}^n x_i = 3 c^n \sqrt{n}\}}$, and (iii) there are multiple local minima and they are all on the hyperplane $\{ x \in \Rn | \sum_{i=1}^n x_i = 3 c^n \sqrt{n}\}$. Case (i) indicates that $\alpha(G) < k,$ and the output of Algorithm U in this case would be the origin. Cases (ii) and (iii) both indicate that $\alpha(G) > k$. The output of Algorithm U is a point away from the origin in case (ii), and the declaration that the local minimum is not unique in case (iii). Thus Algorithm U would reveal which case we are in, and that would allow us to decide if $G$ has a stable set of size $r$ in polynomial time.
\end{proof}

To conclude, we have established intractability of several problems related to local minima of quadratic programs. We hope our results motivate more research on identifying classes of quadratic programs where local minima can be found more efficiently than global minima. One interesting example is the case of the concave knapsack problem, where Mor{\'e} and Vavasis~\cite{more1990solution} show that a local minimum can be found in polynomial time even though, unless P=NP, a global minimum cannot.
\chapter{On Attainment of the Optimal Value in Polynomial Optimization}\label{Chap: Attainment}
\section{Introduction}\label{AOS Sec: intro}

In this chapter, we again consider problems of the form
\begin{equation}\label{AOS Defn: pop}
\begin{aligned}
& \underset{x}{\inf}
& & p(x) \\
& \text{subject to}
&& q_i(x) \ge 0, \forall i \in \{1,\ldots,m\},\\
\end{aligned}
\end{equation}
where $p,q_i$ are polynomial functions, and address the problem of testing whether the optimal value is \emph{attained}, provided that the optimal value $p^*$ is finite. More formally, does there exist a feasible point $x^*$ such that $p(x^*) = p^*$? Such a point $x^*$ will be termed an \emph{optimal solution}. For this chapter, we will refer to sets of the type $\{x \in \Rn\ |\ q_i(x) \ge 0, \forall i \in \{1,\ldots,m\}\}$ as \emph{closed basic semialgebraic sets}.

Existence of optimal solutions is a fundamental question in optimization and its study has a long history, dating back to the nineteenth century with the extreme value theorem of Bolzano and Weierstrass.\footnote{To remove possible confusion, we emphasize that our focus in this chapter is not on the complexity of testing feasibility or unboundedness of problem (\ref{AOS Defn: pop}), which have already been studied extensively. On the contrary, all optimization problems that we consider are by construction feasible and bounded below.} The question of testing attainment of the optimal value for POPs has appeared in the literature explicitly. For example, Nie, Demmel, and Sturmfels describe an algorithm for globally solving an unconstrained POP which requires as an assumption that the optimal value be attained \cite{nie2006minimizing}. This leads them to make the following remark in their conclusion section:
\begin{quote}
	``This assumption is non-trivial, and we do not address the (important and difficult) question of how to verify that a given polynomial $f(x)$ has this property.''
\end{quote}

Prior literature on existence of optimal solutions to POPs has focused on identifying cases where existence is always guaranteed. The best-known result here is the case of linear programming (i.e., when the degrees of $p$ and $q_i$ are one). In this case, the optimal value of the problem is always attained. This result was extended by Frank and Wolfe to the case where $p$ is quadratic and the polynomials $q_i$ are linear~\cite{frank1956algorithm}. Consequently, results concerning attainment of the optimal value are sometimes referred to as ``Frank-Wolfe type'' theorems in the literature \cite{belousov2002frank,luo1999extensions}. Andronov et al. showed that the same statement holds again when $p$ is cubic (and the polynomials $q_i$ are linear)~\cite{andronov1982solvability}. 

Our results in this chapter show that in all other cases, it is strongly NP-hard to determine whether a polynomial optimization problem attains its optimal value. This implies that unless P=NP, there is no polynomial-time (or even pseudo-polynomial time)  algorithm for checking this property. Nevertheless, it follows from the Tarski-Seidenberg quantifier elimination theory \cite{seidenberg1954new,tarski1951decision} that this problem is decidable, i.e., can be solved in finite time. There are also probabilistic algorithms that test for attainment of the optimal value of a POP \cite{greuet2011deciding,greuet2014probabilistic}, but their complexities are exponential in the number of variables.

In this chapter, we also study the complexity of testing several well-known sufficient conditions for attainment of the optimal value (see Section \ref{AOS SSec: Contributions} below). One sufficient condition that we do not consider but that is worth noting is for the polynomials $p,-q_1,\ldots,-q_m$ to all be convex (see \cite{belousov2002frank} for a proof, \cite{luo1999extensions} for the special case where $p$ and $q_i$ are quadratics, and \cite{bertsekas2007set} for other extensions). The reason we exclude this sufficient condition from our study is that the complexity of checking convexity of polynomials has already been analyzed in \cite{ahmadi2013np}. 

\subsection{Organization and Contributions of the Chapter}\label{AOS SSec: Contributions}

As mentioned before, this chapter concerns itself with the complexity of testing attainment of the optimal value of a polynomial optimization problem. More specifically, we show in Section \ref{AOS Sec: AOS} that it is strongly NP-hard to test attainment when the objective function has degree 4, even in absence of any constraints (Theorem \ref{AOS Thm: AOS NP-hard o4c0}), and when the constraints are of degree 2, even when the objective is linear (Theorem \ref{AOS Thm: AOS NP-hard o1c2}).

In Section \ref{AOS Sec: Sufficient Conditions}, we show that several well-known sufficient conditions for attainment of the optimal value in a POP are also strongly NP-hard to test. These include coercivity of the objective function (Theorem \ref{AOS Thm: Coercive NP-hard}), closedness of a bounded feasible set (Theorem \ref{AOS Thm: Closedness} and Remark \ref{AOS rem:boundedness.compactness}), boundedness of a closed feasible set (Corollary \ref{AOS Thm: Boundedness}), a robust analogue of compactness known as stable compactness (Corollary \ref{AOS Thm: Stable Compact NP-hard}), and an algebraic certificate of compactness known as the Archimedean property (Theorem \ref{AOS Thm: Archimedean NP-hard}). The latter property is of independent interest to the convergence of the Lasserre hierarchy, as discussed in Section \ref{AOS SSSec: Archimedean}.

In Section \ref{AOS Sec: Algorithms}, we give semidefinite programming (SDP) based hierarchies for testing compactness of the feasible set and coercivity of the objective function of a POP (Propositions~\ref{AOS Prop: Compactness Stengle} and \ref{AOS Prop: Coercive sublevel sets bounded}). The hierarchy for compactness comes from a straightforward application of Stengle's Positivstellensatz (cf. Theorem \ref{AOS Thm: Stengle}), but the one for coercivity requires us to develop a new characterization of coercive polynomials (Theorem \ref{AOS Thm: Polynomial Radius}). We end the chapter in Section \ref{AOS Sec:conclusion} with a summary and some brief concluding remarks.

\section{NP-hardness of Testing Attainment of the Optimal Value}\label{AOS Sec: AOS}

In this section, we show that testing attainment of the optimal value of a polynomial optimization problem is NP-hard. Throughout this chapter, when we study complexity questions around problem (\ref{AOS Defn: pop}), we fix the degrees of all polynomials involved and think of the number of variables and the coefficients of these polynomials as input. Since we are working in the Turing model of computation, all the coefficients are rational numbers and the input size can be taken to be the total number of bits needed to represent the numerators and denominators of these coefficients.

Our proofs of hardness are based on reductions from ONE-IN-THREE 3SAT which is known to be NP-hard \cite{schaefer1978complexity}. Recall that in ONE-IN-THREE 3SAT, we are given a 3SAT instance (i.e., a collection of clauses, where each clause consists of exactly three literals, and each literal is either a variable or its negation) and we are asked to decide whether there exists a $\{0, 1\}$ assignment to the variables that makes the expression true with the additional property that each clause has \emph{exactly} one true literal.

\begin{theorem}\label{AOS Thm: AOS NP-hard o4c0} Testing whether a degree-4 polynomial attains its unconstrained infimum is strongly NP-hard.
\end{theorem}

\begin{proof}
	
	Consider a ONE-IN-THREE 3SAT instance $\phi$ with $n$ variables $x_1,\ldots,x_n,$ and $k$ clauses. Let $s_\phi(x):\mathbb{R}^n\rightarrow \mathbb{R}$ be defined as
	\begin{equation}\label{AOS eq: 1 in 3 sos}s_\phi(x) = \sum_{i=1}^k (\phi_{i1} + \phi_{i2} + \phi_{i3} + 1)^2 + \sum_{i=1}^n (1-x_i^2)^2,\end{equation}
	where $\phi_{it} = x_j$ if the $t$-th literal in the $i$-th clause is $x_j$, and $\phi_{it} = -x_j$ if it is $\neg x_j$ (i.e., the negation of $x_j$). Now, let 
	\begin{equation}\label{AOS eq: 1 in 3 sos plus unattained}
	p_\phi(x,y,z,\lambda) \mathrel{\mathop{:}}= \lambda^2s_\phi(x) + (1-\lambda)^2(y^2 + (yz-1)^2),\end{equation}where $y,z,\lambda \in \mathbb{R}.$
	We show that $p_{\phi}$ achieves its infimum if and only if $\phi$ is satisfiable. Note that the reduction is polynomial in length and the coefficients of $p_\phi$ are at most a constant multiple of $n+k$ in absolute value.
	
	If $\phi$ has a satisfying assignment, then for any $y$ and $z$, letting $\lambda = 1$, $x_i=1$ if the variable is true in that assignment and $x_i = -1$ if it is false, results in a zero of $p_\phi$. As $p_{\phi}$ is a sum of squares and hence nonnegative, we have shown that it achieves its infimum.
	
	Now suppose that $\phi$ is not satisfiable. We will show that $p_{\phi}$ is positive everywhere but gets arbitrarily close to zero. To see the latter claim, simply set $\lambda=0, z=\frac{1}{y}$, and let $y \rightarrow 0.$ To see the former claim, suppose for the sake of contradiction that $p_{\phi}$ has a zero. Since $y^2+(yz-1)^2$ is always positive, we must have $\lambda=1$ in order for the second term to be zero. Then, in order for the whole expression to be zero, we must also have that $s_{\phi}(x)$ must vanish at some $x.$ But any zero of $s_\phi$ must have each $x\in \{-1,1\}^n$, due to the second term of $s_\phi$. However, because the instance $\phi$ is not satisfiable, for any such $x$, there exists $i \in \{1,\ldots,k\}$ such that $\phi_{i1} + \phi_{i2} + \phi_{i3}+1 \neq 0$, as there must be a clause where not exactly one literal is set to one. This means that $s_{\phi}$ is positive everywhere, which is a contradiction. 
	
	We have thus shown that testing attainment of the optimal value is NP-hard for unconstrained POPs where the objective is a polynomial of degree 6. In the interest of minimality, we now extend the proof to apply to an objective function of degree 4. To do this, we first introduce $n+1$ new variables $\chi_1, \ldots, \chi_n$ and $w$. We replace every occurrence of the product $\lambda x_i$ in $\lambda^2 s_\phi$ with the variable $\chi_i$. For example, the term $\lambda^2 x_1x_2$ would become $\chi_1\chi_2$. Let $\hat{s}_\phi(x,\chi,\lambda)$ denote this transformation on $\lambda^2s_\phi(x)$. Note that $\hat{s}_\phi(x, \chi, \lambda)$ is now a quartic polynomial. Now consider the quartic polynomial (whose coefficients are again at most a constant multiple of $n+k$ in absolute value)
	\begin{equation}\label{AOS eq: 1 in 3 sos new}
	\hat{p}_\phi(x,y,z,\lambda,\chi,w) = \hat{s}_\phi(x, \chi,\lambda) + (1-\lambda)^2(y^2 + (w-1)^2) + (w-yz)^2 + \sum_{i=1}^n (\chi_i - \lambda_ix_i)^2.
	\end{equation}
	Observe that $\hat{p}_\phi$ is a sum of squares as $\hat{s}_\phi$ can be verified to be a sum of squares by bringing $\lambda$ inside every squared term of $s_\phi$. Hence, $\hat{p}_\phi$ is nonnegative. Furthermore, its infimum is still zero, as the choice of variables $\lambda = 0, w = 1, \chi = 0$, $x$ arbitrary, $z = \frac{1}{y}$, and letting $y \to 0$ will result in arbitrarily small values of $\hat{p}_\phi$. Now it remains to show that this polynomial will have a zero if and only if $p_{\phi}$ in (\ref{AOS eq: 1 in 3 sos plus unattained}) has a zero. Observe that if $(x,y,z,\lambda)$ is a zero of $p_{\phi}$, then $(x,y,z, \lambda, \lambda x,yz)$ is a zero of $\hat{p}_{\phi}$. Conversely, if $(x,y,z,\lambda,\chi,w)$ is a zero of $\hat{p}_\phi$, then $(x,y,z,\lambda)$ is a zero of $p_{\phi}$. \end{proof}

\begin{remark} Because we use the ideas behind this reduction repeatedly in the remainder of this chapter, we refer to the quartic polynomial defined in ($\ref{AOS eq: 1 in 3 sos}$) as $s_\phi$ throughout. The same convention for $\phi_{it}$ relating the literals of $\phi$ to the variables $x$ will be assumed as well.
\end{remark}

We next show that testing attainment of the optimal value of a POP is NP-hard when the objective function is linear and the constraints are quadratic. Together with the previously-known Frank-Wolfe type theorems which we reviewed in the introduction, Theorems \ref{AOS Thm: AOS NP-hard o4c0} and \ref{AOS Thm: AOS NP-hard o1c2} characterize the complexity of testing attainment of the optimal value in polynomial optimization problems of any given degree. Indeed, our reductions can trivially be extended to the case where the constraints or the objective have higher degrees. For example to increase the degree of the constraints to some positive integer $d$, one can introduce a new variable $\gamma$ along with the trivial constraint $\gamma^d = 0$. To increase the degree of the objective from four to a higher degree $2d$, one can again introduce a new variable $\gamma$ and add the term $\gamma^{2d}$ to the objective function.

\begin{theorem}\label{AOS Thm: AOS NP-hard o1c2}
	Testing whether a degree-1 polynomial attains its infimum on a feasible set defined by degree-2 inequalities is strongly NP-hard.
\end{theorem}
\begin{proof}
	Consider a ONE-IN-THREE 3SAT instance $\phi$ with $n$ variables and $k$ clauses. Define the following POP, with $x,\chi \in \mathbb{R}^n$ and $\lambda, y,z,w, \gamma, \zeta, \psi \in \mathbb{R}$:
	
	\begin{align}\label{AOS Defn: ONE-IN-THREE 3SAT QCQP}
	& \underset{x, \chi, \lambda, y,z,w,\gamma, \zeta, \psi}{\min}
	& & \gamma \\
	& \text{subject to}
	&& \gamma \ge \lambda \sum_{i=1}^n \chi_i + (1-\lambda) (\psi + \zeta)\label{AOS const: obj surrogate}\\
	&&& 1-x_i^2 = 0,\ \forall i \in \{1,\ldots,n\},\label{AOS const: hypercube}\\
	&&&\chi_i = (\phi_{i1} + \phi_{i2} + \phi_{i3} + 1)^2,\ \forall i \in \{1,...,k\}, \label{AOS const: 3sat sos}\\
	&&& \psi = y^2, \label{AOS const: y2}\\
	&&& yz = w,\label{AOS const: yz}\\
	&&& \zeta = (w-1)^2, \label{AOS const: yz2}\\
	&&& \lambda(1-\lambda) = 0 \label{AOS const: switch}.
	\end{align}
	
	We show that the infimum of this POP is attained if and only if $\phi$ is satisfiable. Note first that the objective value is always nonnegative because of (\ref{AOS const: obj surrogate}) and in view of (\ref{AOS const: 3sat sos}), (\ref{AOS const: y2}), (\ref{AOS const: yz2}), and (\ref{AOS const: switch}). Observe that if $\phi$ has a satisfying assignment, then letting $x_i=1$ if the variable is true in that assignment and $x_i=-1$ if it is false, along with $\lambda = 1$, $y$ and $z$ arbitrary, $\psi = y^2, w = yz,$ and $\zeta = (w-1)^2$, results in a feasible solution with an objective value of 0. 
	
	If $\phi$ is not satisfiable, the objective value can be made arbitrarily close to zero by taking an arbitrary $x \in \{-1,1\}^n,$ $\chi_i$ accordingly to satisfy (\ref{AOS const: 3sat sos}), ${\lambda = 0, \psi = y^2, z = \frac{1}{y}, w = 1,\zeta=0}$, and letting $y \to 0 $. Suppose for the sake of contradiction that there exists a feasible solution to the POP with $\gamma=0.$ As argued before, because of the constraints (\ref{AOS const: 3sat sos}), (\ref{AOS const: y2}), (\ref{AOS const: yz2}), and (\ref{AOS const: switch}), $\lambda \sum_{i=1}^n \chi_i + (1-\lambda) (\psi + \zeta)$ is always nonnegative, and so for $\gamma$ to be exactly zero, we need to have $$\lambda \sum_{i=1}^n \chi_i + (1-\lambda) (\psi + \zeta)=0.$$ From (\ref{AOS const: switch}), either $\lambda=0$ or $\lambda=1.$ If $\lambda=1$, then we must have $\chi_i=0,\forall i=1,\ldots,n$, which is not possible as $\phi$ is not satisfiable. If $\lambda=0$, then we must have $\psi+\zeta=y^2+(yz-1)^2=0$, which cannot happen as this would require $y=0$ and $yz=1$ concurrently.
\end{proof}

\section{NP-hardness of Testing Sufficient Conditions for Attainment}\label{AOS Sec: Sufficient Conditions}

Arguably, the two best-known sufficient conditions under which problem (\ref{AOS Defn: pop}) attains its optimal value are \emph{compactness} of the feasible set and \emph{coercivity} of the objective function. In this section, we show that both of these properties are NP-hard to test for POPs of low degree. We also prove that certain stronger conditions, namely the \emph{Archimedean property} of the quadratic module associated with the constraints and  \emph{stable compactness} of the feasible set, are NP-hard to test.

\subsection{Coercivity of the Objective Function}\label{AOS SSec: Coercive NP-hard}

A function $p: \mathbb{R}^n \to \mathbb{R}$ is \emph{coercive} if for every sequence $\{x_k\}$ such that $\|x_k\| \to \infty$, we have $p(x_k) \to \infty$. It is well known that a continuous coercive function achieves its infimum on a closed set (see, e.g., Appendix A.2 of \cite{bertsekas1999nonlinear}). This is because all sublevel sets of continuous coercive functions are compact.

\begin{theorem}\label{AOS Thm: Coercive NP-hard} 
	Testing whether a degree-4 polynomial is coercive is strongly NP-hard.
\end{theorem}
\begin{proof}
	
	Consider a ONE-IN-THREE 3SAT instance $\phi$ with $n$ variables and $k$ clauses, and the associated quartic polynomial $s_{\phi}(x)$ as in (\ref{AOS eq: 1 in 3 sos}). Let $s_{\phi h}:\mathbb{R}^{n+1} \rightarrow \mathbb{R}$ be the homogenization of this polynomial:
	\begin{equation}\label{AOS eq: 1 in 3 sos homo}
	s_{\phi h}(x_0, x) \defeq x_0^4 s_\phi \left(\frac{x}{x_0}\right) = \sum_{i=1}^k x_0^2(\phi_{i1} + \phi_{i2} + \phi_{i3} + x_0)^2 + \sum_{i=1}^n (x_0^2-x_i^2)^2.
	\end{equation}
	By construction, $s_{\phi h}$ is a homogeneous polynomial of degree 4. We show that $s_{\phi h}$ is coercive if and only if $\phi$ is not satisfiable.
	
	Suppose first that the instance $\phi$ has a satisfying assignment $\hat{x} \in \{-1,1\}^n$. Then it is easy to see that $s_{\phi h}(1,\hat{x}) = 0$. As $s_{\phi h}$ is homogeneous, $s_{\phi h}(\alpha,\alpha\hat{x}) = 0$ for all $\alpha$, showing that $s_{\phi h}$ is not coercive.
	
	Now suppose that $\phi$ is not satisfiable. We show that $s_{\phi h}$ is positive definite (i.e., $s_{\phi h}(x_0,x)>0$ for all $(x_0,x)\neq (0,0) $). This would then imply that $s_{\phi h}$ is coercive as 
	\begin{align*}
	s_{\phi h }(x_0,x)&=||(x_0,x)^T||^4\cdot s_{\phi h} \left( \frac{(x_0,x)}{||(x_0,x)^T||} \right)\\
	&\geq \mu||(x_0,x)^T||^4,
	\end{align*}
	where $\mu>0$ is defined as the minimum of $s_{\phi h}$ on the unit sphere: $$\mu= \min_{(x_0,x)\in S^{n}}s_{\phi h}(x_0,x).$$
	Suppose that $s_{\phi h}$ was not positive definite. Then there exists a point $(\hat{x}_0,\hat{x}) \ne (0,0)$ such that $s_{\phi h}(\hat{x}_0,\hat{x}) = 0$. First observe $\hat{x}_0$ cannot be zero due to the $(x_0 - x_i)^2$ terms in (\ref{AOS eq: 1 in 3 sos homo}). As $\hat{x}_0 \ne 0$, then, by homogeneity, the point $(1, \frac{\hat{x}}{\hat{x}_0})$ is a zero of $s_{\phi h}$ as well. This however implies that $s_{\phi}(\hat{x})=0$, which we have previously argued (cf. the proof of Theorem \ref{AOS Thm: AOS NP-hard o4c0}) is equivalent to satisfiability of $\phi$, hence a contradiction.
\end{proof}

We remark that the above hardness result is minimal in the degree as odd-degree polynomials are never coercive and a quadratic polynomial $x^TQx+b^Tx+c$ is coercive if and only if the matrix $Q$ is positive definite, a property that can be checked in polynomial time (e.g., by checking positivity of the leading principal minors of $Q$).

\subsection{Closedness and Boundedness of the Feasible Set}\label{AOS SSec: Closed and Bounded NP-hard}

The well-known Bolzano-Weierstrass extreme value theorem states that the infimum of a continuous function on a compact (i.e., closed and bounded) set is attained. In this section, we show that testing closedness or boundedness of a basic semialgebraic set defined by degree-2 inequalities is NP-hard. Once again, these hardness results are minimal in degree since these properties can be tested in polynomial time for sets defined by affine inequalities, as we describe next.

To check boundedness of a set $P \defeq \{x \in \mathbb{R}^n\ |\ a_i^Tx \ge b_i, i = 1,\ldots,m\}$ defined by affine inequalities, one can first check that $P$ is nonempty, and if it is, for each $i$ minimize and maximize $x_i$ over $P$. Note that $P$ is unbounded if and only if at least one of these $2n$ linear programs is unbounded, which can be certified e.g. by detecting infeasibility of the corresponding dual problem. Thus, boundedness of $P$ can be tested by solving $2n+1$ linear programming feasibility problems, which can be done in polynomial time.

To check closedness of a set $P \defeq \{x \in \mathbb{R}^n\ |\ a_i^Tx \ge b_i, i = 1, \ldots, m, c_j^Tx > d_j, j = 1, \ldots, r\}$, one can for each $j$ minimize $c_j^Tx$ over $\{x \in \mathbb{R}^n\ |\ a_i^Tx \ge b_i, i = 1, \ldots, m\}$ and declare that $P$ is closed if and only if all of the respective optimal values are greater than $d_j$. Thus, closedness of $P$ can be tested by solving $r$ linear programs, which can be done in polynomial time.

\begin{theorem}\label{AOS Thm: Closedness} 
	Given a set of quadratic polynomials $q_i, i = 1, \ldots, m, t_j, j = 1, \ldots, r$, it is strongly NP-hard to test whether the basic semialgebraic set
	$$\{x \in \mathbb{R}^n |~ q_i(x) \ge 0, i = 1, \ldots, m, t_j(x) > 0, j = 1, \ldots, r\}$$
	is closed\footnote{Note that $m$ is not fixed in this statement or in Corollaries \ref{AOS Thm: Boundedness} and \ref{AOS Thm: Stable Compact NP-hard} below.}.\end{theorem}
\begin{proof}
	Consider a ONE-IN-THREE 3SAT instance $\phi$ with $n$ variables and $k$ clauses. Let $\phi_{ij}$ be as in the proof of Theorem \ref{AOS Thm: AOS NP-hard o4c0} and consider the set
	\begin{equation}\label{AOS eq:def.S.phi}
	S_{\phi} = \big\{(x,y)\in \mathbb{R}^{n+1}|~(\phi_{i1} + \phi_{i2} + \phi_{i3} + 1)y = 0, i=1,\ldots,k, 1 - x_j^2 = 0, j = 1, \ldots, n, y < 1\big\}.
	\end{equation}
	
	We show that $S_{\phi}$ is closed if and only if the instance $\phi$ is not satisfiable. To see this, first note that we can rewrite $S_{\phi}$ as
	$$S_{\phi} = \Big\{\{-1,1\}^n \times \{0\} \Big\} \cup \Big\{\{x \in \mathbb{R}^n|~s_\phi(x) = 0\} \times \{y\in \mathbb{R}|~y < 1\}\Big\},$$
	where $s_{\phi}$ is as in the proof of Theorem \ref{AOS Thm: AOS NP-hard o4c0}. If $\phi$ is not satisfiable, then $S_{\phi}=\{-1,1\} ^n \times \{0\}$, which is closed. If $\phi$ is satisfiable, then $\{x \in \mathbb{R}^n|~s_\phi(x) = 0\}$ is nonempty and $$\{x \in \mathbb{R}^n|~s_\phi(x) = 0\} \times \{y\in \mathbb{R}|~y < 1\}$$ is not closed and not a subset of $\{-1,1\}^n \times \{0\}$. This implies that $S_{\phi}$ is not closed.	
\end{proof}

\begin{remark}\label{AOS rem:boundedness.compactness}
	We note that the problem of testing closedness of a basic semialgebraic set remains NP-hard even if one has a promise that the set is bounded. Indeed, one can add the constraint $y\geq -1$ to the set $S_{\phi}$ in (\ref{AOS eq:def.S.phi}) to make it bounded and this does not change the previous proof.
\end{remark}

\begin{cor}\label{AOS Thm: Boundedness} 
	Given a set of quadratic polynomials $q_i, i = 1, \ldots, m,$ it is strongly NP-hard to test whether the set
	$$\{x \in \mathbb{R}^n |~ q_i(x) = 0, i = 1, \ldots, m\}$$
	is bounded.
\end{cor}
\begin{proof}
	Consider a ONE-IN-THREE 3SAT instance $\phi$ with $n$ variables and $k$ clauses. Let $\phi_{ij}$ be as in the proof of Theorem \ref{AOS Thm: AOS NP-hard o4c0} and consider the set
	$$S = \big\{(x,y)\in \mathbb{R}^{n+1}|~(\phi_{i1} + \phi_{i2} + \phi_{i3} + 1)y = 0, i=1,\ldots,k, 1 - x_j^2 = 0, j = 1, \ldots, n\big\}.$$
	
	This set is bounded if and only if $\phi$ is not satisfiable. One can see this by following the proof of Theorem \ref{AOS Thm: Closedness} and observing that $y$ will be unbounded in the satisfiable case, and only 0 otherwise.
\end{proof}

Note that it follows immediately from either of the results above that testing compactness of a basic semialgebraic set is NP-hard. We end this subsection by establishing the same hardness result for a sufficient condition for compactness that has featured in the literature on polynomial optimization (see, e.g., \cite{marshall2003optimization}, \cite[Section 7]{nie2006minimizing}).

\begin{defn}\label{AOS Defn: Stably Compact}
	A closed basic semialgebraic set $S = \{x \in \mathbb{R}^n|~q_i(x) \ge 0, i = 1,\ldots, m\}$ is \emph{stably compact} if there exists $\epsilon > 0$ such that the set $\{x \in \mathbb{R}^n |~ \delta_i(x) + q_i(x) \ge 0,i = 1,\ldots,m\}$ is compact for any set of polynomials $\delta_i$ having degree at most that of $q_i$ and coefficients at most $\epsilon$ in absolute value.
\end{defn}

Intuitively, a closed basic semialgebraic set is stably compact if it remains compact under small perturbations of the coefficients of its defining polynomials. A stably compact set is clearly compact, though the converse is not true as shown by the set $$S = \big\{(x_1,x_2) \in \mathbb{R}^2|~ (x_1-x_2)^4 + (x_1+x_2)^2 \le 1\big\}.$$ Indeed, this set is contained inside the unit disk, but for $\epsilon>0,$ the set $$S_{\epsilon} = \big\{(x_1,x_2) \in \mathbb{R}^2|~ (x_1-x_2)^4 - \epsilon x_1^4 + (x_1+x_2)^2 \le 1\big\}$$ is unbounded as its defining polynomial tends to $-\infty$ along the line $x_1 = x_2$. 

Section 5 of \cite{marshall2003optimization} shows that the set $S$ in Definition \ref{AOS Defn: Stably Compact} is stably compact if and only if the function $$q(x)=\underset{i,j}\max\{-q_{ij}(x)\}$$ is positive on the unit sphere. Here, $q_{ij}(x)$ is a homogenenous polynomial that contains all terms of degree $j$ in $q_i(x).$ Perhaps because of this characterization, the same section in \cite{marshall2003optimization} remarks that ``stable compactness is easier to check than compactness'', though as far as polynomial-time checkability is concerned, we show that the situation is no better.

\begin{cor}\label{AOS Thm: Stable Compact NP-hard} 
	Given a set of quadratic polynomials $q_i, i = 1, \ldots, m,$ it is strongly NP-hard to test whether the set
	$$\{x \in \mathbb{R}^n | ~q_i(x) = 0, i = 1, \ldots, m\}$$
	is stably compact.
\end{cor}

\begin{proof}
	Consider a ONE-IN-THREE 3SAT instance $\phi$ with $n$ variables and $k$ clauses. Let $\phi_{ij}$ be as in the proof of Theorem \ref{AOS Thm: AOS NP-hard o4c0} and consider the set
	$$T_{\phi} = \big\{(x_0,x)\in \mathbb{R}^{n+1}|~(\phi_{i1} + \phi_{i2} + \phi_{i3} + x_0)^2 = 0, i=1,\ldots,k, x_0^2 - x_j^2 = 0, j = 1, \ldots, n\big\}.$$
	We show that the function $$q_{\phi}(x_0,x)=\max_{i=1,\ldots,k, j=1,\ldots,n}\{-(\phi_{i1} + \phi_{i2} + \phi_{i3} + x_0)^2, (\phi_{i1} + \phi_{i2} + \phi_{i3} + x_0)^2, x_0^2 - x_j^2, x_j^2-x_0^2\}$$ is positive on the unit sphere if and only if the instance $\phi$ is not satisfiable. Suppose first that $\phi$ is not satisfiable and assume for the sake of contradiction that there is a point $(x_0,x)$ on the sphere such that $q_{\phi}(x_0,x)=0.$ This implies that $\phi_{i1} + \phi_{i2} + \phi_{i3} + x_0=0, \forall i=1,\ldots,k$ and $x_0^2=x_j^2,\forall j=1,\ldots,n.$ Hence, $x_0 \neq 0$ and $\frac{x}{x_0}$ is a satisfying assignment to $\phi$, which is a contradiction. 
	Suppose now that $\phi$ has a satisfying assignment $\hat{x} \in \{-1,1\}^n.$ Then it is easy to check that 	$$q_{\phi}\left(\frac{(1,\hat{x})}{||(1,\hat{x})^T||}\right)=0.$$
\end{proof}

\subsubsection{The Archimedean Property}\label{AOS SSSec: Archimedean}

An algebraic notion closely related to compactness is the so-called Archimedean property. This notion has frequently appeared in recent literature at the interface of algebraic geometry and polynomial optimization. The Archimedean property is the assumption needed for the statement of Putinar's Positivstellensatz \cite{putinar1993positive} and convergence of the Lasserre hierarchy \cite{lasserre2001global}. In this subsection, we recall the definition of the Archimedean property and study the complexity of checking it. To our knowledge, the only previous result in this direction is that testing the Archimedean property is decidable \cite[Section 3.3]{wagner2009archimedean}.

Recall that the \emph{quadratic module} associated with a set of polynomials $q_1,\ldots,q_m$ is the set of polynomials that can be written as $$\sigma_0(x)+\sum_{i=1}^m \sigma_i(x) q_i(x),$$ where $\sigma_0,\ldots,\sigma_m$ are sum of squares polynomials.

\begin{defn}\label{AOS Defn: Archimedean}
	A quadratic module $Q$ is \emph{Archimedean} if there exists a scalar $R > 0$ such that $R - \sum_{i=1}^n x_i^2 \in Q$.
\end{defn}

Several equivalent characterizations of this property can be found in \cite[Theorem 3.17]{laurent2009sums}. Note that a set $\{x \in \mathbb{R}^n|~q_i(x) \ge 0\}$ for which the quadratic module associated with the polynomials $\{q_i\}$ is Archimedean is compact. However, the converse is not true. For example, for $n>1$, the sets $$\left\{x \in \mathbb{R}^n|~x_1-\frac{1}{2} \ge 0, \ldots, x_n - \frac{1}{2} \ge 0, 1-\prod_{i=1}^n x_i \ge 0\right\}$$ are compact but not Archimedean; see \cite{laurent2009sums}, \cite{prestel2001positive} for a proof of the latter claim. Hence, hardness of testing the Archimedean property does not follow from hardness of testing compactness. 

As mentioned previously, the Archimedean property has received recent attention in the optimization community due to its connection to the Lasserre hierarchy. Indeed, under the assumption that the quadratic module associated with the defining polynomials of the feasible set of (\ref{AOS Defn: pop}) is Archimedean, the Lasserre hierarchy \cite{lasserre2001global} produces a sequence of SDP-based lower bounds that converge to the optimal value of the POP. Moreover, Nie has shown~\cite{nie2014optimality} that under the Archimedean assumption, convergence happens in a finite number of rounds generically. One way to ensure the Archimedean property---assuming that we know that our feasible set is contained in a ball of radius $R$--- is to add the redundant constraint $R^2 \geq \sum_{i=1}^n x_i^2$ to the constraints of (\ref{AOS Defn: pop}). This approach however increases the size of the SDP that needs to be solved at each level of the hierarchy. Moreover, such a scalar $R$ may not be readily available for some applications.

Our proof of NP-hardness of testing the Archimedean property will be based on showing that the specific sets that arise from the proof of Corollary \ref{AOS Thm: Boundedness} are compact if and only if their corresponding quadratic modules are Archimedean. Our proof technique will use the Stengle's Positivstellensatz, which we recall next.

\begin{theorem}[Stengle's Positivstellensatz \cite{stengle1974nullstellensatz}]\label{AOS Thm: Stengle}
	A basic semialgebraic set $$\mathcal{S} \defeq \{x\in\mathbb{R}^n|~ q_i(x) \ge 0, i= 1,\ldots, m, r_j(x) = 0, j = 1,\ldots, k\}$$ is empty if and only if there exist sos polynomials $\sigma_{c_1, \ldots, c_m}$ and polynomials $t_i$ such that
	$$-1 = \sum_{j=1}^k t_jr_j + \sum_{c_1,\ldots,c_m \in \{0,1\}^m} \sigma_{c_1,\ldots,c_m}(x) \Pi_{i=1}^m q_i(x)^{c_i}.$$
\end{theorem}

\begin{remark} Note that if only equality constraints are considered, the second term on the right hand side is a single sos polynomial $\sigma_{0,\ldots, 0}$. In the next theorem, we only need this special case, which is also known as the Real Nullstellensatz \cite{krivine1964anneaux}.
\end{remark}

\begin{theorem}\label{AOS Thm: Archimedean NP-hard} 
	Given a set of quadratic polynomials $q_1,\ldots,q_m$, it is strongly NP-hard to test whether their quadratic module has the Archimedean property.
	
\end{theorem}

\begin{proof}
	Consider a ONE-IN-THREE 3SAT instance $\phi$ with $n$ variables and $k$ clauses. Let $\phi_{ij}$ be as in the proof of Theorem \ref{AOS Thm: AOS NP-hard o4c0} and consider the set of quadratic polynomials
	$$\big\{(\phi_{i1}+\phi_{i2}+\phi_{i3})y, -(\phi_{i1}+\phi_{i2}+\phi_{i3})y, i=1,\ldots,k; 1-x_j^2, x_j^2-1, j=1,\ldots,n\big\}.$$
	
	We show that the quadratic module associated with these polynomials is Archimedean if and only if $\phi$ is not satisfiable. First observe that if $\phi$ is satisfiable, then the quadratic module cannot be Archimedean as the set 
	$$S = \big\{(x,y)\in \mathbb{R}^{n+1}|~(\phi_{i1} + \phi_{i2} + \phi_{i3} + 1)y = 0, i=1,\ldots,k, 1 - x_j^2 = 0, j = 1, \ldots, n\big\}$$
	is not compact (see the proof of Corollary \ref{AOS Thm: Boundedness}).
	
	Now suppose that the instance $\phi$ is not satisfiable. 
	We need to show that for some scalar $R > 0$ and some sos polynomials $\sigma_0,\sigma_1,\ldots,\sigma_k,$ $\hat{\sigma}_1,\ldots,\hat{\sigma}_k$,$\tau_1,\ldots,\tau_n$,$\hat{\tau}_1,\ldots,\hat{\tau}_n$, we have
	\begin{align*}
	R-\sum_{i=1}^n x_i^2-y^2 &=\sigma_0(x,y)+\sum_{i=1}^k \sigma_i(x,y) (\phi_{i1}+\phi_{i2}+\phi_{i3}+1)y\\
	&+\sum_{i=1}^k \hat{\sigma}_i(x,y) (-\phi_{i1}-\phi_{i2}-\phi_{i3}-1)y
	+\sum_{j=1}^n \tau_j(x,y) (1-x_j^2)+\sum_{j=1}^n\hat{\tau}_j(x,y)(x_j^2-1).
	\end{align*}
	Since any polynomial can be written as the difference of two sos polynomials (see, e.g., \cite[Lemma 1]{ahmadi2015dc}), this is equivalent to existence of a scalar $R>0$, an sos polynomial $\sigma_0$, and some polynomials $v_1,\ldots,v_k, t_1,\ldots,t_n$ such that 
	\begin{align}\label{AOS eq:archimedean}
	R-\sum_{i=1}^n x_i^2-y^2 &=\sigma_0(x,y)+\sum_{i=1}^k v_i(x,y) (\phi_{i1}+\phi_{i2}+\phi_{i3}+1)y +\sum_{j=1}^n t_j(x,y) (1-x_j^2).
	\end{align}
	
	First, note that 
	\begin{align}\label{AOS eq:first.part.arch}
	n-\sum_{i=1}^n x_i^2=\sum_{i=1}^n (1-x_i^2).
	\end{align}
	Secondly, as $\phi$ is not satisfiable, we know that the set 
	$$\{x \in \mathbb{R}^n|~1-x_j^2 = 0, j= 1,\ldots,n, \phi_{i1} + \phi_{i2} + \phi_{i3} + 1 = 0,i = 1, \ldots, k\}$$
	is empty. From Stengle's Positivstellensatz, it follows that there exist an sos polynomial $\tilde{\sigma}_0$ and some polynomials $\tilde{v}_1,\ldots, \tilde{v}_k$, $\tilde{t}_1,\ldots,\tilde{t}_n$ such that 
	\begin{align*}
	-1=\tilde{\sigma}_0(x)+\sum_{i=1}^k \tilde{v}_i(x)(\phi_{i1}+\phi_{i2}+\phi_{i3}+1)+\sum_{j=1}^n \tilde{t}_j(x)(1-x_j^2).
	\end{align*}
	Multiplying this identity on either side by $y^2$, we obtain:
	\begin{align}\label{AOS eq:second.part.arch}
	-y^2=y^2\tilde{\sigma}_0(x)+\sum_{i=1}^k \tilde{v}_i(x)y \cdot(\phi_{i1}+\phi_{i2}+\phi_{i3}+1)y+\sum_{j=1}^n \tilde{t}_j(x)y^2(1-x_j^2).
	\end{align}
	Note that if we sum (\ref{AOS eq:first.part.arch}) and (\ref{AOS eq:second.part.arch}) and take $R=n$, $\sigma_0(x,y)=y^2\tilde{\sigma}_{0}(x)$, $v_i(x,y)=y\tilde{v}_i(x)$ for all $i=1,\ldots,k$, and $t_j(x,y)=y^2 \cdot \tilde{t}_j(x)+1$ for all $j=1,\ldots,n$, we recover (\ref{AOS eq:archimedean}).
	
\end{proof}

\section{Algorithms for Testing Attainment of the Optimal Value}\label{AOS Sec: Algorithms}
In this section, we give a hierarchy of sufficient conditions for compactness of a closed basic semialgebraic set, and a hierarchy of sufficient conditions for coercivity of a polynomial. These hierarchies are amenable to semidefinite programming (SDP) as they all involve, in one way or another, a search over the set of sum of squares polynomials. The connection between SDP and sos polynomials is well known:
recall that a polynomial $\sigma$ of degree $2d$ is sos if and only if there exists a symmetric positive semidefinite matrix $Q$ such that $\sigma(x)=z(x)^TQz(x)$ for all $x$, where $z(x)$ here is the standard vector of monomials of degree up to $d$ in the variables $x$ (see, e.g. \cite{parrilo2003semidefinite}).

The hierarchies that we present are such that if the property in question (i.e., compactness or coercivity) is satisfied on an input instance, then some level of the SDP hierarchy will be feasible and provide a certificate that the property is satisfied. The test for compactness is a straightforward application of Stengle's Positivstellensatz, but the test for coercivity requires a new characterization of this property, which we give in Theorem \ref{AOS Thm: Polynomial Radius}.

\subsection{Compactness of the Feasible Set}\label{AOS SSec: Compactness Alg}

Consider a closed basic semialgebraic set $$\mathcal{S} \defeq \{x \in \mathbb{R}^n | ~q_i(x) \ge 0, i = 1, \ldots, m\},$$ where the polynomials $q_i$ have integer coefficients\footnote{If some of the coefficients of the polynomials $q_i$ are rational (but not integer), we can make them integers by clearing denominators without changing the set $\mathcal{S}$.} and are of degree at most $d$. A result of Basu and Roy \cite[Theorem 3]{basu2010bounding} implies that if $\mathcal{S}$ is bounded, then it must be contained in a ball of radius 
\begin{equation}\label{AOS Eq: Basu Bound}
R^* \defeq \sqrt{n}\left((2d+1)(2d)^{n-1}+1\right)2^{(2d+1)(2d)^{n-1}(2nd+2)\left(2\tau + bit((2d+1)(2d)^{n-1})+(n+1)bit(d+1)+bit(m)\right)},
\end{equation}
where $\tau$ is the largest bitsize of any coefficient of any $q_i$, and bit($\eta$) denotes the bitsize of $\eta$.

With this result in mind, the following proposition is an immediate consequence of Stengle's Positivstellensatz (c.f. Theorem \ref{AOS Thm: Stengle}) after noting that the set $\mathcal{S}$ is bounded if and only if the set $$\{x\in \mathbb{R}^n|~q_i(x) \ge 0, i = 1, \ldots, m, \sum_{i=1}^n x_i^2 \ge R^* + 1\}$$ is empty.
\begin{prop}\label{AOS Prop: Compactness Stengle}
	Consider a closed basic semialgebraic set $\mathcal{S} \defeq \{x \in \mathbb{R}^n |~ q_i(x) \ge 0,$ ${i = 1, \ldots, m\}}$, where the polynomials $q_i$ have integer coefficients and are of degree at most $d$. Let $R^*$ be as in (\ref{AOS Eq: Basu Bound}) and let $q_0(x) = \sum_{i=1}^n x_i^2 - R^* -1$. Then the set $\mathcal{S}$ is compact if and only if there exist sos polynomials $\sigma_{h_0,\ldots, h_m}$ such that
	$$-1 = \sum_{h_0,...,h_m \in \{0,1\}^{m+1}} \sigma_{h_0,...,h_m}(x) \Pi_{i=0}^m q_i(x)^{h_i}.$$
\end{prop}
This proposition naturally yields the following semidefinite programming-based hierarchy indexed by a nonnegative integer $r$:

\begin{equation}\label{AOS Eq: Compactness SDP}
\begin{aligned}
& \underset{\sigma_{h_0, \ldots, h_m}}{\min}
& & 0 \\
& \text{subject to}
&& -1 = \sum_{h_0,\ldots,h_m \in \{0,1\}^{m+1}} \sigma_{h_0,\ldots,h_m}(x) \Pi_{i=0}^m q_i(x)^{h_i},\\
&&&\sigma_{h_0, \ldots, h_m} \text{ is sos and has degree }\leq 2r.\\
\end{aligned}
\end{equation}

Note that for a fixed level $r$, one is solving a semidefinite program whose size is polynomial in the description of $\mathcal{S}$. If for some $r$ the SDP is feasible, then we have an algebraic certificate of compactness of the set $\mathcal{S}.$ Conversely, as Proposition \ref{AOS Prop: Compactness Stengle} implies, if $\mathcal{S}$ is compact, then the above SDP will be feasible for some level $r^*$. One can upper bound $r^*$ by a function of $n, m,$ and $d$ only using the main theorem of \cite{lombardi2014elementary}. This bound is however very large and mainly of theoretical interest.
\subsection{Coercivity of the Objective Function}\label{AOS SSec: Coercivity Alg}

It is well known that the infimum of a continuous coercive function over a closed set is attained. This property has been widely studied, even in the case of polynomial functions; see e.g. \cite{bajbar2015coercive,bajbar2017fast,jeyakumar2014polynomial}.
A simple sufficient condition for coercivity of a polynomial $p$ is for its terms of highest order to form a positive definite (homogeneous) polynomial; see, e.g., \cite[Lemma 4.1]{jeyakumar2014polynomial}. One can give a hierarchy of SDPs to check for this condition as is done in \cite[Section 4.2]{jeyakumar2014polynomial}. However, this condition is sufficient but not necessary for coercivity. For example, the polynomial $x_1^4 + x_2^2$ is coercive, but its top homogeneous component is not positive definite. Theorem \ref{AOS Thm: Polynomial Radius} below gives a necessary and sufficient condition for a polynomial to be coercive which lends itself again to an SDP hierarchy. To start, we need the following proposition, whose proof is straightforward and thus omitted.

\begin{prop}\label{AOS Prop: Coercive sublevel sets bounded}
	A function $f:\mathbb{R}^n \rightarrow \mathbb{R}$ is coercive if and only if the sets $$\mathcal{S}_\gamma \defeq \{x \in \mathbb{R}^n|~f(x) \le \gamma\}$$ are bounded for all $\gamma \in \mathbb{R}$.
\end{prop}

A polynomial $p$ is said to be \emph{$s$-coercive} if $p(x)/\|x\|^s$ is coercive. The \emph{order of coercivity} of $p$ is the supremum over $s \ge 0$ for which $p$ is $s$-coercive. It is known that the order of coercivity of a coercive polynomial is always positive \cite{bajbar2017fast,gorin1961asymptotic}.

\begin{theorem}\label{AOS Thm: Polynomial Radius}
	A polynomial $p$ is coercive if and only if there exist an even integer $c > 0$ and a scalar $k \ge 0$ such that for all $\gamma \in \mathbb{R}$, the $\gamma$-sublevel set of $p$ is contained within a ball of radius $\gamma^c+k$.
\end{theorem}

\begin{proof}
	
	The ``if'' direction follows immediately from the fact that each $\gamma$-sublevel set is bounded. For the converse, suppose that $p$ is coercive and denote its order of coercivity by $q > 0$. Then, from Observation 2 of \cite{bajbar2017fast}, we get that there exists a scalar $M \ge 0$ such that 
	\begin{equation}\label{AOS Eq: M q-coercive}
	\|x\| \ge M \Rightarrow p(x) \ge \|x\|^q,
	\end{equation} or equivalently $\|x\| \le p(x)^\frac{1}{q}$. Now consider the function $R_p: \mathbb{R} \to \mathbb{R}$ which is defined as $$R_p(\gamma) \defeq \underset{p(x) \le \gamma}{\max} \|x\|,$$ i.e. the radius of the $\gamma$-sublevel set of $p$. We note two relevant properties of this function.
	
	\begin{itemize}
		\item The function $R_p(\gamma)$ is nondecreasing. This is because the $\gamma$-sublevel set of $p$ is a subset of the $(\gamma+\epsilon)$-sublevel set of $p$ for any $\epsilon > 0$.
		\item Let $m = \inf \{\gamma| ~R_p(\gamma) \ge M\}$. We claim that \begin{equation}\label{AOS Eq: m ge Mq}
		m \ge M^q.
		\end{equation}
		Suppose for the sake of contradiction that we had $m < M^q$. By the definition of $m$, there exists $\bar{\gamma} \in (m,M^q)$ such that $R_p(\bar{\gamma}) \ge M$. This means that there exists $\bar{x} \in \mathbb{R}^n$ such that $p(\bar{x}) \le \bar{\gamma} < M^q$ and $\|\bar{x}\| \ge M$. From (\ref{AOS Eq: M q-coercive}) we then have $p(\bar{x}) \ge \|\bar{x}\|^q \ge M^q$, which is a contradiction.
	\end{itemize}
	
	We now claim that $R_p(\gamma) \le \gamma^\frac{1}{q}$ for all $\gamma > m$. Suppose for the sake of contradiction that there exists $\gamma_0 > m$ such that $R_p(\gamma_0) > \gamma_0^\frac{1}{q}$. This means that there exists $x_0 \in \mathbb{R}^n$ such that $\|x_0\| > \gamma_0^\frac{1}{q}$ but $p(x_0) \le \gamma_0$.
	
	Consider first the case where $p(x_0) \ge m$. Since $\gamma_0 > m$, we have
	$$\|x_0\| > \gamma_0^{1/q} > m^{1/q} \ge M,$$
	where the last inequality follows from (\ref{AOS Eq: m ge Mq}). It follows from (\ref{AOS Eq: M q-coercive}) that $p(x_0) \ge \|x_0\|^q > \gamma_0$ which is a contradiction.
	
	Now consider the case where $p(x_0) < m$. By definition of $m$, we have $R_p(p(x_0)) < M$, and so $\|x_0\| < M$. Furthermore, since $$\gamma_0 > m \overset{(\ref{AOS Eq: m ge Mq})}{\ge} M^q,$$ we have $M < \gamma_0^{1/q}$, which gives $\|x_0\| < M < \gamma_0^{1/q}$. This contradicts our previous assumption that $\|x_0\| > \gamma_0^{1/q}$.
	
	If we let $c$ be the smallest even integer greater than $1/q$, we have shown that $$R_p(\gamma) \le \gamma^{1/q} \le \gamma^c$$ on the set $\gamma > m$. Finally, if we let $k = R_p(m)$, by monotonicity of $R_p$, we get that $R_p(\gamma) \le \gamma^c + k$, for all $\gamma$.
	
\end{proof}

\begin{remark}\label{AOS Rem: Polynomial R2}
	One can easily show now that for any coercive polynomial $p$, there exist an integer $c' > 0$ and a scalar $k' \ge 0$ (possibly differing from the scalars $c$ and $k$ given in the proof of Theorem \ref{AOS Thm: Polynomial Radius}) such that $$R_p^2(\gamma) < \gamma^{2c'} + k'.$$ For the following hierarchy it will be easier to work with this form.
\end{remark}

In view of the above remark, observe that coercivity of a polynomial $p$ is equivalent to existence of an integer $c' > 0$ and a scalar $k' \ge 0$ such that the set \begin{equation}\label{AOS Eq: Coercivity Set}\left\{(\gamma, x)\in \mathbb{R}^{n+1}|~p(x) \le \gamma, \sum_{i=1}^n x_i^2 \ge \gamma^{2c'} + k'\right\}\end{equation} is empty. This formulation naturally leads to the following SDP hierarchy indexed by a positive integer $r$.

\begin{prop}\label{AOS prop:coercivity.SDP.hierarchy}
	A polynomial $p$ of degree $d$ is coercive if and only if for some integer $r \geq 1$, the following SDP is feasible:
	
	\begin{equation}\label{AOS Eq: Coercivity SDP}
	\begin{aligned}
	& \underset{\sigma_0, \ldots, \sigma_3}{\min}
	& & 0 \\
	& \emph{subject to}
	& -1 =& \sigma_0(x,\gamma) + \sigma_1(x,\gamma)(\gamma - p(x)) + \sigma_2(x,\gamma)\left(\sum_{i=1}^n x_i^2 - \gamma^{2r} - 2^r\right) \\
	&&&+ \sigma_3(x,\gamma)(\gamma - p(x))\left(\sum_{i=1}^n x_i^2 - \gamma^{2r} - 2^r\right),\\
	&&& \sigma_0 \emph{ is sos and of degree } \leq 4r,\\
	&&& \sigma_1 \emph{ is sos and of degree } \leq \max\{4r-d,0\},\\
	&&& \sigma_2 \emph{ is sos and of degree } \leq 2r,\\
	&&& \sigma_3 \emph{ is sos and of degree } \leq \max\{2r-d,0\}.\\
	\end{aligned}
	\end{equation}
\end{prop}

\begin{proof}
	If the SDP in (\ref{AOS Eq: Coercivity SDP}) is feasible for some $r$, then the set
	\begin{equation}\label{AOS Eq: Coercive Empty Set}
	\left\{(\gamma, x) \in \mathbb{R}^{n+1}|~p(x) \le \gamma, \sum_{i=1}^n x_i^2 \ge \gamma^{2r} + 2^r\right\}
	\end{equation}
	must be empty. Indeed, if this was not the case, a feasible $(\gamma,x)$ pair would make the right hand side of the equality constraint of (\ref{AOS Eq: Coercivity SDP}) nonnegative, while the left hand side is negative. As the set in (\ref{AOS Eq: Coercive Empty Set}) is empty, then for all $\gamma$, the $\gamma$-sublevel set of $p$ is contained within a ball of radius $\sqrt{\gamma^{2r} + 2^r}$ and thus $p$ is coercive.
	
	To show the converse, suppose that $p$ is coercive. Then we know from  Theorem \ref{AOS Thm: Polynomial Radius} and Remark \ref{AOS Rem: Polynomial R2} that there exist an integer $c' > 0$ and a scalar $k' \ge 0$ such that the set in (\ref{AOS Eq: Coercivity Set}) is empty. From Stengle's Positivstellensatz (c.f. Theorem \ref{AOS Thm: Stengle}), there exist an even nonnegative integer $\hat{r}$ and sos polynomials $\sigma_0', \ldots, \sigma_3'$ of degree at most $\hat{r}$ such that
	\begin{equation}\label{AOS eq:Stengle.coercive}
	\begin{aligned}
	-1 =& \sigma_0'(x,\gamma) + \sigma_1'(x,\gamma)(\gamma - p(x)) + \sigma_2'(x,\gamma)\left(\sum_{i=1}^n x_i^2 - \gamma^{2c'} - k'\right) \\
	&+ \sigma_3'(x,\gamma)(\gamma - p(x))\left(\sum_{i=1}^n x_i^2 - \gamma^{2c'} - k'\right).
	\end{aligned}	
	\end{equation}
	Let $r^*=\lceil \max\{c', \log_2(k'+1), \frac{\hat{r}+d}{2}\}\rceil$. We show that the SDP in (\ref{AOS Eq: Coercivity SDP}) is feasible for $r=r^*$ by showing that the polynomials
	$$\sigma_0(x,\gamma) = \sigma_0'(x,\gamma) + \sigma_2'(x,y)(\gamma^{2r^*}-\gamma^{2c'} + 2^{r^*}-k'),$$
	$$\sigma_1(x,\gamma) = \sigma_1'(x,\gamma)+\sigma_3'(x,\gamma)(\gamma^{2r^*}-\gamma^{2c'}+ 2^{r^*}-k'),$$
	$$\sigma_2(x,\gamma) = \sigma_2'(x,\gamma),$$
	$$\sigma_3(x,\gamma) = \sigma_3'(x,\gamma)$$
	are a feasible solution to the problem. First, note that $$4r^*-d\geq 2r^*-d\geq \hat{r}\geq 0,$$ and so $\sigma_0$ is of degree at most $\hat{r}+2r^* \leq 4r^*$, $\sigma_1$ is of degree at most ${\hat{r}+2r^*\leq \max\{4r^*-d,0\}}$, $\sigma_2$ is of degree at most $\hat{r}\leq 2r^*$, and $\sigma_3$ is of degree at most $\hat{r}\leq \max\{2r^*-d,0\}$. Furthermore, these polynomials are sums of squares. To see this, note that $\gamma^{2r^*}-\gamma^{2c'}+2^{r^*}-k'$ is nonnegative as $r^*\geq c'$ and $2^{r^*}\geq k'+1$. As any nonnegative univariate polynomial is a sum of squares (see, e.g., \cite{blekherman2012semidefinite}), it follows that $\gamma^{2r^*}-\gamma^{2c'}+ 2^{r^*}-k'$ is a sum of squares. Combining this with the facts that $\sigma_0',\ldots,\sigma_3'$ are sums of squares, and products and sums of sos polynomials are sos again, we get that $\sigma_0,\ldots,\sigma_3$ are sos. Finally, the identity
	\begin{align*}
	-1 =& \left(\sigma_0'(x,\gamma)+ \sigma_2'(x,y)(\gamma^{2r^*}-\gamma^{2c'} + 2^{r^*}-k')\right)\\
	&+ \left(\sigma_1'(x,\gamma)+\sigma_3'(x,\gamma)(\gamma^{2r^*}-\gamma^{2c'}+ 2^{r^*}-k'))(\gamma - p(x)\right) \\
	&+ \sigma_2'(x,\gamma)(\sum_{i=1}^n x_i^2 - \gamma^{2r^*} - 2^{r^*}) + \sigma_3'(x,\gamma)(\gamma - p(x))(\sum_{i=1}^n x_i^2 - \gamma^{2r^*} - 2^{r^*})
	\end{align*}
	holds by a simple rewriting of (\ref{AOS eq:Stengle.coercive}).
	
\end{proof}

As an illustration, we revisit the simple example $p(x) = x_1^4 + x_2^2$, whose top homogeneous component is not positive definite.
The hierarchy in Proposition \ref{AOS prop:coercivity.SDP.hierarchy} with $r=1$ gives an automated algebraic proof of coercivity of $p$ in terms of the following identity:
\begin{equation}\label{AOS Eq: Coercivity Hierarchy Example}
-1 = \left(\frac{2}{3}(x_1^2 - \frac{1}{2})^2 + \frac{2}{3}(\gamma-\frac{1}{2})^2\right) + \frac{2}{3}(\gamma - x_1^4 - x_2^2) + \frac{2}{3}(x_1^2 + x_2^2 - \gamma^2 - 2).
\end{equation}

Note that this is a certificate that the $\gamma$-sublevel set of $p$ is contained in a ball of radius $\sqrt{\gamma^2+2}$.

\begin{remark}
	From a theoretical perspective, our developments so far show that coercivity of multivariate polynomials is a decidable property as it can be checked by solving a finite number of SDP feasibility problems (each of which can be done in finite time \cite{porkolab1997complexity}). Indeed, given a polynomial $p$, one can think of running two programs in parallel. The first one solves the SDPs in Proposition~\ref{AOS prop:coercivity.SDP.hierarchy} for increasing values of $r$. The second uses Proposition~\ref{AOS Prop: Compactness Stengle} and its degree bound to test whether the $\beta$-sublevel set $p$ is compact, starting from $\beta=1$, and doubling $\beta$ in each iteration. On every input polynomial $p$ whose coercivity is in question, either the first program halts with a yes answer or the second program halts with a no answer. We stress that this remark is of theoretical interest only, as the value of our contribution is really in providing proofs of coercivity, not proofs of non-coercivity. Moreover, coercivity can alternatively be decided in finite time by applying the quantifier elimination theory of Tarski and Seidenberg \cite{tarski1951decision,seidenberg1954new} to the characterization in Proposition~\ref{AOS Prop: Coercive sublevel sets bounded}.
\end{remark}

\section{Summary and Conclusions} \label{AOS Sec:conclusion}

We studied the complexity of checking existence of optimal solutions in mathematical programs (given as minimization problems) that are feasible and lower bounded. We showed that unless P=NP, this decision problem does not have a polynomial time (or even pseudo-polynomial time) algorithm when the constraints and the objective function are defined by polynomials of low degree. More precisely, this claim holds if the constraints are defined by quadratic polynomials (and the objective has degree as low as one) or if the objective function is a quartic polynomial (even in absence of any constraints). For polynomial optimization problems with linear constraints and objective function of degrees 1,2, or 3, previous results imply that feasibility and lower boundedness always guarantee existence of an optimal solution.

We also showed, again for low-degree polynomial optimization problems, that several well-known sufficient conditions for existence of optimal solutions are NP-hard to check. These were coercivity of the objective function, closedness of the feasible set (even when bounded), boundedness of the feasible set (even when closed), an algebraic certificate of compactness known as the Archimedean property, and a robust analogue of compactness known as stable compactness.

Our negative results should by no means deter researchers from studying algorithms that can efficiently check existence of optimal solutions---or, for that matter, any of the other properties mentioned above such as compactness and coercivity---on special instances. On the contrary, our results shed light on the intricacies that can arise when studying these properties and calibrate the expectations of an algorithm designer. Hopefully, they will even motivate further research in identifying problem structures (e.g., based on the Newton polytope of the objective and/or constraints) for which checking these properties becomes more tractable, or efficient algorithms that can test useful sufficient conditions that imply these properties.

In the latter direction, we argued that sum of squares techniques could be a natural tool for certifying compactness of basic semialgebraic sets via semidefinite programming. By deriving a new characterization of coercive polynomials, we showed that the same statement also applies to the task of certifying coercivity. This final contribution motivates a problem that we leave for our future research. While coercivity (i.e., boundedness of all sublevel sets) of a polynomial objective function guarantees existence of optimal solutions to a feasible POP, the same guarantee can be made from the weaker requirement that \emph{some} sublevel set of the objective be bounded and have a non-empty intersection with the feasible set. It is not difficult to show that this property is also NP-hard to check. However, it would be useful to derive a hierarchy of sufficient conditions for it, where each level can be efficiently tested (perhaps again via SDP), and such that if the property was satisfied, then a level of the hierarchy would hold.\\
\allowdisplaybreaks

\chapter{Semidefinite Programming Relaxations for Nash Equilibria in Bimatrix Games}\label{Chap: Nash}
\section{Introduction}\label{NASH Sec: Intro}

A bimatrix game is a game between two players (referred to in this chapter as players A and B) defined by a pair of $m \times n$ payoff matrices $A$ and $B$. Let $\triangle_m$ and $\triangle_n$ denote the $m$-dimensional and $n$-dimensional simplices $$\triangle_m = \{x \in \mathbb{R}^m |\  x_i \ge 0, \forall i, \sum_{i = 1}^m x_i = 1\}, \triangle_n = \{y \in \mathbb{R}^n |\  y_i \ge 0, \forall i, \sum_{i = 1}^n y_i = 1\}.$$ These form the strategy spaces of player A and player B respectively. For a strategy pair $(x,y)\in \triangle_m \times \triangle_n $, the payoff received by player A (resp. player B) is $x^TAy$ (resp. $x^TBy$). In particular, if the players pick vertices $i$ and $j$ of their respective simplices (also called pure strategies), their payoffs will be $A_{i,j}$ and $B_{i,j}$. One of the prevailing solution concepts for bimatrix games is the notion of \emph{Nash equilibrium}. At such an equilibrium, the players are playing mutual best responses, i.e., a payoff maximizing strategy against the opposing player's strategy. In our notation, a Nash equilibrium for the game $(A,B)$ is a pair of strategies $(x^*, y^*)\in \triangle_m \times \triangle_n $ such that $$x^{*T}Ay^* \ge x^TAy^*, \forall x\in \triangle_m,$$ and $$x^{*T}By^* \ge x^{*T}By, \forall y\in \triangle_n.\footnote{In this chapter we assume that all entries of $A$ and $B$ are between 0 and 1, and argue at the beginning of Section~\ref{NASH Sec: Intro to SDP} why this is without loss of generality for the purpose of computing Nash equilibria.}$$ 

Nash~\cite{nash1951} proved that for any bimatrix game, such pairs of strategies exist (in fact his result more generally applies to games with a finite number of players and a finite number of pure strategies). While existence of these equilibria is guaranteed, finding them is believed to be a computationally intractable problem. More precisely, a result of ~\cite{daskalakis2009complexity} implies that computing Nash equilibria is PPAD-complete (see~\cite{daskalakis2009complexity} for a definition) even when the number of players is 3. This result was later improved by \cite{chen2006settling} who showed the same hardness result for bimatrix games.

These results motivate the notion of an approximate Nash equilibrium, a solution concept in which players receive payoffs ``close'' to their best response payoffs. More precisely, a pair of strategies $(x^*, y^*)\in \triangle_m \times \triangle_n$ is an \emph{(additive) $\epsilon$-Nash equilibrium} for the game $(A,B)$ if $$x^{*T}Ay^* \ge x^TAy^* - \epsilon, \forall x\in \triangle_m,$$ and $$x^{*T}By^* \ge x^{*T}By- \epsilon, \forall y\in \triangle_n. \footnote{There are also other important notions of approximate Nash equilibria, such as $\epsilon$-well-supported Nash equilibria~\cite{fearnley2016approximate} and relative approximate Nash equilibria~\cite{daskalakis2013complexity} which are not considered in this chapter.}$$ 
Note that when $\epsilon=0$, $(x^*,y^*)$ form an exact Nash equilibrium, and hence it is of interest to find $\epsilon$-Nash equilibria with $\epsilon$ small. Unfortunately, approximation of Nash equilibria has also proved to be computationally difficult. \cite{chen2006computing} have shown that, unless PPAD $\subseteq$ P, there cannot be a fully polynomial-time approximation scheme for computing Nash equilibria in bimatrix games. There have, however, been a series of constant factor approximation algorithms for this problem (\cite{daskalakis2007progress, daskalakis2006note, kontogiannis2006polynomial, tsaknakis2007optimization}), with the current best producing a .3393 approximation via an algorithm by~\cite{tsaknakis2007optimization}.

We remark that there are exponential-time algorithms for computing Nash equilibria, such as the Lemke-Howson algorithm (\cite{lemke1964equilibrium, savani2006hard}). There are also certain subclasses of the problem which can be solved in polynomial time, the most notable example being the case of zero-sum games (i.e. when $B=-A$). This problem was shown to be solvable via linear programming by~\cite{dantzig1951proof}, and later shown to be polynomially equivalent to linear programming by~\cite{adler2013equivalence}. Aside from computation of Nash equilibria, there are a number of related decision questions which are of economic interest but unfortunately NP-hard. Examples include deciding whether a player's payoff exceeds a certain threshold in some Nash equilibrium, deciding whether a game has a unique Nash equilibrium, or testing whether there exists a Nash equilibrium where a particular set of strategies is not played (\cite{gilboa1989nash, conitzer2002complexity}).

Our focus in this chapter is on understanding the power of semidefinite programming\footnote{The unfamiliar reader is referred to~\cite{vandenberghe1996semidefinite} for the theory of SDPs and a description of polynomial-time algorithms for them based on interior point methods.} (SDP) for finding approximate Nash equilibria in bimatrix games or providing certificates for related decision questions. The goal is not to develop a competitive solver, but rather to analyze the algorithmic power of SDP when applied to basic problems around computation of Nash equilibria. Semidefinite programming relaxations have been analyzed in depth in areas such as combinatorial optimization (\cite{goemans1995improved},~\cite{lovasz1979shannon}) and systems theory (\cite{boyd1994linear}), but not to such an extent in game theory. To our knowledge, the appearance of SDP in the game theory literature includes the work of \cite{stein1943exchangeable} for exchangeable equilibria in symmetric games, of \cite{parrilo2006polynomial} on zero-sum polynomial games, of \cite{shah2007polynomial} for zero-sum stochastic games, and of \cite{laraki2012semidefinite} for semialgebraic min-max problems in static and dynamic games.

\subsection{Organization and Contributions of the chapter}
In Section~\ref{NASH Sec: Intro to SDP}, we formulate the problem of finding a Nash equilibrium in a bimatrix game as a nonconvex quadratically constrained quadratic program and pose a natural SDP relaxation for it. In Section~\ref{NASH Sec: Exactness Zero Sum}, we show that our SDP is exact when the game is strictly competitive (see Definition~\ref{NASH Defn: Strictly Competitive Def}).
In Section~\ref{NASH Sec: Algorithms}, we design two continuous but nonconvex objective functions for our SDP whose global minima coincide with rank-1 solutions. We provide a heuristic based on iterative linearization for minimizing both objective functions. We show empirically that these approaches produce $\epsilon$ very close to zero (on average in the order of $10^{-3}$).
In Section~\ref{NASH Sec: Bounds}, we establish a number of bounds on the quality of the approximate Nash equilibria that can be read off of feasible solutions to our SDP. In Theorems \ref{NASH Thm: nksl}, \ref{NASH Thm: Trace minus 2xxt}, and \ref{NASH Thm: 1-1/k}, we show that when the SDP returns solutions which are ``close'' to rank-1, the resulting strategies have have small $\epsilon$. We then present an improved analysis in the rank-2 case which shows how one can recover a $\frac{5}{11}$-Nash equilibrium from the SDP solution (Theorem~\ref{NASH Thm: 5/11e}). We further prove that for symmetric games (i.e., when $B = A^T$), a $\frac{1}{3}$-Nash equilibrium can be recovered in the rank-2 case (Theorem~\ref{NASH Thm: 1/3se}). We do not currently know of a polynomial-time algorithm for finding rank-2 solutions to our SDP. If such an algorithm were found, it would, together with our analysis, improve the best known approximation bound for symmetric games. 	
In Section~\ref{NASH Sec: Bounding Payoffs and Strategy Exclusion}, we show how our SDP formulation can be used to provide certificates for certain (NP-hard) questions of economic interest about Nash equilibria in symmetric games. These are the problems of testing whether the maximum welfare achievable under any symmetric Nash equilibrium exceeds some threshold, and whether a set of strategies is played in every symmetric Nash equilibrium. In Section~\ref{NASH Sec: Connection to Sum of Squares/Lasserre Hierarchy}, we show that the SDP analyzed in this chapter dominates the first level of the Lasserre hierarchy~(Proposition~\ref{NASH Prop: strongerLasserre}). Some directions for future research are discussed in Section~\ref{NASH Sec: Conclusion}. The four appendices of the chapter add some numerical and technical details.

\subsection{Notation}\label{NASH Sec: Notation}
We establish some notation that will be used throughout the chapter. The symbol $\triangle_k$ denotes the $k$-dimensional simplex. For a matrix $A$, the notation $A_i,$ refers to its $i$-th row, and $A_{,j}$ refers to its $j$-th column. The notation $e_i$ refers to the elementary vector $(0,\ldots, 0, 1, 0, \ldots, 0)^T$ with the 1 being in position $i$, $0_m$ refers to the $m$-dimensional vector of zero's, $1_m$ refers to the $m$-dimensional vector of one's, and $J_{m\times n}$ refers to the $m \times n$ matrix of one's. The notation $A \succeq 0$ (resp $A \ge 0$) denotes that a matrix $A$ is positive semidefinite (resp. elementwise nonnegative), $S^{k \times k}$ denotes the set of symmetric $k \times k$ matrices, and $\Tr(A)$ denotes the trace of a matrix $A$, i.e., the sum of its diagonal elements. For two matrices $A$ and $B$, $A \succeq B$ denotes that $A - B$ is positive semidefinite and $A \otimes B$ denotes their Kronecker product. Finally, for a vector $v$, $diag(v)$ denotes the diagonal matrix with $v$ on its diagonal. For a square matrix $M$, $diag(M)$ denotes the vector containing its diagonal entries.

\section{The Formulation of our SDP Relaxation} \label{NASH Sec: Intro to SDP}
In this section we present an SDP relaxation for the problem of finding Nash equilibria in bimatrix games. This is done after a straightforward reformulation of the problem as a nonconvex quadratically constrained quadratic program. We also assume that all entries of the payoff matrices $A$ and $B$ are between 0 and 1. This can be done without loss of generality because Nash equilibria are invariant under certain affine transformations in the payoffs. In particular, the games $(A,B)$ and ${(cA+dJ_{m\times n}, eB+fJ_{m \times n})}$ have the same Nash equilibria for any scalars $c,d,e,$ and $f$, with $c$ and $e$ positive. This is because
\begin{equation*}
\begin{aligned}
x^{*T}Ay &\ge x^TAy\\
\Leftrightarrow c(x^{*T}Ay^*) + d &\ge c(x^TAy^*)+d\\
\Leftrightarrow c(x^{*T}Ay^*) + d(x^{*T}J_{m\times n}y^*) &\ge c(x^TAy^*)+d(x^TJ_{m\times n}y^*)\\
\Leftrightarrow x^{*T}(cA+dJ_{m\times n})y^* &\ge x^T(cA+dJ_{m\times n})y
\end{aligned}
\end{equation*}
Identical reasoning applies for player B.

\subsection{Nash Equilibria as Solutions to Quadratic Programs}\label{NASH SSec: Nash as QP}
Recall the definition of a Nash equilibrium from Section~\ref{NASH Sec: Intro}. An equivalent characterization is that a strategy pair $(x^*, y^*)\in \triangle_m \times \triangle_n$ is a Nash equilibrium for the game $(A,B)$ if and only if
\begin{equation}\label{NASH QP}
\begin{aligned}
x^{*T}Ay^* \ge e_i^TAy^*, \forall i \in \{1,\ldots, m\},\\
x^{*T}By^* \ge x^{*T}Be_i, \forall i \in \{1,\ldots, n\}.
\end{aligned}
\end{equation} 

The equivalence can be seen by noting that because the payoff from playing any mixed strategy is a convex combination of payoffs from playing pure strategies, there is always a pure strategy best response to the other player's strategy.

We now treat the Nash problem as the following quadratic programming (QP) feasibility problem:\\
\begin{equation}\label{NASH Eq: QCQP Formulation}
\begin{aligned}
& \underset{x \in \mathbb{R}^m, y \in \mathbb{R}^n}{\min}
& & 0 \\
& \text{subject to}
& & x^TAy \ge e_i^TAy, \forall i\in \{1,\ldots, m\},\\
&&& x^TBy \ge x^TBe_j, \forall j\in \{1,\ldots, n\},\\
&&& x_i \ge 0, \forall i\in \{1,\ldots, m\},\\
&&& y_i \ge 0, \forall j\in \{1,\ldots, n\},\\
&&& \sum_{i=1}^m x_i = 1,\\
&&& \sum_{i=1}^n y_i = 1.\\
\end{aligned}
\end{equation}

Similarly, a pair of strategies $x^* \in \triangle_m$ and $y^* \in \triangle_n$ form an $\epsilon$-Nash equilibrium for the game $(A,B)$ if and only if
$$x^{*T}Ay^* \ge e_i^TAy^*-\epsilon, \forall i \in \{1,\ldots, m\},$$
$$x^{*T}By^* \ge x^{*T}Be_i-\epsilon, \forall i \in \{1,\ldots, n\}.$$
Observe that any pair of simplex vectors $(x, y)$ is an $\epsilon$-Nash equilibrium for the game $(A,B)$ for any $\epsilon$ that satisfies
$$\epsilon \ge \max\{\underset{i}{\max}\ e_i^TAy - x^TAy, \underset{i}{\max}\ x^TBe_i - x^TBy\}.$$

We use the following notation throughout the chapter:
\begin{itemize}
	\setlength{\itemsep}{1pt}
	\setlength{\parskip}{0pt}
	\setlength{\parsep}{0pt}
	\renewcommand\labelitemi{$\cdot$}
	\item $\epsilon_A(x,y) \mathrel{\mathop:}= \underset{i}{\max}\ e_i^TAy - x^TAy$,
	\item $\epsilon_B(x,y) \mathrel{\mathop:}= \underset{i}{\max}\ x^TBe_i - x^TBy$,
	\item $\epsilon(x,y) \mathrel{\mathop:}= \max\{\epsilon_A(x,y), \epsilon_B(x,y)\}$,
\end{itemize}
and the function parameters are later omitted if they are clear from the context.

\subsection{SDP Relaxation}\label{NASH SSec: SDP Relaxation}

The QP formulation in (\ref{NASH Eq: QCQP Formulation}) lends itself to a natural SDP relaxation. We define a matrix
$$\mathcal{M} \mathrel{\mathop:}= \left[\begin{matrix} X & P\\ Z & Y\end{matrix}\right],$$
and an augmented matrix
$$\mathcal{M}' \mathrel{\mathop:}= \left[\begin{matrix} X & P & x\\ Z & Y & y\\ x & y & 1\end{matrix}\right],$$
with $X \in S^{m\times m}, Z \in \mathbb{R}^{n\times m}, Y \in S^{n\times n}, x \in \mathbb{R}^m, y \in \mathbb{R}^n$ and $P = Z^T$.\\\\
The SDP relaxation can then be expressed as

\begin{align}\label{NASH Eq: SDP1}\tag{SDP1}
& \underset{\mathcal{M}' \in \mathbb{S}^{m+n+1, m+n+1}}{\min}
& & 0 \nonumber\\
& \text{subject to}
& & \Tr(AZ) \ge e_i^TAy, \forall i\in \{1,\ldots, m\}, \label{NASH Eq: SDP1 Relaxed Nash A}\\
&&& \Tr(BZ) \ge x^TBe_j, \forall j\in \{1,\ldots, n\},\label{NASH Eq: SDP1 Relaxed Nash B}\\
&&& \sum_{i=1}^m x_i = 1, \label{NASH Eq: SDP1 Unity x}\\
&&& \sum_{i=1}^n y_i = 1, \label{NASH Eq: SDP1 Unity y}\\
&&& \M' \ge 0, \label{NASH Eq: SDP1 Nonnegativity}\\
&&& \M'_{m+n+1,m+n+1}=1,\\
&&& \mathcal{M}' \label{NASH Eq: SDP1 PSD}\succeq 0.
\end{align}

We refer to the constraints (\ref{NASH Eq: SDP1 Relaxed Nash A}) and (\ref{NASH Eq: SDP1 Relaxed Nash B}) as the relaxed Nash constraints and the constraints (\ref{NASH Eq: SDP1 Unity x}) and (\ref{NASH Eq: SDP1 Unity y}) as the unity constraints. This SDP is motivated by the following observation.
\begin{prop} Let $\M'$ be any rank-1 feasible solution to~\ref{NASH Eq: SDP1}. Then the vectors $x$ and $y$ from its last column constitute a Nash equilibrium for the game $(A,B)$.\end{prop}
\begin{proof} {}We know that $x$ and $y$ are in the simplex from the constraints (\ref{NASH Eq: SDP1 Unity x}), (\ref{NASH Eq: SDP1 Unity y}), and (\ref{NASH Eq: SDP1 Nonnegativity}).\\
	If the matrix $\mathcal{M}'$ is rank-1, then it takes the form \begin{equation}\label{NASH Eq: Rank 1 form}
	\bmat xx^T & xy^T & x \\ yx^T & yy^T & y \\x^T & y^T & 1 \emat
	=\xy1vec \xy1vec^T.\end{equation} Then, from the relaxed Nash constraints we have that\\
	$$e_i^TAy \le \Tr(AZ) = \Tr(Ayx^T) = \Tr(x^TAy) = x^TAy,$$
	$$x^TAe_i \le \Tr(BZ) = \Tr(Byx^T) = \Tr(x^TBy) = x^TBy.$$
	The claim now follows from the characterization given in (\ref{NASH QP}).
\end{proof}

\begin{remark}
	Because a Nash equilibrium always exists, there will always be a matrix of the form~(\ref{NASH Eq: Rank 1 form}) which is feasible to~\ref{NASH Eq: SDP1}. Thus we can disregard any concerns about \ref{NASH Eq: SDP1} being feasible, even when we add valid inequalities to it in Section~\ref{NASH SSec: Valid Inequalities}.
\end{remark}

\begin{remark}
	It is intuitive to note that the submatrix $P=Z^T$ of the matrix $\M'$ corresponds to a probability distribution over the strategies, and that seeking a rank-1 solution to our SDP can be interpreted as making $P$ a product distribution.
\end{remark}

The following theorem shows that~\ref{NASH Eq: SDP1} is a weak relaxation and stresses the necessity of additional valid constraints.

\begin{theorem}\label{NASH Thm: Necessity of VI}
	Consider a bimatrix game with payoff matrices bounded in $[0,1]$. Then for any two vectors $x \in \triangle_m$ and $y \in \triangle_n$, there exists a feasible solution $\M'$ to~\ref{NASH Eq: SDP1} with $\xy1vec$ as its last column.
\end{theorem}
\begin{proof}
	{}Consider any $x,y,\gamma > 0,$ and the matrix
	$$\xy1vec\xy1vec^T + \bmat \gamma J_{m+n,m+n} & 0_{m+n} \\ 0_{m+n}^T & 0\emat.$$ This matrix is the sum of two nonnegative psd matrices and is hence nonnegative and psd. By assumption $x$ and $y$ are in the simplex, and so constraints $(\ref{NASH Eq: SDP1 Unity x})-(\ref{NASH Eq: SDP1 PSD})$ of~\ref{NASH Eq: SDP1} are satisfied. To check that constraints $(\ref{NASH Eq: SDP1 Relaxed Nash A})$ and $(\ref{NASH Eq: SDP1 Relaxed Nash B})$ hold, note that since $A$ and $B$ are nonnegative, as long as the matrices $A$ and $B$ are not the zero matrices, the quantities $\Tr(AZ)$ and $\Tr(BZ)$ will become arbitrarily large as $\gamma$ increases. Since $e_i^TAy$ and $x^TBe_i$ are bounded by 1 by assumption, we will have that constraints $(\ref{NASH Eq: SDP1 Relaxed Nash A})$ and $(\ref{NASH Eq: SDP1 Relaxed Nash B})$ hold for $\gamma$ large enough. In the case where $A$ or $B$ is the zero matrix, the Nash constraints are trivially satisfied for the respective player.
\end{proof}

\subsection{Valid Inequalities}\label{NASH SSec: Valid Inequalities}

In this subsection, we introduce a number of valid inequalities to improve upon the SDP relaxation in~\ref{NASH Eq: SDP1}. These inequalities are justified by being valid if the matrix returned by the SDP is rank-1. The terminology we introduce here to refer to these constraints is used throughout the chapter. Constraints (\ref{NASH Eq: R_X}) and (\ref{NASH Eq: R_Y}) will be referred to as the \emph{row inequalities}, and (\ref{NASH Eq: CE_A}) and (\ref{NASH Eq: CE_B}) will be referred to as the \emph{correlated equilibrium inequalities}.

\begin{prop}\label{NASH Prop: Valid Inequalities}
	Any rank-1 solution $\M'$ to~\ref{NASH Eq: SDP1} must satisfy the following:
	
	\begin{equation}\label{NASH Eq: R_X}
	\sum_{j=1}^m X_{i,j} = \sum_{j=1}^n  P_{i,j} = x_i, \forall i \in \{1,\ldots, m\},\end{equation}
	\begin{equation}\label{NASH Eq: R_Y}
	\sum_{j=1}^n Y_{i,j} = \sum_{j=1}^m Z_{i,j} = y_i, \forall i \in \{1,\ldots, n\}.\end{equation}
	\begin{equation}\label{NASH Eq: CE_A}
	\sum_{j=1}^n A_{i,j}P_{i,j} \ge \sum_{j=1}^n A_{k,j}P_{i,j}, \forall i,k \in \{1,\ldots, m\},\end{equation}
	\begin{equation}\label{NASH Eq: CE_B}
	\sum_{j=1}^m B_{j,i}P_{j,i} \ge \sum_{j=1}^m B_{j,k}P_{j,i}, \forall i,k \in \{1,\ldots, n\}.\end{equation}
	
\end{prop}

\begin{proof}
	{}Recall from (\ref{NASH Eq: Rank 1 form}) that if $\M'$ is rank-1, it is of the form $$\left[\begin{matrix}xx^T & xy^T & x \\ yx^T & yy^T & y \\x^T & y^T & 1\end{matrix}\right]	=\xy1vec\xy1vec^T.$$ To show (\ref{NASH Eq: R_X}), observe that
	$$\sum_{j=1}^m X_{i,j} = \sum_{j=1}^m x_ix_j = x_i, \forall i \in \{1,\ldots, m\}.$$
	An identical argument works for the remaining matrices $P,Z,$ and $Y$. To show (\ref{NASH Eq: CE_A}) and (\ref{NASH Eq: CE_B}), observe that a pair $(x, y)$ is a Nash equilibrium if and only if
	$$\forall i, x_i > 0 \Rightarrow e_i^TAy = x^TAy = \underset{i}{\max}\ e_i^TAy,$$
	$$\forall i, y_i > 0 \Rightarrow x^TBe_i = x^TBy = \underset{i}{\max}\ x^TBe_i.$$
	This is because the Nash conditions require that $x^TAy$, a convex combination of $e_i^TAy$, be at least $e_i^TAy$ for all $i$. Indeed, if $x_i > 0$ but $e_i^TAy < x^TAy$, the convex combination must be less than $\underset{i}{\max}\ x^TAy$.
	
	For each $i$ such that $x_i = 0$ or $y_i = 0$, inequalities (\ref{NASH Eq: CE_A}) and (\ref{NASH Eq: CE_B}) reduce to $0 \ge 0$, so we only need to consider strategies played with positive probability. Observe that if $\M'$ is rank-1, then\\
	$$\sum_{j =1}^n A_{i,j}P_{i,j} = x_i\sum_{j =1}^n A_{i,j}y_j = x_ie_i^TAy \ge x_ie_k^TAy = \sum_{j =1}^nA_{k,j}P_{i,j}, \forall i,k$$
	$$\sum_{j =1}^m B_{j,i}P_{j,i} = y_i\sum_{j =1}^mB_{j,i}x_j = y_i x^TBe_i \ge y_ix^TBe_k = \sum_{j =1}^mB_{j,i}P_{j,k}, \forall i,k.$$
\end{proof}

\begin{remark}There are two ways to interpret the inequalities in (\ref{NASH Eq: CE_A}) and (\ref{NASH Eq: CE_B}): the first is as a relaxation of the constraint $x_i(e_i^TAy - e_j^TAy) \ge 0, \forall i,j$, which must hold since any strategy played with positive probability must give the best response payoff. The other interpretation is to have the distribution over outcomes defined by $P$ be a correlated equilibrium \cite{aumann1974subjectivity}. This can be imposed by a set of linear constraints on the entries of $P$ as explained next.
	
	Suppose the players have access to a public randomization device which prescribes a pure strategy to each of them (unknown to the other player). The distribution over the assignments can be given by a matrix $P$, where $P_{i,j}$ is the probability that strategy $i$ is assigned to player A and strategy $j$ is assigned to player B. This distribution is a correlated equilibrium if both players have no incentive to deviate from the strategy prescribed, that is, if the prescribed pure strategies $a$ and $b$ satisfy
	$$\sum_{j=1}^n A_{i,j}Prob(b=j|a=i) \ge \sum_{j=1}^n A_{k,j}Prob(b=j|a=i),$$
	$$\sum_{i=1}^m B_{i,j}Prob(a=i|b=j) \ge \sum_{i=1}^m B_{i,k}Prob(a=i|b=j).$$
	
	If we interpret the $P$ submatrix in our SDP as the distribution over the assignments by the public device, then because of our row constraints, $Prob(b=j|a=i)=\frac{P_{i,j}}{x_i}$ whenever $x_i \ne 0$ (otherwise the above inequalities are trivial). Similarly, $P(a=i|b=j)=\frac{P_{i,j}}{y_j}$ for nonzero $y_j$. Observe now that the above two inequalities imply (\ref{NASH Eq: CE_A}) and (\ref{NASH Eq: CE_B}). Finally, note that every Nash equilibrium generates a correlated equilibrium, since if $P$ is a product distribution given by $xy^T$, then $Prob(b=j|a=i)=y_j$ and $P(a=i|b=j)=x_i$.
\end{remark}

\subsubsection{Implied Inequalities} \label{NASH SSec: Implied Inequalities}
In addition to those explicitly mentioned in the previous section, there are other natural valid inequalities which are omitted because they are implied by the ones we have already proposed. We give two examples of such inequalities in the next proposition. We refer to the constraints in~(\ref{NASH Eq: Distribution Inequalities}) below as the \emph{distribution constraints}. The constraints in~(\ref{NASH Eq: McCormick}) are the familiar McCormick inequalities \cite{mccormick1976computability} for box-constrained quadratic programming.

\begin{prop}
	Let $z \defeq \bmat x \\ y \emat$. Any rank-1 solution $\M'$ to~\ref{NASH Eq: SDP1} must satisfy the following:
	\begin{equation}\label{NASH Eq: Distribution Inequalities}
	\sum_{i = 1}^m \sum_{j =1}^m X_{i,j} = \sum_{i =1}^n \sum_{j =1}^m Z_{i,j} = \sum_{i =1}^n \sum_{j =1}^n Y_{i,j} = 1.
	\end{equation}
	
	\begin{equation}\label{NASH Eq: McCormick}
	\begin{aligned}
	\M_{i,j} &\le z_i, \forall i,j \in \{1,\ldots,m+n\},\\
	\M_{i,j} + 1 &\ge z_{i}+z_{j}, \forall i,j \in \{1,\ldots,m+n\}.
	\end{aligned}
	\end{equation}
	
\end{prop}

\begin{proof}
	{}The distribution constraints follow immediately from the row constraints (\ref{NASH Eq: R_X}) and (\ref{NASH Eq: R_Y}), along with the unity constraints (\ref{NASH Eq: SDP1 Unity x}) and (\ref{NASH Eq: SDP1 Unity y}).
	
	The first McCormick inequality is immediate as a consequence of (\ref{NASH Eq: R_X}) and (\ref{NASH Eq: R_Y}), as all entries of $\M$ are nonnegative. To see why the second inequality holds, consider whichever submatrix $X, Y, P$, or $Z$ that contains $\M_{i,j}$. Suppose that this submatrix is, e.g., $P$. Then, since $P$ is nonnegative,
	$$0 \le \sum_{k=1, k \ne i}^{m}\sum_{l = 1, l \ne j}^n P_{k,l} \overset{(\ref{NASH Eq: R_X})}{=} \sum_{k=1, k \ne i}^m (x_k - P_{k,j})  \overset{(\ref{NASH Eq: R_Y})}{=} (1-x_i) - (y_j - P_{i,j}) = P_{i,j}+1-x_i-y_j.$$
	
	The same argument holds for the other submatrices, and this concludes the proof.
\end{proof}

\subsection{Simplifying our SDP}\label{NASH SSec: Simplification}

We observe that the row constraints (\ref{NASH Eq: R_X}) and (\ref{NASH Eq: R_Y}) along with the correlated equilibrium constraints (\ref{NASH Eq: CE_A}) and (\ref{NASH Eq: CE_B}) imply the relaxed Nash constraints (\ref{NASH Eq: SDP1 Relaxed Nash A}) and (\ref{NASH Eq: SDP1 Relaxed Nash B}). Indeed, if we fix an index $k~\in~\{1, \ldots, m\}$, then
$$\Tr(AZ) = \sum_{i=1}^m \sum_{j=1}^n A_{i,j} P_{i,j} \overset{(\ref{NASH Eq: CE_A})}{\ge} \sum_{i=1}^m \sum_{j=1}^n A_{k,j} P_{i,j} \ge \sum_{j=1}^n A_{k,j} (\sum_{i=1}^m P_{i,j}) \overset{(\ref{NASH Eq: R_Y}), P = Z^T}{\ge} \sum_{j=1}^n A_{k, j} y_j = e_k^TAy.$$
The proof for player B proceeds identically. Then, after collecting the valid inequalities and removing the relaxed Nash constraints, we arrive at an SDP given by
\begin{samepage}
	\begin{align}\label{NASH Eq: SDP1.1}\tag{SDP1'}
	& \underset{\M' \in S^{(m+n+1)\times(m+n+1)}}{\min}
	& & 0 \\
	& \text{subject to}
	& & (\ref{NASH Eq: SDP1 Unity x})-(\ref{NASH Eq: SDP1 PSD}), (\ref{NASH Eq: R_X})-(\ref{NASH Eq: CE_B}).\nonumber
	\end{align}
\end{samepage}

We make the observation that the last row and column of $\M'$ can be removed from this SDP, that is, there is a one-to-one correspondence between solutions to \ref{NASH Eq: SDP1.1} and those to the following SDP (where $\mathcal{M} \mathrel{\mathop:}= \left[\begin{matrix} X & P\\ Z & Y\end{matrix}\right],$ with $P = Z^T$):

\begin{samepage}
	\begin{align}\label{NASH Eq: SDP2}\tag{SDP2}
	& \underset{\M \in S^{(m+n)\times(m+n)}}{\min}
	& & 0 \\
	& \text{subject to}
	& & \M \succeq 0, &\label{NASH Eq: SDP2 PSD}\\
	&&& \M \ge 0, &\label{NASH Eq: SDP2 Nonnegativity}\\
	&&& \sum_{i=1}^n \sum_{j=1}^n P_{i,j} = 1, &\label{NASH Eq: SDP2 Distribution}\\
	&&& \sum_{j=1}^m X_{i,j} = \sum_{j=1}^n  P_{i,j}, \forall i \in \{1,\ldots, m\}, & \label{NASH Eq: SDP2 Row x}\\
	&&& \sum_{j=1}^n Y_{i,j} = \sum_{j=1}^m Z_{i,j}, \forall i \in \{1,\ldots, n\}, & \label{NASH Eq: SDP2 Row y}\\
	&&& \sum_{j=1}^n A_{i,j}P_{i,j} \ge \sum_{j=1}^n A_{k,j}P_{i,j}, \forall i,k \in \{1,\ldots, m\}, & \label{NASH Eq: SDP2 CE A}\\
	&&& \sum_{j=1}^m B_{j,i}P_{j,i} \ge \sum_{j=1}^m B_{j,k}P_{j,i}, \forall i,k \in \{1,\ldots, n\}. & \label{NASH Eq: SDP2 CE B}
	\end{align}
\end{samepage}

Indeed, it is readily verified that the submatrix $\M$ from any feasible solution $\M'$ to \ref{NASH Eq: SDP1.1} is feasible to \ref{NASH Eq: SDP2}. Conversely, let $\M$ be any feasible matrix to \ref{NASH Eq: SDP2}. Consider an eigendecomposition $\M = \sum_{i=1}^k \lambda_i v_i v_i^T$ and let $\bmat x\\y\emat \defeq \M \frac{1_{m+n}}{2}.$ Then the matrix
\begin{equation}\label{NASH Eq: M psd M' psd} \M' \defeq \bmat \M & \bmat x \\ y\emat \\ \bmat x^T & y^T \emat & 1 \emat = \sum_{i=1}^k \lambda_i \bmat v_i\\ 1_{m+n}^Tv_i/2 \emat \bmat v_i\\ 1_{m+n}^Tv_i/2 \emat^T\end{equation}
is easily seen to be feasible to \ref{NASH Eq: SDP1.1}.

Given any feasible solution $\M$ to \ref{NASH Eq: SDP2}, observe that the submatrix $P$ is a correlated equilibrium. We take our candidate approximate Nash equilibrium to be the pair $x = P1_n$ and $y = P^T1_m$. If the correlated equilibrium $P$ is rank-1, then the pair $(x,y)$ so defined constitutes an exact Nash equilibrium. In Section~\ref{NASH Sec: Algorithms}, we will add certain objective functions to \ref{NASH Eq: SDP2} with the interpretation of searching for low-rank correlated equilibria.

\section{Exactness for Strictly Competitive Games}\label{NASH Sec: Exactness Zero Sum}

In this section, we show that \ref{NASH Eq: SDP1} recovers a Nash equilibrium for any zero-sum game, and that \ref{NASH Eq: SDP2} recovers a Nash equilibrium for any strictly competitive game (see Definition \ref{NASH Defn: Strictly Competitive Def} below). Both these notions represent games where the two players are in direct competition, but strictly competitive games are more general, and for example, allow both players to have nonnegative payoff matrices. These classes of games are solvable in polynomial time via linear programming. Nonetheless, it is reassuring to know that our SDPs recover these important special cases.

\begin{defn}\label{NASH Defn: Zero Sum Def}
	A \emph{zero-sum game} is a game in which the payoff matrices satisfy $A = -B$.
\end{defn}

\begin{theorem}\label{NASH Thm: Zero Sum}
	For a zero-sum game, the vectors $x$ and $y$ from the last column of any feasible solution $\M'$ to~\ref{NASH Eq: SDP1} constitute a Nash equilibrium.
\end{theorem}
\begin{proof}
	{}Recall that the relaxed Nash constraints (\ref{NASH Eq: SDP1 Relaxed Nash A}) and (\ref{NASH Eq: SDP1 Relaxed Nash B}) read
	$$\Tr(AZ) \ge e_i^TAy, \forall i \in \{1,\ldots, m\},$$
	$$\Tr(BZ) \ge x^TBe_j, \forall j \in \{1,\ldots, n\}.$$
	Since $B=-A$, the latter statement is equivalent to
	$$\Tr(AZ) \le x^TAe_j, \forall j \in \{1,\ldots, n\}.$$
	In conjunction these imply
	\begin{equation}\label{NASH eq: zero sum inequality}
	e_i^TAy \le \Tr(AZ) \le x^TAe_j, \forall i \in \{1,\ldots, m\}, j \in \{1,\ldots, n\}.
	\end{equation}
	
	We claim that any pair $x \in \triangle_m$ and $y \in \triangle_n$ which satisfies the above condition is a Nash equilibrium. To see that $x^TAy \ge e_i^TAy, \forall i \in \{1,\ldots, m\},$ observe that $x^TAy$ is a convex combination of $x^TAe_j$, which are at least $e_i^TAy$ by (\ref{NASH eq: zero sum inequality}). To see that $x^TBy \ge x^TBe_j \Leftrightarrow x^TAy \le x^TAe_j, \forall j \in \{1,\ldots, n\}$, observe that $x^TAy$ is a convex combination of $e_i^TAy$, which are at most $x^TAe_j$ by (\ref{NASH eq: zero sum inequality}).
	
\end{proof}

\begin{defn}\label{NASH Defn: Strictly Competitive Def}
	A game $(A, B)$ is \emph{strictly competitive} if for all $x, x' \in \triangle_m, y, y' \in \triangle_n$, $x^TAy - x'^TAy'$ and $x'^TBy' - x^TBy$ have the same sign.
\end{defn}

The interpretation of this definition is that if one player benefits from changing from one outcome to another, the other player must suffer. Adler, Daskalakis, and Papadimitriou show in \cite{adler2009note} that the following much simpler characterization is equivalent.

\begin{theorem}[Theorem 1 of \cite{adler2009note}]\label{NASH Thm: Adler}
	A game is strictly competitive if and only if there exist scalars $c,d,e$, and $f,$ with $c > 0, e > 0,$ such that $cA+dJ_{m \times n} = -eB + fJ_{m \times n}$.
\end{theorem}

One can easily show that there exist strictly competitive games for which not all feasible solutions to~\ref{NASH Eq: SDP1} have Nash equilibria as their last columns (see Theorem~\ref{NASH Thm: Necessity of VI}). However, we show that this is the case for~\ref{NASH Eq: SDP2}.

\begin{theorem}\label{NASH Thm: Strictly Competitive Exact}
	For a strictly competitive game, the vectors $x \defeq P1_n$ and $y \defeq P^T1_m$ from any feasible solution $\M$ to~\ref{NASH Eq: SDP2} constitute a Nash equilibrium.
\end{theorem}

To prove Theorem \ref{NASH Thm: Strictly Competitive Exact} we need the following lemma, which shows that feasibility of a matrix $\M$ in~\ref{NASH Eq: SDP2} is invariant under certain transformations of $A$ and $B$.

\begin{lem}\label{NASH Lem: scaling shifting}
	Let $c,d,e$, and $f$ be any set of scalars with $c > 0$ and $e > 0$. If a matrix $\M$ is feasible to~\ref{NASH Eq: SDP2} with input payoff matrices $A$ and $B$, then it is also feasible to~\ref{NASH Eq: SDP2} with input matrices $cA+dJ_{m \times n}$ and $eB + fJ_{m \times n}$.
\end{lem}
\begin{proof}
	{}It suffices to check that constraints (\ref{NASH Eq: SDP2 CE A}) and (\ref{NASH Eq: SDP2 CE B}) of~\ref{NASH Eq: SDP2} still hold, as only the correlated equilibrium constraints use the matrices $A$ and $B$. We only show that constraint (\ref{NASH Eq: SDP2 CE A}) still holds because the argument for constraint (\ref{NASH Eq: SDP2 CE B}) is identical.
	
	Note from the definition of $x$ that for each $i \in \{1,\ldots, m\}, x_i = \sum_{j=1}^n (J_{m\times n})_{i,j}P_{i,j}$. To check that the correlated equilibrium constraints hold, observe that for scalars $c>0,d$, and for all $i,k \in \{1,\ldots, m\}$,
	\begin{equation*}
	\begin{aligned}
	\sum_{j=1}^n A_{i,j}P_{i,j} &\ge \sum_{j=1}^n A_{k,j}P_{i,j}\\
	\Leftrightarrow c\sum_{j=1}^n A_{i,j}P_{i,j} + d \sum_{j=1}^n P_{i,j} &\ge c\sum_{j=1}^n A_{k,j}P_{i,j} + d \sum_{j=1}^n P_{i,j}\\
	\Leftrightarrow c\sum_{j=1}^n A_{i,j}P_{i,j} + d\sum_{j=1}^n (J_{m \times n})_{i,j}P_{i,j} &\ge c\sum_{j=1}^n A_{k,j}P_{i,j} + d\sum_{j=1}^n (J_{m \times n})_{k,j}P_{i,j}\\
	\Leftrightarrow \sum_{j=1}^n (cA_{i,j}+dJ_{m \times n})_{k,j}P_{i,j} &\ge \sum_{j=1}^n (cA_{i,j}+dJ_{m \times n})_{k,j}P_{i,j}.
	\end{aligned}
	\end{equation*}
\end{proof}

\begin{proof}{\bf Proof (of Theorem \ref{NASH Thm: Strictly Competitive Exact}).}
	{}Let $A$ and $B$ be the payoff matrices of the given strictly competitive game and let $\M$ be a feasible solution to~\ref{NASH Eq: SDP2}. Since the game is strictly competitive, we know from Theorem \ref{NASH Thm: Adler} that $cA + dJ_{m\times n} = -eB + fJ_{m\times n}$ for some scalars $c>0,e>0,d, f$. Consider a new game with input matrices $\tilde{A} = cA+dJ_{m\times n}$ and $\tilde{B} = eB - fJ_{m\times n}$. By Lemma~\ref{NASH Lem: scaling shifting}, $\M$ is still feasible to~\ref{NASH Eq: SDP2} with input matrices $\tilde{A}$ and $\tilde{B}$. By the arguments in Section \ref{NASH SSec: Simplification}, the matrix $\M' \defeq \bmat \M & \bmat x \\ y\emat \\ \bmat x^T & y^T \emat & 1 \emat$ is feasible to \ref{NASH Eq: SDP1.1}, and hence also to~\ref{NASH Eq: SDP1}. Now notice that since $\tilde{A} = -\tilde{B}$, Theorem~\ref{NASH Thm: Zero Sum} implies that the vectors $x$ and $y$ in the last column form a Nash equilibrium to the game $(\tilde{A}, \tilde{B})$. Finally recall from the arguments at the beginning of Section~\ref{NASH Sec: Intro to SDP} that Nash equilibria are invariant to scaling and shifting of the payoff matrices, and hence $(x, y)$ is a Nash equilibrium to the game $(A, B)$.
\end{proof}

\section{Algorithms for Lowering Rank}\label{NASH Sec: Algorithms}

In this section, we present heuristics which aim to find low-rank solutions to~\ref{NASH Eq: SDP2} and present some empirical results. Recall that our~\ref{NASH Eq: SDP2} in Section~\ref{NASH SSec: Simplification} did not have an objective function. Hence, we can encourage low-rank solutions by choosing certain objective functions, in particular the trace of the matrix $\M$, which is a general heuristic for minimizing the rank of symmetric matrices~\cite{recht2010guaranteed,fazel2002matrix}. This simple objective function is already guaranteed to produce a rank-1 solution in the case of strictly competitive games (see Proposition~\ref{NASH Prop: Zero Sum Rank 1} below). For general games, however, one can design better objective functions in an iterative fashion (see Section~\ref{NASH SSec: Linearization Algorithms}).

\emph{Notational Remark:} For the remainder of this section, we will use the shorthand $x \defeq P1_n$ and $y \defeq P^T1_m$, where $P$ is the upper right submatrix of a feasible solution $\M$ to \ref{NASH Eq: SDP2}.

\begin{prop}\label{NASH Prop: Zero Sum Rank 1} For a strictly competitive game, any optimal solution to~\ref{NASH Eq: SDP2} with $\Tr(\M)$ as the objective function must be rank-1.\end{prop}

\begin{proof}	
	{}Let $$\M \mathrel{\mathop:}= \bmat X & P\\ P^T & Y\emat$$ be a feasible solution to~\ref{NASH Eq: SDP2}. In the case of strictly competitive games, from Theorem~\ref{NASH Thm: Strictly Competitive Exact} we know that that $(x, y)$ is a Nash equilibrium. Then because the matrix $\M$ is psd, from (\ref{NASH Eq: M psd M' psd}) and an application of the Schur complement (see, e.g. \cite[Sect. A.5.5]{boyd2004convex}) to $\bmat \M & \bmat x \\ y\emat \\ \bmat x^T & y^T \emat & 1 \emat$, we have that $\M \succeq \xyvec\xyvec^T$. Hence, $\M = \bmat xx^T & xy^T \\ yx^T & yy^T\emat+ \mathcal{P}$ for some psd matrix $\mathcal{P}$ and the Nash equilibrium $(x, y)$. Given this expression, the objective function $\Tr(\M)$ is then $x^Tx + y^Ty + \Tr(\mathcal{P})$. As $(x,y)$ is a Nash equilibrium, the choice of $\mathcal{P} = 0$ results in a feasible solution. Since the zero matrix has the minimum possible trace among all psd matrices, the solution will be the rank-1 matrix $\xyvec\xyvec^T$.\end{proof}

\begin{remark}
	If the row constraints and the nonnegativity constraints on $X$ and $Y$ are removed from \ref{NASH Eq: SDP2}, then this SDP with $\Tr(\M)$ as the objective function can be interpreted as searching for a minimum-rank correlated equilibrium $P$ via the nuclear norm relaxation; see~\cite[Section 2]{recht2010guaranteed}.
\end{remark}

\subsection{Linearization Algorithms}\label{NASH SSec: Linearization Algorithms}
The algorithms we present in this section for minimzing the rank of the matrix $\M$ in \ref{NASH Eq: SDP2} are based on iterative linearization of certain nonconvex objective functions. Motivated by the next proposition, we design two continuous (nonconvex) objective functions that, if minimized exactly, would guarantee rank-1 solutions. We will then linearize these functions iteratively.

\begin{prop}\label{NASH Prop: diagonal sufficiency} Let the matrices $X$ and $Y$ and vectors $x \defeq P1_n$ and $y \defeq P^T1_m$ be taken from a feasible solution to~\ref{NASH Eq: SDP2}. Then the matrix $\M$ is rank-1 if and only if $X_{i,i} = x_i^2$ and $Y_{i,i} = y_i^2$ for all $i$.\end{prop}
\begin{proof}
	{}Note that if $\M$ is rank-1, then it can be written as $zz^T$ for some $z \in \mathbb{R}^{m+n}$. The $i$-th diagonal entry in the $X$ submatrix will then be equal to $$z_i^2 \overset{(\ref{NASH Eq: Distribution Inequalities})}{=} \frac{1}{4}z_i^2(1_{m+n}^Tzz^T1_{m+n}) = (\frac{1}{2}\M_{i,}1_{m+n})^2 \overset{(\ref{NASH Eq: R_X})}{=} (P_{i,}1_n)^2 = x_i^2,$$ where the second equality holds because $\M_{i,}$---the $i$-th row of $\M$---is $z_iz^T$. An analogous statement holds for the diagonal entries of $Y$, and hence the condition is necessary.
	
	To show sufficiency, let $z \defeq \xyvec$. Since $\M$ is psd, we have that $\M_{i,j} \le \sqrt{\M_{i,i}\M_{j,j}}$, which implies $\M_{i,j} \le z_iz_j$ by the assumption of the proposition. Recall from the distribution constraint~(\ref{NASH Eq: Distribution Inequalities}) that $\sum_{i=1}^{m+n}\sum_{j=1}^{m+n} \M_{i,j}=4$. Further, the same constraint along with the definitions of $x$ and $y$ imply that $\sum_{i=1}^{m+n} z_i = 2$, which means that $\sum_{i=1}^{m+n} \sum_{j=1}^{m+n} z_iz_j = 4$. Hence in order to have the equality
	
	$$4 = \sum_{i=1}^{m+n} \sum_{j=1}^{m+n} \M_{i,j} \le \sum_{i=1}^{m+n} \sum_{j=1}^{m+n} z_iz_j = 4,$$
	we must have $\M_{i,j} = z_iz_j$ for each $i$ and $j$. Consequently $\M$ is rank-1.\end{proof}

We focus now on two nonconvex objectives that as a consequence of the above proposition would return rank-1 solutions:
\begin{prop}\label{NASH Prop: nonconvex objectives} All optimal solutions to~\ref{NASH Eq: SDP2} with the objective function $\sum_{i=1}^{m+n} \sqrt{\M_{i,i}}$ or $\Tr(\M) - x^Tx - y^Ty$ are rank-1.\end{prop}
\begin{proof} {}We show that each of these objectives has a specific lower bound which is achieved if and only if the matrix is rank-1.\\
	Observe that since $\M \succeq \xyvec\xyvec^T$, we have $\sqrt{X_{i,i}}\ge x_i$ and $\sqrt{Y_{i,i}} \ge y_i$, and hence $$\sum_{i=1}^{m+n} \sqrt{\M_{i,i}} \ge \sum_{i=1}^{m} x_i + \sum_{i=1}^n y_i= 2.$$
	Further note that $$\Tr(\M) - \xyvec^T\xyvec \ge \xyvec^T\xyvec - \xyvec^T\xyvec = 0.$$
	
	We can see that the lower bounds are achieved if and only if $X_{i,i} = x_i^2$ and $Y_{i,i} = y_i^2$ for all $i$, which by Proposition~\ref{NASH Prop: diagonal sufficiency} happens if and only if $\M$ is rank-1.\end{proof}

We refer to our two objective functions in Proposition~\ref{NASH Prop: nonconvex objectives} as the ``\emph{square root objective}'' and the ``\emph{diagonal gap objective}'' respectively. While these are both nonconvex, we will attempt to iteratively minimize them by linearizing them through a first order Taylor expansion. For example, at iteration $k$ of the algorithm, $$\sum_{i=1}^{m+n} \sqrt{\M_{i,i}^{(k)}} \simeq \sum_{i=1}^{m+n} \sqrt{\M_{i,i}^{(k-1)}} + \frac{1}{2\sqrt{\M_{i,i}^{(k-1)}}}(\M_{i,i}^{(k)} - \M_{i,i}^{(k-1)}).$$ Note that for the purposes of minimization, this reduces to minimizing $\sum_{i=1}^{m+n} \frac{1}{\sqrt{\M_{i,i}^{{(k-1)}}}}\M_{i,i}^{(k)}$.

In similar fashion, for the second objective function, at iteration $k$ we can make the approximation $$\Tr(\M)-\xyvec^{(k)T}\xyvec^{(k)} \simeq \Tr(\M)-\xyvec^{(k-1)T}\xyvec^{(k-1)T} - 2\xyvec^{(k-1)T}(\xyvec^{(k)}-\xyvec^{(k-1)}).$$ Once again, for the purposes of minimization this reduces to minimizing $\Tr(\M)-2\xyvec^{(k-1)T}\xyvec^{(k)}$. This approach then leads to the following two algorithms.\footnote{An algorithm similar to Algorithm \ref{NASH Alg: DG Algorithm} is used in \cite{ibaraki2001rank}.}

\begin{algorithm}[H]
	\caption{Square Root Minimization Algorithm}\label{NASH Alg: SR Algorithm}
	\begin{algorithmic}[1]
		\State Let $x^{(0)} = 1_m, y^{(0)} = 1_n, k = 1$.
		\While {!convergence}
		\State Solve~\ref{NASH Eq: SDP2} with $\sum_{i=1}^m \frac{1}{\sqrt{x_i^{(k-1)}}}X_{i,i} + \sum_{i=1}^n \frac{1}{\sqrt{y_i^{(k-1)}}}Y_{i,i}$ as the objective, and let $\M^*$ be an optimal solution.
		\State Let $x^{(k)} = diag(X^*), y^{(k)}=diag(Y^*)$.
		\State Let $k = k+1$.
		\EndWhile
	\end{algorithmic}
\end{algorithm}

\begin{algorithm}[H]
	\caption{Diagonal Gap Minimization Algorithm}\label{NASH Alg: DG Algorithm}
	\begin{algorithmic}[1]
		\State Let $x^{(0)}=0_m, y^{(0)}=0_n, k=1$.
		\While {!convergence}
		\State Solve~\ref{NASH Eq: SDP2} with $\Tr(X) + \Tr(Y) - 2\xyvec^{(k-1)T}\xyvec^{(k)}$ as the objective, and let $\M^*$ be an optimal solution.
		\State Let $x^{(k)} = P^*1_n, y^{(k)}=P^{*T}1_m$.
		\State Let $k = k+1$.
		\EndWhile
	\end{algorithmic}
\end{algorithm}

\begin{remark}
	Note that the first iteration of both algorithms uses the nuclear norm (i.e. trace) of $\M$ as the objective.
\end{remark}

The square root algorithm has the following property.

\begin{theorem}\label{NASH Thm: SR Monotone}
	Let $\M^{(1)}, \M^{(2)}, \ldots$ be the sequence of optimal matrices obtained from the square root algorithm. Then the sequence
	\begin{equation}\label{NASH sum sqrt}
	\{\sum_{i=1}^{m+n} \sqrt{\M_{i,i}^{(k)}}\}
	\end{equation}
	is nonincreasing and is lower bounded by two. If it reaches two at some iteration $t$, then the matrix $\M^{(t)}$ is rank-1.
\end{theorem}
\begin{proof}
	{}Observe that for any $k>1$,
	$$\sum_{i=1}^{m+n} \sqrt{\M_{i,i}^{(k)}} \le \frac{1}{2}\sum_{i=1}^{m+n} (\frac{\M_{i,i}^{(k)}}{\sqrt{\M_{i,i}^{(k-1)}}}+\sqrt{\M_{i,i}^{(k-1)}}) \le \frac{1}{2}\sum_{i=1}^{m+n} (\frac{\M_{i,i}^{(k-1)}}{\sqrt{\M_{i,i}^{(k-1)}}}+\sqrt{\M_{i,i}^{(k-1)}}) = \sum_{i=1}^{m+n} \sqrt{\M_{i,i}^{(k-1)}},$$
	where the first inequality follows from the arithmetic-mean-geometric-mean inequality, and the second follows from that $\M_{i,i}^{(k)}$ is chosen to minimize $\sum_{i=1}^{m+n} \frac{\M_{i,i}^{(k)}}{\sqrt{\M_{i,i}^{(k-1)}}}$ and hence achieves a no larger value than the feasible solution $\M^{(k-1)}$. This shows that the sequence is nonincreasing.
	
	The proof of Proposition~\ref{NASH Prop: nonconvex objectives} already shows that the sequence is lower bounded by two, and Proposition~\ref{NASH Prop: nonconvex objectives} itself shows that reaching two is sufficient to have the matrix be rank-1.
\end{proof}

The diagonal gap algorithm has the following property.
\begin{theorem}\label{NASH Thm: DG Monotone}
	Let $\M^{(1)}, \M^{(2)}, \ldots$ be the sequence of optimal matrices obtained from the diagonal gap algorithm. Then the sequence \begin{equation}\{\Tr(\M^{(k)})- \xyvec^{(k)T}\xyvec^{(k)}\}\end{equation}
	is nonincreasing and is lower bounded by zero. If it reaches zero at some iteration $t$, then the matrix $\M^{(t)}$ is rank-1.
\end{theorem}
\begin{proof} {}Observe that
	\begin{align*}
	&\ \Tr(\M^{(k)}) - \xyvec^{(k)T}\xyvec^{(k)}\\
	\le& \Tr(\M^{(k)})-\xyvec^{(k)T}\xyvec^{(k)} + \left(\xyvec^{(k)}- \xyvec^{(k-1)}\right)^T\left(\xyvec^{(k)}- \xyvec^{(k-1)}\right)\\
	=&\ \Tr(\M^{(k)})-2\xyvec^{(k)T}\xyvec^{(k-1)}+\xyvec^{(k-1)T}\xyvec^{(k-1)}\\
	\le&\  \Tr(\M^{(k-1)})-2\xyvec^{(k-1)T}\xyvec^{(k-1)}+\xyvec^{(k-1)T}\xyvec^{(k-1)}\\
	=&\ \Tr(\M^{(k-1)})-\xyvec^{(k-1)T}\xyvec^{(k-1)},
	\end{align*}
	where the second inequality follows from that $\M^{(k)}$ is chosen to minimize $$\Tr(\M^{(k-1)})-2\xyvec^{(k-1)T}\xyvec^{(k-1)}$$ and hence achieves a no larger value than the feasible solution $\M^{(k-1)}$. This shows that the sequence is nonincreasing.
	
	The proof of Proposition~\ref{NASH Prop: nonconvex objectives} already shows that the sequence is lower bounded by zero, and Proposition~\ref{NASH Prop: nonconvex objectives} itself shows that reaching zero is sufficient to have the matrix be rank-1.\end{proof}

We also invite the reader to also see Theorem \ref{NASH Thm: Trace minus 2xxt} in the next section which relates the objective value of the diagonal gap minimization algorithm and the quality of approximate Nash equilibria that the algorithm produces.

\subsection{Numerical Experiments}
We tested Algorithms~\ref{NASH Alg: SR Algorithm} and~\ref{NASH Alg: DG Algorithm} on games coming from 100 randomly generated payoff matrices with entries bounded in $[0,1]$ of varying sizes. Below is a table of statistics for $20 \times 20$ matrices; the data for the rest of the sizes can be found in Appendix~\ref{NASH Sec: Appendix Statistics}.\footnote{The code and instance data that produced these results is publicly available at \texttt{https://github.com/jeffreyzhang92/SDP\_Nash}. The function \texttt{nash.m} computes an approximate Nash equilibrium using one of our two algorithms as specified by the user.} We can see that our algorithms return approximate Nash equilibria with fairly low $\epsilon$ (recall the definition from Section~\ref{NASH SSec: Nash as QP}). We ran 20 iterations of each algorithm on each game. Using the SDP solver of \cite{mosek}, each iteration takes on average under 4 seconds to solve on a standard personal machine with a 3.4 GHz processor and 16 GB of memory.

\begin{table}[H]
	\caption{Statistics on $\epsilon$ for $20\times 20$ games after 20 iterations.\label{NASH Tab: 20x20 Main}}
	{\begin{tabular}{|c|c|c|c|c|c|}\hline
			Algorithm & Max & Mean & Median & StDev\\\hline
			Square Root&	0.0198&	0.0046&	0.0039&	0.0034	\\\hline
			Diagonal Gap&	0.0159&	0.0032&	0.0024&	0.0032	\\\hline
	\end{tabular}}
	\centering
\end{table}

The histograms below show the effect of increasing the number of iterations on lowering $\epsilon$ on $20 \times 20$ games. For both algorithms, there was a clear improvement of the $\epsilon$ by increasing the number of iterations.

\begin{figure}[H]
	\includegraphics[height=.3\textheight,keepaspectratio]{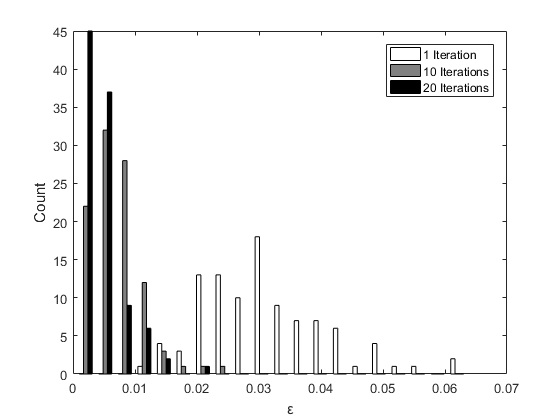}
	\includegraphics[height=.3\textheight,keepaspectratio]{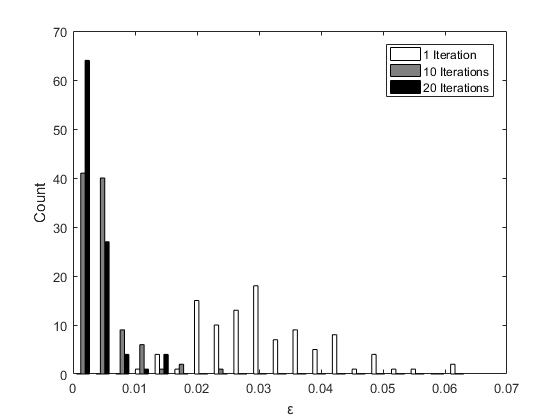}
	\caption{Distribution of $\epsilon$ over numbers of iterations for the square root algorithm (left) and the diagonal gap algorithm (right).\label{NASH Fig: Epsilon Iterations}}
\end{figure}

\section{Bounds on $\epsilon$ for General Games}\label{NASH Sec: Bounds}

Since the problem of computing a Nash equilibrium to an arbitrary bimatrix game is PPAD-complete, it is unlikely that one can find rank-1 solutions to this SDP in polynomial time. In Section~\ref{NASH Sec: Algorithms}, we designed objective functions (such as variations of the nuclear norm) that empirically do very well in finding low-rank solutions to~\ref{NASH Eq: SDP2}. Nevertheless, it is of interest to know if the solution returned by~\ref{NASH Eq: SDP2} is not rank-1, whether one can recover an $\epsilon$-Nash equilibrium from it and have a guarantee on $\epsilon$. Our goal in this section is to study this question.

\emph{Notational Remark:} Recall our notation for the matrix

$$\mathcal{M} \mathrel{\mathop:}= \left[\begin{matrix} X & P\\ Z & Y\end{matrix}\right].$$
Throughout this section, any matrices $X, Z, P=Z^T$ and $Y$ are assumed to be taken from a feasible solution to~\ref{NASH Eq: SDP2}. Furthermore, $x$ and $y$ will be $P1_n$ and $P^T1_m$ respectively.

The ultimate results of this section are the theorems in Sections \ref{NASH SSec: Bounds on e} and \ref{NASH SSec: Rank-2 case}. To work towards them, we need a number of preliminary lemmas which we present in Section~\ref{NASH SSec: Useful Lemmas}.

\subsection{Lemmas Towards Bounds on $\epsilon$}\label{NASH SSec: Useful Lemmas}

We first observe the following connection between the approximate payoffs $\Tr(AZ)$ and $\Tr(BZ)$, and $\epsilon(x,y)$, as defined in Section~\ref{NASH SSec: Nash as QP}.
\begin{lem}\label{NASH Lem: Bound for e}Consider any feasible solution to~\ref{NASH Eq: SDP2}. Then
	$$\epsilon(x,y) \le \max\{\Tr(AZ)-x^TAy, \Tr(BZ) - x^TBy\}.$$\end{lem}
\begin{proof} {}Recall from the argument at the beginning of Section \ref{NASH SSec: Simplification} that constraints (\ref{NASH Eq: CE_A}) and (\ref{NASH Eq: CE_B}) imply $\Tr(AZ) \ge e_i^TAy$ and $\Tr(BZ) \ge x^TBe_i$ for all $i$. Hence, we have ${\epsilon_A \le \Tr(AZ)-x^TAy}$ and $\epsilon_B \le \Tr(BZ) - x^TBy$.\end{proof}

We thus are interested in the difference of the two matrices $P=Z^T$ and $xy^T$. These two matrices can be interpreted as two different probability distributions over the strategy outcomes. The matrix $P$ is the probability distribution from the SDP which generates the approximate payoffs $\Tr(AZ)$ and $\Tr(BZ)$, while $xy^T$ is the product distribution that would have resulted if the matrix had been rank-1. We will see that the difference of these distributions is key in studying the $\epsilon$ which results from \ref{NASH Eq: SDP2}. Hence, we first take steps to represent this difference.

\begin{lem}\label{NASH Lem: Partition Lemma} Consider any feasible matrix $\M$ to~\ref{NASH Eq: SDP2} with an eigendecomposition
	\begin{equation}
	\begin{aligned}
	\M = \sum_{i=1}^k \lambda_i v_iv_i^T =\mathrel{\mathop:} \sum_{i = 1}^k \lambda_i \bmat a_i \\ b_i\emat \bmat a_i \\ b_i\emat^T,
	\end{aligned}
	\end{equation}
	so that the eigenvectors $v_i\in \mathbb{R}^{m+n}$ are partitioned into vectors $a_i \in \mathbb{R}^m$ and $b_i\in \mathbb{R}^n$. Then for all $i,\sum_{j=1}^m (a_i)_j = \sum_{j=1}^n (b_i)_j$.\end{lem}
\begin{proof}
	{}We know from (\ref{NASH Eq: SDP2 Distribution}), (\ref{NASH Eq: SDP2 Row x}), and (\ref{NASH Eq: SDP2 Row y}) that
	\begin{align}
	\sum_{i=1}^k \lambda_i 1_m^T a_ia_i^T 1_m \overset{(\ref{NASH Eq: SDP2 Distribution}),(\ref{NASH Eq: SDP2 Row x})}{=} 1, \label{Nash Eq: pl-1}\\
	\sum_{i=1}^k \lambda_i 1_m^T a_ib_i^T 1_n \overset{(\ref{NASH Eq: SDP2 Distribution})}{=} 1,\label{Nash Eq: pl-2}\\
	\sum_{i=1}^k \lambda_i 1_n^T b_ia_i^T 1_m \overset{(\ref{NASH Eq: SDP2 Distribution})}{=} 1,\label{Nash Eq: pl-3}\\
	\sum_{i=1}^k \lambda_i 1_n^T b_ib_i^T 1_n \overset{(\ref{NASH Eq: SDP2 Distribution}),(\ref{NASH Eq: SDP2 Row y})}{=} 1.\label{Nash Eq: pl-4}
	\end{align}
	
	Then by subtracting terms we have
	\begin{align}
	(\ref{Nash Eq: pl-1})-(\ref{Nash Eq: pl-2})=\sum_{i=1}^k \lambda_i 1_m^Ta_i (a_i^T1_m - b_i^T1_n) = 0, \label{Nash Eq: pl-5}\\
	(\ref{Nash Eq: pl-3})-(\ref{Nash Eq: pl-4})=\sum_{i=1}^k \lambda_i 1_n^Tb_i (a_i^T1_m - b_i^T1_n) = 0. \label{Nash Eq: pl-6}
	\end{align}
	
	By subtracting again these imply
	\begin{equation}
	(\ref{Nash Eq: pl-5})-(\ref{Nash Eq: pl-6})=\sum_{i=1}^k \lambda_i (1_m^T a_i - 1_n^Tb_i)^2 = 0.\label{Nash Eq: pl-7}
	\end{equation}
	As all $\lambda_{i}$ are nonnegative due to positive semidefiniteness of $\M$, the only way for this equality to hold is to have $1_m^Ta_i = 1_n^Tb_i, \forall i$. This is equivalent to the statement of the claim.
\end{proof}

From Lemma~\ref{NASH Lem: Partition Lemma}, we can let $s_i \defeq \sum_{j=1}^m (a_i)_j = \sum_{j=1}^n (b_i)_j$, and furthermore we assume without loss of generality that each $s_i$ is nonnegative. Note that from the definition of $x$ we have
\begin{equation} x_i = \sum_{j = 1}^m P_{ij} = \sum_{l = 1}^k \sum_{j = 1}^m \lambda_l (a_l)_i (b_l)_j = \sum_{j = 1}^k \lambda_j s_j (a_l)_i.
\end{equation}
Hence,
\begin{equation}\label{NASH Eq: Row constraint x}x = \sum_{i =1}^k \lambda_i s_i a_i.\end{equation}
Similarly,
\begin{equation}\label{NASH Eq: Row constraint y}
y = \sum_{i=1}^k \lambda_i s_i b_i.
\end{equation}
Finally note from the distribution constraint (\ref{NASH Eq: Distribution Inequalities}) that this implies \begin{equation}\label{NASH dist const}
\sum_{i=1}^k \lambda_is_i^2 = 1.\end{equation}

\begin{lem}\label{NASH Lem: Representation Lemma}
	Let $$\M = \sum_{i=1}^k \lambda_i \bmat a_i\\b_i \emat \bmat a_i\\b_i \emat^T,$$ be a feasible solution to~\ref{NASH Eq: SDP2}, such that the eigenvectors of $\M$ are partitioned into $a_i$ and $b_i$ with $\sum_{j=1}^m (a_i)_j = \sum_{j=1}^n (b_i)_j=s_i, \forall i$. Then $$P-xy^T= \sum_{i = 1}^k \sum_{j>i}^k \lambda_i\lambda_j (s_ja_i - s_ia_j)(s_jb_i - s_ib_j)^T.$$
\end{lem}
\begin{proof}
	{}Using equations (\ref{NASH Eq: Row constraint x}) and (\ref{NASH Eq: Row constraint y}) we can write
	\begin{equation*}
	\begin{aligned}
	P - xy^T &= \sum_{i=1}^k \lambda_i a_ib_i^T - (\sum_{i =1}^k \lambda_i s_i a_i)(\sum_{j=1}^k \lambda_j s_j b_j)^T\\
	&= \sum_{i=1}^k \lambda_ia_i(b_i - s_i\sum_{j=1}^k \lambda_js_jb_j)^T\\
	&\overset{(\ref{NASH dist const})}{=}\sum_{i=1}^k \lambda_ia_i(\sum_{j=1}^k \lambda_js_j^2b_i - s_i\sum_{j=1}^k \lambda_js_jb_j)^T\\
	&=\sum_{i=1}^k\sum_{j=1}^k \lambda_i\lambda_j a_is_j(s_jb_i-s_ib_j)^T\\
	&=\sum_{i = 1}^k \sum_{j > i}^k \lambda_i\lambda_j (s_ja_i - s_ia_j)(s_jb_i - s_ib_j)^T,
	\end{aligned}
	\end{equation*}
	where the last line follows from observing that terms where $i$ and $j$ are switched can be combined.
\end{proof}

We can relate $\epsilon$ and $P-xy^T$ with the following lemma.
\begin{lem}\label{NASH Lem: L1 norm/2 bound}
	Let the matrix $P$ and the vectors $x \defeq P1_n$ and $y \defeq P^T1_m$ come from any feasible solution to $\ref{NASH Eq: SDP2}$. Then $$\epsilon \le \frac{\|P-xy^T\|_1}{2},$$ where $\|\cdot\|_1$ here denotes the entrywise L-1 norm, i.e., the sum of the absolute values of the entries of the matrix.
\end{lem}
\begin{proof}
	{}Let $D \mathrel{\mathop:}= P - xy^T$. From Lemma~\ref{NASH Lem: Bound for e}, $$\epsilon_A \le \Tr(AZ) - x^TAy = \Tr(A(Z-yx^T)).$$ If we then hold $D$ fixed and restrict that $A$ has entries bounded in [0,1], the quantity $\Tr(AD^T)$ is maximized when $$A_{i,j} = \begin{cases} 1 & D_{i,j} \ge 0 \\ 0 & D_{i,j} < 0\end{cases}.$$ The resulting quantity $\Tr(AD^T)$ will then be the sum of all nonnegative elements of $D$. Since the sum of all elements in $D$ is zero, this quantity will be equal to $\frac{1}{2}\|D\|_1$.\\
	The proof for $\epsilon_B$ is identical, and the result follows from that $\epsilon$ is the maximum of $\epsilon_A$ and $\epsilon_B$.
\end{proof}

\subsection{Bounds on $\epsilon$}\label{NASH SSec: Bounds on e}

We provide a number of bounds on $\epsilon(x,y)$for $x \defeq P1_n$ and $y \defeq P^T1_m$ coming from any feasible solution to \ref{NASH Eq: SDP2}. Our first two theorems roughly state that solutions which are ``close'' to rank-1 provide small $\epsilon$.

\begin{theorem}\label{NASH Thm: nksl}
	Consider any feasible solution $\M$ to~\ref{NASH Eq: SDP2}. Suppose $\M$ is rank-$k$ and its eigenvalues are $\lambda_1 \ge \lambda_2 \ge ... \ge \lambda_k > 0$. Then $x$ and $y$ constitute an $\epsilon$-NE to the game $(A,B)$ with $\epsilon \le \frac{m+n}{2}\sum_{i=2}^k \lambda_i.$
\end{theorem}
\begin{proof}
	{}By the Perron Frobenius theorem (see e.g.~\cite[Chapter 8.3]{meyer2000matrix}), the eigenvector corresponding to $\lambda_1$ can be assumed to be nonnegative, and hence \begin{equation}\label{NASH Eq: s1 value}
	s_1 = \|a_1\|_1 = \|b_1\|_1.
	\end{equation}
	We further note that for all $i$, since $\bmat a_i\\b_i\emat$ is a vector of length $m+n$ with 2-norm equal to 1, we must have
	\begin{equation}\label{NASH Eq: L1 norm sum}\left\|\bmat a_i\\b_i\emat\right\|_1 \le \sqrt{m+n}.\end{equation}
	Since $s_i$ is the sum of the elements of $a_i$ and $b_i$, we know that
	\begin{equation}\label{NASH Eq: si bound} s_i \le \min\{\|a_i\|_1, \|b_i\|_1\} \le \frac{\sqrt{m+n}}{2}.\end{equation}
	This then gives us
	\begin{equation}\label{NASH Eq: L1 norm prod}s_i^2 \le \|a_i\|_1\|b_i\|_1 \le \frac{m+n}{4},\end{equation}
	with the first inequality following from (\ref{NASH Eq: si bound}) and the second from (\ref{NASH Eq: L1 norm sum}).
	Finally note that a consequence of the nonnegativity of $\|\cdot\|_1$ and (\ref{NASH Eq: L1 norm sum}) is that for all $i, j$,
	\begin{equation}\label{NASH Eq: l1 cross product}
	\|a_i\|_1\|b_j\|_1 + \|b_i\|_1\|a_j\|_1 \le (\|a_i\|_1+\|b_i\|_1)(\|a_j\|_1+\|b_j\|_1)= \left\|\bmat a_i\\b_i\emat\right\|_1\left\|\bmat a_j\\b_j\emat\right\|_1 \overset{(\ref{NASH Eq: L1 norm sum})}{\le} m+n.
	\end{equation}
	Now we let $D \mathrel{\mathop:}= P - xy^T$ and upper bound $\frac{1}{2}\|D\|_1$ using Lemma~\ref{NASH Lem: Representation Lemma}.
	\begin{align}
	\frac{1}{2}\|D\|_1 &= \frac{1}{2}\|\sum_{i = 1}^k \sum_{j>i}^k \lambda_i\lambda_j (s_ja_i - s_ia_j)(s_jb_i - s_ib_j)^T\|_1\nonumber \\
	&\le \frac{1}{2}\sum_{i = 1}^k \sum_{j>i}^k \|\lambda_i\lambda_j (s_ja_i - s_ia_j)(s_jb_i - s_ib_j)^T\|_1\nonumber  \\
	& \le \frac{1}{2}\sum_{i = 1}^k \sum_{j > i}^k \lambda_i\lambda_j \|s_ja_i - s_ia_j\|_1\|s_jb_i - s_ib_j\|_1\nonumber \\
	& \le \frac{1}{2}\sum_{i = 1}^k \sum_{j > i}^k \lambda_i\lambda_j (s_j\|a_i\|_1 + s_i\|a_j\|_1)(s_j\|b_i\|_1 + s_i\|b_j\|_1) \label{NASH Eq: Hell} \\
	& \overset{(\ref{NASH Eq: s1 value}),(\ref{NASH Eq: L1 norm prod})}{\le} \frac{1}{2}\sum_{j=2}^k\lambda_1s_1^2\lambda_j(s_j+\|a_j\|_1)(s_j+\|b_j\|_1)\nonumber  \\
	&+\frac{1}{2}\sum_{i = 2}^k \sum_{j > i}^k \lambda_i\lambda_j (s_j^2\frac{m+n}{4} + s_i^2\frac{m+n}{4} + s_is_j\|a_i\|_1\|b_j\|_1+ s_is_j\|a_j\|_1\|b_i\|_1)\nonumber \\
	&\overset{(\ref{NASH Eq: L1 norm sum}),(\ref{NASH Eq: l1 cross product}),(\ref{NASH Eq: si bound})}{\le} \frac{m+n}{2}\lambda_1s_1^2 \sum_{i=2}^k \lambda_i\nonumber \\
	&+ \frac{1}{2}\sum_{i=2}^k \sum_{j > i}^k \lambda_i\lambda_j\frac{m+n}{4}(s_i^2 + s_j^2) + \lambda_i\lambda_js_is_j (m+n)\nonumber \\
	&\overset{\text{AMGM\footnotemark}}{\le} \frac{m+n}{2}\lambda_1s_1^2 \sum_{i=2}^k \lambda_i + \frac{m+n}{2}\sum_{i=2}^k \sum_{j > i}^k \lambda_i\lambda_j (\frac{s_i^2 + s_j^2}{4} + \frac{s_i^2+s_j^2}{2})\nonumber \\
	&= \frac{m+n}{2}\lambda_1s_1^2 \sum_{i=2}^k \lambda_i + \frac{3(m+n)}{8}\sum_{i=2}^k \sum_{j > i}^k \lambda_i\lambda_j (s_i^2+s_j^2)\nonumber \\
	&= \frac{m+n}{2}\lambda_1s_1^2 \sum_{i=2}^k \lambda_i + \frac{3(m+n)}{8}(\sum_{i=2}^k\lambda_is_i^2 \sum_{j > i}^k \lambda_j + \sum_{i=2}^k \lambda_i \sum_{j > i}^k \lambda_js_j^2)\nonumber \\
	&= \frac{m+n}{2}\lambda_1s_1^2 \sum_{i=2}^k \lambda_i + \frac{3(m+n)}{8}(\sum_{j = 2}^k \lambda_j \sum_{2\le i<j}^k\lambda_is_i^2 + \sum_{i=2}^k \lambda_i \sum_{j > i}^k \lambda_js_j^2)\nonumber \\
	&\le \frac{m+n}{2}\lambda_1s_1^2 \sum_{i=2}^k \lambda_i + \frac{3(m+n)}{8}(\sum_{j =2}^k \lambda_js_j^2)\sum_{i=2}^k \lambda_i\nonumber \\
	&\overset{(\ref{NASH dist const})}{=}
	\frac{m+n}{2}\lambda_1s_1^2 \sum_{i=2}^k \lambda_i + \frac{3(m+n)}{8}(1-\lambda_1s_1^2)\sum_{i=2}^k \lambda_i\nonumber \\
	&= \frac{m+n}{8}(3+\lambda_1s_1^2)\sum_{i=2}^k \lambda_i\nonumber \\
	&\overset{(\ref{NASH dist const})}{\le} \frac{m+n}{2} \sum_{i=2}^k \lambda_i.\nonumber 
	\end{align}
\end{proof}
\footnotetext{AMGM is used to denote the arithmetic-mean-geometric-mean inequality.}

The following theorem quantifies how making the objective of the diagonal gap algorithm from Section \ref{NASH Sec: Algorithms} small makes $\epsilon$ small. The proof is similar to the proof of Theorem~\ref{NASH Thm: nksl}.
\begin{theorem}\label{NASH Thm: Trace minus 2xxt}
	Let $\M$ be a feasible solution to~\ref{NASH Eq: SDP2}. Then, $x$ and $y$ constitute an $\epsilon$-NE to the game $(A,B)$ with $\epsilon \le \frac{3(m+n)}{8}(\Tr(\M)-x^Tx-y^Ty).$
\end{theorem}
\begin{proof}
	{}Let $\M$ be rank-$k$ with eigenvalues $\lambda_1 \ge \lambda_2 \ge \ldots \ge \lambda_k > 0$ and eigenvectors $v_1, \ldots, v_k$ partitioned as in Lemma \ref{NASH Lem: Partition Lemma} so that $v_i = \bmat a_i\\ b_i\emat$ with $\sum_{j=1}^m (a_i)_j = \sum_{j=1}^n (b_i)_j$ for $i = 1, \ldots, k$. Let $s_i \defeq \sum_{j=1}^m (a_i)_j$. Then we have $\Tr(\M) = \sum_{i=1}^k \lambda_i$, and
	\begin{equation}
	\label{NASH Eq: xynorm}
	x^Tx + y^Ty \overset{(\ref{NASH Eq: Row constraint x}),(\ref{NASH Eq: Row constraint y})}{=} (\sum_{i=1}^k \lambda_i s_i v_i)^T(\sum_{i=1}^k \lambda_i s_i v_i) = \sum_{i=1}^k \lambda_i^2s_i^2.\end{equation}
	We now get the following chain of inequalities (the first one follows from Lemma~\ref{NASH Lem: L1 norm/2 bound} and inequality~(\ref{NASH Eq: Hell})):
	\begin{align*}
	\epsilon &\le \frac{1}{2}\sum_{i = 1}^k \sum_{j > i}^k \lambda_i\lambda_j (s_j\|a_i\|_1 + s_i\|a_j\|_1)(s_j\|b_i\|_1 + s_i\|b_j\|_1)\\
	& \overset{(\ref{NASH Eq: s1 value}),(\ref{NASH Eq: L1 norm prod})}{\le} \frac{1}{2}\sum_{i = 1}^k \sum_{j > i}^k \lambda_i\lambda_j (s_j^2\frac{m+n}{4} + s_i^2\frac{m+n}{4} + s_is_j\|a_i\|_1\|b_j\|_1+ s_is_j\|a_j\|_1\|b_i\|_1)\\
	&\overset{(\ref{NASH Eq: l1 cross product})}{\le} \frac{1}{2}\sum_{i=1}^k \sum_{j > i}^k \lambda_i\lambda_j\frac{m+n}{4}(s_i^2 + s_j^2) + \lambda_i\lambda_js_is_j (m+n)\\
	&\overset{AMGM}{\le} \frac{m+n}{2}\sum_{i=1}^k \sum_{j > i}^k \lambda_i\lambda_j (\frac{s_i^2 + s_j^2}{4} + \frac{s_i^2+s_j^2}{2})\\
	&= \frac{3(m+n)}{8}\sum_{i=1}^k \sum_{j > i}^k \lambda_i\lambda_j (s_i^2+s_j^2)\\
	&= \frac{3(m+n)}{8}(\sum_{i=1}^k\lambda_is_i^2 \sum_{j > i}^k \lambda_j + \sum_{i=1}^k \lambda_i \sum_{j > i}^k \lambda_js_j^2)\\
	&= \frac{3(m+n)}{8}(\sum_{j = 1}^k \lambda_j \sum_{1\le i<j}^k\lambda_is_i^2 + \sum_{i=1}^k \lambda_i \sum_{j > i}^k \lambda_js_j^2)\\
	&= \frac{3(m+n)}{8}(\sum_{i = 1}^k \lambda_i \sum_{j \ne i}\lambda_js_j^2)\\
	&\overset{(\ref{NASH dist const})}{=} \frac{3(m+n)}{8}(\sum_{i = 1}^k \lambda_i (1-\lambda_is_i^2))\\
	&= \frac{3(m+n)}{8}(\sum_{i=1}^k \lambda_i - \sum_{i=1}^k \lambda_i^2s_i^2) \overset{(\ref{NASH Eq: xynorm})}{=} \frac{3(m+n)}{8}(\Tr(\M)-x^Tx-y^Ty).
	\end{align*}
\end{proof}

We now give a bound on $\epsilon$ which is dependent on the nonnegative rank of the matrix returned by \ref{NASH Eq: SDP2}. Our analysis will also be useful for the next subsection. To begin, we first recall the definition of the nonnegative rank.

\begin{defn}\label{NASH Defn: Nonnegative Rank}
	The \emph{nonnegative rank} of a (nonnegative) $m \times n$ matrix $M$ is the smallest $k$ for which there exist a nonnegative $m \times k$ matrix $U$ and a nonnegative $n \times k$ matrix $V$ such that $M = UV^T$. Such a decomposition is called a \emph{nonnegative matrix factorization} of $M$.
\end{defn}

\begin{theorem}\label{NASH Thm: 1-1/k}
	Consider the matrix $P$ from any feasible solution to \ref{NASH Eq: SDP2}. Suppose its nonnegative rank is $k$. Then $x \defeq P1_n$ and $y \defeq P^T1_m$ constitute an $\epsilon$-NE to the game $(A,B)$ with $\epsilon \le 1-\frac{1}{k}$.
\end{theorem}
\begin{proof}	
	{}Since $P$ has nonnegative rank $k$ and its entries sum up to 1, we can write ${P = \sum_{i=1}^k \sigma_ia_ib_i^T}$, where $a_i \in \triangle_m, b_i \in \triangle_n$, and $\sum_{i=1}^k \sigma_i = 1$. From Lemma~\ref{NASH Lem: L1 norm/2 bound} and inequality~(\ref{NASH Eq: Hell}) (keeping in mind that $s_i = 1,\ \forall\ i$) we have
	\begin{align*}
	\epsilon &\le \frac{1}{2} \sum_{i=1}^k \sum_{j > i}^k \sigma_i\sigma_j (\|a_i\|_1 + \|a_j\|_1)(\|b_i\|_1 + \|b_j\|_1)\\
	& \le 2 \sum_{i=1}^k \sum_{j > i}^k \sigma_i\sigma_j\\
	&= 2(\frac{1}{2}(\sum_{i=1}^k \sigma_i \sum_{j=1}^k \sigma_j - \sum_{i=1}^k \sigma_i^2))\\
	&= 1 - \sum_{i=1}^k \sigma_i^2\\
	&\le 1 - \frac{1}{k},
	\end{align*}
	where the last line follows from  the fact that $\|v\|_2^2 \ge \frac{1}{k}$ for any vector $v \in \triangle_k$.
\end{proof}

\subsection{Bounds on $\epsilon$ in the Rank-2 Case}\label{NASH SSec: Rank-2 case}

We now provide a number of bounds on $\epsilon(x,y)$ with $x \defeq P1_n$ and $y \defeq P^T1_m$ which hold for rank-2 feasible solutions $\M$ to \ref{NASH Eq: SDP2} (note that $P$ will have rank at most 2 in this case). This is motivated by our ability to show stronger (constant) bounds in this case, and the fact that we often recover rank-2 (or rank-1) solutions with our algorithms in Section~\ref{NASH Sec: Algorithms}. Furthermore, our analysis will use the special property that a rank-2 nonnegative matrix will have nonnegative rank also equal to two, and that a nonnegative factorization of it can be computed in polynomial time (see, e.g., Section 4 of \cite{cohen1993nonnegative}). We begin with the following observation, which follows from Theorem \ref{NASH Thm: 1-1/k} when $k = 2$.

\begin{cor}\label{NASH Cor: 1/2e} If the matrix $P$ from a feasible solution to~\ref{NASH Eq: SDP2} is rank-2, then $x$ and $y$ constitute a $\frac{1}{2}-$NE.\end{cor}

We now show how this pair of strategies can be refined.

\begin{theorem}\label{NASH Thm: 5/11e}
	If the matrix $P$ from a feasible solution to~\ref{NASH Eq: SDP2} is rank-2, then either $x$ and $y$ constitute a $\frac{5}{11}$-NE, or a $\frac{5}{11}$-NE can be recovered from $P$ in polynomial time.
\end{theorem}
\begin{proof}
	{}We consider 3 cases, depending on whether $\epsilon_A(x,y)$ and $\epsilon_B(x,y)$ are greater than or less than .4. If $\epsilon_A \le .4, \epsilon_B \le .4$, then $(x, y)$ is already a .4-Nash equilibrium. Now consider the case when $\epsilon_A \ge .4, \epsilon_B \ge .4$. Since $\epsilon_A \le \Tr(A(P-xy^T)^T)$ and $\epsilon_B \le \Tr(B(P-xy^T)^T)$ as seen in the proof of Lemma \ref{NASH Lem: Bound for e}, we have, reusing the notation in the proof of Theorem \ref{NASH Thm: 1-1/k},
	\begin{equation*}
	\sigma_1\sigma_2(a_1-a_2)^T A (b_1-b_2) \ge .4,
	\sigma_1\sigma_2(a_1-a_2)^T B (b_1-b_2) \ge .4.
	\end{equation*}
	Since $A, a_1, a_2, b_1,$ and $b_2$ are all nonnegative and $\sigma_1\sigma_2 \le \frac{1}{4}$,
	$$a_1^TAb_1 + a_2^TAb_2 \ge (a_1-a_2)^T A (b_1-b_2) \ge 1.6,$$ and the same inequalities hold for for player B. In particular, since $A$ and $B$ have entries bounded in [0,1] and $a_1,a_2,b_1,$ and $b_2$ are simplex vectors, all the quantities $a_1^TAb_1, a_2^TAb_2, a_1^TBb_1,\ \text{and}\ a_2^TBb_2$ are at most 1, and consequently at least .6. Hence $(a_1,a_2)$ and $(a_2,b_2)$ are both .4-Nash equilibria.
	
	Now suppose that $(x,y)$ is a .4-NE for one player (without loss of generality player A) but not for the other (without loss of generality player B). Then $\epsilon_A \le .4$, and $\epsilon_B \ge .4$. Let $y^*$ be a best response for player B to $x$, and let $p = \frac{1}{1+\epsilon_B - \epsilon_A}$. Consider the strategy profile $(\tilde{x},\tilde{y}) \defeq (x, py + (1-p)y^*)$. This can be interpreted as the outcome $(x,y)$ occurring with probability $p$, and the outcome $(x,y^*)$ happening with probability $1-p$. In the first case, player A will have $\epsilon_A(x,y) = \epsilon_A$ and player B will have $\epsilon_B(x,y) = \epsilon_B$. In the second outcome, player A will have $\epsilon_A(x,y^*)$ at most 1, while player~B will have $\epsilon_B(x,y^*) = 0$. Then under this strategy profile, both players have the same upper bound for $\epsilon$, which equals $\epsilon_B p = \frac{\epsilon_B}{1 + \epsilon_B - \epsilon_A}$. To find the worst case for this value, let $\epsilon_B = .5$ (note from Theorem~\ref{NASH Cor: 1/2e} that $\epsilon_B \le \frac{1}{2}$) and $\epsilon_A = .4$, and this will return $\epsilon = \frac{5}{11}$.
	
\end{proof}

We now show a stronger result in the case of symmetric games.

\begin{defn}\label{NASH Defn: Symmetric Game} A \emph{symmetric game} is a game in which the payoff matrices $A$ and $B$ satisfy $B=A^T$.\end{defn}
\begin{defn} A Nash equilibrium strategy $(x,y)$ is said to be \emph{symmetric} if $x=y$.\end{defn}
\begin{theorem}[see Theorem 2 in~\cite{nash1951}]	Every symmetric bimatrix game has a symmetric Nash equilibrium.
\end{theorem}

For the proof of Theorem~\ref{NASH Thm: 1/3se} below we modify \ref{NASH Eq: SDP2} so that we are seeking a symmetric solution. We also need a more specialized notion of the nonnegative rank.

\begin{defn}\label{NASH Defn: Completely Positive} A matrix $M$ is \emph{completely positive} (CP) if it admits a decomposition $M=UU^T$ for some nonnegative matrix $U$.\end{defn}
\begin{defn} The \emph{CP-rank} of an $n \times n$ CP matrix $M$ is the smallest $k$ for which there exists a nonnegative $n \times k$ matrix $U$ such that $M = UU^T$.\end{defn}
\begin{theorem}[see e.g.~Theorem 2.1 in~\cite{berman2003completely}]\label{NASH CP rank 2}
	A rank-2, nonnegative, and positive semidefinite matrix is CP and has CP-rank 2.
\end{theorem}

It is also known (see e.g., Section 4 in~\cite{kalofolias2012computing}) that the CP factorization of a rank-2 CP matrix can be found to arbitrary accuracy in polynomial time.

\begin{theorem}\label{NASH Thm: 1/3se} Suppose the constraint $P \succeq 0$ is added to~\ref{NASH Eq: SDP2}. Then if in a feasible solution to this new SDP the matrix $P$ is rank-2, either $x$ and $y$ constitute a symmetric $\frac{1}{3}$-NE, or a symmetric $\frac{1}{3}$-NE can be recovered from $P$ in polynomial time.\end{theorem} 
\begin{proof}
	{}If $(x,y)$ is already a symmetric $\frac{1}{3}$-NE, then the claim is established. Now suppose that $(x,y)$ does not constitute a $\frac{1}{3}$-Nash equilibrium. Similarly as in the proof of Theorem \ref{NASH Thm: 1-1/k}, we can decompose $P$ into $\sum_{i=1}^2 \sigma_i a_ia_i^T$, where $\sum_{i=1}^2 \sigma_i = 1$ and each $a_i$ is a vector on the unit simplex. Then we have
	\begin{equation*}
	\sigma_1\sigma_2(a_1-a_2)^T A (a_1-a_2) \ge \frac{1}{3}.
	\end{equation*}
	Since $A, a_1,$ and $a_2$ are all nonnegative, and $\sigma_1\sigma_2 \le \frac{1}{4}$, we get
	\begin{equation*}a_1^TAa_1+ a_2^TAa_2 \ge (a_1-a_2)^T A (a_1-a_2) \ge \frac{4}{3}.\end{equation*}
	In particular, at least one of $a_1^TAa_1$ and $a_2^TAa_2$ is at least $\frac{2}{3}$. Since the maximum possible payoff is 1, at least one of $(a_1,a_1)$ and $(a_2,a_2)$ is a (symmetric) $\frac{1}{3}$-Nash equilibrium.
\end{proof}

\begin{remark}
	For symmetric games, instead of the construction stated in Theorem \ref{NASH Thm: 1/3se}, one can simply optimize over a smaller $m \times m$ matrix (note $m=n$). This is the relaxed version of exchangeable equilibria~\cite{stein1943exchangeable}, with the completely positive constraint relaxed to a psd constraint.
\end{remark}

\begin{remark}\label{NASH Rem: Rank of M}
	
	The statements of Corollary \ref{NASH Cor: 1/2e}, and Theorem \ref{NASH Thm: 5/11e}, and Theorem \ref{NASH Thm: 1/3se} hold for any rank-2 correlated equilibrium. Indeed, given any rank-2 (equivalently, nonnegative-rank-2) correlated equilibrium $P$, one can complete it to a (rank-2) feasible solution to \ref{NASH Eq: SDP2} as follows. Let $P = \sum_{i=1}^2 \sigma_ia_ib_i^T$, where $a_i \in \triangle_m, b_i \in \triangle_n$, and $\sigma_1 + \sigma_2 = 1$. It is easy to check that $$\M~\defeq~\sum_{i=1}^2 \sigma_i \bmat a_i \\ b_i \emat \bmat a_i \\ b_i \emat^T$$ is feasible to \ref{NASH Eq: SDP2}.	\end{remark}

\section{Bounding Payoffs and Strategy Exclusion in Symmetric Games}\label{NASH Sec: Bounding Payoffs and Strategy Exclusion}
In addition to finding $\epsilon$-additive Nash equilibria, our SDP approach can be used to answer certain questions of economic interest about Nash equilibria without actually computing them. For instance, economists often would like to know the maximum welfare (sum of the two players' payoffs) achievable under any Nash equilibrium, or whether there exists a Nash equilibrium in which a given subset of strategies (corresponding, e.g., to undesirable behavior) is not played. Both these questions are NP-hard for bimatrix games \cite{gilboa1989nash}, even when the game is symmetric and only symmetric equilibria are considered \cite{conitzer2008new}. In this section, we consider these two problems in the symmetric setting and compare the performance of our SDP approach to an LP approach which searches over symmetric correlated equilibria. For general equilibria, it turns out that for these two specific questions, our SDP approach is equivalent to an LP that searches over correlated equilibria.

\subsection{Bounding Payoffs}\label{NASH SSec: Bounding Payoffs}
When designing policies that are subject to game theoretic behavior by agents, economists would often like to find one with a good socially optimal outcome, which usually corresponds to an equilibrium giving the maximum welfare. Hence, given a game, it is of interest to know the highest achievable welfare under any Nash equilibrium. For symmetric games, symmetric equilibria are of particular interest as they reflect the notion that identical agents should behave similarly given identical options.

Note that the maximum welfare of a symmetric game under any symmetric Nash equilibrium is equal to the optimal value of the following quadratic program:
\begin{equation}\label{NASH Eq: Payoff QP}
\begin{aligned}
& \underset{x \in \triangle_m}{\max}
& & 2x^TAx\\
& \text{subject to}
& & x^TAx \ge e_i^TAx, \forall i\in \{1,\ldots, m\}.\\
\end{aligned}
\end{equation}
One can find an upper bound on this number by solving an LP which searches over symmetric correlated equilibria:
\begin{align}\label{NASH Eq: LP1}\tag{LP1}
& \underset{P \in \mathbb{S}^{m, m}}{\max}
& & \Tr(AP^T)& \nonumber\\
& \text{subject to}
& & \sum_{i=1}^m \sum_{j=1}^m P_{i,j} = 1\label{NASH Eq: SDP3 Distribution}&\\
&&& \sum_{j=1}^m A_{i,j}P_{i,j} \ge \sum_{j=1}^m A_{k,j}P_{i,j}, \forall i,k \in \{1,\ldots, m\}, & \label{NASH Eq: SDP3 CE}\\
&&& P \ge 0.\label{NASH Eq: SDP3 Nonnegativity}&
\end{align}
A potentially better upper bound on the maximum welfare can be obtained from a version of \ref{NASH Eq: SDP2} adapted to this specific problem:
\begin{samepage}
	\begin{align}\label{NASH Eq: SDP3}\tag{SDP3}
	& \underset{P \in \mathbb{S}^{m, m}}{\max}
	& & \Tr(AP^T)& \nonumber\\
	& \text{subject to}
	& & (\ref{NASH Eq: SDP3 Distribution}), (\ref{NASH Eq: SDP3 CE}), (\ref{NASH Eq: SDP3 Nonnegativity})\nonumber&\\
	&&& P \succeq 0.\nonumber&
	\end{align}
\end{samepage}
To test the quality of these upper bounds, we tested this LP and SDP on a random sample of one hundred $5\times 5$ and $10 \times 10$ games\footnote{The matrix $A$ in each game was randomly generated with diagonal entries uniform and independent in [0,.5] and off-diagonal entries uniform and independent in [0,1].}. The resulting upper bounds are in Figure~\ref{NASH Fig: welfareapprox}, which shows that the bound returned by~\ref{NASH Eq: SDP3} was exact in a large number of the experiments.\footnote{The computation of the exact maximum payoffs was done with the \texttt{lrsnash} software~\cite{avis2010enumeration}, which computes all extreme Nash equilibria. For a definition of extreme Nash equilibria and for understanding why it is sufficient for us to compare against extreme Nash equilibria (both in Section \ref{NASH SSec: Bounding Payoffs} and in Section~\ref{NASH SSec: Strategy Exclusion}), see Appendix~\ref{NASH Sec: Lemmas for Extreme Nash Equilibria}. The computation of the SDP upper bound has been implemented in the file {nashbound.m}, which is publicly available at \texttt{https://github.com/jeffreyzhang92/SDP\_Nash} along with the instance data. This file more generally computes an SDP-based lower bound on the minimum of an input quadratic function over the set of Nash equilibria of a bimatrix game. The file also takes as an argument whether one wishes to only consider symmetric equilibria when the game is symmetric.}

\begin{figure}[H]
	\includegraphics[height=.25\textheight,keepaspectratio]{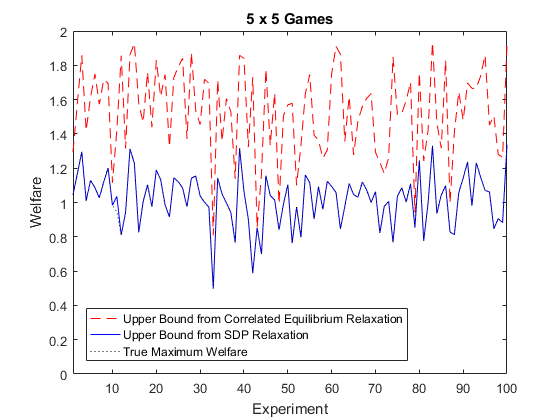}
	\includegraphics[height=.25\textheight,keepaspectratio]{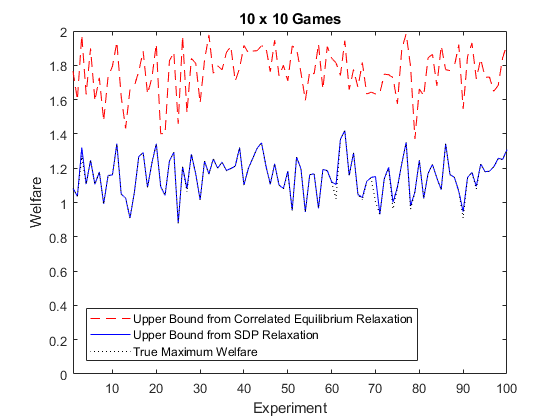}
	\caption{The quality of the upper bound on the maximum welfare obtained by~\ref{NASH Eq: LP1} and~\ref{NASH Eq: SDP3} on 100 $5\times 5$ games (left) and 100 $10\times 10$ games (right).\label{NASH Fig: welfareapprox}}
\end{figure}

\subsection{Strategy Exclusion}\label{NASH SSec: Strategy Exclusion}
The strategy exclusion problem asks, given a subset of strategies $\mathcal{S}=(\mathcal{S}_x,\mathcal{S}_y)$, with ${\mathcal{S}_x \subseteq\{1,\ldots,m\}}$ and $\mathcal{S}_y\subseteq\{1,\ldots,n\}$, is there a Nash equilibrium in which no strategy in $\mathcal{S}$ is played with positive probability. We will call a set $\mathcal{S}$ ``persistent'' if the answer to this question is negative, i.e. at least one strategy in $\mathcal{S}$ is played with positive probability in every Nash equilibrium. One application of the strategy exclusion problem is to understand whether certain strategies can be discouraged in the design of a game, such as reckless behavior in a game of chicken or defecting in a game of prisoner's dilemma. In these particular examples these strategy sets are persistent and cannot be discouraged.

As in the previous subsection, we consider the strategy exclusion problem for symmetric strategies in symmetric games (such as the aforementioned games of chicken and prisoner's dilemma). A quadratic program which addresses this problem is as follows:
\begin{samepage}
	\begin{equation}\label{NASH Eq: Exclusion QP}
	\begin{aligned}
	& \underset{x\in\triangle_m}{\min}
	& & \sum_{i \in \mathcal{S}_x} x_i\\
	& \text{subject to}
	& & x^TAx \ge e_i^TAx, \forall i\in \{1,\ldots, m\}.\\
	\end{aligned}
	\end{equation}
\end{samepage}
Observe that by design, $\mathcal{S}$ is persistent if and only if this quadratic program has a positive optimal value. As in the previous subsection, an LP relaxation of this problem which searches over symmetric correlated equilibria is given by
\begin{samepage}
	\begin{align}\label{NASH Eq: LP2}\tag{LP2}
	& \underset{P \in \mathbb{S}^{m, m}}{\min}
	& & \sum_{i \in \mathcal{S}_x}\sum_{j=1}^m P_{ij} &\nonumber\\
	& \text{subject to}
	& & (\ref{NASH Eq: SDP3 Distribution}), (\ref{NASH Eq: SDP3 CE}), (\ref{NASH Eq: SDP3 Nonnegativity}).\nonumber&
	\end{align}
\end{samepage}
The SDP relaxation that we propose for the strategy exclusion problem is the following:
\begin{samepage}
	\begin{align}\label{NASH Eq: SDP4}\tag{SDP4}
	& \underset{P \in \mathbb{S}^{m, m}}{\min}
	& & \sum_{i \in \mathcal{S}_x}\sum_{j=1}^m P_{ij} &\nonumber\\
	& \text{subject to}
	& & (\ref{NASH Eq: SDP3 Distribution}), (\ref{NASH Eq: SDP3 CE}), (\ref{NASH Eq: SDP3 Nonnegativity})&\nonumber\\
	&&& P \succeq 0.\nonumber&
	\end{align}
\end{samepage}
Our approach would be to declare that the strategy set $\mathcal{S}_x$ is persistent if and only if~\ref{NASH Eq: SDP4} has a positive optimal value.

Note that since the optimal value of~\ref{NASH Eq: SDP4} is a lower bound for that of (\ref{NASH Eq: Exclusion QP}),~\ref{NASH Eq: SDP4} carries over the property that if a set $\mathcal{S}$ is not persistent, then the SDP for sure returns zero. Thus, when using~\ref{NASH Eq: SDP4} on a set which is not persistent, our algorithm will always be correct. However, this is not necessarily the case for a persistent set. While we can be certain that a set is persistent if~\ref{NASH Eq: SDP4} returns a positive optimal value (again, because the optimal value of~\ref{NASH Eq: SDP4} is a lower bound for that of (\ref{NASH Eq: Exclusion QP})), there is still the possibility that for a persistent set~\ref{NASH Eq: SDP4} will have optimal value zero. The same arguments hold for the optimal value of \ref{NASH Eq: LP2}.

To test the performance of \ref{NASH Eq: LP2} and \ref{NASH Eq: SDP4}, we generated 100 random games of size $5\times 5$ and $10\times 10$ and computed all their symmetric extreme Nash equilibria\footnote{The exact computation of the exact Nash equilibria was done again with the \texttt{lrsnash} software \cite{avis2010enumeration}, which computes extreme Nash equilibria. To understand why this suffices for our purposes see Appendix~\ref{NASH Sec: Lemmas for Extreme Nash Equilibria}.}. We then, for every strategy set $\mathcal{S}$ of cardinality one and two, checked whether that set of strategies was persistent, first by checking among the extreme Nash equilibria, then through \ref{NASH Eq: LP2} and \ref{NASH Eq: SDP4}. The results are presented in Tables~\ref{NASH Tab: SE 5x5 1} and~\ref{NASH Tab: SE 10x10 1}. As can be seen,~\ref{NASH Eq: SDP4} was quite effective for the strategy exclusion problem.

\begin{table}[h]
	\caption{Performance of~\ref{NASH Eq: LP2} and~\ref{NASH Eq: SDP4} on $5\times 5$ games}
	{\begin{tabular}{|c|c|c|}\hline
			$|\mathcal{S}|$ & 1 & 2\\\hline
			Number of total sets&	500 & 1000\\\hline
			Number of persistent sets&	245 & 748 \\\hline
			Persistent sets certified (\ref{NASH Eq: LP2}) & 177 (72.2\%) & 661 (88.7\%)\\\hline
			Persistent sets certified (\ref{NASH Eq: SDP4})& 245 (100\%) & 748 (100\%)\\\hline
		\end{tabular}\label{NASH Tab: SE 5x5 1}}
	\centering
	{}
\end{table}
\begin{table}[h]
	\caption{Performance of~\ref{NASH Eq: LP2} and~\ref{NASH Eq: SDP4} on $10\times 10$ games}
	{\begin{tabular}{|c|c|c|}\hline
			$|\mathcal{S}|$ & 1 & 2\\\hline
			Number of total sets& 1000 & 4500\\\hline
			Number of persistent sets&	326 & 2383 \\\hline
			Persistent sets certified (\ref{NASH Eq: LP2}) & 39 (12.0\%) & 630 (26.4\%)\\\hline
			Persistent sets certified (\ref{NASH Eq: SDP4})& 318 (97.5\%) & 2368 (99.4\%)\\\hline
		\end{tabular}\label{NASH Tab: SE 10x10 1}}
	\centering
	{}
\end{table}

\section{Connection to the Sum of Squares/Lasserre Hierarchy}\label{NASH Sec: Connection to Sum of Squares/Lasserre Hierarchy}

In this section, we clarify the connection of the SDPs we have proposed in this chapter to those arising in the sum of squares/Lasserre hierarchy. We start by briefly reviewing this hierarchy.

\subsection{Sum of Squares/Lasserre Hierarchy}\label{NASH SSec: Lassere's Hierarchy}
The sum of squares/Lasserre hierarchy\footnote{The unfamiliar reader is referred to~\cite{lasserre2001global,parrilo2003semidefinite,laurent2009sums} for an introduction to this hierarchy and the related theory of moment relaxations.} gives a recipe for constructing a sequence of SDPs whose optimal values converge to the optimal value of a given polynomial optimization problem. Recall that for a POP of the form
\begin{equation}\label{NASH Defn: pop}
\begin{aligned}
& \underset{x \in \mathbb{R}^n}{\min}
& & p(x) \\
& \text{subject to}
&& q_i(x) \ge 0, \forall i \in \{1,\ldots,m\},\\
\end{aligned}
\end{equation}
where $p,q_i$ are polynomial functions, $k$-th level of the Lasserre hierarchy is given by

\begin{equation}\label{NASH Defn: Lasserre Hierarchy}
\begin{aligned}
\gamma_{sos}^k \defeq & \underset{\gamma,\sigma_i}{\max}
&& \gamma\\
& \text{subject to}
&& p(x)-\gamma = \sigma_0(x) + \sum_{i=1}^m \sigma_i(x)q_i(x),\\
&&&\sigma_i \text{ is sos, }\forall i \in \{0,\ldots,m\},\\
&&&\sigma_0, g_i\sigma_i \text{ have degree at most } 2k,\ \forall i \in \{1,\ldots,m\}.
\end{aligned}
\end{equation}

Recall that any fixed level of this hierarchy gives an SDP of size polynomial in $n$ and that, if the quadratic module generated by $\{x\in\mathbb{R}^n|g_i(x)\ge 0\}$ is Archimedean (see, e.g. \cite{laurent2009sums} for definition), then $\underset{k \to \infty}{\lim} \gamma_{sos}^k = p^*$, where $p^*$ is the optimal value of the pop in (\ref{NASH Defn: pop}). The latter statement is a consequence of Putinar's positivstellensatz (see, e.g.~\cite{putinar1993positive},~\cite{lasserre2001global}).

\subsection{The Lasserre Hierarchy and~\ref{NASH Eq: SDP1}}\label{NASH SSec: The Lasserre Hierarchy and SDP1}

One can show, e.g. via the arguments in \cite{lasserre2009convex}, that the feasible sets of the SDPs dual to the SDPs underlying the hierarchy we summarized above produce an arbitrarily tight outer approximation to the convex hull of the set of Nash equilibria of any game. The downside of this approach, however, is that the higher levels of the hierarchy can get expensive very quickly. This is why the approach we took in this chapter was instead to improve the first level of the hierarchy. The next proposition formalizes this connection.

\begin{prop}\label{NASH Prop: strongerLasserre}
	Consider the problem of minimizing any quadratic objective function over the set of Nash equilibria of a bimatrix game. Then,~\ref{NASH Eq: SDP1} (and hence~\ref{NASH Eq: SDP2}) gives a lower bound on this problem which is no worse than that produced by the first level of the Lasserre hierarchy.
\end{prop}

\begin{proof}
	{}To prove this proposition we show that the first level of the Lasserre hierarchy is dual to a weakened version of~\ref{NASH Eq: SDP1}.
	
	\textbf{Explicit parametrization of first level of the Lasserre hierarchy.} Consider the formulation of the Lasserre hierarchy in (\ref{NASH Defn: Lasserre Hierarchy}) with $k=1$. Suppose we are minimizing a quadratic function $$f(x,y)=\xy1vec^T \mathcal{C} \xy1vec$$ over the set of Nash equilibria as described by the linear and quadratic constraints in (\ref{NASH Eq: QCQP Formulation}). If we apply the first level of the Lasserre hierarchy to this particular pop, we get
	
	\begin{equation}
	\begin{aligned}\label{NASH Eq: LH1 long}
	\underset{Q, \alpha, \chi, \beta, \psi, \eta}{\max}
	&& \gamma\\
	\text{subject to}
	&& \xy1vec^T \mathcal{C} \xy1vec-\gamma &= \xy1vec^T Q \xy1vec+\sum_{i=1}^m \alpha_i(x^TAy - e_i^TAy)\\
	&& &+ \sum_{i=1}^n \beta_i(x^TBy - x^TBe_i)\\
	&& &+ \sum_{i=1}^m \chi_i x_i + \sum_{i=1}^n \psi_i y_i\\
	&& &+\eta_1 (\sum_{i=1}^m x_i - 1)+ \eta_2 (\sum_{i=1}^n y_i -1),\\
	&& Q &\succeq 0,\\
	&& \alpha, \chi, \beta, \psi &\ge 0,
	\end{aligned}
	\end{equation}
	
	where $Q \in \mathbb{S}^{m+n+1 \times m+n+1},\alpha, \chi \in \mathbb{R}^m, \beta, \psi \in \mathbb{R}^n, \eta \in \mathbb{R}^2$.
	
	By matching coefficients of the two quadratic functions on the left and right hand sides of (\ref{NASH Eq: LH1 long}), this SDP can be written as
	\begin{equation}\label{NASH Eq: Lassere Level 1}
	\begin{aligned}
	& \underset{\gamma,\alpha,\beta,\chi,\psi,\eta}{\max}
	&& \gamma\\
	& \text{subject to}
	&& \mathcal{H} \succeq 0,\\
	&&& \alpha, \beta, \chi, \psi \ge 0,
	\end{aligned}
	\end{equation}
	where
	\begin{small}
	\begin{equation}\label{NASH Eq: LH1 H}
	\mathcal{H} \mathrel{\mathop:}= \frac{1}{2}\bmat 0 & (-\sum_{i=1}^m \alpha_i )A + (-\sum_{i=1}^m \beta_i)B & \sum_{i=1}^n \beta_i B_{,i}-\chi - \eta_11_m\\
	(-\sum_{i=1}^m \alpha_i)A + (-\sum_{i=1}^n \beta_i)B & 0 & \sum_{i=1}^m \alpha_i A_{i,}^T-\psi - \eta_21_n\\
	\sum_{i=1}^n \beta_iB_{,i}^T-\chi^T-\eta_11_m^T & \sum_{i=1}^m \alpha_i A_{i,}-\psi^T - \eta_21_n^T  & 2\eta_1+2\eta_2-2\gamma
	\emat+\mathcal{C}.\end{equation}\end{small}
	
	\textbf{Dual of a weakened version of SDP1.} With this formulation in mind, let us consider a weakened version of~\ref{NASH Eq: SDP1} with only the relaxed Nash constraints, unity constraints, and nonnegativity constraints on $x$ and $y$ in the last column (i.e., the nonegativity constraint is not applied to the entire matrix). Let the objective be $\Tr(C\M')$. To write this new SDP in standard form, let\\
	$$\mathcal{A}_i \mathrel{\mathop:}= \frac{1}{2}\bmat 0 & A & 0\\ A^T & 0 & -A_{i,}^T\\ 0 & -A_{i,} & 0\emat, \mathcal{B}_i \mathrel{\mathop:}= \frac{1}{2}\bmat 0 & B & -B_{,i}\\ B^T & 0 & 0\\-B_{,i}^T & 0 & 0 \emat,$$
	$$\mathcal{S}_1 \mathrel{\mathop:}= \frac{1}{2}\bmat 0 & 0 & 1_m\\ 0 & 0 & 0\\1_m^T & 0 & -2\emat, \mathcal{S}_2 \mathrel{\mathop:}= \frac{1}{2}\bmat 0 & 0 & 0\\ 0 & 0 & 1_n\\ 0 & 1_n^T & -2 \emat.$$
	Let $\mathcal{N}_i$ be the matrix with all zeros except a $\frac{1}{2}$ at entry $(i,m+n+1)$ and $(m+n+1,i)$ (or a 1 if $i=m+n+1$).\\
	Then this SDP can be written as\\
	\begin{align}\label{NASH Eq: SDP0}\tag{SDP0}
	& \underset{\M'}{\min}
	& & \Tr(\mathcal{C}\M') \\
	& \text{subject to}
	& & \mathcal{M}' \succeq 0,\label{NASH Eq: SDP0 PSD}\\
	&&& \Tr(\mathcal{N}_{i}\M') \ge 0, \forall i \in \{1,\ldots,m+n\}, \label{NASH Eq: SDP0 nonnegativity}\\
	&&& \Tr(\mathcal{A}_i\M') \ge 0, \forall i \in \{1,\ldots,m\}, \label{NASH Eq: SDP0 Nash A}\\
	&&& \Tr(\mathcal{B}_i\M') \ge 0, \forall i \in \{1,\ldots,n\}, \label{NASH Eq: SDP0 Nash B}\\
	&&& \Tr(\mathcal{S}_1\M') = 0, \label{NASH Eq: SDP0 simplex x}\\
	&&& \Tr(\mathcal{S}_2\M') = 0, \label{NASH Eq: SDP0 simplex y}\\
	&&& \Tr({\mathcal{N}}_{m+n+1}) = 1 \label{NASH Eq: SDP0 Corner}.
	\end{align}
	
	We now create dual variables for each constraint; we choose $\alpha_i$ and $\beta_i$ for the relaxed Nash constraints (\ref{NASH Eq: SDP0 Nash A}) and (\ref{NASH Eq: SDP0 Nash B}), $\eta_1$ and $\eta_2$ for the unity constraints (\ref{NASH Eq: SDP0 simplex x}) and (\ref{NASH Eq: SDP0 simplex y}), $\chi$ for the nonnegativity of $x$ (\ref{NASH Eq: SDP0 nonnegativity}), $\psi$ for the nonnegativity of $y$ (\ref{NASH Eq: SDP0 nonnegativity}), and $\gamma$ for the final constraint on the corner (\ref{NASH Eq: SDP0 Corner}). These variables are chosen to coincide with those used in the parametrization of the first level of the Lasserre hierarchy, as can be seen more clearly below.\\
	
	We then write the dual of the above SDP as\\
	\begin{equation*}
	\begin{aligned}
	& \underset{\alpha, \beta, \lambda, \gamma}{\max}&& \gamma\\
	&\text{subject to} && \sum_{i = 1}^m \alpha_i \mathcal{A}_i + \sum_{i = 1}^n \beta_i \mathcal{B}_i + \sum_{i = 1}^2 \eta_i\mathcal{S}_i+\sum_{i=1}^m \mathcal{N}_{i+n}\chi_i+\sum_{i=1}^n \mathcal{N}_{i}\psi_i + \gamma \mathcal{N}_{m+n+1} \preceq \mathcal{C},\\\
	&&& \alpha, \beta, \chi, \psi \ge 0.
	\end{aligned}
	\end{equation*}
	which can be rewritten as
	\begin{equation} \label{NASH Eq: SDP1 Dual}
	\begin{aligned}
	& \underset{\alpha, \beta, \chi, \psi, \gamma}{\max}&& \gamma\\
	&\text{subject to} && \mathcal{G} \succeq 0,\\\
	&&& \alpha, \beta, \chi,\psi \ge 0,
	\end{aligned}
	\end{equation}
	
	where
	\begin{small}
	\begin{equation*}
	\mathcal{G} \defeq \frac{1}{2}\bmat 0 & (-\sum_{i=1}^m \alpha_i )A + (-\sum_{i=1}^m \beta_i)B & \sum_{i=1}^n \beta_i B_{,i}-\chi - \eta_11_m\\
	(-\sum_{i=1}^m \alpha_i)A + (-\sum_{i=1}^n \beta_i)B & 0 & \sum_{i=1}^m \alpha_i A_{i,}^T-\psi - \eta_21_n\\
	\sum_{i=1}^n \beta_iB_{,i}^T-\chi^T-\eta_11_m^T & \sum_{i=1}^m \alpha_i A_{i,}-\psi^T - \eta_21_n^T  & 2\eta_1+2\eta_2-2\gamma
	\emat+\mathcal{C}.
	\end{equation*}
	\end{small}
	
	We can now see that the matrix $\mathcal{G}$ coincides with the matrix $\mathcal{H}$ in the SDP (\ref{NASH Eq: Lassere Level 1}). Then we have $$(\ref{NASH Eq: LH1 long})^{opt}=(\ref{NASH Eq: Lassere Level 1})^{opt} = (\ref{NASH Eq: SDP1 Dual})^{opt} \le~\ref{NASH Eq: SDP0}^{opt} \le~\ref{NASH Eq: SDP1}^{opt},$$
	where the first inequality follows from weak duality, and the second follows from that the constraints of~\ref{NASH Eq: SDP0} are a subset of the constraints of~\ref{NASH Eq: SDP1}.
\end{proof}

\begin{remark}
	The Lasserre hierarchy can be viewed in each step as a pair of primal-dual SDPs: the sum of squares formulation which we have just presented, and a moment formulation which is dual to the sos formulation~\cite{lasserre2001global}. All our SDPs in this chapter can be viewed more directly as an improvement upon the moment formulation.
\end{remark}

\begin{remark}
	One can see, either by inspection or as an implication of the proof of Theorem~\ref{NASH Thm: Necessity of VI}, that in the case where the objective function corresponds to maximizing player A's and/or B's payoffs\footnote{This would be the case, for example, in the maximum social welfare problem of Section~\ref{NASH SSec: Bounding Payoffs}, where the matrix of the quadratic form in the objective function is given by $$\mathcal{C}=\bmat 0 & -A-B & 0\\-A-B&0&0\\0&0&0\emat.$$}, SDPs (\ref{NASH Eq: Lassere Level 1}) and (\ref{NASH Eq: SDP1 Dual}) are infeasible. This means that for such problems the first level of the Lasserre hierarchy gives an upper bound of $+\infty$ on the maximum payoff. On the other hand, the additional valid inequalities in~\ref{NASH Eq: SDP2} guarantee that the resulting bound is always finite.
\end{remark}

\section{Future Work} \label{NASH Sec: Conclusion}

Our work leaves many avenues of further research. Are there other interesting subclasses of games (besides strictly competitive games) for which our SDP is guaranteed to recover an exact Nash equilibrium? Can the guarantees on $\epsilon$ in Section~\ref{NASH Sec: Bounds} be improved in the rank-2 case (or the general case) by improving our analysis? Is there a polynomial time algorithm that is guaranteed to find a rank-2 solution to~\ref{NASH Eq: SDP2}? Such an algorithm, together with our analysis, would improve the best known approximation bound for symmetric games (see Theorem \ref{NASH Thm: 1/3se}). Can this bound be extended to general games? We show in Appendix \ref{NASH Sec: Irreducibility} that some natural approaches based on symmetrization of games do not immediately lead to a positive answer to this question. Can SDPs in a higher level of the Lasserre hierarchy be used to achieve better $\epsilon$ guarantees? What are systematic ways of adding valid inequalities to these higher-order SDPs by exploiting the structure of the Nash equilibrium problem? For example, since any strategy played with positive probability must give the same payoff, one can add a relaxed version of the cubic constraints $$x_ix_j(e_i^TAy - e_j^TAy)=0, \forall i,j \in \{1,\ldots,m\}$$ to the SDP underlying the second level of the Lasserre hierarchy. What are other valid inequalities for the second level? Finally, our algorithms were specifically designed for two-player one-shot games. This leaves open the design and analysis of semidefinite relaxations for repeated games or games with more than two players.
\appendix
\chapter{Appendices for Nash Equilibria}

\section{Statistics on $\epsilon$ from Algorithms in Section~\ref{NASH Sec: Algorithms}}\label{NASH Sec: Appendix Statistics}
	Below are statistics for the $\epsilon$ recovered in 100 random games of varying sizes using the algorithms of Section~\ref{NASH Sec: Algorithms}.
	
	\begin{table}[H]
		\caption{Statistics on $\epsilon$ for $5\times 5$ games after 20 iterations.\label{NASH Tab: 5x5}}
		{\begin{tabular}{|c|c|c|c|c|c|}\hline
				Algorithm & Max & Mean & Median & StDev\\\hline
				Square Root&	0.0702&	0.0040&	0.0004&	0.0099	\\\hline
				Diagonal Gap&	0.0448&	0.0027&	0&	0.0061	\\\hline
		\end{tabular}}
		\centering
		{}
	\end{table}
	
	\begin{table}[H]
		\caption{Statistics on $\epsilon$ for $10\times 5$ games after 20 iterations.\label{NASH Tab: 10x5}}
		{\begin{tabular}{|c|c|c|c|c|c|}\hline
				Algorithm & Max & Mean & Median & StDev\\\hline
				Square Root&	0.0327&	0.0044&	0.0021&	0.0064	\\\hline
				Diagonal Gap&	0.0267&	0.0033&	0.0006&	0.0053	\\\hline
		\end{tabular}}
		\centering
		{}
	\end{table}
	
	\begin{table}[H]
		\caption{Statistics on $\epsilon$ for $10\times 10$ games after 20 iterations.\label{NASH Tab: 10x10}}
		{\begin{tabular}{|c|c|c|c|c|c|}\hline
				Algorithm & Max & Mean & Median & StDev\\\hline
				Square Root&	0.0373&	0.0058&	0.0039&	0.0065	\\\hline
				Diagonal Gap&	0.0266&	0.0043&	0.0026&	0.0051	\\\hline
		\end{tabular}}
		\centering
		{}
	\end{table}
	
	\begin{table}[H]
		\caption{Statistics on $\epsilon$ for $15\times 10$ games after 20 iterations.\label{NASH Tab: 15x10}}
		{\begin{tabular}{|c|c|c|c|c|c|}\hline
				Algorithm & Max & Mean & Median & StDev\\\hline
				Square Root&	0.0206&	0.0050&	0.0034&	0.0045	\\\hline
				Diagonal Gap&	0.0212&	0.0038&	0.0025&	0.0039	\\\hline
		\end{tabular}}
		\centering
		{}
	\end{table}
	
	\begin{table}[H]
		\caption{Statistics on $\epsilon$ for $15\times 15$ games after 20 iterations.\label{NASH Tab: 15x15}}
		{\begin{tabular}{|c|c|c|c|c|c|}\hline
				Algorithm & Max & Mean & Median & StDev\\\hline
				Square Root&	0.0169&	0.0051&	0.0042&	0.0039	\\\hline
				Diagonal Gap&	0.0159&	0.0038&	0.0029&	0.0034	\\\hline
		\end{tabular}}
		\centering
		{}
	\end{table}
	
	\begin{table}[H]
		\caption{Statistics on $\epsilon$ for $20\times 15$ games after 20 iterations.\label{NASH Tab: 15x20}}
		{\begin{tabular}{|c|c|c|c|c|c|}\hline
				Algorithm & Max & Mean & Median & StDev\\\hline
				Square Root&	0.0152&	0.0046&	0.0035&	0.0036	\\\hline
				Diagonal Gap&	0.0119&	0.0032&	0.0022&	0.0027	\\\hline
		\end{tabular}}
		\centering
		{}
	\end{table}
	
	\begin{table}[H]
		\caption{Statistics on $\epsilon$ for $20\times 20$ games after 20 iterations.\label{NASH Tab: 20x20}}
		{\begin{tabular}{|c|c|c|c|c|c|}\hline
				Algorithm & Max & Mean & Median & StDev\\\hline
				Square Root&	0.0198&	0.0046&	0.0039&	0.0034	\\\hline
				Diagonal Gap&	0.0159&	0.0032&	0.0024&	0.0032	\\\hline
		\end{tabular}}
		\centering
		{}
	\end{table}
	
	\section{Comparison with an SDP Approach from \cite{laraki2012semidefinite}}
	
	In this section, at the request of a referee, we compare the first level of the SDP hierarchy given in \cite[Section 4]{laraki2012semidefinite} to~\ref{NASH Eq: SDP2} using $\Tr(M)$ as the objective function on 100 randomly generated games for each size given in the tables below. The first level of the hierarchy in \cite{laraki2012semidefinite} optimizes over a matrix which is slightly bigger than the one in \ref{NASH Eq: SDP2}, though it has a number of constraints linear in the size of the game considered, as opposed to the quadratic number in \ref{NASH Eq: SDP2}. We remark that the approach in \cite{laraki2012semidefinite} is applicable more generally to many other problems, including several in game theory.
	
	The scalar $\epsilon$ reported in Table \ref{NASH Table: SDP5 Statistics} is computed using the strategies $(x,y)$ extracted from the first row of the optimal matrix $M_1$ as described in Section 4.1 of \cite{laraki2012semidefinite}. The scalar $\epsilon$ reported in Table \ref{NASH Table: Trace Statistics} is computed using $x = P1_n$ and $y = P^T1_m$ from the optimal solution to \ref{NASH Eq: SDP2} with $\Tr(\M)$ as the objective function.
	
	\begin{table}[H]
		\caption{Statistics on $\epsilon$ for first level of the hierarchy in \cite{laraki2012semidefinite}.\label{NASH Table: SDP5 Statistics}}
		{\begin{tabular}{|c|c|c|c|c|c|c|c|c|c|c|}\hline
				& $5 \times 5$& $10 \times 5$ & $10 \times 10$ & $15 \times 10$ & $15 \times 15$ & $20 \times 15$ & $20 \times 20$ \\\hline
				Max &	0.3357 &	0.3304 &	0.2557 &	0.2189 &	0.1987 &	0.1837 &	0.1828 \\\hline
				Mean &	0.1883 &	0.1889 &	0.1513 &	0.1446 &	0.1262 &	0.1217 &	0.1087 \\\hline
				Median &	0.1803 &	0.1865 &	0.1452 &	0.1418 &	0.1271 &	0.1208 &	0.1070 \\\hline
		\end{tabular}}
		\centering
	\end{table}
	
	\begin{table}[H]
		\caption{Statistics on $\epsilon$ for~\ref{NASH Eq: SDP2} with $\Tr(M)$ as the objective function.\label{NASH Table: Trace Statistics}}
		{\begin{tabular}{|c|c|c|c|c|c|c|c|c|c|c|}\hline
				& $5 \times 5$ & $10 \times 5$ & $10 \times 10$ & $15 \times 10$ & $15 \times 15$ & $20 \times 15$ & $20 \times 20$ \\\hline
				Max &	0.1581 &	0.1589 &	0.115 &	0.1335 &	0.0878 &	0.082 &	0.0619	\\\hline
				Mean &	0.0219 &	0.0332 &	0.0405 &	0.04 &	0.0366 &	0.0356 &	0.0298 	\\\hline
				Median &	0.0046 &	0.0233 &	0.036 &	0.0346 &	0.0345 &	0.0325 &	0.0293 	\\\hline
		\end{tabular}}
		\centering
	\end{table}
	
	We also ran the second level of the hierarchy in \cite{laraki2012semidefinite} on the same 100 $5 \times 5$ games. The maximum $\epsilon$ observed was .3362, while the mean was .1880 and the median was .1800. The size of the variable matrix that needs to be positive semidefinite for this level is $78 \times 78$.
	
	\section{Lemmas for Extreme Nash Equilibria}\label{NASH Sec: Lemmas for Extreme Nash Equilibria}
	The results reported in Section~\ref{NASH Sec: Bounding Payoffs and Strategy Exclusion} were found using the \texttt{lrsnash}~\cite{avis2010enumeration} software which computes extreme Nash equilibria (see definition below). In particular the true maximum welfare and the persistent strategy sets were found in relation to extreme symmetric Nash equilibria only. We show in this appendix why this is sufficient for the claims we made about \emph{all} symmetric Nash equilibria. We prove a more general statement below about general games and general Nash equilibria since this could be of potential independent interest. The proof for symmetric games is identical once the strategies considered are restricted to be symmetric.
	\begin{defn}
		An \emph{extreme Nash equilibrium} is a Nash equilibrium which cannot be expressed as a convex combination of other Nash equilibria.
	\end{defn}
	
	\begin{lem}\label{NASH lem: convex combination ENE}
		All Nash equilibria are convex combinations of extreme Nash equilibria.
	\end{lem}
	
	\begin{proof}
		{}It suffices to show that any extreme point of the convex hull of the set of Nash equilibria must be an extreme Nash equilibrium, as any point in a compact convex set can be written as a convex combination of its extreme points. Note that this convex hull contains three types of points: extreme Nash equilibria, Nash equilibria which are not extreme, and convex combinations of Nash equilibria which are not Nash equilibria. The claim then follows because any extreme point of the convex hull cannot be of the second or third type, as they can be written as convex combinations of other points in the hull.
	\end{proof}
	
	The next lemma shows that checking extreme Nash equilibria are sufficient for the maximum welfare problem.
	
	\begin{lem}
		For any bimatrix game, there exists an extreme Nash equilibrium giving the maximum welfare among all Nash equilibria.
	\end{lem}
	\begin{proof}
		{}Consider any Nash equilibrium $(\tilde{x}, \tilde{y})$, and let it be written as $\bmat \tilde{x} \\ \tilde{y} \emat = \sum_{i=1}^r \lambda_i \bmat x^i \\ y^i \emat$ for some set of extreme Nash equilibria $\bmat x^1 \\ y^1 \emat, \ldots, \bmat x^r \\ y^r \emat$ and $\lambda \in \triangle_r$. Observe that for any $i,j$, \begin{equation}\label{NASH eq: defn of NE}x^{iT}Ay^{j} \le x^{jT}Ay^{j}, x^{iT}By^{j} \le x^{iT}By^{i},\end{equation} from the definition of a Nash equilibrium. Now note that
		
		\begin{equation*}
		\begin{aligned}
		\tilde{x}^T(A+B)\tilde{y} & = (\sum_{i=1}^r \lambda_i x^i)^T(A+B)(\sum_{i=1}^r \lambda_i y^i)\\
		&=\sum_{i=1}^r \sum_{j=1}^r \lambda_i\lambda_j x^{iT}(A+B)y^j\\
		&= \sum_{i=1}^r \sum_{j=1}^r \lambda_i\lambda_j x^{iT}Ay^j + \sum_{i=1}^r \sum_{j=1}^r \lambda_i\lambda_j x^{iT}By^j\\
		& \overset{(\ref{NASH eq: defn of NE})}{\le} \sum_{i=1}^r \sum_{j=1}^r \lambda_i\lambda_j x^{jT}Ay^j + \sum_{i=1}^r \sum_{j=1}^r \lambda_i\lambda_j x^{iT}By^i\\
		&= \sum_{i=1}^r \lambda_i x^{iT}Ay^i + \sum_{i=1}^r \lambda_i x^{iT}By^i\\
		&= \sum_{i=1}^r \lambda_i x^{iT}(A+B)y^i.
		\end{aligned}
		\end{equation*}
		
		In particular, since each $(x^i, y^i)$ is an extreme Nash equilibrium, this tells us for any Nash equilibrium $(\tilde{x}, \tilde{y})$ there must be an extreme Nash equilibrium which has at least as much welfare.
	\end{proof}
	
	Similarly for the results for persistent sets in Section~\ref{NASH SSec: Strategy Exclusion}, there is no loss in restricting attention to extreme Nash equilibria.
	\begin{lem}
		For a given strategy set $\mathcal{S}$, if every extreme Nash equilibrium plays at least one strategy in $\mathcal{S}$ with positive probability, then every Nash equilibrium plays at least one strategy in $\mathcal{S}$ with positive probability.
	\end{lem}
	
	\begin{proof}
		{}Let $\mathcal{S}$ be a persistent set of strategies. Since all Nash equilibria are composed of nonnegative entries, and every extreme Nash equilibrium has positive probability on some entry in $\mathcal{S}$, any convex combination of extreme Nash equilibria must have positive probability on some entry in $\mathcal{S}$.
	\end{proof}
	
	\section{A Note on Reductions from General Games to Symmetric Games}\label{NASH Sec: Irreducibility}
	
	An anonymous referee asked us if our guarantees for symmetric games transfer over to general games by symmetrization. Indeed, there are reductions in the literature that take a general game, construct a symmetric game from it, and relate the Nash equilibria of the original game to symmetric Nash equilibria of its symmetrized version. In this Appendix, we review two well-known reductions of this type, which are shown in \cite{griesmer1963symmetric} and \cite{jurg1992symmetrization}, and show that the quality of approximate Nash equilibria can differ greatly between the two games. We hope that our examples can be of independent interest.
	
	\subsection{The Reduction of \cite{griesmer1963symmetric}}
	
	Consider a game $(A,B)$ with $A, B > 0$ and a Nash equilibrium $(x^*, y^*)$ of it with payoffs $p_A~\defeq~x^{*T}Ay^*$ and $p_B~\defeq~x^{*T}By^*$. Then the symmetric game $(S_{AB}, S_{AB}^T)$ with
	
	$$S_{AB} \defeq \bmat
	0 & A\\
	B^T & 0
	\emat$$
	admits a symmetric Nash equilibrium in which both players play $\bmat \frac{p_A}{p_A+p_B}x^* \\ \frac{p_B}{p_A+p_B}y^*\emat$. In the reverse direction, any symmetric equilibrium $\left( \bmat x\\ y\emat, \bmat x\\ y\emat\right)$ of $(S_{AB}, S_{AB}^T)$ yields a Nash equilibrium $(\frac{x}{1_m^Tx}, \frac{y}{1_n^Ty})$ to the original game $(A,B)$.
	
	To demonstrate that high-quality approximate Nash equilibria in the symmetrized game can map to low-quality approximate Nash equilibria in the original game, consider the game given by $(A,B) = \left( \bmat \epsilon & 0 \\ 1 & 1\emat, \bmat \epsilon^2 & 0 \\ 0 & 1 \emat \right)$ for some $\epsilon > 0$. The symmetric strategy $$\left( \bmat \frac{1}{1 + \epsilon} \\ 0 \\ \frac{\epsilon}{1+\epsilon} \\ 0 \emat, \bmat \frac{1}{1 + \epsilon} \\ 0 \\ \frac{\epsilon}{1+\epsilon} \\ 0 \emat \right)$$ is an $\epsilon \frac{1-\epsilon}{1+\epsilon}$-NE for $(S_{AB}, S_{AB}^T)$, but the strategy pair $\left(\bmat 1 \\ 0 \emat, \bmat 1 \\ 0 \emat\right)$ is a $(1-\epsilon)$-NE for $(A,B)$.
	
	\subsection{The Reduction of \cite{jurg1992symmetrization}}
	Consider a game $(A,B)$ with $A > 0, B < 0$ and a Nash equilibrium $(x^*, y^*)$ of it with payoffs $p_A~\defeq~x^{*T}Ay^*$ and $p_B~\defeq~x^{*T}By^*$. Then the symmetric game $(S_{AB}, S_{AB}^T)$ with
	
	$$S_{AB} \defeq \bmat
	0_{m \times m} & A & -1_m\\
	B^T & 0_{n \times n} & 1_n\\
	1_m^T & -1_n^T & 0\\
	\emat$$
	admits a symmetric Nash equilibrium in which both players play $$\bmat \frac{x^*}{2 - p_B} \\ \frac{y^*}{2+p_A} \\ 1 - \frac{1}{2-p_B} - \frac{1}{2+p_A}\emat.$$ In the reverse direction, any symmetric equilibrium $\left( \bmat x\\ y\\z \emat, \bmat x\\ y\\z\emat\right)$ of $(S_{AB}, S_{AB}^T)$ yields a Nash equilibrium $(\frac{x}{1_m^Tx}, \frac{y}{1_n^Ty})$ to the original game $(A,B)$. This reduction has some advantages over the previous one (see \cite[Section 1]{jurg1992symmetrization}).
	
	To demonstrate that high-quality approximate Nash equilibria in the new symmetrized game can again map to low-quality approximate Nash equilibria in the original game, consider the game given by $(A,B) = \left( \bmat 0 & 0 \\ 0 & 1\emat, \bmat -1 & -1 \\ 0 & 0 \emat \right)$. Let $\epsilon \in (0,\frac{1}{2})$. The symmetric strategy $$\left( \bmat \epsilon \\ 0 \\ 1 - \epsilon \\ 0 \\ 0\emat, \bmat \epsilon \\ 0 \\ 1 - \epsilon \\ 0 \\ 0\emat \right)$$ is an $\frac{\epsilon}{2}(1-\epsilon)$-NE\footnote{Note that approximation factor is halved since the range of the entries of the payoff matrix in the symmetrized game is $[-1,1]$.} for $(S_{AB}, S_{AB}^T)$, but the strategy pair $\left(\bmat 1 \\ 0 \emat, \bmat 1 \\ 0 \emat\right)$ is a $1$-NE for $(A,B)$.

\singlespacing
\bibliographystyle{plain}

\cleardoublepage
\ifdefined\phantomsection
\phantomsection  
\else
\fi
\addcontentsline{toc}{chapter}{Bibliography}

\bibliography{Thesis_refs}

\begin{thebibliography}{100}

\bibitem{adler2013equivalence}
Ilan Adler.
\newblock The equivalence of linear programs and zero-sum games.
\newblock {\em International Journal of Game Theory}, 42(1):165--177, 2013.

\bibitem{adler2009note}
Ilan Adler, Constantinos Daskalakis, and Christos~H Papadimitriou.
\newblock A note on strictly competitive games.
\newblock In {\em International Workshop on Internet and Network Economics},
  pages 471--474. Springer, 2009.

\bibitem{ahmadi2013np}
A.~A. Ahmadi, A.~Olshevsky, P.~A. Parrilo, and J.~N. Tsitsiklis.
\newblock {NP}-hardness of deciding convexity of quartic polynomials and
  related problems.
\newblock {\em Mathematical Programming}, 137(1-2):453--476, 2013.

\bibitem{ahmadi2015dc}
Amir~Ali Ahmadi and Georgina Hall.
\newblock {DC} decomposition of nonconvex polynomials with algebraic
  techniques.
\newblock {\em Mathematical Programming}, pages 1--26, 2015.

\bibitem{ahmadi2019complexity}
Amir~Ali Ahmadi and Georgina Hall.
\newblock On the complexity of detecting convexity over a box.
\newblock {\em Mathematical Programming}, pages 1--15, 2019.

\bibitem{ahmadi2016some}
Amir~Ali Ahmadi and Anirudha Majumdar.
\newblock Some applications of polynomial optimization in operations research
  and real-time decision making.
\newblock {\em Optimization Letters}, 10(4):709--729, 2016.

\bibitem{cubicpaper}
Amir~Ali Ahmadi and Jeffrey Zhang.
\newblock Complexity aspects of local minima and related notions.
\newblock {\em In Preparation}.

\bibitem{alizadeh1995interior}
Farid Alizadeh.
\newblock Interior point methods in semidefinite programming with applications
  to combinatorial optimization.
\newblock {\em SIAM Journal on Optimization}, 5(1):13--51, 1995.

\bibitem{anandkumar2016efficient}
Animashree Anandkumar and Rong Ge.
\newblock Efficient approaches for escaping higher order saddle points in
  non-convex optimization.
\newblock In {\em Conference on learning theory}, pages 81--102, 2016.

\bibitem{andronov1982solvability}
VG~Andronov, EG~Belousov, and VM~Shironin.
\newblock On solvability of the problem of polynomial programming.
\newblock {\em Izvestija Akadem. Nauk SSSR, Tekhnicheskaja Kibernetika},
  4:194--197, 1982.

\bibitem{aumann1974subjectivity}
Robert~J Aumann.
\newblock Subjectivity and correlation in randomized strategies.
\newblock {\em Journal of Mathematical Economics}, 1(1):67--96, 1974.

\bibitem{avis2010enumeration}
David Avis, Gabriel~D Rosenberg, Rahul Savani, and Bernhard Von~Stengel.
\newblock Enumeration of {N}ash equilibria for two-player games.
\newblock {\em Economic Theory}, 42(1):9--37, 2010.

\bibitem{bajbar2017fast}
Tom{\'a}{\v{s}} Bajbar and S{\"o}nke Behrends.
\newblock How fast do coercive polynomials grow?
\newblock Technical report, Instituts f{\"u}r Numerische und Angewandte
  Mathematik, Georg-August-Universit{\"a}t G{\"o}ttingen, 2017.

\bibitem{bajbar2015coercive}
Tomas Bajbar and Oliver Stein.
\newblock Coercive polynomials and their {N}ewton polytopes.
\newblock {\em SIAM Journal on Optimization}, 25(3):1542--1570, 2015.

\bibitem{bareiss1968sylvester}
Erwin~H Bareiss.
\newblock Sylvester{’}s identity and multistep integer-preserving {G}aussian
  elimination.
\newblock {\em Mathematics of computation}, 22(103):565--578, 1968.

\bibitem{basu2010bounding}
Saugata Basu and Marie-Fran{\c{c}}oise Roy.
\newblock Bounding the radii of balls meeting every connected component of
  semi-algebraic sets.
\newblock {\em Journal of Symbolic Computation}, 45(12):1270--1279, 2010.

\bibitem{belousov2002frank}
Evgeny~G Belousov and Diethard Klatte.
\newblock A {F}rank--{W}olfe type theorem for convex polynomial programs.
\newblock {\em Computational Optimization and Applications}, 22(1):37--48,
  2002.

\bibitem{berman2003completely}
Abraham Berman and Naomi Shaked-Monderer.
\newblock {\em Completely positive matrices}.
\newblock World Scientific, 2003.

\bibitem{bertsekas1999nonlinear}
Dimitri~P Bertsekas.
\newblock {\em Nonlinear {P}rogramming}.
\newblock Athena {S}cientific, 1999.

\bibitem{bertsekas2007set}
Dimitri~P Bertsekas and Paul Tseng.
\newblock Set intersection theorems and existence of optimal solutions.
\newblock {\em Mathematical {P}rogramming}, 110(2):287--314, 2007.

\bibitem{blekherman2012semidefinite}
Grigoriy Blekherman, Pablo~A Parrilo, and Rekha~R Thomas.
\newblock {\em Semidefinite {O}ptimization and {C}onvex {A}lgebraic
  {G}eometry}.
\newblock SIAM, 2012.

\bibitem{boyd2004convex}
Stephen Boyd and Lieven Vandenberghe.
\newblock {\em Convex optimization}.
\newblock Cambridge university press, 2004.

\bibitem{boyd1994linear}
Stephen~P Boyd, Laurent El~Ghaoui, Eric Feron, and Venkataramanan Balakrishnan.
\newblock {\em Linear matrix inequalities in system and control theory},
  volume~15.
\newblock SIAM, 1994.

\bibitem{chen2006settling}
Xi~Chen and Xiaotie Deng.
\newblock Settling the complexity of two-player {N}ash equilibrium.
\newblock In {\em FOCS}, volume~6, page 47th, 2006.

\bibitem{chen2006computing}
Xi~Chen, Xiaotie Deng, and Shang-Hua Teng.
\newblock Computing {N}ash equilibria: Approximation and smoothed complexity.
\newblock In {\em 2006 47th Annual IEEE Symposium on Foundations of Computer
  Science (FOCS'06)}, pages 603--612. IEEE, 2006.

\bibitem{cohen1993nonnegative}
Joel~E Cohen and Uriel~G Rothblum.
\newblock Nonnegative ranks, decompositions, and factorizations of nonnegative
  matrices.
\newblock {\em Linear Algebra and its Applications}, 190:149--168, 1993.

\bibitem{conitzer2002complexity}
Vincent Conitzer and Tuomas Sandholm.
\newblock Complexity results about nash equilibria.
\newblock {\em arXiv preprint cs/0205074}, 2002.

\bibitem{conitzer2008new}
Vincent Conitzer and Tuomas Sandholm.
\newblock New complexity results about {N}ash equilibria.
\newblock {\em Games and Economic Behavior}, 63(2):621--641, 2008.

\bibitem{dantzig1951proof}
George~B Dantzig.
\newblock A proof of the equivalence of the programming problem and the game
  problem.
\newblock {\em Activity analysis of production and allocation}, 13:330--338,
  1951.

\bibitem{daskalakis2013complexity}
Constantinos Daskalakis.
\newblock On the complexity of approximating a {N}ash equilibrium.
\newblock {\em ACM Transactions on Algorithms (TALG)}, 9(3):23, 2013.

\bibitem{daskalakis2009complexity}
Constantinos Daskalakis, Paul~W Goldberg, and Christos~H Papadimitriou.
\newblock The complexity of computing a {N}ash equilibrium.
\newblock {\em SIAM Journal on Computing}, 39(1):195--259, 2009.

\bibitem{daskalakis2006note}
Constantinos Daskalakis, Aranyak Mehta, and Christos Papadimitriou.
\newblock A note on approximate {N}ash equilibria.
\newblock In {\em International Workshop on Internet and Network Economics},
  pages 297--306. Springer, 2006.

\bibitem{daskalakis2007progress}
Constantinos Daskalakis, Aranyak Mehta, and Christos Papadimitriou.
\newblock Progress in approximate {N}ash equilibria.
\newblock In {\em Proceedings of the 8th ACM conference on Electronic
  commerce}, pages 355--358. ACM, 2007.

\bibitem{de2006aspects}
Etienne De~Klerk.
\newblock {\em Aspects of semidefinite programming: interior point algorithms
  and selected applications}, volume~65.
\newblock Springer Science \& Business Media, 2006.

\bibitem{de2002approximation}
Etienne De~Klerk and Dmitrii~V Pasechnik.
\newblock Approximation of the stability number of a graph via copositive
  programming.
\newblock {\em SIAM Journal on Optimization}, 12(4):875--892, 2002.

\bibitem{fazel2002matrix}
Maryam Fazel.
\newblock {\em Matrix rank minimization with applications}.
\newblock PhD thesis, PhD thesis, Stanford University, 2002.

\bibitem{fearnley2016approximate}
John Fearnley, Paul~W Goldberg, Rahul Savani, and Troels~Bjerre S{\o}rensen.
\newblock Approximate well-supported {N}ash equilibria below two-thirds.
\newblock {\em Algorithmica}, 76(2):297--319, 2016.

\bibitem{frank1956algorithm}
Marguerite Frank and Philip Wolfe.
\newblock An algorithm for quadratic programming.
\newblock {\em Naval {R}esearch {L}ogistics (NRL)}, 3(1-2):95--110, 1956.

\bibitem{garey2002computers}
Michael~R Garey and David~S Johnson.
\newblock {\em Computers and {I}ntractability}, volume~29.
\newblock {WH F}reeman New York, 2002.

\bibitem{gilboa1989nash}
Itzhak Gilboa and Eitan Zemel.
\newblock {N}ash and correlated equilibria: Some complexity considerations.
\newblock {\em Games and Economic Behavior}, 1(1):80--93, 1989.

\bibitem{goemans1995improved}
Michel~X Goemans and David~P Williamson.
\newblock Improved approximation algorithms for maximum cut and satisfiability
  problems using semidefinite programming.
\newblock {\em Journal of the ACM (JACM)}, 42(6):1115--1145, 1995.

\bibitem{gorin1961asymptotic}
Evgenii~Alekseevich Gorin.
\newblock Asymptotic properties of polynomials and algebraic functions of
  several variables.
\newblock {\em Russian mathematical surveys}, 16(1):93--119, 1961.

\bibitem{greuet2011deciding}
Aur{\'e}lien Greuet and Mohab Safey El~Din.
\newblock Deciding reachability of the infimum of a multivariate polynomial.
\newblock In {\em Proceedings of the 36th {I}nternational {S}ymposium on
  {S}ymbolic and {A}lgebraic {C}omputation}, pages 131--138. ACM, 2011.

\bibitem{greuet2014probabilistic}
Aur{\'e}lien Greuet and Mohab Safey El~Din.
\newblock Probabilistic algorithm for polynomial optimization over a real
  algebraic set.
\newblock {\em SIAM Journal on Optimization}, 24(3):1313--1343, 2014.

\bibitem{griesmer1963symmetric}
JH~Griesmer, AJ~Hoffman, and A~Robinson.
\newblock On symmetric bimatrix games.
\newblock {\em IBM Research Paper RC-959. IBM Corp, Thomas J Watson Research
  Center, Yorktown Heights, New York}, 1963.

\bibitem{grotschel2012geometric}
Martin Gr{\"o}tschel, L{\'a}szl{\'o} Lov{\'a}sz, and Alexander Schrijver.
\newblock {\em Geometric Algorithms and Combinatorial Optimization}, volume~2.
\newblock Springer Science \& Business Media, 2012.

\bibitem{helton2010semidefinite}
J~William Helton and Jiawang Nie.
\newblock Semidefinite representation of convex sets.
\newblock {\em Mathematical Programming}, 122(1):21--64, 2010.

\bibitem{hemaspaandra2012sigact}
Lane~A Hemaspaandra and Ryan Williams.
\newblock {SIGACT} {N}ews {C}omplexity {T}heory {C}olumn 76: An atypical survey
  of typical-case heuristic algorithms.
\newblock {\em ACM SIGACT News}, 43(4):70--89, 2012.

\bibitem{hilbert1888darstellung}
David Hilbert.
\newblock {\"U}ber die darstellung definiter formen als summe von
  formenquadraten.
\newblock {\em Mathematische Annalen}, 32(3):342--350, 1888.

\bibitem{horn2012matrix}
Roger~A Horn and Charles~R Johnson.
\newblock {\em Matrix analysis}.
\newblock Cambridge university press, 2012.

\bibitem{huneault1991survey}
M~Huneault and FD~Galiana.
\newblock A survey of the optimal power flow literature.
\newblock {\em IEEE Transactions on Power Systems}, 6(2):762--770, 1991.

\bibitem{ibaraki2001rank}
Soichi Ibaraki and Masayoshi Tomizuka.
\newblock Rank minimization approach for solving {BMI} problems with random
  search.
\newblock In {\em American Control Conference, 2001. Proceedings of the 2001},
  volume~3, pages 1870--1875. IEEE, 2001.

\bibitem{jeyakumar2014polynomial}
Vaithilingam Jeyakumar, Jean~B Lasserre, and Guoyin Li.
\newblock On polynomial optimization over non-compact semi-algebraic sets.
\newblock {\em Journal of Optimization Theory and Applications},
  163(3):707--718, 2014.

\bibitem{jurg1992symmetrization}
AP~Jurg, MJM Jansen, Jos~AM Potters, and SH~Tijs.
\newblock A symmetrization for finite two-person games.
\newblock {\em Zeitschrift f{\"u}r Operations Research}, 36(2):111--123, 1992.

\bibitem{kalofolias2012computing}
Vassilis Kalofolias and Efstratios Gallopoulos.
\newblock Computing symmetric nonnegative rank factorizations.
\newblock {\em Linear Algebra and its Applications}, 436(2):421--435, 2012.

\bibitem{karmarkar1984new}
Narendra Karmarkar.
\newblock A new polynomial-time algorithm for linear programming.
\newblock In {\em Proceedings of the Sixteenth Annual ACM Symposium on Theory
  of Computing}, pages 302--311, 1984.

\bibitem{khachiyan1979polynomial}
Leonid~Genrikhovich Khachiyan.
\newblock A polynomial algorithm in linear programming.
\newblock In {\em Doklady Akademii Nauk}, volume 244, pages 1093--1096. Russian
  Academy of Sciences, 1979.

\bibitem{kontogiannis2006polynomial}
Spyros~C Kontogiannis, Panagiota~N Panagopoulou, and Paul~G Spirakis.
\newblock Polynomial algorithms for approximating {N}ash equilibria of bimatrix
  games.
\newblock In {\em International Workshop on Internet and Network Economics},
  pages 286--296. Springer, 2006.

\bibitem{krivine1964anneaux}
Jean-Louis Krivine.
\newblock Anneaux pr{\'e}ordonn{\'e}s.
\newblock {\em Journal d'analyse math{\'e}matique}, 12(1):307--326, 1964.

\bibitem{laraki2012semidefinite}
Rida Laraki and Jean~B Lasserre.
\newblock Semidefinite programming for min--max problems and games.
\newblock {\em Mathematical {P}rogramming}, 131(1-2):305--332, 2012.

\bibitem{lasserre2001global}
Jean~B Lasserre.
\newblock Global optimization with polynomials and the problem of moments.
\newblock {\em SIAM Journal on Optimization}, 11(3):796--817, 2001.

\bibitem{lasserre2009convex}
Jean~B Lasserre.
\newblock Convex sets with semidefinite representation.
\newblock {\em Mathematical programming}, 120(2):457--477, 2009.

\bibitem{laurent2009sums}
Monique Laurent.
\newblock Sums of squares, moment matrices and optimization over polynomials.
\newblock In {\em Emerging Applications of Algebraic Geometry}, pages 157--270.
  Springer, 2009.

\bibitem{lemke1964equilibrium}
Carlton~E Lemke and Joseph~T Howson, Jr.
\newblock Equilibrium points of bimatrix games.
\newblock {\em Journal of the Society for Industrial and Applied Mathematics},
  12(2):413--423, 1964.

\bibitem{loiola2007survey}
Eliane~Maria Loiola, Nair Maria~Maia de~Abreu, Paulo~Oswaldo Boaventura-Netto,
  Peter Hahn, and Tania Querido.
\newblock A survey for the quadratic assignment problem.
\newblock {\em European journal of operational research}, 176(2):657--690,
  2007.

\bibitem{lombardi2014elementary}
Henri Lombardi, Daniel Perrucci, and Marie-Fran{\c{c}}oise Roy.
\newblock An elementary recursive bound for effective {P}ositivstellensatz and
  {H}ilbert's 17th problem.
\newblock {\em Available at arXiv:1404.2338}, 2014.

\bibitem{lovasz1979shannon}
L{\'a}szl{\'o} Lov{\'a}sz.
\newblock On the shannon capacity of a graph.
\newblock {\em IEEE Transactions on Information theory}, 25(1):1--7, 1979.

\bibitem{luo1999extensions}
Z.-Q. Luo and S.~Zhang.
\newblock On extensions of the {F}rank-{W}olfe theorems.
\newblock {\em Computational Optimization and Applications}, 13(1-3):87--110,
  1999.

\bibitem{marshall2003optimization}
Murray Marshall.
\newblock Optimization of polynomial functions.
\newblock {\em Canadian {M}athematical {B}ulletin}, 46(4):575--587, 2003.

\bibitem{mccormick1976computability}
Garth~P McCormick.
\newblock Computability of global solutions to factorable nonconvex programs:
  Part {I} : {C}onvex underestimating problems.
\newblock {\em Mathematical {P}rogramming}, 10(1):147--175, 1976.

\bibitem{meyer2000matrix}
Carl~D Meyer.
\newblock {\em Matrix analysis and applied linear algebra}, volume~2.
\newblock SIAM, 2000.

\bibitem{more1990solution}
Jorge~J Mor{\'e} and Stephen~A Vavasis.
\newblock On the solution of concave knapsack problems.
\newblock {\em Mathematical programming}, 49(1-3):397--411, 1990.

\bibitem{mosek}
MOSEK.
\newblock {\em {MOSEK} reference manual}, 2013.
\newblock Version 7. Latest version available at http://www.mosek.com/.

\bibitem{motzkin1965maxima}
Theodore~S Motzkin and Ernst~G Straus.
\newblock Maxima for graphs and a new proof of a theorem of tur{\'a}n.
\newblock {\em Canadian Journal of Mathematics}, 17:533--540, 1965.

\bibitem{murty1987some}
Katta~G Murty and Santosh~N Kabadi.
\newblock Some {NP}-complete problems in quadratic and nonlinear programming.
\newblock {\em Mathematical {P}rogramming}, 39(2):117--129, 1987.

\bibitem{murty1988linear}
Katta~G Murty and Feng-Tien Yu.
\newblock {\em Linear Complementarity, Linear and Nonlinear Programming},
  volume~3.
\newblock Berlin: Heldermann, 1988.

\bibitem{nash1951}
John {N}ash.
\newblock Non-cooperative games.
\newblock {\em Annals of mathematics}, pages 286--295, 1951.

\bibitem{nesterov2019implementable}
Yurii Nesterov.
\newblock Implementable tensor methods in unconstrained convex optimization.
\newblock {\em Mathematical Programming}, pages 1--27, 2019.

\bibitem{nesterov2003random}
Yurii Nesterov et~al.
\newblock Random walk in a simplex and quadratic optimization over convex
  polytopes.
\newblock Technical report, CORE, 2003.

\bibitem{nesterov1994interior}
Yurii Nesterov and Arkadii Nemirovskii.
\newblock {\em Interior-Point Polynomial Algorithms in Convex Programming},
  volume~13.
\newblock SIAM, 1994.

\bibitem{netzer2010semidefinite}
Tim Netzer.
\newblock On semidefinite representations of non-closed sets.
\newblock {\em Linear algebra and its applications}, 432(12):3072--3078, 2010.

\bibitem{nie2014optimality}
Jiawang Nie.
\newblock Optimality conditions and finite convergence of the {Lasserre}
  hierarchy.
\newblock {\em Mathematical programming}, 146(1-2):97--121, 2014.

\bibitem{nie2015hierarchy}
Jiawang Nie.
\newblock The hierarchy of local minimums in polynomial optimization.
\newblock {\em Mathematical Programming}, 151(2):555--583, 2015.

\bibitem{nie2006minimizing}
Jiawang Nie, James Demmel, and Bernd Sturmfels.
\newblock Minimizing polynomials via sum of squares over the gradient ideal.
\newblock {\em Mathematical {P}rogramming}, 106(3):587--606, 2006.

\bibitem{pardalos1988checking}
Panos~M Pardalos and Georg Schnitger.
\newblock Checking local optimality in constrained quadratic programming is
  np-hard.
\newblock {\em Operations Research Letters}, 7(1):33--35, 1988.

\bibitem{pardalos1991quadratic}
Panos~M Pardalos and Stephen~A Vavasis.
\newblock Quadratic programming with one negative eigenvalue is np-hard.
\newblock {\em Journal of Global optimization}, 1(1):15--22, 1991.

\bibitem{pardalos1992open}
Panos~M Pardalos and Stephen~A Vavasis.
\newblock Open questions in complexity theory for numerical optimization.
\newblock {\em Mathematical Programming}, 57(1-3):337--339, 1992.

\bibitem{parrilo2003semidefinite}
Pablo~A Parrilo.
\newblock Semidefinite programming relaxations for semialgebraic problems.
\newblock {\em Mathematical {P}rogramming}, 96(2):293--320, 2003.

\bibitem{parrilo2006polynomial}
Pablo~A Parrilo.
\newblock Polynomial games and sum of squares optimization.
\newblock In {\em Proceedings of the 45th IEEE Conference on Decision and
  Control}, pages 2855--2860. IEEE, 2006.

\bibitem{porkolab1997complexity}
L.~Porkolab and L.~Khachiyan.
\newblock On the complexity of semidefinite programs.
\newblock {\em Journal of Global Optimization}, 10(4):351--365, 1997.

\bibitem{prestel2001positive}
Alexander Prestel and Charles~N Delzell.
\newblock Positive {P}olynomials: from {H}ilbert's 17th {P}roblem to {R}eal
  {A}lgebra.
\newblock {\em Springer Monographs in Mathematics. Springer, Berlin, Germany},
  2001.

\bibitem{putinar1993positive}
Mihai Putinar.
\newblock Positive polynomials on compact semi-algebraic sets.
\newblock {\em Indiana University Mathematics Journal}, 42(3):969--984, 1993.

\bibitem{ramana1995some}
Motakuri Ramana and AJ~Goldman.
\newblock Some geometric results in semidefinite programming.
\newblock {\em Journal of Global Optimization}, 7(1):33--50, 1995.

\bibitem{ramana1997exact}
Motakuri~V Ramana.
\newblock An exact duality theory for semidefinite programming and its
  complexity implications.
\newblock {\em Mathematical Programming}, 77(1):129--162, 1997.

\bibitem{ramana1993algorithmic}
Motakuri~Venkata Ramana.
\newblock {\em An algorithmic analysis of multiquadratic and semidefinite
  programming problems}.
\newblock PhD thesis, Citeseer, 1993.

\bibitem{recht2010guaranteed}
Benjamin Recht, Maryam Fazel, and Pablo~A Parrilo.
\newblock Guaranteed minimum-rank solutions of linear matrix equations via
  nuclear norm minimization.
\newblock {\em SIAM {R}eview}, 52(3):471--501, 2010.

\bibitem{rockafellar1970convex}
R~Tyrrell Rockafellar.
\newblock {\em Convex {A}nalysis}, volume~28.
\newblock Princeton {U}niversity {P}ress, 1970.

\bibitem{savani2006hard}
Rahul Savani and Bernhard Stengel.
\newblock Hard-to-solve bimatrix games.
\newblock {\em Econometrica}, 74(2):397--429, 2006.

\bibitem{schaefer1978complexity}
Thomas~J Schaefer.
\newblock The complexity of satisfiability problems.
\newblock In {\em Proceedings of the tenth annual {ACM} symposium on Theory of
  computing}, pages 216--226. ACM, 1978.

\bibitem{seidenberg1954new}
A.~Seidenberg.
\newblock A new decision method for elementary algebra.
\newblock {\em Annals of Mathematics}, pages 365--374, 1954.

\bibitem{shah2007polynomial}
Parikshit Shah and Pablo~A Parrilo.
\newblock Polynomial stochastic games via sum of squares optimization.
\newblock In {\em Decision and Control, 2007 46th IEEE Conference on}, pages
  745--750. IEEE, 2007.

\bibitem{stein1943exchangeable}
Noah~D Stein.
\newblock {\em Exchangeable equilibria}.
\newblock PhD thesis, Massachusetts Institute of Technology, 2011.

\bibitem{stengle1974nullstellensatz}
Gilbert Stengle.
\newblock A {N}ullstellensatz and a {P}ositivstellensatz in semialgebraic
  geometry.
\newblock {\em Mathematische Annalen}, 207(2):87--97, 1974.

\bibitem{suykens1999least}
Johan~AK Suykens and Joos Vandewalle.
\newblock Least squares support vector machine classifiers.
\newblock {\em Neural processing letters}, 9(3):293--300, 1999.

\bibitem{tarski1951decision}
Alfred Tarski.
\newblock A decision method for elementary algebra and geometry.
\newblock In {\em Quantifier {E}limination and {C}ylindrical {A}lgebraic
  {D}ecomposition}, pages 24--84. Springer, 1998.

\bibitem{tibshirani1996regression}
Robert Tibshirani.
\newblock Regression shrinkage and selection via the lasso.
\newblock {\em Journal of the Royal Statistical Society: Series B
  (Methodological)}, 58(1):267--288, 1996.

\bibitem{tsaknakis2007optimization}
Haralampos Tsaknakis and Paul~G Spirakis.
\newblock An optimization approach for approximate {N}ash equilibria.
\newblock In {\em International Workshop on Web and Internet Economics}, pages
  42--56. Springer, 2007.

\bibitem{vandenberghe1996semidefinite}
Lieven Vandenberghe and Stephen Boyd.
\newblock Semidefinite programming.
\newblock {\em SIAM {R}eview}, 38(1):49--95, 1996.

\bibitem{vavasis1990quadratic}
Stephen~A Vavasis.
\newblock Quadratic programming is in np.
\newblock {\em Information Processing Letters}, 36(2):73--77, 1990.

\bibitem{vinzant2014spectrahedron}
Cynthia Vinzant.
\newblock What is... a spectrahedron?
\newblock {\em Notices Amer. Math. Soc}, 61(5):492--494, 2014.

\bibitem{wagner2009archimedean}
Sven Wagner.
\newblock {\em Archimedean quadratic modules: a decision problem for real
  multivariate polynomials}.
\newblock PhD thesis, University of Konstanz, 2009.

\end{thebibliography}

\end{document}